\documentclass[12pt,letterpaper]{report}
\usepackage{natbib}
\usepackage{geometry}
\usepackage{fancyhdr}
\usepackage{afterpage}
\usepackage{graphicx}
\usepackage{epstopdf}
\usepackage{multirow}
\usepackage{subfigure}
\usepackage{amsmath,amssymb,amsbsy,amsthm}
\usepackage{dcolumn,array}
\usepackage{tocloft}
\usepackage{arydshln}
\usepackage[boxed]{algorithm2e}
\usepackage{url}
\usepackage{asudis}
\usepackage{pbox}

\usepackage{slashbox,booktabs,amsmath}

\newtheorem{mydef}{Definition}

\newtheorem{mylem}{Lemma}

\newtheorem{mythm}{Theorem}
\newtheorem{myass}{Assumption}

\DeclareMathOperator*{\argmin}{arg\,min}

\begin{document}
\pagenumbering{roman}
\title{Parallel Optimization of Polynomials for Large-scale Problems in \\ Stability and Control}
\author{Reza Kamyar}
\degreeName{Doctor of Philosophy}
\defensemonth{January}
\gradmonth{May}
\gradyear{2016}
\chair{Matthew Peet}
\memberOne{Daniel Rivera}
\memberTwo{Georgios Fainekos}
\memberThree{Panagiotis Artemiadis}
\memberFour{Spring Berman}	
\maketitle
\doublespace

\begin{abstract}

In today's world, optimal operation of ever-growing industries and markets often requires solving optimization problems with unprecedented sizes. Economic dispatch of generating units in power companies, frequency assignment in large mobile communication networks, profit maximization in competitive markets, and optimal operation of smart grids are few examples of many real-world problems which can be closely modeled as optimization over a large number (tens of thousands) of integer- and real-valued decision variables. Unfortunately, majority of the existing commercial off-the-shelf software are not designed to scale to optimization problems of this size. Moreover, in theory, these optimization problems often fall into the class NP-hard - meaning that despite the tremendous effort towards modernization of optimization algorithms, it is widely suspected that no algorithm can find exact solutions to these problems in a reasonable amount of time.


In this thesis, we focus on some of the NP-hard problems in control theory. Thanks to the converse Lyapunov theory, these problems can often be modeled as optimization over polynomials. To avoid the problem of intractability, we establish a trade off between accuracy and complexity. In particular, we develop a sequence of tractable optimization problems - in the form of Linear Programs (LPs) and/or Semi-Definite Programs (SDPs) - whose solutions converge to the exact solution of the NP-hard problem. However, the computational and memory complexity of these LPs and SDPs grow exponentially with the progress of the sequence - meaning that improving the accuracy of the solutions requires solving SDPs with
tens of thousands of decision variables and constraints. Setting up and solving such problems is a significant challenge. Unfortunately, the existing optimization algorithms and software are only designed to use desktop computers or small cluster computers - machines which do not have sufficient memory for solving such large SDPs. Moreover, the speed-up of these algorithms does not scale beyond dozens of processors. This in fact is the reason we seek parallel algorithms for setting-up and solving large SDPs on large cluster- and/or super-computers.


We propose parallel algorithms for stability analysis of two classes of systems: 1) Linear systems with a large number of uncertain parameters; 2) Nonlinear systems defined by polynomial vector fields. First, we develop a distributed parallel algorithm which applies Polya's and/or Handelman's theorems to some variants of parameter-dependent Lyapunov inequalities with parameters defined over the standard simplex. The result is a sequence of  SDPs which possess a block-diagonal structure. We then develop a parallel SDP solver which exploits this structure in order to map the computation, memory and communication to a distributed parallel environment. We produce a Message Passing Interface (MPI) implementation of our parallel algorithms and provide a comprehensive theoretical and experimental analysis on its complexity and scalability. Numerical tests on a supercomputer demonstrate the ability of the algorithm to efficiently utilize hundreds and potentially thousands of processors and analyze systems with 100+ dimensional state-space. We then apply our algorithms to two real-world problems: Stability of plasma in a Tokamak reactor, and optimal electricity pricing in a smart grid environment. Finally, we extend our algorithms to analyze robust stability over more complicated geometries such as hypercubes and arbitrary convex polytopes. Our algorithms can be readily extended to address a wide variety of problems in control; e.g., $\mathcal{H}_2/\mathcal{H}_{\infty}$ control synthesis for systems with parametric uncertainty, computing control Lyapunov functions for optimal control problems, and analysis and control of switched/hybrid systems.
\end{abstract}

\tableofcontents
\addtocontents{toc}{~\hfill Page\par}
\newpage

\addcontentsline{toc}{part}{LIST OF TABLES}
\renewcommand{\cftlabel}{Table}
\addtocontents{lot}{Table~\hfill Page \par}
\listoftables

\newpage

\addcontentsline{toc}{part}{LIST OF FIGURES}
\renewcommand{\cftlabel}{Figure}
\addtocontents{lof}{Figure~\hfill Page \par}
\listoffigures

\newpage

\addtocontents{toc}{CHAPTER \par}


\doublespace
\pagenumbering{arabic}

\chapter{INTRODUCTION}
\label{chp:introduction}

Consider problems such as portfolio optimization, path-planning, structural design, local stability of nonlinear ordinary differential equations, control of time-delay systems and control of systems with uncertainties. These problems can all be formulated as \textit{polynomial optimization} and/or \textit{optimization of polynomials}. In this dissertation, we show how computation can be applied in a variety of ways to solve these classes of problems.
A simple example of polynomial optimization is $\beta^* = \min\limits_{x \in Q} p(x)$, where $p: \mathbb{R}^n \rightarrow \mathbb{R}$ is a multi-variate polynomial and $Q \subset \mathbb{R}^n$. In general, since $p(x)$ and $Q$ are not convex, this is not a convex optimization problem. In fact, it has been proved that polynomial optimization is NP-hard~(\cite{blum}). Fortunately, algorithms such as branch-and-bound can find arbitrarily precise solutions to polynomial optimization problems by repeatedly partitioning $Q$ into subsets $Q_i$ and computing lower and upper bounds on $p(x)$ over each $Q_i$. To find an upper bound for $p(x)$ over each $Q_i$, one could use a local optimization algorithm such as sequential quadratic programming. To find a lower bound on $p(x)$ over each $Q_i$, one can solve the following optimization problem.
\begin{align}
& \quad \beta^* =\max_{y \in \mathbb{R}} \; y \nonumber \\
&\text{subject to } \;\; p(x) - y \geq 0 \text{ for all } x \in Q_i.
\label{optim_poly_exp}
\end{align}
This problem is in fact an instance of the problem of optimization of polynomials. Optimization of polynomials is convex, yet again NP-hard. We will discuss optimization of polynomials in more depth in Chapter~\ref{chp:background}. In the following, we discuss some of the state-of-the-art methods for solving optimization of polynomials - hence finding lower bounds on $\beta^*$.

\section{Sum of Squares Method} 
\label{sec:intro_SOS}

One approach to find lower bounds on the optimal objective $\beta^*$ is to apply \textit{Sum of Squares (SOS) programming}~(\cite{parillo_thesis}, \cite{sostools2013}). A polynomial $p$ is SOS if there exist polynomials $q_i$ such that $p(x)=\sum_{i=1}^r q_i(x)^2$.
The set $\{ q_i \in \mathbb{R}[x], i=1, \cdots,r \} $ is called an SOS \textit{decomposition} of $p(x)$, where $\mathbb{R}[x]$ is the ring of real polynomials. An SOS program is an optimization problem of the form
\begin{align}
& \quad \min_{x \in \mathbb{R}^m} \qquad c^T x \nonumber \\
& \text{subject to} \quad  A_{i,0}(y)+\sum_{j=1}^m x_j A_{i,j}(y) \text{ is SOS, } i=1, \cdots, k,
\label{eq:SOS_prog}
\end{align}
where $c \in \mathbb{R}^m$ and $A_{i,j} \in \mathbb{R}[y]$ are given. If $p(x)$ is SOS, then clearly $p(x) \geq 0$ on $\mathbb{R}^n$. While verifying $p(x) \geq 0$ on $\mathbb{R}^n$ is NP-hard, verifying whether $p(x)$ is SOS - hence non-negative - can be done in polynomial time (\cite{parillo_thesis}). It was first shown in~\cite{parillo_thesis} that verifying the existence of a SOS decomposition is a Semi-Definite Program (SDP). Fortunately, there exist several algorithms~(\cite{monteiro, helmberg, alizadeh}) and solvers~(\cite{sdpa,sedumi,sdpt3}) which solve SDPs to arbitrary precision in polynomial time. To find lower bounds on $\beta^* = \min_{x \in \mathbb{R}^n} p(x)$, consider the SOS program
\begin{align*}
& y^*=\max_{y \in \mathbb{R}} \quad  y \\
& \text{ subject to } \; p(x) - y \text{ is SOS}.
\end{align*}
Clearly $y^* \leq \beta^*$. One can compute $y^*$ by performing a bisection search on $y$ and using semi-definite programming to verify $p(x) - y$ is SOS. SOS programming can also be used to find lower bounds on global minimum of polynomials over a semi-algebraic set $S:=\{x \in \mathbb{R}^n: g_i(x) \geq 0, \, i = 1, \cdots, m \}$ generated by $g_i \in \mathbb{R}[x]$. Given Problem~\eqref{optim_poly_exp} with $x \in S$, \textit{Positivstellensatz} results~(\cite{stengle}, \cite{putinar}, \cite{schmudgen}) define a sequence of SOS programs whose objective values form a sequence of lower bounds on the global minimum $\beta^*$. For instance, Putinar's Positivstellensatz defines the optimization problem
\begin{align}
& 
y_d := \max_{y \in \mathbb{R}} \quad y \nonumber  \\ 
& \text{subject to } \quad p(x) - y = s_0(x) + \sum_{i=1}^m s_i(x) g_i(x), \; s_i \in \Sigma_{2d},
\label{eq:SOS_putinar_exp}
\end{align} 
where $\Sigma_{2d}$ denotes the cone of SOS polynomials of degree $2d$. \cite{putinar} has shown that under certain conditions (verifiable by semi-definite programming) on $S$ and for sufficiently large $d$, $y_d = \beta^*$. See~\cite{laurent_survey} for a comprehensive discussion on the Positivstellensatz results. 

\section{Moments Method}

As a dual to SOS program,~\cite{lasserre2001global} used the theory of \textit{moments} to define a sequence of lower bounds for global optima of polynomials. Let $\beta^*:= \min_{x \in S} \, p(x)$, where $S:=\{ x \in \mathbb{R}^n: g_i(x) \geq 0, i= 1, \cdots,m\}$ is compact and $p(x):= \sum\limits_{\alpha \in W_p} p_\alpha x^\alpha$ with the index set $W_p:=\{ \alpha \in \mathbb{N}^n: \Vert \alpha \Vert_1 \leq p \}$. Let us denote the degree of $g_i$ by $e_i$. Then,~\cite{lasserre2001global} showed that $z_d$ defined as
\begin{align}
& z_d := \min_{z} \;  \sum_{\alpha \in W_p} p_\alpha z_\alpha  \nonumber \\
& \text{subject to } \;\,\, M_d(z) \geq 0 \nonumber	 \\
& \qquad\qquad\quad\;  M_{d-e_i}(g_i \, z) \geq 0 \quad \text{ for } i = 1, \cdots, m,
\label{eq:moments}
\end{align}
is a lower bound on $\beta^*$. In Equation \eqref{eq:moments}, $z:= \left\lbrace z_{\alpha} \right\rbrace_{\alpha \in I_{2d}}$, where $ z_\alpha :=  \int_{S} x^\alpha \mu(dx)$ is called the \textit{moment of order} $\alpha$ and is represented by any probability measure\footnote{Let $X$ be a set and $M$ be a $\sigma-$algebra over $X$. Then $\mu: M \rightarrow [0,1]$ is a probability measure if
\begin{enumerate}
\item $\mu(\emptyset) = 0$ and $\mu(X) = 1$.
\item For all countable collections $\{S_i\}_{i \in N}$ of pairwise disjoint subsets of $M$, $\mu(\bigcup\limits_{i \in N} S_i) = \sum\limits_{i \in N} \mu(S_i)$.
\end{enumerate}
}  $\mu$ on $\mathbb{R}^n$ such that $\mu(\mathbb{R}\backslash S) = 0$. Moreover, $M_d(z)$ is called the \textit{moment matrix} associated with sequence $z$ and in two dimensions is defined as
\[
M_d(z) \hspace{-0.05in} = \hspace{-0.05in} \left( \hspace{-0.05in} \begin{array}{c:cc:ccc:c:cccccc} 
1 &  z_{[1,0]} & z_{[0,1]}  & z_{[2,0]} & z_{[1,1]} & z_{[0,2]} &  \cdots & z_{[d,0]} & \cdots & z_{[0,d]}  \\  \hdashline
z_{[1,0]} & z_{[2,0]} & z_{[1,1]} & z_{[3,0]} & z_{[2,1]} & z_{[1,2]} & \cdots & z_{[d+1,0]} & \cdots & z_{[1,d]}\\ 
z_{[0,1]} & z_{[1,1]} & z_{[0,2]} & z_{[2,1]} & z_{[1,2]} & z_{[0,3]} & \cdots & z_{[d,1]} & \cdots  & z_{[0,d+1]} \\  \hdashline
z_{[2,0]} & z_{[3,0]} & z_{[2,1]} & z_{[4,0]} & z_{[3,1]} & z_{[2,2]} & \cdots & z_{[d+2,0]} & \cdots & z_{[2,d]}\\ 
z_{[1,1]} & z_{[2,1]} & z_{[1,2]} & z_{[3,1]} & z_{[2,2]} & z_{[1,3]} & \cdots & z_{[d+1,1]} & \cdots & z_{[1,d+1]}\\ 
z_{[0,2]} & z_{[1,2]} & z_{[0,3]} & z_{[2,2]} & z_{[1,3]} & z_{[0,4]} & \cdots & z_{[d,2]} & \cdots & z_{[0,d+2]} \\ \hdashline
\vdots    & \vdots    & \vdots    & \vdots    & \vdots    &  \vdots   & \ddots & \cdots & \cdots & \cdots\\  \hdashline
z_{[d,0]} & z_{[d+1,0]} & z_{[d,1]}  & z_{[d+2,0]} & z_{[d+1,1]} & z_{[d,2]} & \cdots & z_{[2d,0]} & \cdots & z_{[d,d]} \\ 
\vdots    &  \vdots     &\vdots  & \vdots   & \vdots & \vdots & \vdots & \vdots & \ddots & \vdots\\ 
z_{[0,d]} &   z_{[1,d]} &  z_{[0,d+1]} & z_{[2,d]}  & z_{[1,d+1]} & z_{[0,d+2]} & \cdots & z_{[d,d]} & \cdots & z_{[0,2d]}\\ 
\end{array} \hspace{-0.05in} \right).
\]
It can be shown that the SDPs in~\eqref{eq:moments} are duals to the SDPs in~\eqref{eq:SOS_putinar_exp} - implying that $y_d \leq z_d$. Indeed, if $S$ has a non-empty interior, then for all sufficiently large $d$, the duality gap is zero, i.e., $y_d = z_d$. See~\cite{laurent_survey} and~\cite{lasserre_noncompact2014} for conditions on convergence of the lower bounds to global minima and extension of moments method to polynomial optimization over non-compact semi-algebraic sets.
 
In the sequel, we explore the merits of some of the alternatives to SOS programming and moments method. There exist several results in the literature that can be applied to polynomial optimization; e.g., Quantifier Elimination (QE) algorithms~(\cite{CAD}) for testing the feasibility of semi-algebraic sets, Reformulation Linear Techniques (RLTs)~(\cite{sherali_1992,sherali_1997}) for linearizing polynomial optimizations, Polya's theorem~(\cite{inequalities}) for positivity over the positive orthant, Bernstein's~(\cite{roy,leroy}) and Handelman's~(\cite{handelman_1988}) theorems for positivity over simplices and convex polytopes, and other results based on Groebner bases~(\cite{adams_groebner}) and Blossoming~(\cite{blossoming}) techniques. In particular, we will focus on Polya's, Bernstein's and Handelman's results in more depth and elaborate on the computational advantages of these results over the others. The discussion of the other results are beyond the scope of this dissertation, however the ideas behind these results can be summarized as follows.

\section{Quantifier Elimination}

QE algorithms apply to First-Order Logic formulae, e.g.,
\[
\forall x \, \exists y \, (f(x,y) \geq 0 \Rightarrow ((g(a) < x y) \wedge (a > 2)),
\]
to eliminate the \textit{quantified} variables $x$ and $y$ (preceded by quantifiers $\forall,\exists$) and construct an equivalent formula in terms of the \textit{unquantified} variable $a$. The key result underlying QE algorithms is Tarski-Seidenberg theorem~(\cite{tarski}). The theorem implies that for every formula of the form $\forall x \in \mathbb{R}^n \, \exists y \in \mathbb{R}^m ( f_i(x,y,a) \geq 0 )$, where $f_i \in \mathbb{R}[x,y,a]$, there exists an equivalent quantifier-free formula of the form $\wedge_i (g_i(a) \geq 0) \vee_j (h_j(a) \geq 0)$ with $g_i,h_j \in \mathbb{R}[a]$. QE implementations (e.g.,~\cite{QEPCAD} and~\cite{Redlog}) with a bisection search yields the exact solution to optimization of polynomials, however the complexity scales double exponentially in the dimension of variables $x,y$. 

\section{Reformulation Linear Techniques}

RLT was initially developed to find the convex hull of feasible solutions of zero-one linear programs~(\cite{sherali_1990}). It was later generalized by~\cite{sherali_1992} to address polynomial optimizations of the form $\min_x p(x)$ subject to $x \in [0,1]^n \cap S$. RLT constructs a $\delta-$hierarchy of linear programs by performing two steps. In the first step (reformulation), RLT introduces the new constraints $\prod_i x_i \prod_j (1-x_j) \geq 0$ for all $i,j: i+j=\delta$. In the second step (linearization), RTL defines a linear program by replacing every product of variables $x_i$ by a new variable. By increasing $\delta$ and repeating the two steps, one can construct a $\delta-$hierarchy of lower bounding linear programs. A combination of RLT and branch-and-bound partitioning of $[0,1]^n$ was developed by~\cite{sherali_1997} to achieve tighter lower bounds on the global minimum. For a survey of different extensions of RLT see~\cite{sherali_global_2007}. 

\section{Groebner Basis Technique}

Groebner bases can be used to reduce a polynomial optimization over a semi-algebraic set $S:=\{ x \in \mathbb{R}^n : g_i(x) \geq 0, \, h_j(x) = 0\}$ to the problem of finding the roots of univariate polynomials~(\cite{poly_groebner}). First, one needs to construct the system of polynomial equations
\begin{equation}
[\nabla_x L(x,\lambda,\mu), \nabla_\lambda L(x,\lambda,\mu), \nabla_\mu L(x,\lambda,\mu)] = 0,
\label{eq:nabla_sys}
\end{equation}
 where $L:= p(x)+ \sum_i \lambda_i g_i(x)+ \sum_j \mu_j h_j(x)$ is the Lagrangian function. It is well-known that the set of solutions to~\eqref{eq:nabla_sys} is the set of extrema of the polynomial optimization $\min_{x \in S} p(x)$.
Let
 \[
\left[ f_1(x,\lambda,\mu), \cdots,f_N(x,\lambda,\mu) \right] := \left[\nabla_x L(x,\lambda,\mu), \nabla_\lambda L(x,\lambda,\mu), \nabla_\mu L(x,\lambda,\mu) \right].
 \]
Using the elimination property~(\cite{adams_groebner}) of the Groebner bases, the minimal Groebner basis of the ideal of $f_1, \cdots, f_N$ defines a triangular-form system of polynomial equations. This system can be solved by calculating one variable at a time and back-substituting into other polynomials. The most computationally expensive part is the calculation of the Groebner basis, which in the worst case scales double-exponentially in the number of decision variables. 

\section{Blossoming Technique}

The blossoming technique involves a bijective map between the space of polynomials $p: \mathbb{R}^n \rightarrow \mathbb{R}$ and the space of multi-affine functions $q: \mathbb{R}^{d_1+d_2+ \cdots + d_n} \rightarrow \mathbb{R}$ (polynomials that are affine in each variable), where $d_i$ is the degree of $p$ in variable $x_i$. For instance, the blossom of a cubic polynomial $p(x) = ax^3 + b x^2 + cx + d$ is the multi-affine function
\[
q(z_1,z_2,z_3) = a z_1 z_2 z_3 + \dfrac{b}{3} (z_1z_2 + z_1 z_3 + z_2 z_3) + \dfrac{c}{3} (z_1 + z_2 + z_3) + d.
\]
It can be shown that the blossom, $q$, of any polynomial $p \in \mathbb{R}[x]$ with degree $d_i$ in variable $x_i$  satisfies the so-called \textit{diagonal} property~(\cite{blossoming}), i.e.,
\[
p(z_1,z_2, \cdots, z_n) = q( \underbrace{z_1, \cdots, z_1}_{d_1 \text{ times}}, \cdots, \underbrace{z_n, \cdots, z_n}_{d_n \text{ times}}) \quad \text{ for all } z \in \mathbb{R}. 
\]
By using this property, one can reformulate any polynomial optimization $\min_{x \in S} p(x)$ as 
\begin{align}
&\qquad \; \min_{z \in Q} \;\; \, q(z) \nonumber \\
& \text{subject to } \; z_{\phi(i)}=z_{\phi(i)-j}   \quad \text{ for } i=1, \cdots, n   \text{ and for } j = 1, \cdots,  d_i-1,
\label{eq:optim_blossom}
\end{align}
where $\phi(i) := \sum\limits_{k=1}^i d_i $ and $Q$ is the semi-algebraic set defined by the blossoms of the generating polynomials of $S$. In the special case, where $S$ is a hypercube,  \cite{girard_blossom} showed that the Lagrangian dual optimization problem to Problem~\eqref{eq:optim_blossom} is a linear program. Hence, the optimal objective value of this linear program is a lower bound on the minimum of $p(x)$ over the hypercube. Application of blossoming in estimation of reachability sets of discrete-time dynamical systems can be found in~\cite{sassi2012reachability}.

\section{Bernstein, Polya and Handelman Theorems}

While QE, RLT, Groebner bases and blossoming are all useful techniques with
advantages and disadvantages (such as exponential complexity), we focus on Polya's, Bernstein's and Handelman's theorems - results which yield polynomial-time tests for positivity of polynomials. Polya's theorem yields a basis to represent the cone of polynomials that are positive over the positive orthant. Bernstein's and Handelman's theorems yield bases which represent the cones of polynomials that are positive over simplices and convex polytopes, respectively. Similar to SOS programming, one can find certificates of positivity using Polya's, Bernstein's and Handelman's representations by solving a sequence of Linear Programs (LPs) and/or SDPs. However, unlike the SDPs associated with SOS programming, the SDPs associated with these theorems have a block-diagonal structure. In this dissertation, we exploit this structure to design parallel algorithms for optimization of polynomials of high degrees with several independent variables. See~\cite{kamyar_ACC2012},~\cite{kamyar_CDC2012},~\cite{kamyar_CDC2013} and~\cite{kamyar_TAC2013} for parallel implementations of variants of Polya's theorem applied to various Lyapunov inequalities.

Unfortunately, unlike the SOS methodology, the bases given by Polya's theorem, Bernstein's theorem and  Handelman's theorem cannot be used to represent the cone of non-negative polynomials which have zeros in the interior of simplices and polytopes. This is indeed a barrier against using these theorems to compute polynomial Lyapunov functions, since Lyapunov functions, by definition, have a zero at the origin. There do, however, exist some variants of Polya's theorem which consider zeros at the corners~\cite{polya_corner} and edges~\cite{polya_edge} by constructing local certificates of non-negativity over closed subsets, $C_i$, of the simplex such that $\cup C_i $ is the simplex. These results apply to non-negative polynomials whose zeros are on the corners and/or edges of the simplex. Moreover,~\cite{peres_multisimplex} and~\cite{kamyar_CDC2012} propose versions of Polya's theorem which prove positivity over hypercubes by: 1) Providing certificates of positivity on the Cartesian product of unit simplices; and 2) Introducing a one-to-one map between products of unit simplices (multi-simplex) and hypercubes. A generalization of Polya's theorem for proving positivity on the entire $\mathbb{R}^n$ was introduced by~\cite{polya_Rn}. This generalization first applies Polya's theorem to each orthant of $\mathbb{R}^n$ to compute a certificate of positivity over each orthant. Then, it uses the merging technique in~\cite{lombardi1991effective} to obtain a unified certificate - in the form of SOS of rationals - over $\mathbb{R}^n$. A recent extension of Polya's theorem by~\cite{polya_positivstellensatz_2014} can be used to prove positivity over an intersection of a semi-algebraic set with the positive orthant. Finally, positivity of polynomials with rational exponents can be verified by a weak version of Polya's theorem in~\cite{polya_rational}.

\section{Motivations and Summary of Contributions}

The novelty of our research centers on the areas of: computation and energy. In the realm of computation, we observed that processors speeds are not growing at the rate they once were. The entire controls community seems to have ignored this fact, since
everyone speaks of polynomial-time algorithms as the gold standard for what the solution to
a control problem should look like. But what good is a polynomial-time algorithm when the
degree of the polynomial is bounded by the current state-of-the-art computers. Our solution was to look at the only area where the computing world was getting faster (growing) - supercomputers. Surprisingly, there have been no studies on the use of parallel
computers for controls since the 1970's. The reason was that the mathematical machinery for
analysis and control is based on Semidefinite Programming, which is inherently sequential
(NC-hard). Our idea, however, was that if the SDP problem has special structure, then this
structure can be exploited to distribute computation among processors. With this in mind,
we decided to seek out alternatives to the classical Sum-of-Squares approach to nonlinear and robust stability analyses. We identified more than seven different alternatives to
the Sum-of-Squares approach. In the end, not all of these had usable structures for parallelization. However, we identified three which did: polynomial positivity results by Handelman, Polya and Bernstein. To demonstrate how well this approach works in practice, we developed a Message Passing Interface code for Polya's theorem. The result enabled stability analysis for systems three times larger (in terms of number of states) than
any other algorithm. As a real-world application, we further used our code to analyze robust stability of plasma in the Tore Supra Tokamak reactor.

In the realm of energy, we noticed that the two electrical utility companies of
Arizona (APS and SRP) have recently started charging their customers for their maximum rate of electricity usage. This intrigued us as a mathematical problem of how to optimize the thermostat settings of HVAC systems (the major sources of electricity consumption in Arizona) in order to minimize the electricity bill. This problem is interesting in that the time of peak electricity use is not usually at the hottest time of day, but rather a couple of hours after - a behaviour which is usually associated with a diffusion PDE. We used the heat equation to model the thermostat programming problem as an optimal control problem and it turned out to be unsolved. The mathematical reason being that the cost function is not separable in time - a property which is necessary for optimal control algorithms to converge to an optimal solution. We noticed that an arbitrarily precise approximation of the cost function however, satisfy certain properties which make it solvable on a Pareto-optimal front. The result is an optimal thermostat which can significantly reduce
the electricity bills and peak demand of both solar and nonsolar customers under the current pricing plans. Expanding this approach, we started thinking about related topics, such as how to set the demand price on order to influence customers' behavior in an optimal manner. Based on that, we proposed an optimal pricing algorithm which resulted in a moderate reduction in the cost of generating, transmission and distribution of electricity at SRP.

We highlight our contributions as follows. In Chapter~\ref{chp:linear}, we propose a parallel set-up algorithm which applies Polya's theorem to the parameter-dependent Lyapunov inequalities $P(\alpha) > 0$ and $A^T(\alpha) P (\alpha) + P(\alpha) A(\alpha) < 0$ with $\alpha$ belonging to the standard simplex. Feasibility of these inequalities implies robust stability of the system of linear Ordinary Differential Equations (ODEs) $\dot{x}(t) = A(\alpha)x(t)$ over the simplex. The output of our set-up algorithm is a sequence of SDPs of increasing size and precision. A solution to any of these SDPs yield a Lyapunov function which is quadratic in the states and depends polynomially on the uncertain parameters. An interesting property of these SDPs is that they possess a block-diagonal structure. We show how this structure can be exploited to design a parallel \textit{interior-point primal-dual} SDP solver which distributes the computation of search direction among a large number of processors. We then produce a Message Passing Interface (MPI) implementation of our set-up and solver algorithms. Through numerical experiments, we show that these algorithms achieve a near-linear theoretical and experimental speed-up (the increase in processing speed per additional processor). Moreover, our numerical experiments on cluster computers demonstrate the ability of our algorithms in utilizing hundreds and potentially thousands of processors to analyze systems with 100+ states. 

In Chapter~\ref{chp:multisim}, we generalize our methodology to perform robust stability analysis over hypercubes. We first propose an extended version of Polya's theorem. This theorem parameterizes every homogeneous polynomial which is positive over a hypercube. We then propose an extended set-up algorithm which maps the computation and memory - associated with applying the extended Polya's theorem to stability analysis problems - to parallel machines. This set-up algorithm has no centralized computation and its per-core communication complexity scales polynomially with the state-space dimension and the number of uncertain parameters. As the result, it demonstrates a near-linear speed-up. 

In Chapter~\ref{chp:Nonlinear}, we further extend our analysis to address stability of nonlinear ODEs defined by a polynomial vector field $f$. Our proposed solution to this problem is to reformulate the nonlinear stability problem using only strictly positive forms. Specifically, we use our extended version of Polya's theorem in Chapter~\ref{chp:multisim} to compute a matrix-valued homogeneous polynomial $P(x)$ such that $P(x) > 0$ and $\langle \nabla (x^TP(x)x), f(x) \rangle <0$ for all $x$ inside a hypercube containing the origin in its interior. This yields a Lyapunov function of the form $V(x)= x^TP(x)x$ for the system $\dot{x}(t) = f(x(t))$. To do this, we design a new parallel set-up algorithm which applies Polya's theorem to the inequalities $P(x) > 0$ and $\langle \nabla (x^TP(x)x), f(x) \rangle <0$. The result is a sequence of SDPs with coefficients of $P$ as decision variables. Again, we show that these SDPs have a block-diagonal structure - thus can be solved in parallel using our SDP solver in Chapter~\ref{chp:linear}. As an extension to stability analysis over arbitrary convex polytopes, we then propose an algorithm which applies Handelman's theorem to the aforementioned Lyapunov inequalities. Unfortunately, as in the case of Polya's theorem, Handelman's theorem is incapable of parameterizing polynomials which possess zeros in the interior of a polytope. However, we show that this is not the case if the zeros are on the vertices of the polytope. By using this property, we propose the following methodology: 1) Decompose the polytope into several convex sub-polytopes with a common vertex on the equilibrium; 2) Apply Handelman's theorem to Lyapunov inequalities defined on each sub-polytope. The result is a sequence of linear programs whose solutions define a piecewise polynomial Lyapunov function $V$- hence proving asymptotic stability over the sublevel-set of $V$ inscribed in the original polytope. We provide a comprehensive comparison between the computational complexities of SOS algorithm, our Polya's algorithms and our Handelman algorithm. Our analysis shows that by using a certain decomposition scheme, our algorithm (based on Handelman's theorem) has the lowest computational complexity compared to the SOS and Polya's algorithms.


\chapter{FUNDAMENTAL RESULTS FOR OPTIMIZATION OF POLYNOMIALS}
\label{chp:background}

In this chapter, we first provide an overview of fundamental theorems on positivity of polynomials over various sets. Then, we show how applying these theorems to optimization of polynomials problems of the Form~\eqref{optim_poly_exp} yields tractable convex optimization problems in the forms of LPs and/or SDPs. Any solution to these LPs and/or SDPs yields a lower-bound on the global minimum of the polynomial optimization problem $\min\limits_{x \in Q} \; p(x)$. \\

\section{Background on Positivity Results}
\label{sec:history}

In 1900, Hilbert published a list of mathematical problems, one of which is: For every non-negative $f \in \mathbb{R}[x]$, does there exist any non-zero $q \in \mathbb{R}[x]$ such that $q^2f$ is a sum of squares? In other words, is every non-negative polynomial a sum of squares of rational functions? This question was motivated by his earlier works~(\cite{hilbert1,hilbert2}), in which he proved: 1) Every non-negative bi-variate degree 4 \textit{homogeneous} polynomial (A polynomial whose monomials all have the same degree) is a SOS of three polynomials; 2) Every bi-variate non-negative polynomial is a SOS of four rational functions; 3) Not every non-negative homogeneous polynomial with more than two variables and degree greater than 5 is SOS of polynomials. While there exist systematic ways (e.g., semi-definite programming) to prove that a non-negative polynomial is SOS, proving that a non-negative polynomial is not a SOS of polynomials is not straightforward. Indeed, the first example of a non-negative non-SOS polynomial was published eighty years after Hilbert posed his 17$^{th}$ problem.~\cite{motzkin} constructed a PSD degree 6 polynomial with three variables which is not SOS:
\begin{equation}
M(x_1,x_2,x_3)=x_1^4x_2^2+x_1^2x_2^4-3x_1^2x_2^2x_3^2+x_3^6.
\label{eq:motzkin}
\end{equation}
Non-negativity of $M$ follows directly from the inequality of arithmetic and geometric means, i.e., $(a_1+\cdots+a_n)/n \geq \sqrt[n]{a_1 \cdots a_n}$, by letting $n=3, a_1=x_1^4 x_2^2, a_2=x_1^2 x_2^4$ and $a_3=x_3^6$. To show that $M$ is not SOS, first by contradiction suppose that there exist some $N \in \mathbb{N}$ and coefficients $b_{i,j} \in \mathbb{R}$ such that
\begin{align}
\hspace{-0.1in} M(x_1,x_2,x_3) =  \sum_{i=1}^N \Bigg(
\Big[
 b_{i,1} \;\; \cdots \;\; b_{i,20} 
\Big]
 \Big[\begin{matrix} x_1^3  & x_1^2 x_2 &  x_1^2 x_3 & x_1 x_2^2 & x_1 x_2 x_3 & x_1 x_3^2 & x_2^3  \end{matrix} \nonumber \\
\renewcommand\arraystretch{0.8}  
 \begin{matrix} x_2^2 x_3 & x_2 x_3^2 & x_3^3 & x_1^2 & x_1 x_2 
 & x_2 x_3 & x_2^2 & x_2 x_3 & x_3^2 & x_1 \;\; x_2 \;\;  x_3 \;\;  1 \end{matrix}  \Big]^T \Bigg)^2.
\label{eq:motzkin_expand}
\end{align}
By substituting~\eqref{eq:motzkin} in~\eqref{eq:motzkin_expand} and equating the coefficients of both sides of~\eqref{eq:motzkin_expand}, it follows that $\sum_{i=1}^N b^2_{i,5} = -3$. This is a contradiction, thus $M$ is not SOS of polynomials. A generalization of Motzkin's example is given by Robinson (\cite{Hil17_reznick}). Polynomials of the form $(\prod_{i=1}^n x_i^2) f(x_1,\cdots,x_n)+1$ are not SOS if polynomial $f$ of degree $<2n$ is not SOS. Hence, although the non-homogeneous Motzkin polynomial $M(x_1,x_2,1)=x_1^2 x_2^2(x_1^2+x_2^2-3)+1$ is non-negative it is not SOS.

\cite{artin} answered Hilbert's problem in the following theorem.
\begin{mythm} (Artin's theorem)
A polynomial $f\in \mathbb{R}[x]$ satisfies $f(x) \geq 0$ on $\mathbb{R}^n$ if and only if there exist SOS polynomials $N$ and $D \neq 0$ such that $f(x)= \frac{N(x)}{D(x)}$.
\end{mythm}
Although Artin settled Hilbert's problem, his proof was neither constructive nor gave a characterization of the numerator $N$ and denominator $D$. 
In 1939, Habicht provided some structure on $N$ and $D$ for a certain class of polynomials $f$. \cite{habicht} showed that if a homogeneous polynomial $f$ is positive definite and can be expressed as $f(x_1, \cdots, x_n)$ $=g(x_1^2,\cdots,x_n^2)$ for some polynomial $g$, then one can choose the denominator $D=\sum_{i=1}^n x_i^2$. Moreover, he showed that by using $D=\sum_{i=1}^n x_i^2$, the numerator $N$ can be expressed as a sum of squares of monomials. Habicht used Polya's theorem (\cite{polya_book}, Theorem 56) to obtain the above characterizations for $N$ and $D$.
\begin{mythm}(Polya's theorem)
Suppose a homogeneous polynomial $p$ satisfies $p(x) > 0$ for all $x \in  \{ x \in \mathbb{R}^n : x_i \geq 0, \sum_{i=1}^n {x_i} \neq 0 \}$. Then $p(x)$ can be expressed as
\[
p(x) = \dfrac{N(x)}{D(x)},
\]
where $N(x)$ and $D(x)$ are homogeneous polynomials with all positive coefficients. Furthermore, for every homogeneous $p(x)$ and some $e \geq 0$, the denominator $D(x)$ can be chosen as $(x_1+ \cdots +x_n)^e$.
\label{thm:polya}
\end{mythm}
To see Habicht's result, suppose $f$ is homogeneous and positive on the positive orthant and can be expressed as $f(x_1, \cdots, x_n)=g(x_1^2,\cdots,x_n^2)$ for some homogeneous polynomial $g$. By using Polya's theorem, $g(y)=\frac{N(y)}{D(y)}$, where $y:=(y_1,\cdots,y_n)$ and polynomials $N$ and $D$ have all positive coefficients. Furthermore, from Theorem~\ref{thm:polya} we may choose $D(y)=\left( \sum_{i=1}^n y_i \right)^e$. Then, $\left( \sum_{i=1}^n y_i \right)^e g(y)=N(y)$. Now let $x_i=\sqrt{y_i} $, then $\left( \sum_{i=1}^n x_i^2 \right)^e f(x_1,\cdots,x_n)=N(x_1^2,\cdots,x_n^2)$. Since $N$ has all positive coefficients, $N(x_1^2,\cdots,x_n^2)$ is a sum of squares of monomials.

Similar to the case of positive definite polynomials, ternary positive semi-definite polynomials of the form $g(x_1^2, x_2^2, x_3^2)$ can be parameterized using the denominator $D=(x_1^2+x_2^2+x_3^2)^N$~(\cite{scheiderer2006sums}). However, in any dimension higher than three, there exist positive semi-definite polynomials $f$ such that if $h^2 f$ is SOS, then $h$ has a zero other than the origin. Thus, for such polynomials $f$, $D f$ cannot be SOS.
Indeed, it has been shown by~\cite{reznick_no_denominator} that there exists no single SOS polynomial $D \neq 0$ which satisfies $f=\frac{N}{D}$ for every positive semi-definite $f$ and some SOS polynomial $N$.

As in the case of positivity on $\mathbb{R}^n$, there has been an extensive research regarding positivity of polynomials on bounded sets. A pioneering result on local positivity is Bernstein's theorem~(\cite{bernstein_1915}). Bernstein's theorem uses the polynomials $h_{i,j}=(1+x)^i(1-x)^j$ as a basis to parameterize univariate polynomials which are positive on $[-1,1]$.
\begin{mythm}(Bernstein's theorem)
If a polynomial $f(x) > 0$ on $[-1,1]$, then there exist $c_{i,j} > 0$ such that \[
f(x)= \sum_{ i,j \in \mathbb{N}: \, i+j=d} c_{i,j} (1+x)^i(1-x)^j
\]
for some $d > 0$.
\end{mythm}

\cite{Reznick_powers_interval2000} used Goursat's transformation of $f$ to find an upper bound on $d$. Unfortunately, the bound itself is a function of the minimum of $f$ on $[-1,1]$. In order to reduce the computational complexity of testing positivity,~\cite{roy} proposed a decomposition scheme for breaking $[-1,1]$ into a collection of sub-intervals. Subsequently, Bernstein's theorem was applied to $f$ over each sub-interval to find a certificate of positivity over each sub-interval. An extension of this technique was proposed in~\cite{leroy} to verify positivity over simplices (a simplex is the convex hull of $n + 1$ vertices in $\mathbb{R}^n$). Moreover,~\cite{leroy} provided a degree bound as a function of the minimum of $f$ over the simplex, the number of variables in $f$, the degree of $f$ and the maximum of certain affine combinations of the coefficients $c_{i,j}$.

\cite{handelman1988} also used products of affine functions as a basis (the Handelman basis) to extend Bernstein's theorem to multi-variate polynomials which are positive on convex polytopes. 
\begin{mythm} (Handelman's Theorem)
\label{thm:Handelman} Given $w_i \in \mathbb{R}^n$ and $u_i \in \mathbb{R}$, define the polytope $\Gamma^K := \{ x\in \mathbb{R}^n : w_i^Tx + u_i\geq 0, i=1,\cdots,K \}$. If a polynomial $f(x) > 0$ on $\Gamma^K$, then there exist $b_\alpha \geq 0$, $\alpha \in \mathbb{N}^K$ such that for some $d \in \mathbb{N}$,
\begin{equation}
f(x) = \sum _{\substack{\alpha \in \mathbb{N}^K \\ \alpha_1+\cdots+\alpha_K \leq d}} b_\alpha (w_1^T x+u_1)^{\alpha_1} \cdots (w_K^T x+u_K)^{\alpha_K}.
\label{eq:handelman_representation}
\end{equation}
\end{mythm}
Recently,~\cite{Handelman_Sankaranarayanan} combined the Handelman basis with positive basis functions 
\[x_1^{\alpha_1} \cdots x_n^{\alpha_n}- l_\alpha \quad \text{and} \quad  u_\alpha - x_1^{\alpha_1} \cdots x_n^{\alpha_n}
\] to compute Lyapunov functions over a hypercube $\Phi$, where $l_{\alpha}$ and $u_{\alpha}$ are the minimum and maximum of $x_1^{\alpha_1} \cdots x_n^{\alpha_n}$ over the hypercube $\Phi$. A generalization of Handelman's theorem was made by~\cite{schweighofer_nonnegativity} to verify non-negativity of polynomials over compact semi-algebraic sets. Schweighofer used the cone of polynomials\footnote{A set of polynomials $S \subset \mathbb{R}[x_1, \cdots, x_n]$ is a \textit{cone} if:
1) $f_1 \in S$ and $f_2 \in S$ imply $f_1f_2 \in S$ and $f_1+f_2 \in S$; and
2) $f \in \mathbb{R}[x_1, \cdots, x_n]$ implies $f^2 \in S$.
} defined in~\eqref{eq:cone}
to parameterize any polynomial $f$ which has the following properties:
\begin{enumerate}
\item $f$ is non-negative over the compact semi-algebraic set $S$ defined in~\eqref{eq:semialgebraic_set}
\item $f=q_1 p_1 + q_2 p_2 + \cdots$ for some $q_i$ in the Cone~\eqref{eq:cone} and for some $p_i > 0$ over 
\[
S \cap \{x \in \mathbb{R}^n:f(x)=0 \}
\]
\end{enumerate}

\begin{mythm}(Schweighofer's theorem)
Suppose 
\begin{equation}
S:=\{x \in \mathbb{R}^n: g_i(x) \geq 0, g_i \in \mathbb{R}[x] \text{ for } i=1, \cdots, K\}
\label{eq:semialgebraic_set}
\end{equation}
is compact. Define the following set of polynomials which are positive on $S$.
\begin{equation}
\Theta_d:= \left\lbrace  \sum_{\substack{\lambda \in \mathbb{N}^K: \lambda_1+\cdots+\lambda_K \leq d}} s_\lambda g_1^{\lambda_1} \cdots g_K^{\lambda_K}: s_\lambda \text{ are SOS }   \right\rbrace
\label{eq:cone}
\end{equation}
If $f \geq 0$ on $S$ and there exist $q_i \in \Theta_d$ and polynomials $p_i > 0$ on $S  \cap \{x \in \mathbb{R}^n:f(x)=0 \}$ such that $f= \sum_{i} q_i p_i$
for some $d$, then $f \in \Theta_d$.
\label{Schweighofer}
\end{mythm}
 On the assumption that $g_{i}$ are affine functions, $p_i=1$ and $s_\lambda$ are constant, Schweighofer's theorem gives the same parameterization of $f$ as in Handelman's theorem. Another special case of Schweighofer's theorem is when $\lambda \in \{ 0,1\}^K$. In this case, Schweighofer's theorem reduces to 
Schmudgen's Positivstellensatz~(\cite{schmudgen}). Schmudgen's Positivstellensatz states that the cone
\begin{equation}
\Lambda_g:=\left\lbrace \sum_{\lambda \in \{0,1 \}^K }s_{\lambda} g_1^{\lambda_1} \cdots g_K^{\lambda_K}: s_{\lambda} \text{ are } SOS \right\rbrace \subset \Theta_d
\label{schmudgen}
\end{equation}
is sufficient to parameterize every $f > 0$ over the semi-algebraic set $S$ generated by $\{g_1,\cdots,g_K\}$. Unfortunately, the cone $\Lambda_g$ contains $2^K$ products of $g_i$, thus finding a representation of Form~\eqref{schmudgen} for $f$ requires a search for at most $2^K$ SOS polynomials.
Putinar's Positivstellensatz~(\cite{putinar}) reduces the complexity of Schmudgen's parameterization in the case where the quadratic module $M_g$ (as defined in~\eqref{eq:putinar_cone}) of polynomials $g_i$ is \textit{Archimedean}, i.e., there exists $N \in \mathbb{N}$ such that
\begin{equation}
N - \sum_{i=1}^n x_i^2 \in M_g.
\label{eq:Archimedean}
\end{equation}
Equivalently, if there exists some $f \in M_g$ such that $\{ x \in \mathbb{R}^n: f(x) \geq 0 \}$ is compact, then $M_g$ is Archimedean.

\begin{mythm}(Putinars's Positivstellensatz)
Let $S:=\{x \in \mathbb{R}^n: g_i(x) \geq 0, g_i \in \mathbb{R}[x] \text{ for } i=1, \cdots, K\}$ and define 
\begin{equation}
M_g:=\left\lbrace s_0+\sum_{i=1}^K s_i g_i: s_i \text{ are SOS}  \right\rbrace.
\label{eq:putinar_cone}
\end{equation}
If there exist some $N > 0$ such that $N-\sum_{i=1}^n x_i^2 \in M_g$, then $M_g$ is Archimedean. If $M_g$ is Archimedean and $f > 0$ over $S$, then $f \in M_g$.
\label{putinar}
\end{mythm}
Finding a representation of Form~\eqref{eq:putinar_cone} for $f$, only requires a search for $K+1$ SOS polynomials using SOS programming. Verifying the Archimedian Condition~\eqref{eq:Archimedean} is also an SOS program. Observe that if $M_g$ is not Archimedean, one can add a redundant constraint $r-\sum_{i=1}^n x_i^2 \geq 0$ (for sufficiently large $r \in \mathbb{R}$) to $S$ in order to make $M_g$ Archimedean. Archimedean condition clearly implies compactness of the semi-algebraic set $S$ because for any $f \in M_g$, $S \subset \{ x \in \mathbb{R}^n: f(x) \geq 0 \}$. The following theorem lifts the compactness requirement for the semi-algebraic set $S$.

\begin{mythm}(Stengle's Positivstellensatz)
Let $S:=\{x \in \mathbb{R}^n: g_i(x) \geq 0, g_i \in \mathbb{R}[x] \text{ for } i=1, \cdots, K\}$ and define the cone
\[
\Lambda_g:=\left\lbrace \sum_{\lambda \in \{0,1 \}^K }s_{\lambda} g_1^{\lambda_1} \cdots g_K^{\lambda_K}: s_{\lambda} \text{ are } SOS \right\rbrace.
\]
If $f>0$ on $S$, then there exist $p,g \in \Lambda_g$ such that $qf=p+1$.
\label{stengle}
\end{mythm}
Notice that the Parameterziation~\eqref{eq:handelman_representation} in Handelman's theorem is affine in $f$ and the coefficients $b_\alpha$. Likewise, the parameterizations in Theorems~\ref{Schweighofer} and~\ref{putinar}, i.e., $f=\sum_{\lambda} s_\lambda g_1^{\lambda_1} \cdots g_K^{\lambda_K}$ and $f=s_0+\sum_i s_i g_i$ are affine in $f,s_\lambda$ and $s_i$. Thus, one can use convex optimization to find $b_\alpha$, $s_\lambda,s_i$ and $f$ efficiently. Unfortunately, since the parameterization $qf=p+1$ in Stengle's Positivstellensatz is non-convex (bilinear in $q$ and $f$), it is more difficult to verify $qf=p+1$ compared to Handelman's and Putinar's parameterizations.

For a comprehensive discussion on the Positivstellensatz and other results on polynomial positivity in algebraic geometry see~\cite{laurent_survey,scheiderer_survey}, and~\cite{delzell_book}.

\section{Polynomial Optimization and Optimization of Polynomials}
\label{sec:poly_optim2}

Given $f,g_i,$ $h_j \in \mathbb{R}[x]$ for $i=1, \cdots,m$ and $j=1, \cdots, r$, define a semi-algebraic set $S$ as
\begin{equation}
S:=\{ y \in \mathbb{R}^n: g_i(y) \geq 0, \, h_j(y) = 0 \text{ for }  i=1, \cdots,m \text{ and } j=1, \cdots, r \}.
\label{eq:SASdef}
\end{equation}
We then define \textit{polynomial optimization} problems as
\begin{equation}
\beta^* =  \min_{x \in S}  \; f(x).
\label{eq:polynomial_optimization2}
\end{equation}
 For example, the integer program
\begin{align}
\min_{x \in \mathbb{R}^n}  \quad &p(x) \nonumber  \\
\text{subject to} \quad &a_i^T x \geq b_i \text{ for } i=1, \cdots, m, \nonumber \\
&x \in \{-1,1\}^n,
\label{eq:polynomial_optimization_exp}
\end{align}
with given $a_i \in \mathbb{R}^n, b_i \in \mathbb{R}$ and $p \in \mathbb{R}[x]$, can be formulated as a polynomial optimization problem by setting $f=p$ in~\eqref{eq:polynomial_optimization2} and setting
\begin{align*}
& g_i(x)= a_i^Tx-b_i && \hspace{-1.1in} \text{ for } i=1,\cdots,m \\
& h_j(x)=x_j^2-1     && \hspace{-1.1in} \text{ for } j=1,\cdots,n.
\end{align*}
in the definition of $S$ in~\eqref{eq:SASdef}.

Given $c \in \mathbb{R}^n$ and $g_i,$ $h_j \in \mathbb{R}[x]$ for $i=1, \cdots,m$ and $j=1, \cdots, r$, we define \textit{Optimization of polynomials} problems as
\begin{align}
&\gamma^*=\max_{x \in \mathbb{R}^q} \quad c^Tx \nonumber \\
&\text{subject to } \;\;\, F(x,y):= F_0(y)+\sum_{i=1}^q x_i F_i(y) \geq 0 \text{ for all } y \in S,
\label{eq:optimization_of_polynomials2}
\end{align}
where $S$ is defined in~\eqref{eq:SASdef} 
and
\[
F_i(y) := \sum_{\alpha \in E_{d_i}} F_{i,\alpha} y_1^{\alpha_1} \cdots y_n^{\alpha_n}
\]
with $E_{d_i}:=\{ \alpha \in \mathbb{N}^n : \sum_{i=1}^n \alpha_i \leq d_i \}$, where coefficients $F_{i,\alpha} \in \mathbb{R}^{t \times t}, \, i=0, \cdots, q$ are given. If the goal is to optimize over a polynomial variable, $p(y)$, this may be achieved using a basis of monomials for $F_i(y)$ so that the polynomial variable becomes $p(y)=\sum_i x_i F_i(y)$. 
 Optimization of polynomials can be used to find $\beta^*$ in~\eqref{eq:polynomial_optimization2}. For example, we can compute the optimal objective value $\eta^*$ of the polynomial optimization problem
\begin{align*}
\eta^* = \min_{x \in \mathbb{R}^n} \quad &p(x)  \nonumber  \\
\text{subject to} \quad & a_i^Tx-b_i \geq 0 && \hspace{-0.9in} \text{for } i=1, \cdots, m, \nonumber \\
&x_j^2-1=0 && \hspace{-0.9in} \text{for } j=1,\cdots,n ,
\end{align*}
by solving the problem
\begin{align}
&\eta^*= \max_{\eta \in \mathbb{R}} \quad \eta \nonumber \\
&\text{subject to} \quad p(y) \geq \eta \qquad \text{ for } y \in \{ y \in \mathbb{R}^n: a_i^T y \geq b_i, \, y_j^2-1=0 \text{ for } i=1, \cdots, m \nonumber \\
& \hspace{4in} \text{ and } j=1,\cdots,n\},
\label{eq:optimization_of_polynomials_exp}
\end{align}
where Problem~\eqref{eq:optimization_of_polynomials_exp} can be expressed in the Form~\eqref{eq:optimization_of_polynomials2} by setting
\[
c=1, \qquad q=1, \qquad t=1, \qquad F_0 = p \qquad F_1=-1,
\]
\[
S:=\{y \in \mathbb{R}^n: a_i^T y \geq b_i, \, y_j^2-1=0 \text{ for } i=1, \cdots, m,  \text{ and } j=1,\cdots,n \}.
\]
Optimization of polynomials~\eqref{eq:optimization_of_polynomials2} can be reformulated as the feasibility problem
\begin{align}
& \hspace{-0.09in} \gamma^* = \min_\gamma \;\,  \gamma \nonumber \\
& \hspace{-0.13in} \text{ subject to }
S_\gamma:= \left\lbrace   x \in \mathbb{R}^q : c^T x > \gamma, \,  F(x,y) \geq 0 \text{ for all } y \in S\right\rbrace = \emptyset,
\label{eq:feasibility}
\end{align}
where $c$ and $F$ are given and
\[
S:=\{ y \in \mathbb{R}^n: g_i(y) \geq 0, \, h_j(y) = 0 \text{ for }  i=1, \cdots,m \text{ and } j=1, \cdots, r \},
\]
where polynomials $g_i$ and $h_j$ are given.
The question of feasibility of a semi-algebraic set is NP-hard~(\cite{blum}). However, if we have a test to verify $S_\gamma = \emptyset$, we can find $\gamma^*$ by performing a bisection on $\gamma$.
In the following section, we use the results of Section~\ref{sec:history} to provide sufficient conditions, in the form of Linear Matrix Inequalities (LMIs), for $S_\gamma = \emptyset$.

\section{Algorithms for Optimization of Polynomials}
\label{sec:poly_optim}

In this section, we discuss how to find lower bounds on $\beta^*$ for different classes of polynomial optimization problems. The results in this section are primarily expressed as methods for verifying $S_{\gamma}=\emptyset$ and can be used with bisection to solve polynomial optimization problems.

\subsection{Case 1: Optimization over the Standard Simplex $\Delta^n$}
\label{sec:optim_simplex}

Define the standard unit simplex as
\begin{equation}
\Delta^n : = \{ x \in \mathbb{R}^n: \sum_{i=1}^n x_i=1, x_i \geq 0\}.
\label{eq:simplex}
\end{equation}
Consider the polynomial optimization problem
\begin{equation*}
\gamma^* = \min_{x \in \Delta^n} \quad f(x),
\end{equation*}
where $f$ is a homogeneous polynomial of degree $d$. If $f$ is not homogeneous, we can homogenize it by multiplying each monomial $x_1^{\alpha_1} \cdots x_n^{\alpha_n}$ in $f$ by $(\sum_{i=1}^n x_i)^{d-\Vert \alpha \Vert_1}$. Notice that since $\sum_{i=1}^n x_i=1$ for all $x \in \Delta^n$, the homogenized $f$ is equal to $f$ for every $x \in \Delta^n$. To find $\gamma^*$, one can solve the following optimization of polynomials problem. 
\begin{align}
& \gamma^* = \max_{\gamma \in \mathbb{R}} \quad \gamma \nonumber  \\
&   \text{subject to } f(x) \geq \gamma \text{ for all } x \in \Delta^n
\label{eq:optimization_of_polynomials_simplex}
\end{align}
Clearly, Problem~\eqref{eq:optimization_of_polynomials_simplex} can be re-stated as the following feasibility problem 
\begin{align*}
&\gamma^*=\min_{\gamma \in \mathbb{R}} \quad \gamma \\
&\text{subject to } S_\gamma  := \{x \in \mathbb{R}^n: f(x) - \gamma < 0, \sum_{i=1}^n x_i=1,\, x_i \geq 0 \} = \emptyset.
\end{align*}
For a given $\gamma$, we can use the following version of Polya's theorem to verify $S_\gamma = \emptyset$.

\begin{mythm}(Polya's theorem, simplex version)
If a homogeneous matrix-valued polynomial $F$ satisfies $F(x) > 0$ for all $x \in \Delta^n := \{ x \in \mathbb{R}^n: \sum_{i=1}^n x_i = 1, x_i \geq 0 \}$, then there exists $e \geq 0$ such that all the coefficients of
\[
\left( \sum_{i=1}^n x_i  \right)^e F(x)
\]
are positive definite.
\label{thm:polya_simplex}
\end{mythm}
See pages 57-59 of~\cite{inequalities} for a proof. The converse of the theorem only implies $F \geq 0$ over the unit simplex. Given $\gamma \in \mathbb{R}$, it follows from the converse of Theorem~\ref{thm:polya_simplex} that $S_\gamma=\emptyset$ if there exists some $e \geq 0$ such that
\begin{equation}
\left( \sum_{i=1}^n x_i \right)^e \left( f(x)-\gamma \left(\sum_{i=1}^n x_i\right)^d \right)
\label{eq:polya_product}
\end{equation}
has all positive coefficients, where recall that $d$ is the degree of $f$. We can compute lower bounds on $\gamma^*$ by performing a bisection on $\gamma$. For each $\gamma$ of the bisection, if there exists some $e \geq 0$ such that all of the coefficients of~\eqref{eq:polya_product} are positive, then $\gamma \leq \gamma^*$. We have detailed this procedure in Algorithm~\ref{alg:Polya_simplex}.

\begin{algorithm}
\textbf{\textit{Input:}}\\
Polynomial $f$; maximum polya's exponent $e_{\max}$; lower-bound $\gamma_l$ and upper-bound $\gamma_u$ for bisection search; number of bisection iterations $b_{\max}$; \vspace{0.1in} \\ 

\textbf{\textit{Initialization:}}\\
Set Polya's exponent $e=0$.\\
Set $k=0$. \vspace{0.1in}\\

\textbf{\textit{Main Loop:}}\\
\While{$d \leq b_{\max}$}{
	Set $\gamma = \frac{\gamma_u + \gamma_l}{2}$.\\

	\While{Eq.~\eqref{eq:polya_product} has some negative coefficient or $e \leq e_{\max}$}{
		Set $e=e+1$.\\
		Calculate the Product~\eqref{eq:polya_product}.\\

	}
	\eIf{Eq.~\eqref{eq:polya_product} has all positive coefficients}{
		Set $\gamma_l = \gamma$.
	}{
		Set $\gamma_u = \gamma$.
	}
	
	Set $k=k+1$.\\
} \vspace{0.1in}

\textbf{\textit{Output:}}\\
 $\gamma$: a lower bound on the minimum of $f$ over the standard simplex.
\vspace{0.15in}\caption{Polya's algorithm for polynomial optimization over the simplex}
\label{alg:Polya_simplex}
\end{algorithm}
In Chapter~\ref{chp:linear}, we will propose a decentralized version of Algorithm~\ref{alg:Polya_simplex} to perform robust stability analysis over a simplex.

\subsection{Case 2: Optimization over The Hypercube $\Phi^n$}
\label{sec:optim_hypercube}

Given $r_i \in \mathbb{R}$, define the hypercube 
\begin{equation}
\Phi^n:=\{x \in \mathbb{R}^n:  \vert x_i \vert \leq r_i, i=1, \cdots, n  \}.\
\label{eq:hypercube}
\end{equation}
Define the set of $n$-variate \textit{multi-homogeneous} polynomials of degree vector $d \in \mathbb{N}^n$ as 
\begin{equation}
\left\lbrace
p \in \mathbb{R}[x,y]:
p(x,y) = \sum_{\substack{ h,g \in \mathbb{N}^n  \\ h+g=d  }} p_{h,g} x_1^{h_1}y_1^{g_1} \cdots x_n^{h_n}y_n^{g_n},\, p_{h,g} \in \mathbb{R}
\right\rbrace.
\label{eq:multi-homog_poly}
\end{equation}
In a more general case, if the coefficients $p_{h,g}$ are matrices, we call $p$ a \textit{matrix-valued multi-homogeneous} polynomial. 
Now consider the polynomial optimization problem
\[
\gamma^*=\min_{x \in \Phi^n} f(x).
\]
To find $\gamma^*$, one can solve the following feasibility problem.
\begin{align}
&\gamma^*=\min_{\gamma \in \mathbb{R}} \quad \gamma \nonumber \\
&\text{subject to } \;\, S_{\gamma,r}  := 
\{x \in \mathbb{R}^n: f(x) - \gamma < 0, \, \vert x_i \vert \leq r_i,\, i=1, \cdots,n \} = \emptyset
\label{eq:feasibility_hypercube}
\end{align}
For a given $\gamma$, we propose the following version of Polya's theorem~(\cite{kamyar_CDC2012}) to verify $S_{\gamma,r} = \emptyset$.

\begin{mythm}(Polya's theorem: multi-simplex version)
A matrix-valued multi-homogeneous polynomial $F$ satisfies $F(x,y) > 0$ for all $(x_i,y_i) \in \Delta^2, i=1, \cdots, n$, if there exist $e \geq 0$ such that all the coefficients of
\[
\left( \prod_{i=1}^n \left( x_i+y_i \right)^e \right) F(x,y)
\]
are positive definite.
\label{thm:polya_multi-simplex}
\end{mythm}
We will prove this result in Section~\ref{sec:notation_multi-homog}. The converse of Theorem~\ref{thm:polya_multi-simplex} only implies non-negativity of $F$ over the hypercube. To find lower bounds on $\gamma$, we first obtain the multi-homogeneous form $p$ of the polynomial $f$ in~\eqref{eq:feasibility_hypercube}. In~\ref{sec:notation_multi-homog} we have provided a procedure to construct $p$. Given $\gamma \in \mathbb{R}$ and $r \in \mathbb{R}^n$, it follows from the converse of Theorem~\ref{thm:polya_multi-simplex} that $S_{\gamma,r}$ defined in~\eqref{eq:feasibility_hypercube} is empty if there exists some $e \geq 0$ such that 
\begin{equation}
\left( \prod_{i=1}^n \left( x_i+y_i \right)^e \right) \left( p(x,y) - \gamma \left( \prod_{i=1}^n \left( x_i+y_i \right)^{d_i} \right) \right)
\label{eq:polya_product_hypercube}
\end{equation}
has all positive coefficients, where $d_i$ is the degree of $x_i$ in $p(x,y)$. We can compute lower bounds on $\gamma^*$, as defined in~\eqref{eq:feasibility_hypercube}, by performing a bisection on $\gamma$. For each $\gamma$ of the bisection, if there exists some $e \geq 0$ such that all of the coefficients of~\eqref{eq:polya_product_hypercube} are positive, then $\gamma \leq \gamma^*$. By replacing~\eqref{eq:polya_product} with~\eqref{eq:polya_product_hypercube} in Algorithm~\ref{alg:Polya_simplex}, this algorithm computes $\gamma$. In Chapter~\ref{chp:multisim}, we will propose a parallel algorithm to perform robust stability analysis for systems with uncertain parameters inside a hypercube.

\subsection{Case 3: Optimization over The Convex Polytope $\Gamma^K$}
\label{sec:optim_polytope}

Given $w_i \in \mathbb{R}^n$ and $u_i \in \mathbb{R}$, define the convex polytope 
\begin{equation}
\Gamma^K := \{ x\in \mathbb{R}^n : w_i^Tx + u_i \geq 0, i=1,\cdots,K \}.
\label{eq:polytope}
\end{equation}
 Suppose $\Gamma^K$ is bounded. Consider the polynomial optimization problem
\[
\gamma^*=\min_{x \in \Gamma^K} f(x),
\]
where $f$ is a polynomial of degree $d_f$. To find $\gamma^*$, one can solve the feasibility problem
\begin{align*}
&\gamma^*=\min_{\gamma \in \mathbb{R}} \quad \gamma \nonumber \\
&\text{subject to } \;\, S_{\gamma,K}  := 
\{x \in \mathbb{R}^n: f(x) - \gamma <  0, \, w_i^Tx+u_i \geq 0, \, i = 1, \cdots, K \} = \emptyset.
\end{align*}
Given $\gamma$, one can use Handelman's theorem to verify $S_{\gamma,K}=\emptyset$.

\begin{mythm} (Handelman's Theorem)
\label{thm:Handelman_rep} Given $w_i \in \mathbb{R}^n$ and $u_i \in \mathbb{R}$, define the polytope $\Gamma^K := \{ x\in \mathbb{R}^n : w_i^Tx + u_i\geq 0, i=1,\cdots,K \}$. If a polynomial $f(x) > 0$ on $\Gamma^K$, then there exist $b_\alpha \geq 0$, $\alpha \in \mathbb{N}^K$ such that for some $d \in \mathbb{N}$,
\begin{equation}
f(x) = \sum _{\substack{\alpha \in \mathbb{N}^K \\ \alpha_1+\cdots+\alpha_K \leq d}} b_\alpha (w_1^T x+u_1)^{\alpha_1} \cdots (w_K^T x+u_K)^{\alpha_K}.
\label{eq:handelman_representation_rep}
\end{equation}
\end{mythm}

Consider the Handelman basis associated with polytope $\Gamma^K$ defined as
\begin{equation*}
H_s:=\left\lbrace \lambda_{\alpha} \in \mathbb{R}[x] : \lambda_{\alpha}(x)= \prod_{i=1}^K \left( w_i^Tx+u_i \right)^{\alpha_i}, \alpha \in \mathbb{N}^K, \, \sum_{i=1}^K  \alpha_i \leq s \right\rbrace.
\end{equation*} 
Basis $H_s$ spans the space of polynomials of degree $s$ or less, however it is not minimal. As a special case, if we take $\Gamma^K$ to be the standard unit simplex of $\mathbb{R}^K$, i.e.,
\[
\Gamma^K := \{ x \in \mathbb{R}^K : 1- \sum_{i=1}^K x_i \geq 0, x_i \geq 0 \text{ for } i=1,\cdots,K \},
\]
then the following set of polynomials is called the \textit{Bernstein basis} associated with $\Gamma^K$.
\begin{align*}
& B_s := \\ 
& \left\lbrace \lambda_{\alpha} \in \mathbb{R}[x]  \hspace{-0.03in} :  \hspace{-0.03in} \lambda_{\alpha}(x)  
  = \dfrac{s!}{\alpha_1! \cdots \alpha_{K+1}!} \left( \prod_{i=1}^K x_i^{\alpha_i} \right) \hspace{-0.08in} \left( 1-  \sum_{i=1}^{K} x_i \right)^{\alpha_{K+1}} \hspace{-0.28in} , \alpha \in \mathbb{N}^{K+1}, \sum_{i=1}^K  \alpha_i = s \right\rbrace.
\end{align*}
Unlike $H_s$, $B_s$ is a minimal basis\footnote{This follows from the fact that every polynomial can be uniquely represented in the canonical basis and every member of the canonical basis is a unique linear combination of $\lambda_\alpha \in B_s$. A derivation for these linear combinations can be found in~\cite{farin2002curves}} for the vector space of polynomials of degree $\leq s$.

Given $\gamma \in \mathbb{R}$, polynomial $f(x)$ of degree $d_f$ and $d_{\max} \in \mathbb{N}$, if there exist 
\begin{equation}
c_\alpha \geq 0 \text{ for all } \alpha\in I_d:=\{\alpha \in \mathbb{N}^K: \Vert \alpha \Vert_1 \leq d \}
\label{eq:handelman_representation1}
\end{equation}
such that
\begin{equation}
f(x) - \gamma = \sum_{  \alpha \in  I_d} c_\alpha \prod_{i=1}^K (w_i^Tx+u_i)^{\alpha_i}
\label{eq:handelman_representation2}
\end{equation}
for some $d \geq d_f$, then $f(x) - \gamma \geq 0$ for all $x \in \Gamma^K$. Thus $S_{\gamma,K} = \emptyset$. Since $H_s$ is not a minimal basis, if~\eqref{eq:handelman_representation1} is feasible, then $c_\alpha$ are not unique. Feasibility of Conditions~\eqref{eq:handelman_representation1} and~\eqref{eq:handelman_representation2} can be determined using linear programming. To set-up the linear program, we first represent the right and left hand side of~\eqref{eq:handelman_representation2} in the canonical basis as
\begin{equation}
f(x) - \gamma = \begin{bmatrix}
b_1 & b_2 & \cdots & b_M
\end{bmatrix} z_{n,d}(x)
\label{eq:dummy5}
\end{equation}
\begin{equation}
\sum_{\alpha \in  I_d} c_\alpha \prod_{i=1}^K (w_i^Tx+u_i)^{\alpha_i} = \begin{bmatrix}
l_1(c_\alpha,w,u) & l_2(c_\alpha,w,u) & \cdots & l_M(c_\alpha,w,u) 
\end{bmatrix} z_{n,d}(x),
\label{eq:dummy6}
\end{equation}
where it can be shown that $l_i: \mathbb{R}^N \times \mathbb{R}^{n \times K} \times \mathbb{R}^{K} \rightarrow \mathbb{R}$ are \textit{affine} in $c_\alpha$ and $N:=\dbinom{K+d}{d}$ is the cardinality of the index set $\{ \alpha \in \mathbb{N}^K: \Vert \alpha \Vert_1 \leq d \}$. In~\eqref{eq:dummy6},
\[
\renewcommand\arraystretch{0.8}  
w := \begin{bmatrix}
w_1 & w_2 & \cdots & w_K
\end{bmatrix}  \text{ and }
u := \begin{bmatrix}
u_1 & u_2 & \cdots & u_K
\end{bmatrix}^T, 
\]
where $w_i \in \mathbb{R}^n$ and $u_i \in \mathbb{R}$ define the polytope $\Gamma^K$ in~\eqref{eq:polytope}. Recall that in~\eqref{eq:dummy6}, $z_{n,d}(x)$ denotes the vector of all $n-$variate monomials of degree $d$ or less. By equating~\eqref{eq:dummy5} and~\eqref{eq:dummy6} and cancelling $z_{n,d}(x)$ from both sides, the problem of finding a lower bound $\gamma_d$ on $\gamma^*$ can be expressed as the following linear program.
\begin{align}
& \gamma_d := \max_{\gamma \in \mathbb{R}, c_\alpha \geq 0} \quad\; \gamma \nonumber \\
& \text{subject to } \;\; l_i(c_\alpha, w,u) = b_i \qquad \text{ for } \; i = 1,\cdots,M.
\label{eq:Handelman_LP1}
\end{align}
If Linear Program~\eqref{eq:Handelman_LP1} is infeasible for some $d$, then one can increase $d$ and repeat setting-up and solving Linear Program~\eqref{eq:Handelman_LP1}. From Handelman's theorem, if $f(x) - \gamma > 0$ for all $x \in \Gamma^K$, then for some $d \geq d_f$, Conditions~\eqref{eq:handelman_representation1} and~\eqref{eq:handelman_representation2} hold and Linear Program~\eqref{eq:Handelman_LP1} will have a solution. We have outlined this procedure in Algorithm~\ref{alg:Handelman_polyoptim}. Unfortunately, to this date all the proposed upper-bounds on $d$ (see e.g.,~\cite{reznick_powers_polyhedra} and~\cite{leroy}) are functions of the minimum of $f(x)-\gamma$ over the polytope $\Gamma^K$. In Chapter~\ref{chp:Nonlinear}, we will combine this algorithm with a polytope decomposition scheme to construct Lyapunov functions for nonlinear systems with polynomial vector fields.


\begin{algorithm}
\textbf{Input:}
Polynomial $f$ of degree $d_f$; maximum degree $d_{\max}$ for Handelman's basis; polytope data: $w_i \in \mathbb{R}^n$ and $u_i \in \mathbb{R}$ in~\eqref{eq:polytope}; tolerance $\epsilon > 0$ for stopping criterion.  \vspace{0.1in} \\ 

\textbf{Initialization:}\\
Set $d=d_f$.\\
Set $\gamma_{\text{old}} = -10^{100}$.\\
Set $\gamma_{\text{new}} = -10^{100} +2 \epsilon$.\vspace{0.1in}\\

\textbf{Main Loop:}\\
\While{$(d < d_{\max})$ and $(\gamma_{\text{new}} - \gamma_{\text{old}} < \epsilon)$}{
	Express $f(x)-\gamma_{\text{new}}$ in the canonical basis as in~\eqref{eq:dummy5}.\\
	Express $\sum_{ \alpha \in  I_d} c_\alpha \prod_{i=1}^K (w_i^Tx+u_i)^{\alpha_i} $ in the canonical basis as in~\eqref{eq:dummy6}.\\
	Set-up LP~\eqref{eq:Handelman_LP1}.\\
	
	\If{LP~\eqref{eq:Handelman_LP1} is feasible}{
		Set $\gamma_{\text{old}} = \gamma_{\text{new}}$.\\
		Set $\gamma_{\text{new}} = \gamma_d$. 
	}
	Set $d = d+1$.
} \vspace{0.1in}

\textbf{Output:} $\gamma_{\text{new}}$: a lower bound on the minimum of $f$ over polytope $\Gamma^K$.
\vspace{0.15in}\caption{Polynomial optimization over convex polytopes using Handelman's theorem}
\label{alg:Handelman_polyoptim}
\end{algorithm}

\subsection{Case 4: Optimization over Compact Semi-algebraic Sets}
\label{sec:optim_semialg} 

Recall that we defined a semi-algebraic set as
\begin{equation}
S:=\{ x \in \mathbb{R}^n : g_i(x) \geq 0,i=1,\cdots,m, \, h_j(x) = 0,j=1,\cdots,r \}.
\label{eq:semialgebraic_set2}
\end{equation}
Suppose $S$ is bounded. Consider the polynomial optimization problem
\begin{align*}
& \gamma^*= \min_{x \in \mathbb{R}^n}  \quad f(x) \nonumber \\
& \text{subject to } \;\;\, g_i(x) \geq 0 \text{ for } i=1,\cdots,m \nonumber \\
& \hspace{0.8in} h_j(x)=0 \text { for } j=1,\cdots,r.
\end{align*}
Define the following cone of polynomials which are positive over $S$.
\begin{equation}
M_{g,h} \hspace{-0.03in} := \hspace{-0.03in} \left\lbrace \hspace{-0.03in} m \hspace{-0.03in} \in \hspace{-0.03in} \mathbb{R}[x] :  m(x) \hspace{-0.03in} - \hspace{-0.04in} \sum_{i=1}^m s_i(x) g_i(x) \hspace{-0.03in} - \hspace{-0.04in} \sum_{i=1}^r t_i(x) h_i(x) \text{ is SOS}, s_i \in \Sigma_{2d} , t_i \in \mathbb{R}[x]  \right\rbrace,
\label{eq:putinar_cone2}
\end{equation}
where $\Sigma_{2d}$ denotes the cone of SOS polynomials of degree $2d$. 
From Putinar's Positivstellensatz (Theorem~\ref{putinar}) it follows that if the Cone~\eqref{eq:putinar_cone2} is Archimedean, then the solution to the following SOS program is a lower bound on $\gamma^*$. Given $d \in \mathbb{N}$, define
\begin{align}
& \gamma^d := \max_{\gamma \in \mathbb{R}, s_i, t_i} \;\; \gamma \nonumber \\
&\text{subject to } \; f(x)-\gamma - \sum_{i=1}^m s_i(x) g_i(x) -  \sum_{i=1}^r t_i(x) h_i(x) \text{ is SOS },\, t_i \in \mathbb{R}[x], s_i \in \Sigma_{2d}.
\label{eq:SOS_putinar}
\end{align}
On the other hand, every $F \in \Sigma_{2d}$ has a quadratic representation with a positive semi-definite matrix. To see this, suppose $F(x) = \sum_{i} q_i(x)^2$, where $q_i$ are polynomials of degree $d$. Each $q_i$ can be written in the canonical basis as $q_i(x) = c_i^T z_{n,d}(x)$, where $z_{n,d}(x)$ is the vector of all $n-$variate monomials of degree $d$ or less.
Hence, we can write $F$ as
\begin{align*}
F(x) & = \sum_{i} q_i(x)^2 = \sum_{i} z_{n,d}(x)^T c_i c_i^T z_{n,d}(x) \\
& = z_{n,d}(x)^T \left( \sum_{i} c_i c_i^T \right) z_{n,d}(x) = z_{n,d}(x)^T Q z_{n,d}(x),
\end{align*}
where clearly $Q \geq 0$. Therefore, for given $\gamma \in \mathbb{R}$ and $d \in \mathbb{N}$, Problem~\eqref{eq:SOS_putinar} can be formulated as the following linear matrix inequality.
\begin{align}
&\text{Find } \quad\, Q_i \geq 0 \text{ and } P_j \; \text{ for } i=0,\cdots,m \text{ and } j=1, \cdots, r \nonumber \\ 
& \hspace*{-0.15in} \text{such that } \; f(x)-\gamma = z_{n,d}^T(x)\left( Q_0+\sum_{i=1}^m Q_i g_i(x) + \sum_{j=1}^r P_j h_j(x) \right)z_{n,d}(x),
\label{eq:putinar_feasibility}
\end{align}
where $Q_i$ and $P_j \in \mathbb{S}^{N}$, where $\mathbb{S}^N$ is the subspace of symmetric matrices in $\mathbb{R}^{N \times N}$ and $N:=\dbinom{n+d}{d}$. See~\cite{SDP_SIAMbook_parrilo} for methods of solving SOS programs. Also~\cite{sostools2013} provide a MATLAB package called SOSTOOLs for solving SOS programs. 

%

If the Cone~\eqref{eq:putinar_cone2} is not Archimedean, then we can use Schmudgen's Positivstellensatz to obtain the following SOS program with solution $\gamma^d \leq \gamma^*$.
\begin{align}
& \gamma^d =  \max_{\gamma \in \mathbb{R}, s_i \in \Sigma_{2d}, t_i \in \mathbb{R}[x]} \quad \gamma \nonumber \\
&\text{subject to } \; f(x)-\gamma = 1 + \hspace*{-0.15in} \sum_{\lambda \in \{0,1\}^m} \hspace*{-0.085in} s_\lambda(x) g_1(x)^{\lambda_1} \cdots g_m(x)^{\lambda_m} + \sum_{i=1}^r t_i(x) h_i(x).
\label{eq:SOS_schmudgen}
\end{align}

The Positivstellensatz and SOS programming can also be applied to polynomial optimization over a more general form of semi-algebraic sets defined as
\[
T:= \{ x\in \mathbb{R}^n \hspace{-0.05in} : \hspace{-0.05in} g_i(x) \geq 0,\, i =1, \cdots,m, h_j(x) = 0, j=1, \cdots,r,  q_k(x) \neq 0, k=1, \cdots, l \}.
\]
It can be shown that $T = \emptyset$ if and only if 
\begin{align*}
 \hat{T}:= \{ (x,y) \in \mathbb{R}^{n+l}: g_i(x) \geq 0,\, i =1, \cdots,m, \,  h_j(x) & =  0, \, j=1, \cdots,r, \\
& y_k q_k(x) = 1, \, k=1, \cdots, l \} = \emptyset.
\end{align*}
Thus, for any $f \in \mathbb{R}[x]$, we have 
\[
\min\limits_{x\in T} f(x)= \min\limits_{(x,y) \in \hat{T}} f(x).
\]
Therefore, to find lower bounds on $\min_{x\in T} f(x)$, one can apply SOS programming and Putinar's Positivstellensatzs to $\min\limits_{(x,y) \in \hat{T}} f(x)$. \\

\subsection{Case 5: Tests for Non-negativity on $\mathbb{R}^n$:}

The following result from~\cite{habicht} defines a test for non-negativity of homogeneous polynomials over $\mathbb{R}^n$.
\begin{mythm}(Habicht theorem)
For every homogeneous polynomial $f$ that satisfies $f(x_1,\cdots,x_n) > 0$ for all $x \in \mathbb{R}^n \setminus \{0\}$, there exists some $e \geq 0$ such that all of the coefficients of
\begin{equation}
\left( \sum_{i=1}^n x_i^2 \right)^e f(x_1,\cdots,x_n)
\label{eq:habicht_product}
\end{equation}
are positive. In particular, the product is a sum of squares of monomials. 
\label{thm:habicht}
\end{mythm}

Using this theorem, one can verify non-negativity of any homogeneous polynomial $f$ over $\mathbb{R}^n$ by multiplying $F$ repeatedly by $\sum_{i=1}^n x_i^2$. If for some $e \in \mathbb{N}$, the Product~\eqref{eq:habicht_product} has all positive coefficients, then $f \geq 0$. We can define an alternative test for non-negativity over $\mathbb{R}^n$ using the following theorem~(\cite{polya_Rn}).

\begin{mythm}
Define $E_n:=\{-1,1\}^n$. Suppose a polynomial $f(x_1,\cdots,x_n)$ of degree $d$ satisfies $f(x_1,\cdots,x_n) > 0$ for all $x \in \mathbb{R}^n$ and its homogenization\footnote{Associated to every polynomial $f(x_1, \cdots, x_n), \, x \in \mathbb{R}^n$ of degree $d$, there exists a degree $e$ homogeneous polynomial $h(x_1,\cdots,x_n,y):= y^e f(\frac{x_1}{y}, \cdots, \frac{x_n}{y})$, where $e \geq d$.
}
is positive definite. Then
\begin{enumerate}
\item there exist $\lambda_e \geq 0$ and coefficients $c_\alpha \in \mathbb{R}$ such that
\begin{equation}
\left(1+ e^Tx \right)^{\lambda_e} f(x_1, \cdots,x_n) = \sum_{\alpha \in  I_e} c_\alpha x_1^{\alpha_1} \cdots x_n^{\alpha_n} \text{ for all } e \in E_n,
\label{eq:santos_product}
\end{equation}  
where $I_e := \{ \alpha \in \mathbb{N}^n: \Vert \alpha \Vert_1 \leq d+\lambda_e \}$ and $sgn(c_\alpha) = e_1^{\alpha_1} \cdots e_n^{\alpha_n}$.

\item there exist positive $N,D \in \mathbb{R}[x_1^2,\cdots,x_n^2,f^2]$ such that $f=\frac{N}{D}$.
\end{enumerate}
\label{thm:polya_satos}
\end{mythm}
Based on the converse of Theorem~\ref{thm:polya_satos}, we propose the following test for non-negativity of polynomials over the cone $\Lambda_e:=\{x \in \mathbb{R}^n: sgn(x_i)=e_i, i=1,\cdots,n \}$ for some $e \in E_n$. Multiply a given polynomial $f$ repeatedly by $1+e^Tx$ for some $ e \in E_n$. If there exists some $\lambda_e \geq 0$ such that $sgn(c_\alpha) = e_1^{\alpha_1} \cdots e_n^{\alpha_n}$, then~\eqref{eq:santos_product} clearly implies that $f(x) \geq 0$ for all $x \in \Lambda_e$.
Since $\mathbb{R}^n = \cup_{e \in E_n} \Lambda_e$, we can repeat the test $2^n$ times to obtain a test for non-negativity of $f$ over $\mathbb{R}^n$.

The second part of Theorem~\ref{thm:polya_satos} gives a solution to Hilbert's $17^{th}$ problem  (see Section~\ref{sec:history}). For a construction of this solution (i.e., numerator $N$ and denominator $D$) see~\cite{polya_Rn}.


\chapter{SEMI-DEFINITE PROGRAMMING AND INTERIOR-POINT ALGORITHMS} \label{chp:convex_optim} 

As discussed in Chapter~\ref{chp:background}, Polya's theorem, Handelman's theorem and the Positivstellensatz results can be used to approximate the minimum of a polynomial over simplicies, hypercubes, polytopes and semi-algebraic sets. We showed that these theorems define sequences of Linear/Semi-Definite Programs (SDPs) whose solutions define lower bounds on the objective of the polynomial optimization problem. In this section, we focus on solving these SDPs. In particular, we discuss the primal and dual forms of semi-definite programming problems and introduce a state-of-the-art primal-dual interior-point algorithm for solving SDPs. In Section~\ref{sec:SDPSOLVER}, we will propose a new parallel version of this algorithm - an algorithm which is specifically designed to solve the SDPs defined by applying Polya's theorem to  optimization of polynomials arising in robust stability and control problems.\\

\section{Convex Optimization and Duality}

Let us define the constrained optimization problem
\begin{align}
& f^* := \min_{x \in \mathbb{R}^n} \;\; f_0(x) \nonumber \\
& \text{subject to } \;\; f_i(x) \leq 0, \quad i=1, \cdots, p \nonumber \\
& \hspace{0.8in} h_i(x) = 0, \quad i=1, \cdots, q,
\label{eq:constrnd_optim}
\end{align}
where $f_i: \mathbb{R}^n \rightarrow \mathbb{R}$ and $h_i: \mathbb{R}^n  \rightarrow \mathbb{R}$. For every problem of Form~\eqref{eq:constrnd_optim}, one can define the \textit{Lagrangian function} $L: \mathbb{R}^n \times \mathbb{R}^p \times \mathbb{R}^q$ as
\begin{equation}
L(x,\lambda, \nu) := f_0(x) + \sum_{i=1}^p \lambda_i f_i(x) + \sum_{i=1}^q \nu_i h_i(x),
\label{eq:Lagrangian}
\end{equation}
where $\lambda_i \in \mathbb{R}$ and $\nu_i \in \mathbb{R}$ are called the \textit{Lagrange multipliers} associated with the inequality constraints and the equality constraints in~\eqref{eq:constrnd_optim}, respectively. The vectors $\lambda=[\lambda_1, \cdots, \lambda_p]$ and $\nu=[\nu_1, \cdots, \nu_q]$ are called the \textit{dual variables} of Problem~\eqref{eq:constrnd_optim}. Let us define the \textit{Lagrange dual function} $g: \mathbb{R}^p \times \mathbb{R}^q \rightarrow \mathbb{R}$ as
\[
g(\lambda,\nu) := \inf_{x} \left( f_0(x) + \sum_{i=1}^p \lambda_i f_i(x) + \sum_{i=1}^q \nu_i h_i(x) \right).
\]
The Lagrange dual functions have some interesting properties. First, because the Lagrangian is affine in $\lambda_i$ and $\nu_i$ and the pointwise infimum of a family of affine functions is concave (\cite{boyd2004convex}), $g$ is a concave function. Second, it is easy to show that the dual functions yield lower bounds on $f^*$ as define in~\eqref{eq:constrnd_optim}, i.e.,
$g(\lambda,\nu) \leq p^*$. To find the best lower bound on $f^*$ using the Lagrange dual function, one can solve the \textit{Lagrange dual problem} defined as
\begin{align}
& d^* :=  \max_{\lambda, \nu} \;\; g(\lambda, \nu) \nonumber \\
& \text{subject to }\;\; \lambda \geq 0.
\label{eq:dual}
\end{align}
Every pair $(\lambda,\nu)$ which satisfies $\lambda \geq 0$ and $g(\lambda,\nu) > -\infty$ is called a \textit{dual feasible} point for Problem~\eqref{eq:dual}. Likewise, every $x \in \mathbb{R}^n$ satisfying $f_i(x) \leq 0$ for $i=1, \cdots, p$ and $h_i(x) \leq 0$ for  $i=1, \cdots, q$ is a \textit{primal feasible} point for Problem~\eqref{eq:constrnd_optim}. Dual feasible points can be used to bound sub-optimality of a primal feasible point. In particular, for every primal feasible point $x$ and dual feasible point $(\lambda,\nu)$, 
\[
f_0(x) - f^* \leq f_0(x) - g(\lambda,\nu),
\]
where $f_0(x) - g(\lambda,\nu)$ is called the \textit{duality gap} associated with $x$ and $(\lambda,\nu)$. For certain problems, the duality gap associated with primal optimal point $x^*$ and dual optimal point $(\lambda^*, \nu^*)$ is zero, i.e.,
\[
f_0(x^*) = f^* = d^* = g(\lambda^*, \nu^*).
\]
This property is often called \textit{strong duality}. One important class of problems which usually posses this property is convex optimization problems. A convex optimization problem is an optimization problem of Form~\eqref{eq:constrnd_optim}, where the functions $f_i, \, i=0, \cdots, p$ are convex\footnote{A function $f: \mathbb{R}^n \rightarrow \mathbb{R}$ is convex if $f(\alpha x + \beta y) \leq \alpha f(x) + \beta f(y)$ for all $x,y \in \mathbb{R}^n$ and for all $\alpha, \beta \in \mathbb{R}$ such that $\alpha+\beta=1$ and $\alpha,\beta \geq 0$.} and $h_i, \, i=1, \cdots,q$ are affine. For example, the Lagrange dual problem~\eqref{eq:dual} is by definition a convex problem (it is a maximization of a concave function) whether or not its primal (Eq. \eqref{eq:constrnd_optim}) is convex. It can be shown that~(\cite{slater2014lagrange}) if the primal problem~\eqref{eq:constrnd_optim} is convex and there exists some $x \in \mathbb{R}^n$ such that
\begin{align}
& f_i(x) < 0 \quad \text{ for } i=1, \cdots, p \;\; \text{ and } \nonumber \\
& h_i(x) = 0 \quad \text{ for } i=1, \cdots, q,
\label{eq:slater}
\end{align}
then strong duality holds. Strong duality can be exploited to solve the primal problem via its dual. This is useful specially when the dual is easier or computationally less expensive to solve. Suppose a dual optimal solution $(\lambda^*,\nu^*)$ is known and strong duality holds. If
\[
x^* := \argmin L(x, \lambda^*, \nu^*)
\] is unique and primal feasible, $x^*$ is the primal optimal solution.

\section{Descent Algorithms for Convex Optimization}

Suppose $f_0: \mathbb{R}^n \rightarrow \mathbb{R}$ is differentiable. For $\hat{x}$ to be a minimum of $f_0$, the necessary condition is that $\left[ \nabla_x \, f_0(x) \right]_{x=\hat{x}} = 0$. The Karush-Kuhn-Tucker (KKT) conditions~(\cite{KKT}) generalize this necessary condition for the constrained optimization problem~\eqref{eq:constrnd_optim}, under the assumption that the functions $f_i$ and $g_i$ are differentiable. The KKT conditions can be stated as follows: Suppose $x^* \in \mathbb{R}^n$ is a primal optimal point for~\eqref{eq:constrnd_optim} and $\lambda_i^*, \, i=1, \cdots, p$ and $\nu_i^*, \, i=1, \cdots,q$ are dual optimal points for~\eqref{eq:dual}. Moreover, suppose the strong duality holds, i.e., $f_0(x^*) = g(\lambda^*, \nu^*)$. Then, the optimal primal and dual points satisfy the following.
\begin{align}
 \left[ \nabla_x f_0(x)  +  \sum_{i=1}^p \lambda_i^* \nabla_x f_i(x) + \sum_{i=1}^q \nu_i^*  \nabla_x  h_i(x) \right]_{x=x^*} = & \; 0 \nonumber \\
 f_i(x^*) \leq \; 0 && \text{for } \; i=1, \cdots,p \nonumber \\
 h_i(x^*) = \; 0 && \text{for } \; i=1, \cdots,q \nonumber \\
 \lambda_i^* \geq \; 0 && \text{for } \; i=1, \cdots,p \nonumber \\
 \lambda_i^* f_i(x^*) = \; 0 && \text{for } \; i=1, \cdots,p.
 \label{eq:KKT}
\end{align}
The first line follows from the fact that $x^*$ is a minimizer of the Lagrangian $L(x,\lambda^*, \nu^*)$. The second, third and fourth lines indicate that $x^*$ and $(\lambda^*, \nu^*)$ are primal and dual feasible. The last line is called the \textit{complementary slackness} and follows from strong duality. This condition implies that for $i=1, \cdots, p$, either the $i^{th}$ primal constraint must be active at $x^*$ (i.e., $f_i(x^*) = 0$) or its corresponding optimal dual variable $\lambda_i^*$ must be zero.

In general, the KKT conditions are only necessary conditions for optimality.
Indeed, under certain regularity conditions, local minima of the primal Problem~\eqref{eq:constrnd_optim} satisfy the KKT conditions. However, when the primal problem is convex and there exists  $x \in \mathbb{R}^n$ which satisfies~\eqref{eq:slater}, the KKT conditions become necessary and sufficient. Motivated by this result, many of the existing convex optimization algorithms are in principle algorithms for solving the KKT conditions iteratively. These algorithms are often called \textit{descent} \textit{algorithms} because they generate a sequence $\{x^k\}_{k=1,2,  \cdots}$ of primal feasible solutions which satisfy
\begin{equation}
f_0(x^{k}) > f_0(x^{k+1}) \quad \text{ for } k=1,2,3, \cdots,
\label{eq:descent}
\end{equation}
unless $x^k$ is optimal. One example of descent algorithms is the Newton's algorithm. Given a primal feasible starting point $x^0$, Newton's algorithm finds a sequence of \textit{search directions} $\Delta x^k \in \mathbb{R}^n$ and \textit{step length} $t^k \in \mathbb{R}^+$ such that all the iterates
\[
x^{k+1} = x^k + t^k \Delta x^k \quad k = 1, 2, 3, \cdots  
\]
are feasible and satisfy~\eqref{eq:descent}. Given a primal feasible point $x^k$, Newton's algorithm calculates the search directions $\Delta x^k$ by first defining the convex optimization problem
\begin{align}
& \hat{f}_k := \min_{v \in \mathbb{R}} \quad  \left[ f_0(x) + \nabla_{x} f_0(x)^T v + \dfrac{1}{2} v^T \nabla_x^2 f_0(x) v \right]_{x=x^k} \nonumber \\
& \text{subject to } \;\; h_i(x^k+v) = 0 \qquad \text{ for } i=1, \cdots, m.
\label{eq:2nd_taylor}
\end{align}
The objective function of Problem~\eqref{eq:2nd_taylor} is the second-order Taylor's approximation of the objective function $f_0(x)$ at $x=x^k$. Then, the KKT optimality conditions for Problem~\eqref{eq:2nd_taylor} yield the following system of linear equations.
\begin{equation}
\begin{bmatrix} 
\left[ \nabla^2_x f_0(x) \right]_{x=x^k}  & D_x h(x)^T \\ 
D_x h(x)  & \mathbf{0}
\end{bmatrix} 
\begin{bmatrix}
\Delta x^k \\ 
\nu^k
\end{bmatrix} = 
\begin{bmatrix}
- \left[ \nabla_x f_0(x) \right]_{x=x^k} \\ 
\mathbf{0}
\end{bmatrix},
\label{eq:KKT_taylor}
\end{equation}
where $D_x h(x):= \left[ \nabla_x h_1(x), \cdots, \nabla_x h_m(x) \right]^T$. If the coefficient matrix in~\eqref{eq:KKT_taylor} is nonsingular, then there exist a unique Newton's search direction $\Delta x^k$ and optimal dual point $\nu^k$ for the dual to Problem~\eqref{eq:KKT_taylor}. Finally, Newton's algorithm calculates the new iterate as $x^{k+1} = x^k + t^k \Delta x^k$, where a step length $t^k$ can be obtained using 
a line search method such as \textit{backtracking}~(\cite{dennis1996numerical}) or bisection. 
A typical stopping criterion for Newton's algorithm is $f_0(x^k) - \hat{f}^k \leq \epsilon$ for some desired $\epsilon > 0$, where recall that $\hat{f}^k$ is the minimum of the second-order Taylor's approximation of $f_0$ at $x^k$, subject to the equality constraints in~\eqref{eq:2nd_taylor}. The difference between $f_0(x^k)$ and  $\hat{f}^k$ can also be interpreted as the size of Newton's search direction defined by the following weighted norm of $\Delta x^k$:
\begin{equation}
f_0(x^k)-\hat{f}^k = \Vert \Delta x^k \Vert_{\nabla^2_x f_0(x) \vert_{x=x^k}} := 2 (\Delta x{^k})^T \, \nabla_x^2 f_0(x) \vert_{x=x^k} \, \Delta x^k.
\label{eq:Newton_stop}
\end{equation}
For a comprehensive discussion on the complexity and convergence of Newton's algorithm, refer to~\cite{boyd2004convex}.

\section{Interior-point Algorithms for Convex Problems with Inequality Constraints}
\label{sec:barrier}

Suppose in Problem~\eqref{eq:constrnd_optim}, $f_i$ are convex and differentiable and $h_i$ are affine and differentiable. One of the most successful class of algorithms for solving this type of problems is \textit{interior-point} algorithms. Typically,
interior-point algorithms solve this problem in two steps: 1- Reducing the problem to a sequence of convex optimization programs with only linear equality constraints; and 2- Applying a descent algorithm, e.g., Newton's algorithm, to solve the equality constrained problem. One way to define this sequence of equality constrained problems is to incorporate the inequality constraints into the objective function using \textit{barrier functions}. For example, by using logarithmic barrier functions one can approximate Problem~\eqref{eq:constrnd_optim} as
\begin{align}
& \min_{x \in \mathbb{R}^n} \;\; f_0(x) - \sum_{i=1}^p \left( \dfrac{1}{b} \right) \log (-f_i(x)) \nonumber \\
& \text{subject to } \;\; h_i(x) = 0, \quad i=1, \cdots, q. 
\label{eq:barrier}
\end{align}
for some $b > 0$. Clearly, if any of the inequality constraints becomes active ($f_i(x) \rightarrow 0$), then the objective function blows up. Thus, any solution to Problem~\eqref{eq:barrier} lies in the interior (as the name `interior-point' suggests) of the feasible set 
\[
\{ x \in \mathbb{R}^n : f_i(x) \leq 0 , h_j(x)=0,\, \text{ for } i=1, \cdots,p \; \text{ and } \; j=1, \cdots, q \}.
\]
 Since Problem~\eqref{eq:barrier} is convex, one can use Newton's algorithm to find the optimal solution $x^*_b$ for any $b > 0$. In particular, given $b > 0$ and feasible $x^0$, Newton's algorithm finds a sequence $\{x^k\}_{k=1,2,\cdots} \rightarrow x^*_b$ by solving the modified KKT conditions
\begin{align}
\renewcommand\arraystretch{0.9} 
\begin{bmatrix}
\left[ b \nabla^2_x f_0(x) - \nabla_x^2  \sum\limits_{i=1}^p \log(-f_i(x)) \right]_{x=x^k}  & D_x h(x)^T \\ 
D_x h(x)  & \mathbf{0}
\end{bmatrix} &
\begin{bmatrix}
\Delta x^k \\ 
\nu^k
\end{bmatrix} = \nonumber \\
 & \hspace{-0.75in} \begin{bmatrix}
- \left[ \nabla_x f_0(x) + \nabla \sum\limits_{i=1}^p \log(-f_i(x)) \right]_{x=x^k} \\ 
\mathbf{0}
\end{bmatrix}
\label{eq:KKTmod_taylor}
\end{align}
for $\Delta x^k$ and $\nu^k$ and setting $x^{k+1} = x^k + t^k x^k$.
The set of optimal solutions $x^*_b$ for all $b > 0$ is called the \textit{central path}. Corresponding to any $x^*_b$ in the central path, one can verify that 
\[
\lambda_i^* =  -\dfrac{1}{b \, f_i(x^*_b)} \; \text{ for } i = 1, \cdots, p \quad \text{ and } \quad \nu^* = \dfrac{\nu^k}{b} \in \mathbb{R}^q, 
\]
are dual feasible and together with $x^*_b$ yield the duality gap $\frac{p}{b}$.
This indicates that as $b \rightarrow \infty$, $x^*_b$ converges to the optimal solution of Problem~\ref{eq:constrnd_optim} under the assumption that $f_i$ are convex and differentiable and $h_i$ are affine and differentiable. Based on this result, we can summarize the interior-point barrier algorithm for inequality constrained problems in Algorithm~\ref{alg:barrier}.

\begin{algorithm}
\textbf{Input:}
Convex functions $f_0, \cdots, f_p$; affine functions $h_1, \cdots, h_q$; a feasible starting point $x^0$; initial barrier parameter $b_0$;  tolerances $\epsilon_b > 0 $ and $\epsilon_N > 0$  for stopping criteria. \vspace{0.1in}

\textbf{Initialization:}\\
Set $b = b_0$.\\
Choose $\mu > 1$.\vspace{0.1in}\\

\textbf{Barrier Algorithm:}\\
\While{$\dfrac{p}{b} > \epsilon_b$}{\vspace{0.05in}
	Set $x=x^0$.\\	
	Set $\Delta x = 10^{100} \cdot \mathbf{1}_n$.\vspace{0.05in} \\
	\textbf{Newton's Algorithm:}\\
	\While{$\Vert \Delta x \Vert$ as defined in~\eqref{eq:Newton_stop} is greater than or equal to $\epsilon_N$}{ \vspace{0.05in}

	Calculate Newton's search direction $\Delta x$ by solving the system of linear equations in~\eqref{eq:KKTmod_taylor}. \vspace{0.05in}	\\	

	Choose step length $t$ using backtracking line search. \vspace{0.05in} \\
	Update Newton's iterate as $x := x + t \, \Delta x$. \vspace{0.05in} \\
	}
	Set $x^*_b = x$. \vspace{0.05in} \\
	Update the barrier parameter as $b := \mu b$.

} \vspace{0.1in}

\textbf{Output:} $x^*_b$: A $\frac{p}{b}$-suboptimal solution  to Problem~\eqref{eq:constrnd_optim}.
\vspace{0.15in}\caption{Barrier algorithm for inequality constrained convex optimization problems}
\label{alg:barrier}
\end{algorithm}

An alternative subclass of interior-point algorithms for solving inequality constrained problems is the \textit{primal-dual} algorithms. Similar to the barrier algorithm, primal-dual algorithms find their search direction by solving the KKT optimality conditions. However, instead of incorporating the inequality constraints into the objective function (equivalently, eliminating the dual variable $\lambda$ from the KKT condition~\eqref{eq:KKT}), primal-dual algorithms simultaneously solve the primal problem and its dual by computing independent Newton's search directions $\Delta x \in \mathbb{R}^n$, $\Delta \lambda \in \mathbb{R}^p$ and $\Delta \nu \in \mathbb{R}^q$ for primal and dual variables $x, \lambda$ and $\nu$. Given a feasible point $(x^k, \lambda^k, \nu^k) $ for Problem~\eqref{eq:constrnd_optim} and $b > 0$, the basic version of primal-dual algorithms computes the search directions $(\Delta x^k, \Delta \lambda^k, \Delta \nu^k)$ by approximating the modified KKT conditions
\[
R(x, \lambda, \nu, b) = 
\begin{bmatrix}
\nabla f_0(x) + D_x f(x)^T \lambda + D_x h(x)^T \nu\\ 
\begin{array}{c}
 \lambda_1 f_1(x) - \dfrac{1}{b} \\ 
 \vdots \\ 
 \lambda_p f_p(x) - \dfrac{1}{b}
 \end{array}  \\ 
h_1(x) \\ 
\vdots \\
h_q(x)
\end{bmatrix} = 0
\]
at the point $(x^k, \lambda^k, \nu^k)$ as
\begin{align}
 R(x^k + \Delta x^k, \lambda^k + \Delta \lambda^k,& \nu^k + \nabla \nu^k, b) \nonumber  \\
\approx &\; R(x,\lambda,\nu,b)
+
\begin{bmatrix}
\left[ \nabla R_1(x,\lambda,\nu,b)^T \right]_{\substack{x = x^k \\ \lambda = \lambda^k \\ \nu = \nu^k}} \\ 
\vdots \\ 
\left[ \nabla R_{n+p+q}(x,\lambda,\nu,b)^T \right]_{\substack{x = x^k \\ \lambda = \lambda^k \\ \nu = \nu^k}}
\end{bmatrix}
\begin{bmatrix}
\Delta x^k \\ 
\Delta \lambda^k \\ 
\Delta \nu^k
\end{bmatrix}
 = 0,
\label{eq:basic_primaldual}
\end{align}
and solving for $(\Delta x^k, \Delta \lambda^k, \Delta \nu^k)$.
The primal-dual iterates are then updated according to
\[
x^{k+1} = x^k + t^k \Delta x^k, \quad \lambda^{k+1} = \lambda^k + t^k \Delta \lambda^k, \quad \nu^{k+1} = \nu^k + t^k \Delta \nu^k.
\]
Similar to the barrier algorithm, the duality gap corresponding to any feasible primal-dual iterate $(x^k, \lambda^k, \nu^k)$ is $\dfrac{p}{b}$. Thus, as $b \rightarrow \infty$ in~\eqref{eq:basic_primaldual}, the resulting iterates converge to the optimal solution of Problem~\eqref{eq:constrnd_optim}, assuming that $f_i$ are convex and $h_i$ are affine. In the sequel, we describe a primal-dual algorithm for solving semi-definite programs - a class of convex optimization problems which has several applications in control~theory.

\section{Semi-definite Programming}
\label{sec:SDP_SDP}

Consider the delay-differential equation
\begin{equation}
\dot{x}(t) = A x(t) + \sum_{i=1}^N A_i x(t-\tau_i)
\label{eq:sys_delay}
\end{equation} 
where $x(t) \in \mathbb{R}^n$ and $\tau_i > 0, \, i=1, \cdots,N$. From~\cite{repin1965quadratic}, a sufficient condition for asymptotic stability of this system is existence of $P_0 > 0, \cdots, P_N > 0$ such that the quadratic functional 
\[
V(x,t) = x^T(t) P_0 x(t) + \sum_{i=1}^{N} \int_{0}^{\tau_i} x(t-s)^T P_i \, x(t-s) ds
\]
satisfies $\dot{V}(x,t) < 0$ for all $x(t) \in \mathbb{R}^n \setminus \{0\}$ and $t > 0$. The derivative $\dot{V}(x,t)$ can be expanded as
\begin{align*}
\dot{V}(x,t) = x(t)^T &\left(  A^T P_0 + P_0 A + \sum_{i=1}^N P_i  \right) x(t) + x(t)^T \left( \sum_{i=1}^N P_0 A_i x(t-\tau_i) \right)  \\
 +  &\left( \sum_{i=1}^N x(t-\tau_i)^T A_i^T P_0  \right) x(t) - \sum_{i=1}^N x(t-\tau_i)^T P_i x(t-\tau_i).
\end{align*}
Thus, $\dot{V}(x,t) = z(t)^T Q(P_0, \cdots,P_N) z(t)$, where $z(t):= \left[
x(t) \;\; x(t-\tau_1) \;\; \cdots \;\; x(t-\tau_N) \right]^T$ and
\[
Q(P_0, \cdots,P_N):=
\begin{bmatrix}
A^TP_0 + P_0 A + \sum_{i=1}^N  P_i& P_0 A_1 & \cdots & P_0 A_N \\ 
A_1^T P_0 & -P_1 & \cdots & 0 \\ 
\vdots & \vdots & \ddots & \vdots \\ 
A_N^T P_0 & 0 & \cdots & -P_N
\end{bmatrix}. \]
Thus, stability of System~\eqref{eq:sys_delay} can be verified by solving the following feasibility problem:
\begin{align}
 &  \quad\text{Find } \quad\; P_0 > 0, \cdots, P_N > 0 \nonumber \\
& \text{such that } \;Q(P_0, \cdots,P_N) < 0.
\label{eq:delay_feasibility}
\end{align}
Now let us parameterize each $P_i$ as
\[
P_i(y_{iL+1}, \cdots, y_{(i+1)L}) = 
\begin{bmatrix}
y_{iL+1} & y_{iL+2} & \cdots & y_{iL+n} \\ 
y_{iL+2} & y_{iL+n+1} & \cdots & y_{iL+2n-1} \\ 
\vdots &  & \ddots & \vdots \\ 
y_{iL+n} & y_{iL+2n-1} & \cdots & y_{(i+1)L}
\end{bmatrix}
\]
for $i=0, \cdots, N$, where $L:=n(n+1)/2$ and $y_j \in \mathbb{R}$ for $j=0, \cdots, (N+1)L$. Then, we can formulate the problem of stability of System~\eqref{eq:sys_delay} as the convex optimization problem
\begin{align}
& \min_{ \substack{ y \in \mathbb{R}^{(N+1)L} \\ Z \in \mathbb{S}^{(N+2)n} }} \quad \;\;\;   \langle \mathbf{1}_{(N+1)L},y \rangle \nonumber \\
& \text{subject to }\;\;  \sum_{i=1}^{(N+1)L} F_i y_i = Z \nonumber \\
& \hspace{0.91in} Z \geq 0,
\label{eq:SDP_exmp}
\end{align}
where the matrices $F_i \in \mathbb{S}^{(N+2)n}$ for $i=1, \cdots, (N+1)L$ are defined as
\begin{equation}
F_i = \text{diag}\{P_0(x_0, \cdots, x_L),Q(P_0(x_0, \cdots, x_L), \cdots, P_N(x_{NL+1}, \cdots, x_{(N+1)L}))\},
\label{eq:Fi}
\end{equation}
where 
\[
x_j = \begin{cases}
1 & j=i \\
0 & j \neq i 
\end{cases} \qquad \text{for } j=1, \cdots, (N+1)L.
\]
Problem~\eqref{eq:SDP_exmp} is an example of the dual form of the Semi-Definite Programming (SDP) problem. We define SDP as the optimization of a linear objective function over the cone of positive definite matrices subject to linear matrix equality and linear matrix inequality constraints. Given $C \in \mathbb{S}^n, B_i \in \mathbb{S}^n$ for $i=1, \cdots,k$, $G_i \in \mathbb{S}^n$ for $i=1, \cdots,l$, $a \in \mathbb{R}^k$ and $b \in \mathbb{R}^l$, the \textit{primal} SDP problem is
\begin{align}
& p^* := \max_{X \in \mathbb{S}^n} \;\;  \text{tr}(CX) \nonumber \\
& \text{subject to } \;\; B(X) = a \nonumber \\
& \hspace{0.8in} G(X) \leq b \nonumber \\
& \hspace{1.05in} X \geq 0,
\label{eq:SDP_primal}
\end{align}
where the linear maps $B: \mathbb{S}^n \rightarrow \mathbb{R}^k$ and $G: \mathbb{S}^n \rightarrow \mathbb{R}^l$ are defined as
\begin{equation}
B(X) = \begin{bmatrix}
\text{tr}(B_1X) \\ 
\text{tr}(B_2X) \\ 
\vdots \\ 
\text{tr}(B_kX)
\end{bmatrix}
\qquad \text{and} \qquad 
G(X) = \begin{bmatrix}
\text{tr}(G_1X) \\ 
\text{tr}(G_2X) \\ 
\vdots \\ 
\text{tr}(G_lX)
\end{bmatrix}.
\label{eq:B_G}
\end{equation}
To derive the dual SDP to Problem~\eqref{eq:SDP_primal}, we employ Lagrange multipliers $t \in \mathbb{R}^l_+$ and $y \in \mathbb{R}^k$ as follows.
\[
p^* = \max_{X \geq 0} \min_{y \in \mathbb{R}^k,t \in \mathbb{R}^l_+ } \quad \text{tr}(CX) + t^T(b-G(X)) + y^T(a-B(x))
\]
Then, from the \textit{min-max inequality}, i.e.,
\[
\max_{u \in U} \min_{v \in V} \; f(u,v) \leq \min_{v \in V} \max_{u \in U} \; f(u,v)
\]
it follows that
\begin{align*}
p^* &\leq \max_{y \in \mathbb{R}^k,t \in \mathbb{R}^l_+ } \min_{X \geq 0} \; \text{tr}(CX) + t^T(b-G(X)) + y^T(a-B(x)) \\
    &= \min_{y \in \mathbb{R}^k,t \in \mathbb{R}^l_+ } \max_{X \geq 0} \; \text{tr}(C- \sum_{i=1}^k B_i y_i - \sum_{i=1}^l G_i t_i)X + a^Ty + b^Tt.
\end{align*}
Note that 
\[
\max\limits_{X \geq 0} \; \text{tr}(C- \sum_{i=1}^k B_i y_i - \sum_{i=1}^l G_i t_i)X < \infty
\]
only if $C- \sum_{i=1}^k B_i y_i - \sum_{i=1}^l G_i t_i \leq 0$. In this case, clearly the maximum occurs when 
\[
C- \sum_{i=1}^k B_i y_i - \sum_{i=1}^l G_i t_i = 0.
\]
Therefore, we can the write \textit{dual} SDP problem as
\begin{align}
& \max_{y \in \mathbb{R}^k, t \in \mathbb{R}^l_+} \;\;\;  a^T y + b^T t  \nonumber \\
& \text{subject to } \;\; \sum_{i=1}^k B_i y_i + \sum_{i=1}^l G_i t_i - C = Z \nonumber \\
& \hspace{0.85in} Z \geq 0.
\label{eq:SDP_dual}
\end{align}
From~\eqref{eq:SDP_exmp} and~\eqref{eq:SDP_dual} it is clear that the problem of stability of the delay-differential Equation~\eqref{eq:sys_delay} can be formulated as the dual SDP defined by the elements
\[
a := \mathbf{1}_{(N+1)L}, \quad b := \mathbf{0}, \quad G_i = \mathbf{0}, \quad C= \mathbf{0}, \quad B_i = F_i,
\]
where we have defined $F_i$ in~\eqref{eq:Fi}.  

SDPs are popular among controls community because not only they can be solved efficiently using convex optimization algorithms, but also a wide variety of problems in controls can be formulated as SDPs; e.g., robust stability~(\cite{bliman2004convex,peres2007}) and robust performance~(\cite{robust_performanceTAC2001,scherer2006lmi}) of uncertain systems, $H_2/H_{\infty}$-optimal filter design~(\cite{li1997linear, H2_Hinf_filtering}), estimation of regions of attraction~(\cite{lall_west2005polynomial,packard_ROA, topcu2010robust}) and reachability sets~(\cite{wang2013polynomial}) of nonlinear systems, stability and control of hybrid systems~(\cite{boukas2006static,papchristodoulou2009robust}) and game theory~(\cite{parrilo2006polynomial}). In the next section, we describe a state-of-the-art primal-dual algorithm by~\cite{helmberg2005interior} for solving SDPs.

\section{A Primal-dual Interior-point Algorithm for Semi-definite Programming}
\label{sec:primal_dual}

Fortunately, there exists several interior-point algorithms in the literature for solving SDPs; e.g., dual scaling~(\cite{dual_scaling,primal}), primal-dual~(\cite{alizadeh,monteiro,helmberg}) and cutting-plane/spectral bundle~(\cite{helmberg2000spectral,krishnan,mahdu}) algorithms. In our study, we are particularly interested in a state-of-the-art primal-dual algorithm proposed by~\cite{helmberg2005interior} mainly because at each iteration, it preserves a certain \textit{property} (see~\eqref{eq:structure}) of the primal and dual search directions. In Section~\ref{sec:SDPSOLVER}, we will exploit this property to propose a distributed parallel version of this algorithm for solving large-scale SDPs in robust and/or nonlinear stability analysis. In the following, we briefly discuss the original version of this algorithm algorithm.

Similar to the barrier method described in Section~\ref{sec:barrier}, we can incorporate the inequality constraints in the dual SDP~\eqref{eq:SDP_dual} using logarithmic barrier functions and the barrier parameter $\mu > 0$ as
\begin{align}
& \max_{y \in \mathbb{R}^k, t \in \mathbb{R}^l} \;\;\;  a^T y + b^T t - \mu \left( \log \det Z + \sum_{i=1}^l \log t_i \right) \nonumber  \\
& \text{subject to } \;\; \sum_{i=1}^k B_i y_i + \sum_{i=1}^l G_i t_i - C = Z.
\label{eq:SDP_barrier}
\end{align}
The Lagrangian for Problem~\eqref{eq:SDP_barrier} is defined as
\begin{align*}
L(X,y,t,Z) :=  a^Ty + b^Tt - \mu & \left( \log \det Z + \sum_{i=1}^l \log t_i \right) \nonumber \\
& \qquad \qquad + \text{tr}\left( \left(Z+C - \sum_{i=1}^k B_i y_i - \sum_{i=1}^l G_i t_i \right)X \right).
\end{align*}
Then, the KKT optimality conditions for Problem~\eqref{eq:SDP_barrier} is $\nabla L(X,y,t,Z) = 0$, which can be expanded as
\begin{align}
&\nabla_X L(X,y,t,Z) = Z + C - \sum_{i=1}^k B_i y_i - \sum_{i=1}^l G_i t_i = 0 \label{eq:KKT_SDP_1}  \\
&\nabla_y L(X,y,t,Z) = a - B(X) = 0 \label{eq:KKT_SDP_2} \\
&\nabla_t L(X,y,t,Z) = b - G(X) -\mu \left[1/{t_1} , \cdots, 1/{t_l} \right]^T = 0 \label{eq:KKT_SDP_3} \\
&\nabla_Z L(X,y,t,Z) = X - \mu Z^{-1} = 0,
\label{eq:KKT_SDP_4}
\end{align}
where $B(X)$ and $G(x)$ are defined in~\eqref{eq:B_G}.

Given a barrier parameter $\mu >0$, at each iteration, the primal-dual algorithm finds a search direction $\Delta s:=[\Delta X, \Delta y, \Delta t, \Delta Z]$ such that the new iterate 
$
[X + \Delta X, y + \Delta y, t + \Delta t, Z + \Delta Z]
$
belongs to the central path, i.e.,
\[
\left\lbrace [X_\mu, y_\mu, t_\mu, Z_\mu] : \mu \in [0, \infty] \text{ and } X_\mu, y_\mu, t_\mu, Z_\mu \text{ satisfy Conditions~\eqref{eq:KKT_SDP_1}-\eqref{eq:KKT_SDP_4}}  \right\rbrace.
\]
Conversely, given a point $[X,y,t,Z]$, one can use~\eqref{eq:KKT_SDP_3} and~\eqref{eq:KKT_SDP_4} to find its corresponding barrier parameter as
\begin{equation}
\mu = \dfrac{\text{tr}(ZX) + \left[1/{t_1} , \cdots, 1/{t_l} \right] (b - G(X))}{n+l}.
\label{eq:mu_corrector}
\end{equation}
The search direction $\Delta s$ of the primal-dual algorithm is the sum of two steps: the \textit{predictor} step $ \Delta \hat{s} := [\Delta \hat{X}, \Delta \hat{y}, \Delta \hat{t}, \Delta \hat{Z}]$ and the \textit{corrector} or \textit{centering} step $ \Delta \bar{s} :=[\Delta \bar{X}, \Delta \bar{y}, \Delta \bar{t}, \Delta \bar{Z}]$. The predictor step is defined as the Newton's step for solving the optimality conditions~\eqref{eq:KKT_SDP_1}-\eqref{eq:KKT_SDP_4} with $\mu = 0$, starting at any point $(X,y,t,Z)$ which satisfies
\begin{equation}
X > 0, \quad Z > 0, \quad t >0, \quad G(X) < b.
\label{eq:starting_point}
\end{equation}
Similar to the Taylor's approximation in~\eqref{eq:basic_primaldual}, we find the Newton's step by solving 
\begin{equation}
\nabla L(X,y,t,Z) + \nabla^2 L(X,y,t,Z) \Delta \hat{s}^T = 0
\label{eq:Taylor_PD_SDP}
\end{equation}
for $\Delta \hat{s}$. Substituting for $\nabla L$ from~\eqref{eq:KKT_SDP_1}-\eqref{eq:KKT_SDP_4} into~\eqref{eq:Taylor_PD_SDP} yields the following system of equations for the predictor step.
\begin{equation}
\begin{bmatrix}
\Lambda_{11} & \Lambda_{12} \\
\Lambda_{21} & \Lambda_{22}
\end{bmatrix}
\begin{bmatrix}
\Delta \hat{y} \\
\Delta \hat{t}
\end{bmatrix} = 
\begin{bmatrix}
B\left(Z^{-1} T X \right) - a) \\
G(Z^{-1} T X) - b)
\end{bmatrix},
\label{eq:deltay_deltat}
\end{equation}
\begin{align}
& \Delta \hat{X} = Z^{-1} T X - Z^{-1} \left( \sum_{i=1}^k B_i \Delta \hat{y} + \sum_{i=1}^l G_i \Delta \hat{t} \right) X -X \label{eq:deltaX} \\
& \Delta \hat{Z} = - T + \sum_{i=1}^k B_i \Delta \hat{y} + \sum_{i=1}^l G_i \Delta \hat{t}
\label{eq:deltaZ}
\end{align}
where
\begin{align*}
& T = -\sum_{i=1}^k B_i y + \sum_{i=1}^l G_i t + C + Z \\
& \Lambda_{11} = \left[B(Z^{-1} B_1 X) \quad \cdots \quad B(Z^{-1} B_k X)\right] \\
& \Lambda_{12} = \left[B(Z^{-1} G_1 X) \quad \cdots \quad B(Z^{-1} G_l X)\right] \\
& \Lambda_{21} = \left[G(Z^{-1} B_1 X) \quad \cdots \quad G(Z^{-1} B_k X)\right] \\
& \Lambda_{22} = \left[G(Z^{-1} G_1 X) \quad \cdots \quad G(Z^{-1} G_l X)\right] + 
\text{diag} \left\lbrace \dfrac{b_1-\text{tr}(G_1X)}{t_1}, \cdots, \dfrac{b_l-\text{tr}(G_lX)}{t_l} \right\rbrace.
\end{align*}
The corrector step is defined as the Newton's step for solving the KKT conditions~\eqref{eq:KKT_SDP_1}-\eqref{eq:KKT_SDP_4}, using the barrier parameter $\mu$ as defined in~\eqref{eq:mu_corrector} and starting at 
\[
[X+\Delta \hat{X}, y + \Delta \hat{y}, t+ \Delta \hat{t}, Z+ \Delta \hat{Z}],
\] where $[X,y,t,Z]$ can be any point satisfying~\eqref{eq:starting_point} and $[\Delta \hat{X}, \Delta \hat{y},\Delta \hat{t}, \Delta \hat{Z}]$ can be calculated using~\eqref{eq:deltay_deltat}-\eqref{eq:deltaZ}. Thus, to derive the corrector step, we substitute for $\nabla L$ from KKT conditions~\eqref{eq:KKT_SDP_1}-\eqref{eq:KKT_SDP_4} into
\begin{equation*}
\left[ \nabla L(\bar{X},\bar{y},\bar{t},\bar{Z}) + \nabla^2 L(\bar{X},\bar{y},\bar{t},\bar{Z}) \right]_{\substack{\bar{X}=X+\Delta \hat{X}  \\ \bar{Z}= Z + \Delta \hat{Z} \\ \bar{y}= y + \Delta \hat{y} \\  \bar{t}=t+ \Delta \hat{t} }} \Delta \bar{s}^T = 0
\end{equation*}
This yields the following set of equations for the corrector step.
\begin{equation}
\begin{bmatrix}
\Lambda_{11} & \Lambda_{12} \\
\Lambda_{21} & \Lambda_{22}
\end{bmatrix}
\begin{bmatrix}
\Delta \bar{y} \\
\Delta \bar{t}
\end{bmatrix}
=
\begin{bmatrix}
\mu B(Z^{-1}) - B(Z^{-1} \Delta \hat{Z} \Delta \hat{X}) \\
\mu \left[\dfrac{1}{t_1} \;\; \cdots \;\; \dfrac{1}{t_l} \right] + G \left( X+\Delta \hat{X} + \mu Z^{-1} -Z^{-1} \Delta \hat{Z} \Delta \hat{X} \right)
\end{bmatrix}
\label{eq:deltay_deltat_bar}
\end{equation}
\begin{align}
& \Delta \bar{X} = Z^{-1} \left( -\Delta \hat{Z} \Delta \hat{X} + \mu I -\Delta \bar{Z}X \right) \label{eq:deltaXbar} \\
& \Delta \bar{Z} = \sum_{i=1}^k B_i \Delta \bar{y} + \sum_{i=1}^l G_i \Delta \bar{t}
\label{eq:deltaZbar}
\end{align}
By solving~\eqref{eq:deltay_deltat}-\eqref{eq:deltaZ} for the predictor step and solving~\eqref{eq:deltay_deltat_bar}-\eqref{eq:deltaZbar} for the corrector step, we can calculate the search direction as
\begin{equation}
\Delta s = \left[\text{Sym}(\Delta \hat{X} + \Delta \bar{X}), \Delta \hat{y} + \Delta \bar{y}, \Delta \hat{t} + \Delta \bar{t}, \Delta \hat{Z} + \Delta \bar{Z} \right],
\label{eq:primal-dual_step}
\end{equation}
where $\text{Sym}(W):=(W+W^T)/2$ is the symmetric part of matrix $W$.
We have provided an outline of the discussed primal-dual algorithm in Algorithm~\ref{alg:primal_dual_centeralized}.

\begin{algorithm}
\textbf{Input:}
SDP elements $C, a, b, B_i$ for $i=1, \cdots,k$ and $G_i$ for $i=1, \cdots, l$; starting point satisfying~\eqref{eq:starting_point}; tolerance $\epsilon > 0 $ the stopping criterion. 

\textbf{Initialization:}\\
Set the duality gap $\gamma = 2 \epsilon$.\vspace{0.1in} \\

\While{duality gap $\gamma > \epsilon$}{\vspace{0.05in}
	
	\textbf{Calculating the predictor step:}\\
	Solve $\Delta \hat{y}$ and $\Delta \hat{t}$ by solving system of the equations in~\eqref{eq:deltay_deltat}.\\	
    Calculate $\Delta \hat{y}$ and $\Delta \hat{t}$ using~\eqref{eq:deltaX} and~\eqref{eq:deltaZ}. \vspace{0.05in}\\

	\textbf{Calculating the corrector step:}\\	
	Calculate the barrier parameter $\mu$ using~\eqref{eq:mu_corrector}.\\	
	Solve $\Delta \bar{y}$ and $\Delta \bar{t}$ by solving system of the equations in~\eqref{eq:deltay_deltat_bar}.\\	
	Calculate $\Delta \bar{y}$ and $\Delta \bar{t}$ using~\eqref{eq:deltaXbar} and~\eqref{eq:deltaZbar}. \vspace{0.05in}\\

	\textbf{Updating the primal and dual variables:}\\
	Calculate the search direction as
	\[
	\Delta X := \text{Sym}(\Delta \hat{X} + \Delta \bar{X}), \; \Delta y := \Delta \hat{y} + \Delta \bar{y}, \; \Delta t:= \Delta \hat{t} + \Delta \bar{t}, \; \Delta Z := \Delta \hat{Z} + \Delta \bar{Z}.
	\]
	Calculate primal and dual step lengths $\alpha_p$ and $\alpha_d$ using an appropriate line-search algorithm. \\
	
	Set the primal and dual variables as
	\[
	X := X + \alpha_p \Delta X, \quad y := y + \Delta y, \quad t := t+ \Delta t, \quad Z:= Z + \Delta Z.
	\]
	
	Calculate the duality gap as $ \gamma = \text{tr}(CX) - (a^Ty + b^Tt)$.

}

\textbf{Output:} $\left[X^*, y^*, t^*, Z^*\right]$: A $\gamma$-suboptimal solution to Problems~\eqref{eq:SDP_primal} and~\eqref{eq:SDP_dual}.
\vspace{0.15in}\caption{An interior-point central-path primal-dual algorithm for SDP}
\label{alg:primal_dual_centeralized}
\end{algorithm}


\chapter{PARALLEL ALGORITHMS FOR ROBUST STABILITY ANALYSIS OVER SIMPLEX}
\label{chp:linear}

\section{Background and Motivations}

Control system theory when applied in practical situations often involves the use of large state-space models, typically due to inherent complexity of the system, the interconnection of subsystems, or the reduction of an infinite-dimensional or PDE model to a finite-dimensional approximation. One approach to dealing with such large scale models has been to use model reduction techniques such as \textit{balanced truncation}~(\cite{gugercin2004survey}). However, the use of model reduction techniques are not necessarily robust and can result in arbitrarily large errors. In addition to large state-space, practical problems often contain uncertainty in the model due to modeling errors, linearization, or fluctuation in the operating conditions. The problem of stability and control of systems with uncertainty has been widely studied. See, e.g. the  texts~\cite{ackermann_2001,bhattacharyya_1995,green_1994,zhou_1998,dullerud2000course}. Famous results such as the \textit{small-gain} theorem, \textit{Popov's} criterion, \textit{passivity} theorems and \textit{Kharitonov's} theorem have been widely used to find tractable solutions to certain robust stability problems of a single and/or interconnected uncertain systems. As an example, Kharitonov's theorem reduces the stability problem of an infinite family of differential equations
\begin{equation}
a_1 \dfrac{d^n}{dt^n} x(t) + a_2 \dfrac{d^{n-1}}{dt^{n-1}} x(t) + \cdots + a_{n-2} \dfrac{d}{dt} x(t) + a_{n+1} x(t) + a_{n+2} = 0, \quad a_i \in [\underline{u}_i, \bar{u}_i] \subset \mathbb{R}
\label{eq:ODE_family}
\end{equation}
to verifying whether the following four characteristic polynomials
\[
k_1(s) = \underline{u}_{n+2} + \underline{u}_{n+1}s + \bar{u}_{n}s^2 + \bar{u}_{n-1}s^3 + \underline{u}_{n-2}s^4 + \underline{u}_{n-3}s^5 + \cdots
\]
\[
k_2(s) = \bar{u}_{n+2} + \bar{u}_{n+1}s + \underline{u}_{n}s^2 + \underline{u}_{n-1}s^3 + \bar{u}_{n-2}s^4 + \bar{u}_{n-3}s^5 + \cdots
\]
\[
k_3(s) = \underline{u}_{n+2} + \bar{u}_{n+1}s + \bar{u}_{n}s^2 + \underline{u}_{n-1}s^3 + \underline{u}_{n-2}s^4 + \bar{u}_{n-3}s^5 + \cdots
\]
\[
k_4(s) = \bar{u}_{n+2} + \underline{u}_{n+1}s + \underline{u}_{n}s^2 + \bar{u}_{n-1}s^3 + \bar{u}_{n-2}s^4 + \underline{u}_{n-3}s^5 + \cdots
\]
have all their roots in the open left half-plane - a problem which can be tractably solved (in $\mathcal{O}(n^2)$ operations) using the Routh-Hurwitz criterion. Despite all the progress in robust control theory during the past few decades, a drawback of existing computational methods for analysis and control of systems with uncertainty is high computational complexity. This is a consequence of the fact that a wide range of problems in robust stability and control of systems with parametric uncertainty are known to be NP-hard. For example, even the classical problem of stability of $\dot{x}(t) = A(a) x(t)$ for all $a$ inside a hypercube (the matrix analog of System~\eqref{eq:ODE_family}) is NP-hard\footnote{\cite{nemirovskii1993several} proves that the $\{-1,+1\}$-integer linear programming problem (a well-known NP-complete problem) admits a polynomial-time reduction to the problem of verifying positive semi-definiteness of a family of symmetric matrices with entries belonging to an interval on $\mathbb{R}$.
}. Other examples are calculation of structured singular values for robust performance analysis and $\mu$-synthesis (\cite{zhou1996robust}), deciding null-controllability\footnote{A system $x(k+1) = f(x(k),u(k))$ is called null-controllable if for every initial state $x(0)$, there exist some $T > 1$ and controls $u(k), k=0, \cdots, T-1$ such that $x(T) = 0$} of $x(k+1) = f(x(k),u(k))$ for a given $f: \mathbb{R}^n \times \mathbb{R}^m \rightarrow \mathbb{R}^n$~(\cite{blondel1999complexity}), and computing arbitrarily precise bounds on the joint spectral radius of matrices for stability analysis of systems with time-varying uncertainty~(\cite{gripenberg1996computing}). See~\cite{blondel2000survey} for a comprehensive survey on NP-hard problems in control theory. The result of such complexity is that for systems with parametric uncertainty and with hundreds of states, existing algorithms fail with the primary point of failure usually being lack of unallocated memory.

In this dissertation, we seek to distribute the computation over an array of processors within the context of existing computational resources; specifically cluster-computers and supercomputers. When designing algorithms to run in a parallel computing environment, one must both synchronize computational tasks among the processors while minimizing communication overhead among the processors. This can be difficult, as each architecture has a specific memory hierarchy and communication graph (See Figure~\ref{fig:CPU-GPU}). Likewise, in a lower level, individual computing units may have different processing architectures and memory hierarchies; e.g., see a comparison of the memory hierarchy of a multi-core CPU and a GPU in Figure~\ref{fig:CPU-GPU}. We account for communication by explicitly modeling the required communication graph between processors. This communication graph is then mapped to the processor architecture using the Message-Passing Interface (MPI)~(\cite{walker1996mpi}). While there are many algorithms for robust stability analysis and control of linear systems, ours is the first which explicitly accounts for the processing architecture in the emerging multi-core computing environment.

\begin{figure}[h]
\centering
\includegraphics[scale=0.2]{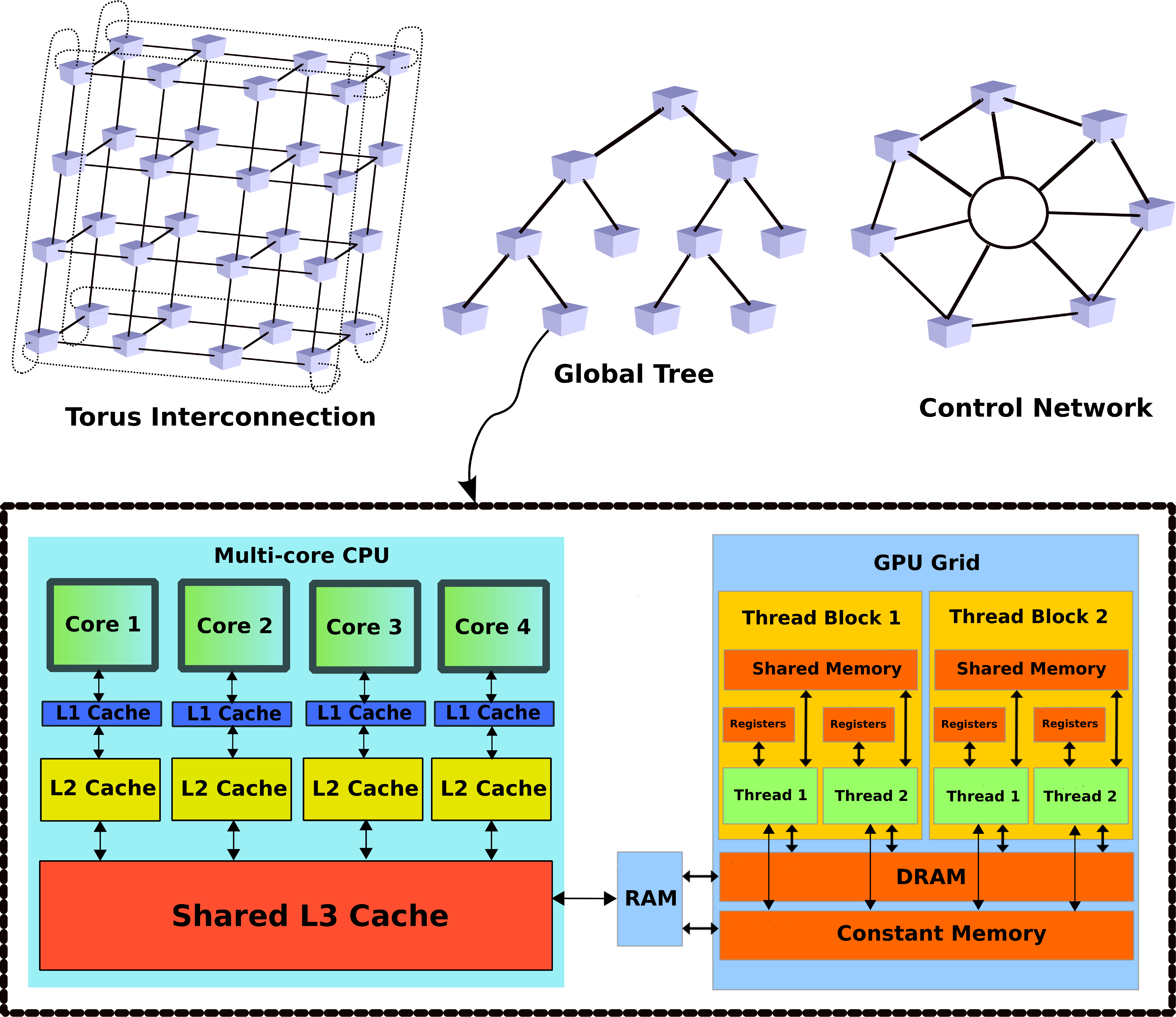}
\caption{Various Interconnections of Nodes in a Cluster Computer (Top), Typical Memory Hierarchies of a GPU and a Multi-core CPU (bottom)}
\label{fig:CPU-GPU}
\end{figure}


Our approach to robust stability is based on the well-established use of parameter-dependent Quadratic-In-The-State (QITS) Lyapunov functions. The use of parameter-dependent Lyapunov QITS functions eliminates the conservativity associated with e.g. quadratic stability~(\cite{quadratic1}), at the cost of requiring some restriction on the rate of parameter variation. Specifically, our QITS Lyapunov variables are polynomials in the vector of uncertain parameters. This is a generalization of the use of QITS Lyapunov functions with affine parameter dependence as in~\cite{affine_orig} and expanded in, e.g.~\cite{affine_dependent0,affine_dependent1,affine_dependent2} and~\cite{affine_dependent3}. The use of polynomial QITS Lyapunov variables can be motivated by~\cite{Bliman_existence}, wherein it is shown that any feasible parameter-dependent LMI with parameters inside a compact set has a polynomial solution or~\cite{peet2009exponentially} wherein it is shown that local stability of a nonlinear vector field implies the existence of a polynomial Lyapunov function.

There are several results which use polynomial QITS Lyapunov functions to prove robust stability. In most cases, the stability problem is reduced to the general problem of optimization of polynomial variables subject to LMI constraints - an NP-hard problem~(\cite{Np_hard}). To avoid NP-hardness, the optimization of polynomials problem is usually solved in an asymptotic manner by posing a sequence of sufficient conditions of increasing accuracy and decreasing conservatism. For example, building on the result in~\cite{Bliman_existence},~\cite{bliman2004convex} proposes a sequence of increasingly precise LMIs for robust stability analysis of linear systems with affine dependency on uncertain parameters on the complex unit ball. Necessary and sufficient stability conditions for linear systems with one uncertain parameter are derived in~\cite{bound1}, providing an explicit bound on the degree of the polynomial-type Lyapunov function. This result is extended to multi-parameter-dependent linear systems in~\cite{bound2}. Another important approach to optimization of polynomials is the SOS methodology  which replaces the polynomial positivity constraint with the constraint that the polynomial admits a representation as a sum of squares of polynomials. See Sections~\ref{sec:optim_semialg} and~\ref{sec:intro_SOS} for a review of this approach. Applications of the SOS methodology in robust stability analysis of linear and nonlinear systems can be found in~\cite{sos3,lavaei2008robust} and~\cite{packard_ROA}.
While the SOS methodology have been extensively utilized in the literature, we have not, as of yet, been able to adapt algorithms for solving the resulting LMI conditions to a parallel-computing environment. Finally, there have been multiple results in recent years on the use of Polya's theorem to solve optimization of polynomials problems~(\cite{peres2007}) on the simplex. An extension of Polya's theorem for uncertain parameters on the multisimplex or hypercube can be found in~\cite{peres_multisimplex}. In this section, we propose an extension of Polya's theorem and its use for solving optimization of polynomials problems in a parallel computing environment.

Our goal is to create algorithms which explicitly map computation, communication and storage to existing parallel processing architectures. This goal is motivated by the failure of existing general-purpose Semi-Definite Programming (SDP) solvers to efficiently utilize platforms for large-scale computation. Specifically, it is well-established that linear programming and semi-definite programming both belong to the complexity class P-Complete, also known as the class of inherently sequential problems.
Although there have been several attempts to map certain SDP solvers to a parallel computing environment~(\cite{csdp,sdpara}), certain critical steps cannot be distributed. The result is that as the number of processors increases, certain computational and communication bottlenecks dominate - leading to a saturation in the speed-up (the increase in processing speed per additional processor) of these solvers (Amdahl's law~(\cite{amdahl})). We avoid these bottlenecks by exploiting the particular structure of the LMI conditions associated with Polya's theorem. Note that, in principle, a perfectly designed general-purpose SDP algorithm could identify the structure of the SDP, as we have, and map the communication, computation and memory constraints to a parallel architecture. Indeed, there has been a great deal of research on creating programming languages which attempt to do just this~(\cite{kale1994charm,deitz2005high}). However, at present such languages are mostly theoretical and have certainly not been incorporated into existing SDP solvers.

In addition to parallel SDP solvers, there have been some efforts to exploit structure in certain polynomial optimization algorithms to reducing the size and complexity of the resulting LMI's. For example, for the case of finding SOS representations for symmetric polynomials\footnote{A symmetric polynomial is a polynomial which is invariant under all permutations of its variables, e.g., $f(x,y,z) = x^4+y^4+z^4-4xyz+x+y+z$.},~\cite{parrilo_sym} exploited symmetry to reduce the number of decision variables and constraints in the associated SDPs. Another example is the use of an specific sparsity structure in~\cite{parrilo_sructure,kim2005generalized} and~\cite{waki} to reduce the complexity of the linear algebra calculations associated with the SOS methodology. The use of generalized Lagrangian duals and Groebner basis techniques for reducing the complexity of the SDPs associated

\pagebreak

\hspace{-0.32in} with the SOS decompositions of sparse polynomial optimization problems can be found in~\cite{kim2005generalized} and~\cite{permenter2012selecting}.


\subsection{Our Contributions}

In this section, we focus on robust stability analysis of: 1- Systems with parametric uncertainty inside a simplex; and 2- Systems with parametric uncertainty inside a hypercube. We solve each problem in two phases by proposing the following algorithms: 1- A decentralized algorithm for Setting up the sequence of structured SDPs associated with Polya's theorem; and 2- A parallel SDP solver to solve the SDPs. Note that the problem of decentralizing the set-up algorithm is significant in that for large-scale systems, the instantiation of the problem may be beyond the memory and computational capacity of a single processing node. For the set-up problem, the algorithm that we propose has no centralized memory/computational requirements whatsoever. Furthermore, we show that for a sufficiently large number of available processors, the communication complexity is independent of the size of the state-space or the number of Polya's iterations.

In the second phase, we propose a variant of Helmberg's primal-dual algorithm (\cite{helmberg2005interior}) and map the computational, memory and communication requirements to a parallel computing environment. Unlike the set-up algorithm, the primal-dual algorithm does have a ``relatively small" centralized computation associated with the update of the dual variables. However, we have structured the algorithm so that the size of this centralized computation is solely a function of the degree of the polynomial Lyapunov function and does not depend on the number of Polya's iterations. 
 In addition, there is no point-to-point communication between the processors, which means that the algorithm is compatible with most of the existing parallel \\
 
 \pagebreak
 
\hspace{-0.3in} computing architectures. We will provide a graph representation of the communication architecture of both the set-up and SDP algorithms.

By linking the set-up and SDP algorithms and conducting tests on various cluster computers, we demonstrate the ability of our algorithms in performing robust stability analysis on systems with 100+ states and several uncertain parameters. Specifically, we ran a series of numerical experiments using the Linux-based cluster computer Karlin at Illinois Institute of Technology and the Blue Gene supercomputer (with 200 processor allocation). First, we applied the algorithm to a current problem in robust stability analysis of magnetic confinement fusion using a discretized PDE model. Next, we examine the accuracy of the algorithm as Polya's iterations progress and compare this accuracy with the SOS approach. We show that unlike the general-purpose parallel SDP solver SDPARA~\cite{sdpara}, the speed-up of our algorithm shows no evidence of saturation. Finally, we calculate the envelope of the algorithm on the cluster computer Karlin in terms of the maximum state-space dimension, number of processors and Polya's iterations.

\section{Notation and Preliminaries on Homogeneous Polynomials}
\label{sec:notation_simplex}

Let us denote an $l-$variate monomial as $\alpha^{{\gamma}} = \prod_{i=1}^l\alpha_i^{\gamma_i}$, where $\alpha \in \mathbb{R}^l $ is the vector of variables, $\gamma \in \mathbb{N}^l$ is the vector of exponents and $\sum\limits_{i=1}^l \gamma_i = d$ is the degree of the monomial. We define 
\begin{equation}
W_d:= \left\lbrace \gamma \in \mathbb{N}^l: \sum_{i=1}^l \gamma_i = d \right\rbrace
\label{eq:W_d}
\end{equation}
 as the totally ordered set of the exponents of $l-$variate monomials of degree $d$, where the ordering is lexicographic. Recall that in lexicographical ordering $\gamma \in W_d$ precedes $\eta \in W_d$, if the left most non-zero entry of $\gamma-\eta$ is positive. The lexicographical index of every $ \gamma \in W_d$ can be calculated using the map $\langle{{\cdot}}\rangle:\mathbb{N}^l \rightarrow \mathbb{N}$ defined as~(\cite{peet_acc}) 
\begin{equation*}
\langle{\gamma}\rangle = \sum_{j=1}^{l-1} \sum_{i=1}^{\gamma_i} f \left(l-j,d+1-\sum_{k=1}^{j-1}\gamma_k -i \right) + 1
\label{eq:lex} 
\end{equation*}
where
\begin{equation}
f(l,d) :=
\begin{cases}
\hspace*{0.2in}	0 &\text{for} \;\; l = 0 \\
\dbinom{l+d-1}{l-1}= \dfrac{(d+l-1)!}{d!(l-1)!} & \text{for} \;\; l > 0, 
\end{cases}
\label{eq:f_TAC}
\end{equation}
is the cardinality of $W_d$, i.e., the number of $l-$variate monomials of degree $d$. For convenience, we also denote the index of a monomial $\alpha^\gamma$ by $\langle \gamma \rangle$. We represent $l-$variate homogeneous polynomials of degree $d_p$ as 
\begin{equation*}
P(\alpha)=\sum_{\gamma \in W_{d_p}} P_{\langle{\gamma}\rangle} \alpha^{\gamma}, 
\end{equation*}
where $P_{\langle{\gamma}\rangle} \in \mathbb{R}^{n \times n}$ is the matrix coefficient of the monomial $\alpha^{\gamma}$.

Now consider the linear system
\begin{equation}
\dot{x}(t)= A(\alpha)x(t),  
\label{eq:system_TAC}
\end{equation}
where $ A(\alpha) \in \mathbb{R}^{n \times n} $ and $\alpha \in Q \subset \mathbb{R}^{l}$ is a vector of uncertain parameters.
We assume that $A(\alpha)$ is a homogeneous polynomial and
$Q=\Delta^l \subset \mathbb{R}^{l}$, where $\Delta^l$ is the unit simplex, i.e., 
\begin{equation*}
\Delta^l=\left\lbrace \alpha\in \mathbb{R}^l : \sum_{i=1}^{l} \alpha_i=1, \alpha_i\geqslant 0 \right\rbrace. 
\end{equation*}
If $A(\alpha)$ is not homogeneous, we can homogenize it in the following manner.
Suppose $A(\alpha)$ with $\alpha \in \Delta^l$ is a non-homogeneous polynomial of degree $d_a$ and has $N_a$ monomials with non-zero coefficients. Define $D = \left( d_{a_1}, \cdots, d_{a_{N_a}} \right)$, where $d_{a_i}$ is the degree of the $i^{th}$ monomial of $A(\alpha)$ according to the lexicographical ordering.
Now define the polynomial $B(\alpha)$ as per the following:
\begin{enumerate}
\item Let $B=A$.
\item For $i=1,\cdots, N_a$, multiply the $i^{th}$ monomial of $B(\alpha)$, according to lexicographical ordering, by $\left( \sum\limits_{j=1}^l \alpha_j \right)^{d_a-d_{a_i}}$.
\end{enumerate}
Then, since  $\sum\limits_{j=1}^l \alpha_j =1$,  $B(\alpha)=A(\alpha)$ for all $\alpha \in \Delta^l$ and hence all properties of $\dot x(t) = A(\alpha)x(t)$ for any $\alpha \in \Delta^l$ are retained by the homogeneous system $\dot x(t) = B(\alpha)x(t)$. To further clarify the homogenization procedure, we provide the following example. \vspace{0.1in}

\noindent \textit{Example: Construction of the homogeneous system $\dot x(t) = B(\alpha) x(t)$.}

Consider the non-homogeneous polynomial $
A(\alpha) = C \alpha_1^2  + D \alpha_2 + E \alpha_3 + F
$ of degree $d_a=2$, where $[\alpha_1,\alpha_2,\alpha_3] \in \Delta^3$. Using the above procedure, the homogeneous polynomial $B(\alpha)$ can be constructed as 
\begin{align}
B(\alpha) & = C \alpha_1^2  + D \alpha_2 (\alpha_1+\alpha_2+\alpha_3) + E  \alpha_3 (\alpha_1+\alpha_2+\alpha_3) + F (\alpha_1+\alpha_2+\alpha_3)^2  \nonumber \\
& = \underbrace{(C+F)}_{B_1}\alpha_1^2 + \underbrace{(D+2F)}_{B_2} \alpha_1 \alpha_2  
+ \underbrace{(E+2F)}_{B_3} \alpha_1 \alpha_3 + \underbrace{(D+F)}_{B_4} \alpha_2^2  \nonumber \\
& + \underbrace{(D+E+2F)}_{B_5} \alpha_2 \alpha_3 + \underbrace{(E+F)}_{B_6} \alpha_3^2 = \sum_{\gamma \in W_2} B_{\langle{\gamma}\rangle} \alpha^{\gamma}.
\end{align}
\vspace{0.05in}

\section{Setting-up the Problem of Robust Stability Analysis over a Simplex}

 In this section, we show that applying Polya's Theorem to the robust stability problem, i.e., the inequalities in Theorem~\ref{thm:thm1} yields a semi-definite program with a block-diagonal structure - hence can be an efficiently distributed among processing units.
We start by stating the following well-known Lyapunov result on stability of System~\eqref{eq:system_TAC}.
\begin{mythm}\label{thm:thm1}
System~\eqref{eq:system_TAC} is stable if and only if there exists a polynomial matrix $P(\alpha)$ such that $P(\alpha) \succ 0$ for all $\alpha \in \Delta^l$ and
\begin{equation}
A^T(\alpha)P(\alpha)+P(\alpha)A(\alpha) \prec 0  \quad \text{for all } \alpha \in \Delta^l.
 \label{eq:Lyap_LMI}  
\end{equation}
\end{mythm}
A similar condition also holds for discrete-time linear systems. The conditions associated with Theorem~\ref{thm:thm1} are infinite-dimensional LMIs, meaning they must hold at infinite number of points. Such problems are known to be NP-hard~(\cite{Np_hard}). Our goal is to derive a sequence of polynomial-time algorithms such that their outputs converge to a solution of the parameter-dependent LMI in~\eqref{eq:Lyap_LMI}. Key to this result is Polya's Theorem~(\cite{polya_book}). A variation of this theorem for matrices is given as follows.

\begin{mythm}(Polya's theorem, simplex version)
If a homogeneous matrix-valued polynomial $F$ satisfies $F(\alpha) > 0$ for all $\alpha \in \Delta^l$, then there exists $d \geq 0$ such that all the coefficients of
\begin{equation}
\left( \sum_{i=1}^l \alpha_i  \right)^d F(\alpha)
\label{eq:polya_product_simplex}
\end{equation}
are positive definite.
\label{thm:polya_simplex_TAC}
\end{mythm}

See Chapter~\ref{chp:background} for a more detailed discussion on this result.

Consider the stability of the system described by Equation~\eqref{eq:system_TAC}. We are interested in finding a $P(\alpha)$ which satisfies the conditions of Theorem~\ref{thm:thm1}. According to Polya's theorem, the constraints of Theorem~\ref{thm:thm1} are satisfied if for some sufficiently large $d_1$ and $d_2$, the polynomials  
\begin{equation}
\left(\sum_{i=1}^l \alpha_i \right)^{d_1} P(\alpha ) \label{eq:LMI_1} \qquad \text{and} 
\end{equation}
\begin{equation}
-\left(\sum_{i=1}^l \alpha_i \right)^{d_2} \left(A^T(\alpha)P(\alpha)+P(\alpha)A(\alpha)\right) \label{eq:LMI_2} 
\end{equation}  
have all positive definite coefficients.

Let $P(\alpha)$ be a homogeneous polynomial of degree $d_p$ which can be represented as 
\begin{equation}
P(\alpha)=\sum_{\gamma \in W_{d_p} } P_{\langle{\gamma}\rangle} \alpha^{{\gamma}}, \label{eq:P_alpha} 
\end{equation}
where the coefficients $P_{\langle \gamma \rangle} \in \mathbb{S}^n$. Recall that $W_{d_p} := \left\lbrace \gamma \in \mathbb{N}^l: \sum_{i=1}^l \gamma_i = d_p \right\rbrace$ is the set of the exponents of all $l$-variate monomials of degree $d_p$. Since $A(\alpha)$ is a homogeneous polynomial of degree $d_a$, we can write it as 
\begin{equation}
A(\alpha)=\sum_{\gamma \in W_{d_a} } A_{\langle{\gamma}\rangle} \alpha^{{\gamma}}, \label{eq:A_alpha} 
\end{equation}
where the coefficients $A_{\langle \gamma \rangle} \in \mathbb{R}^{n \times n}$. By substituting~\eqref{eq:P_alpha} and~\eqref{eq:A_alpha} into~\eqref{eq:LMI_1} and~\eqref{eq:LMI_2} and defining $d_{pa}$ as the degree of $P(\alpha)A(\alpha)$, the conditions of Theorem~\ref{thm:polya_simplex_TAC} can be represented in the form
\begin{equation*}
\left( \sum_{i=1}^l \alpha_i \right)^{d_1} \left( \sum_{h \in W_{d_p}}  P_{\langle h \rangle} \alpha^{h}  \right) =   \sum_{g \in W_{d_p+d_1}}  \left( \sum_{h \in W_{d_p}} \beta_{\langle h \rangle,\langle \gamma \rangle } P_{ \langle h \rangle} \right) \alpha^{\gamma}
\end{equation*}
and
\begin{align*}
- &\left( \sum_{i=1}^l \alpha_i \right)^{d_2}  \left( \hspace{-0.035in}  \left( \sum_{h \in W_{d_a}}  A^T_{\langle h \rangle} \alpha^h \right) \hspace{-0.035in}  \left( \sum_{h \in W_{d_p}}  P_{\langle h \rangle} \alpha^h  \right) +  \left( \sum_{h \in W_{d_p}}  P_{\langle h \rangle} \alpha^h \right)  \hspace{-0.035in}  \left( \sum_{h \in W_{d_a}} A_{\langle h \rangle} \alpha^h \right) \hspace{-0.035in}  \right) \nonumber \\
&  \hspace{2.45in} =  \sum_{\gamma \in W_{d_{pa}+d_2}} \hspace{-0.05in}  \left( \sum_{h \in W_{d_p}} H_{\langle h \rangle, \langle \gamma \rangle}^T P_{\langle h \rangle} + P_{\langle h \rangle} H_{\langle h \rangle, \langle \gamma \rangle}  \right)\alpha^\gamma
\end{align*}
have all positive coefficients. This means that
\begin{align}
&\sum_{h \in W_{d_p}} \beta_{\langle{h}\rangle, \langle{\gamma}\rangle }P_{\langle{h}\rangle} > 0 \quad &&\text{ for all } \quad \gamma \in W_{d_p + d_1} \; \text{and} \label{eq:LMI_3} \\
&\sum_{h \in W_{d_p}} (H_{\langle{h}\rangle, \langle{\gamma}\rangle}^T P_{\langle{h}\rangle} + P_{\langle{h}\rangle} H_{\langle{h}\rangle, \langle{\gamma}\rangle}) < 0 \quad &&\text{ for all } \gamma \in W_{d_{pa} + d_2} \label{eq:LMI_4}.  
\end{align}
Here we have defined $\beta_{\langle{h}\rangle,\langle{\gamma}\rangle}$ to be the scalar coefficient which multiplies $P_{\langle{h}\rangle}$ in the $\langle{\gamma}\rangle$-th monomial of the homogeneous polynomial $\left( \sum_{i=1}^l \alpha_i \right)^{d_1} P(\alpha)$ using the lexicographical ordering. Likewise, $H_{\langle{h}\rangle, \langle{\gamma}\rangle} \in \mathbb{R}^{n \times n}$ is the term which left or right multiplies $P_{\langle{h}\rangle}$ in the $\langle{\gamma}\rangle$-th monomial of $\left( \sum_{i=1}^l \alpha_i \right)^{d_2} \left( A^T(\alpha)P(\alpha)+P(\alpha)A(\alpha) \right)$ using the lexicographical ordering. For an intuitive explanation as to how these $\beta$ and $H$ terms are calculated, we consider a simple example. Precise formulae for these terms will follow the example. \vspace{0.1in}

\noindent \textit{Example:} \textit{Calculating the $\beta$ and $H$ coefficients.}

Consider $
A(\alpha)=A_1 \alpha_1 + A_2 \alpha_2$ and $P(\alpha)=P_1 \alpha_1+P_2 \alpha_2$.
By expanding Equation~\eqref{eq:LMI_1} for $d_1=1$ we have $
(\alpha_1+\alpha_2)P(\alpha)= P_{1}\alpha_1^2+(P_{1}+P_{2})\alpha_1 \alpha_2 + P_{2} \alpha_2^2 
$. The coefficients $ \beta_{\langle{h}\rangle,\langle{\gamma}\rangle}$ are then extracted as 
\begin{equation*}
\beta_{1,1}=1,\; \beta_{2,1}=0,\; \beta_{1,2}=1,\;
\beta_{2,2}=1, \;\beta_{1,3}=0, \;\beta_{2,3}=1. 
\end{equation*}
Next, by expanding Equation~\eqref{eq:LMI_2} for $d_2=1$ we have 
\begin{align*}
& (\alpha_1+\alpha_2) \left(A^T(\alpha)P(\alpha)+P(\alpha)A(\alpha) \right)=  \nonumber\\ 
& \left( A^T_{1}P_{1}+P_{1}A_{1} \right)\alpha_1^3 + \left(  A^T_{1}P_{1}+P_{1} A_{1} +A^T_{2}P_{1} +P_{1}A_{2} +A^T_{1}P_{2}+P_{2}A_{1}\right) \alpha_1^2 \alpha_2   +\left( A^T_{2}P_{1} \right.  \\
 & \hspace{1in} \left. +P_{1}A_{2} +A^T_{1}P_{2} +P_{2}A_{1}+A^T_{2}P_{2}+P_{2}A_{2} \right)\alpha_1 \alpha_2^2 
  +\left(A^T_{2}P_{2}+P_{2}A_{2} \right)\alpha_2^3. 
\end{align*}
The coefficients $ H_{\langle{h}\rangle,\langle{\gamma}\rangle} $ are then extracted as
\begin{align*}
&\hspace*{-0.1in} H_{1,1}=A_{1}, && \hspace*{-0.07in} H_{2,1}=\textbf{0},  && \hspace*{-0.4in} H_{1,2}=A_{1}+A_{2},    && \hspace*{-0.07in}  H_{2,2}=A_{1},  \\ 
& \hspace*{-0.1in} H_{1,3}=A_{2}, &&  \hspace*{-0.07in} H_{2,3}=A_{1}+A_{2}, && \hspace*{-0.07in} H_{1,4}=\textbf{0},   && \hspace*{-0.07in} H_{2,4}=A_{2}.
\end{align*}

\subsection{General Formulae for Calculating Coefficients $\beta$ and H}
\label{sec:betaH_simplex}

The set $\{ \beta_{\langle{h}\rangle,\langle{\gamma}\rangle} \}$ of coefficients can be formally defined recursively as follows. Let the initial values for $ \beta_{\langle{h}\rangle,\langle{\gamma}\rangle}$ be defined as 
\begin{equation}
\beta^{(0)}_{\langle{h}\rangle,\langle{\gamma}\rangle} = \begin{cases}1& \text{if } h=\gamma\\ 0 & \text{otherwise} \end{cases}\qquad \text{for all } \; \gamma \in W_{d_p} \; \text{and} \; h \in W_{d_p}. \label{eq:beta_init} 
\end{equation}
Then, iterating for $i=1,\ldots d_1$, we let 
\begin{equation}
 \beta^{(i)}_{\langle{h}\rangle,\langle{\gamma}\rangle}=\sum_{\lambda \in W_1} \beta^{(i-1)}_{\langle{h}\rangle,\langle{\gamma-\lambda}\rangle} \qquad \text{for all } \; \gamma \in W_{d_p + i} \; \text{and} \; h \in W_{d_p}. \label{eq:beta}
\end{equation}
Finally, we set $\{\beta_{\langle{h}\rangle,\langle{\gamma}\rangle}\} =  \{\beta^{d_1}_{\langle{h}\rangle,\langle{\gamma}\rangle}\}$. 

To obtain the set $\lbrace{ H_{\langle{h}\rangle, \langle{\gamma}\rangle}} \rbrace$ of coefficients, set the initial values as 
\begin{equation}
H^{(0)}_{\langle{h}\rangle,\langle{\gamma}\rangle} = \sum_{\lambda \in W_{d_a}: \lambda + h = \gamma} A_{\langle{\lambda} \rangle}
 \;\; \text{for all } \; \gamma \in W_{d_p+d_a}\; \text{and} \; h \in W_{d_p}.  \label{eq:H_init} 
\end{equation}
Then, iterating for $i=1,\ldots d_2$, we let 
\begin{equation}
H^{(i)}_{\langle{h}\rangle,\langle{\gamma}\rangle}=\sum_{\lambda \in W_1 } H^{(i-1)}_{\langle{h}\rangle,\langle{\gamma-\lambda}\rangle}  \quad \text{for all } \; \gamma \in W_{d_{pa} + i}\; \text{and} \;h \in W_{d_p}.  \label{eq:H} 
\end{equation}
Finally, set $\{H_{\langle{h}\rangle,\langle{\gamma}\rangle}\} = \{H^{d_2}_{\langle{h}\rangle,\langle{\gamma}\rangle}\}$.

For the case of large-scale systems, computing and storing $\lbrace \beta_{\langle{h}\rangle,\langle{\gamma}\rangle} \rbrace$ and $ \lbrace {H_{\langle{h}\rangle,\langle{\gamma}\rangle}} \rbrace$ is a significant challenge due to the number of these coefficients. Specifically, the number of terms increases with $l$ (number of uncertain parameters in System~\eqref{eq:system_TAC}), $d_{p}$ (degree of $P(\alpha)$), $d_{pa}$ (degree of $P(\alpha)A(\alpha)$) and $d_1, d_2$ (Polya's exponents) as follows. \vspace*{0.1in}

\subsection{Number of Coefficients $ \beta_{\langle{h}\rangle,\langle{\gamma}\rangle}$ and $H_{\langle{h}\rangle,\langle{\gamma}\rangle}$}

Given $l,d_p$ and $d_1$, since $h \in W_{d_p}$ and $\gamma \in W_{d_p+d_1}$, the number of coefficients $ \beta_{\langle{h}\rangle,\langle{\gamma}\rangle}$  is the product of $L_0:= \text{card}(W_{d_p})$ and $L:=\text{card}(W_{d_p+d_1})$. Recall that card$(W_{d_p})$ is the number of all $l$-variate monomials of degree $d_p$ and can be calculated using~\eqref{eq:f_TAC} as follows. 
\begin{equation}
L_0 = f(l,d_p) =
\begin{cases}
\hspace*{0.2in}	0 &\text{for} \;\; l=0\\
\dbinom{d_p+l-1}{l-1}= \dfrac{(d_p+l-1)!}{d_p!(l-1)!} & \text{for} \;\; l > 0. \label{eq:L0} 
\end{cases}
\end{equation}
Likewise, card$(W_{d_p+d_1})$, i.e., the number of all $l-$variate monomials of degree $d_p+d_1$ is calculated using~\eqref{eq:f_TAC} as follows. 
\begin{align}
&  L = f(l,d_p+d_1) = 
 \begin{cases}
\hspace*{0.2in}	0 &\text{for} \;\; l=0\\
\dbinom{d_p+d_1+l-1}{l-1}= \dfrac{(d_p+d_1+l-1)!}{(d_p+d_1)!(l-1)!} & \text{for} \;\; l > 0. \label{eq:L}
\end{cases}
\end{align}
The number of coefficients $ \beta_{\langle{h}\rangle,\langle{\gamma}\rangle}$  is $L_0 \cdot L$. In Figure~\ref{fig:beta_TAC}, we have plotted the number of coefficients $ \beta_{\langle{h}\rangle,\langle{\gamma}\rangle}$ in terms of the number of uncertain parameters $l$ and for different polya's exponents.

\begin{figure}
\centering
\includegraphics[scale=0.4]{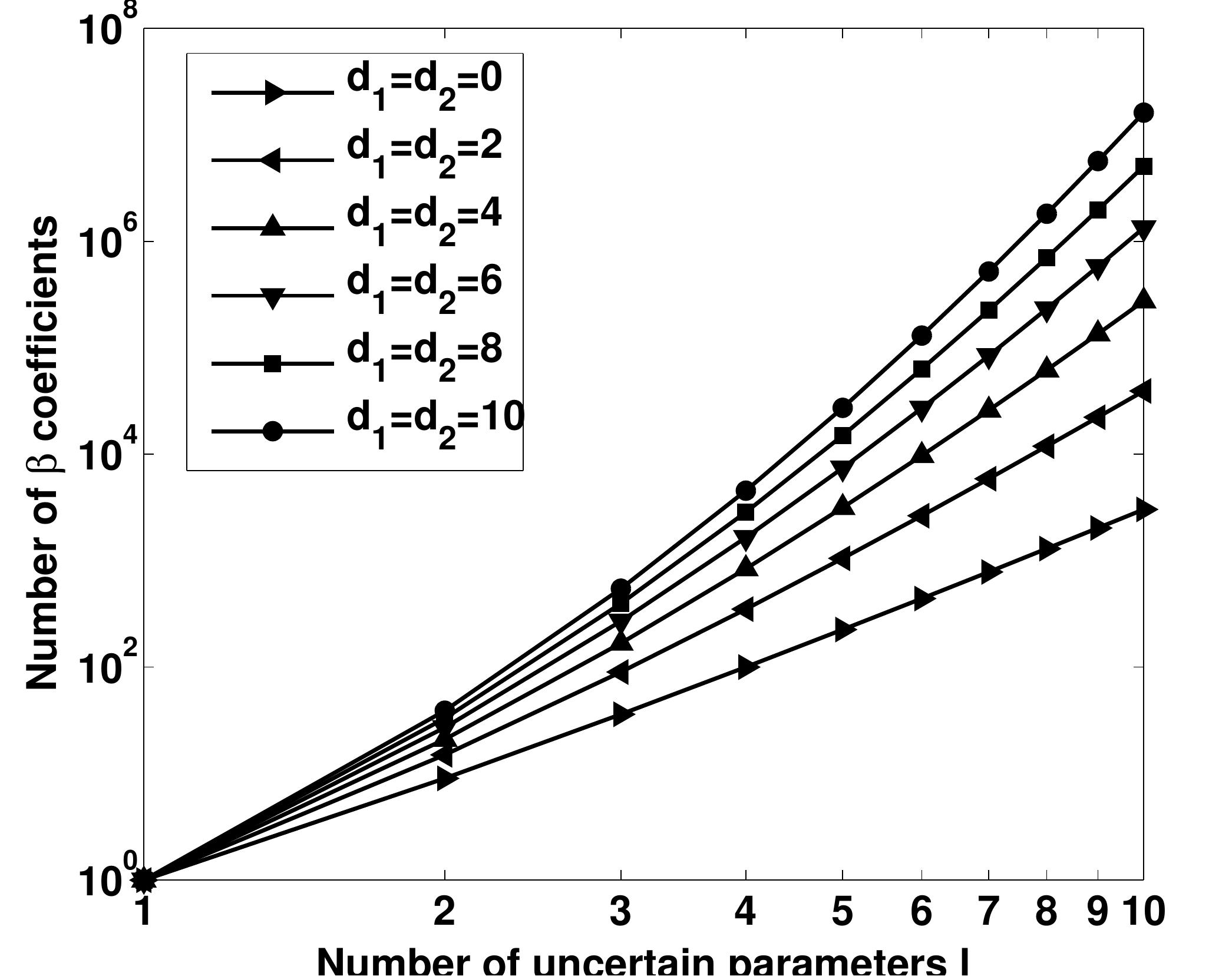}
\caption{Number of $ \beta_{\langle{h}\rangle,\langle{\gamma}\rangle}$ Coefficients vs. the Number of Uncertain Parameters for Different Polya's Exponents and for $d_p=2$}
\label{fig:beta_TAC} 
 \includegraphics[scale=0.4]{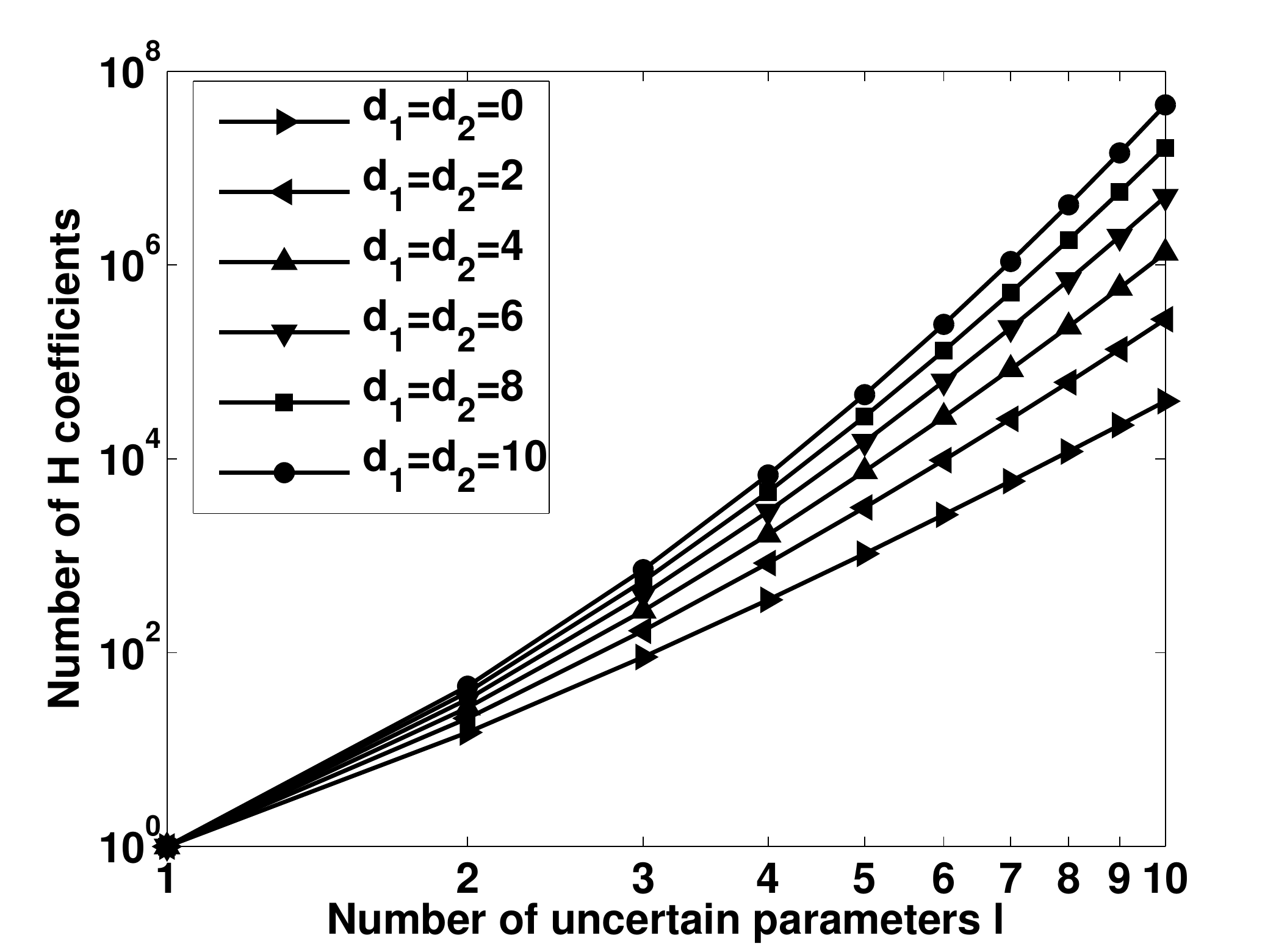}
\caption{Number of $H_{\langle{h}\rangle,\langle{\gamma}\rangle}$ Coefficients vs. the Number of Uncertain Parameters for Different Polya's Exponents and for $d_p=d_a=2$}
\label{fig:H_TAC} 
\end{figure}

Given $l,d_p, d_a$ and $d_2$, since $h\in W_{d_p}$ and $\gamma \in W_{d_{pa}+d_2}$, the number of coefficients $ H_{\langle{h}\rangle,\langle{\gamma}\rangle}$ is the product of $L_0:= \text{card}(W_{d_p})$ and $M:=\text{card}(W_{d_{pa}+d_2})$. By using~\eqref{eq:f_TAC}, we have 
\begin{align}
M = f(l,d_{pa}+d_2) = 
\begin{cases}
\hspace*{0.2in}	0 &\text{for} \;\; l=0\\
\dbinom{d_{pa}+d_2+l-1}{l-1}= \dfrac{(d_{pa}+d_2+l-1)!}{(d_{pa}+d_2)!(l-1)!} & \text{for} \;\; l > 0. 
\end{cases}
\label{eq:M}
\end{align}
The number of $ H_{\langle{h}\rangle,\langle{\gamma}\rangle}$ coefficients is $L_0 \cdot M$. In Figure~\ref{fig:H_TAC}, we have plotted the number of coefficients $ H_{\langle{h}\rangle,\langle{\gamma}\rangle} $ in terms of the number of uncertain parameters $l$ and for different polya's exponents.

We have shown the required memory to store the coefficients $\beta_{\langle{h}\rangle,\langle{\gamma}\rangle}$ and $H_{\langle{h}\rangle,\langle{\gamma}\rangle}$ in Figure~\ref{fig:memory_beta_h} in terms of the number of uncertain parameters $l$ and for different Polya's exponents. It is observed from Figure~\ref{fig:memory_beta_h} that even for small degree $d_p$ of $P(\alpha)$ and small degree $d_a$ of the system matrix $A(\alpha)$, the required memory is in the Terabyte range.
~\cite{peet_acc} proposed a decentralized computing approach to the
calculation of $\{ \beta_{\langle{h}\rangle,\langle{\gamma}\rangle} \}$ on a cluster computer. In the work, we extend this method to the calculation of $ \{ H_{\langle{h}\rangle,\langle{\gamma}\rangle} \}$
and the SDP elements which will be discussed in the following section. We express the LMIs associated with conditions~\eqref{eq:LMI_3} and~\eqref{eq:LMI_4} as an SDP in both primal and dual forms. We will also discuss the structure of the primal and dual SDP variables and the constraints.

\begin{figure}
\centering
\includegraphics[scale=0.4]{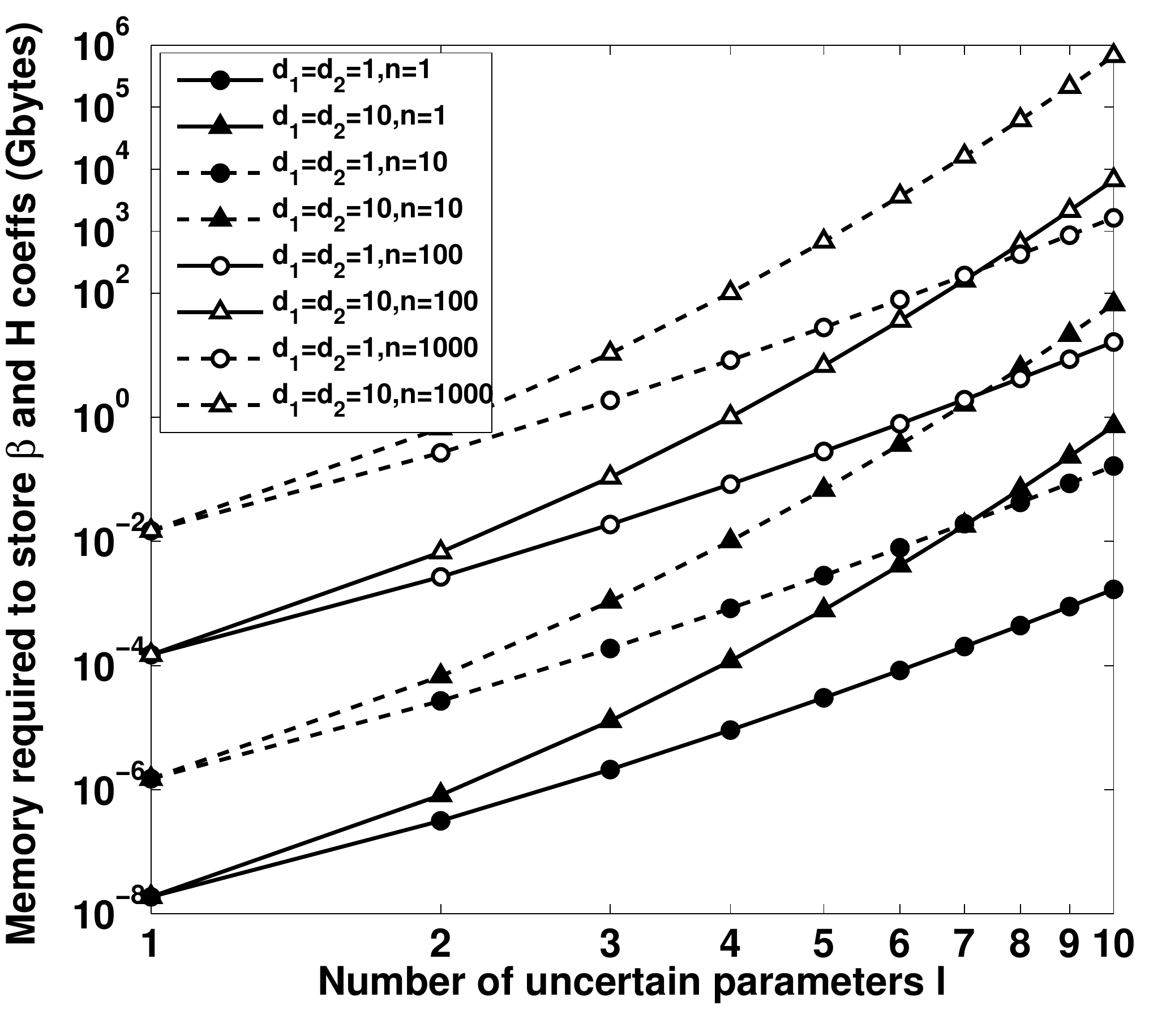}
\caption{Memory Required to Store the Coefficients $\beta_{\langle{h}\rangle,\langle{\gamma}\rangle}$ and $H_{\langle{h}\rangle,\langle{\gamma}\rangle}$ vs. Number of Uncertain Parameters, for Different $d_1,d_2$ and $d_p=d_a=2$}
\label{fig:memory_beta_h} 
\end{figure}

\subsection{The Elements of the SDP Problem Associated with Polya's Theorem}
\label{sec:SDPelements_TAC}

Recall from Section~\ref{sec:SDP_SDP} that a semi-definite program can be stated either in primal or in dual format. Given $C \in \mathbb{S}^{m} $,  $a \in \mathbb{R}^K$ and $B_i \in \mathbb{S}^m$, here we consider
\begin{equation}
\max_{X \in \mathbb{S}^m} \quad \text{tr}(CX) \nonumber 
\end{equation}
\begin{equation}
\hspace{-0.1in} \text{subject to } \;  B(X) = a \nonumber 
\end{equation}
\begin{equation}
\hspace{0.45in} X \geq 0,  \label{eq:primal_TAC} 
\end{equation}
as the primal SDP form, where the linear operator $B: \mathbb{S}^{m} \rightarrow \mathbb{R}^K$ is defined in~\eqref{eq:B_G}.  
The associated dual problem is 
\begin{equation}
\hspace*{-0.2in}\min_{y,Z} \quad a^Ty \nonumber 
\end{equation}
\begin{equation}
 \hspace*{0.2in}\text{subject to} \quad  \sum_{i=1}^K B_i y_i-C = Z \nonumber 
\end{equation}
\begin{equation}
 \hspace*{1in}Z \geq 0 \;,\; y \in \mathbb{R}^K.  \label{eq:dual_TAC} 
\end{equation}
The elements $C$, $B_i$ and $a$ of the SDP problem associated with the LMIs in~\eqref{eq:LMI_3} and~\eqref{eq:LMI_4} are defined as follows.
We define the element $C$ as 
\begin{equation}
C := \text{diag}(C_1, \cdots C_L, C_{L+1}, \cdots C_{L+M}), \label{eq:C}
\end{equation}
where 
\begin{equation}
C_i :=
\begin{cases}
\delta I_n \cdot \left( \sum_{h \in W_{d_p}} \beta_{\langle h \rangle,i} \, \frac{d_p!}{h_1! \, \cdots \, h_l!} \right), & \text{for } \;\; 1 \le i \le L\\
0_n, & \text{for } \;\;   L+1 \le  i  \le L+M,
\end{cases}  
\label{eq:Cj}
\end{equation}
where recall that $L= \text{card}(W_{d_p+d_1})$ is the number of monomials in $\left(\sum_{i=1}^l \alpha_i \right)^{d_1} P(\alpha)$, $M=\text{card}(W_{d_{pa}+d_2})$ is the number of monomials in $\left(\sum_{i=1}^l \alpha_i \right)^{d_2} P(\alpha)A(\alpha)$, $n$ is the dimension of System~\eqref{eq:system_TAC}, $l$ is the number of uncertain parameters and $\delta$ is a small positive parameter.

For $i=1, \cdots, K$, define $B_i$ elements as
\begin{equation}
B_i := \text{diag}(B_{i,1}, \cdots B_{i,L}, B_{i,L+1}, \cdots B_{i,L+M}), \label{eq:Ai} 
\end{equation}
where $K$ is the number of dual variables in~\eqref{eq:dual_TAC} and is equal to the product of  the number of upper-triangular elements in each $P_\gamma \in \mathbb{S}^n$ (the coefficients in $P(\alpha)$) and the number of monomials in $P(\alpha)$ (i.e. the cardinality of $W_{d_p}$). Since there are $f(l,d_p)=\dbinom{d_p+l-1}{l-1}$ coefficients in $P(\alpha)$ and each coefficient has $\tilde{N}:=\frac{1}{2}n(n+1)$ upper-triangular elements, we find $K$ as
\begin{equation}
K= \frac{(d_p + l-1)!}{d_p! (l-1)!}\tilde{N}. \label{eq:K} 
\end{equation}
To define the $B_{i,j}$ blocks, first we define the map $V_{\langle h \rangle}: \mathbb{Z}^K \rightarrow \mathbb{Z}^{n \times n}$, 
\begin{equation}
V_{\langle h \rangle}(x) := \sum_{j=1}^{\tilde{N}} E_j \; x_{j+\tilde{N}(\langle h \rangle-1)} \quad \text{for all} \quad h \in W_{d_p}, 
\end{equation}
which maps each variable to $E_j$, where $E_j, \, j=1, \cdots \tilde{N}$ define the canonical basis for $\mathbb{S}^n$ (subspace of symmetric matrices) as follows.
\begin{equation}
[E_j]_{i,k}:=\begin{cases}
1  &  \text{if } i=k=j \\
0  &  \text{otherwise}
\end{cases} , \quad \text{for} \;\, j \leq n \quad
\text{and} 
\label{eq:Ej_basis} 
\end{equation}
\begin{equation}
[E_j]_{i,k}:=[F_j]_{i,k}+[F_j]^T_{i,k}, \quad \text{for} \;\, j > n,  
\end{equation}
where 
\begin{equation}
[F_j]_{i,k}:=\begin{cases}
1  &  \text{if }  i=k-1=j-n  \\
0  &  \text{otherwise}.
\end{cases} 
\end{equation} 
Note that a different choice of basis would require a different function $V_{\langle h \rangle}$. Then, for $i=1, \cdots, K$, we define $B_{i,j}$ matrices as  
\begin{equation}
B_{i,j}:= 
\begin{cases}
\sum\limits_{h \in W_{d_p}} \beta_{\langle h \rangle,j} V_{\langle h \rangle}(e_i), & \text{for }  1 \le j \le L \hspace{0.8in}  (I)\\
-\sum\limits_{h \in W_{d_p}} \hspace{-0.05in} \Big( H_{{\langle h \rangle} ,j-L}^T V_{\langle h \rangle}(e_i)+  
 V_{\langle h \rangle}(e_i) H_{{\langle h \rangle},j-L} \Big), & \text{for }  L+1 \le j \le L+M, \; (II)
\end{cases} 
\label{eq:Aij} 
\end{equation}
where we have denoted the canonical basis for $\mathbb{R}^n$ by $e_i=[0 \: ... \: 0 \overbrace{1}^{i^{th}} 0 \: ... \: 0], \, i=1, \cdots,n$.
Finally, to complete the SDP problem associated with Polya's algorithm, we choose $a$ as
\begin{equation}
a=\vec{1} \in \mathbb{R}^K. \label{eq:a} 
\end{equation}

\subsection{A Parallel Algorithm for Setting-up the SDP}
In this section, we propose a decentralized, iterative algorithm for calculating the terms $\{ \beta_{\langle h \rangle, \langle \gamma \rangle} \}$, $\{H_{\langle h \rangle, \langle \gamma \rangle}\}$, $C$ and $B_{i}$ as defined in~\eqref{eq:beta},~\eqref{eq:H},~\eqref{eq:C} and~\eqref{eq:Ai}. We have provided an MPI implementation of this algorithm in C++. The source code is available at \url{https://www.sites.google.com/a/asu.edu/kamyar/Software}. In Algorithm~\ref{alg:setup}, we have presented a pseudo-code for this algorithm, wherein $N$ is the number of available processors. 

\begin{algorithm}
\textbf{\textit{Inputs}:} 
$d_p$: degree of $P(\alpha)$, $d_a$: degree of $A(\alpha)$, $n$: number of states, $l$: No. of uncertain parameters, $d_1, d_2$: number of Polya's iterations, Coefficients of $A(\alpha)$. 

\textbf{\textit{Initialization}}: Set $\hat{d}_1=\hat{d}_2=0$ and $d_{pa}=d_p+d_a$. \\ 
Calculate $L_0$ as the No. of monomials in $P(\alpha)$ using~\eqref{eq:L0}. Set $L=L_0$. \\
Calculate $M$ as the No. of monomials in $P(\alpha)A(\alpha)$ using \eqref{eq:M}.\\
Calculate $ L'=\mathtt{floor}(\frac{L}{N})$ as No. of monomials in $P$ assigned to each processor. \\
Calculate $M'=\mathtt{floor}(\frac{M}{N})$ as the No. of monomials in $P(\alpha)A(\alpha)$ assigned to each processor.\\
 \For{$i=1, \cdots,N$, processor $i$}{
 Initialize $\beta_{k,j}$ for $j=(i-1)L'+1, \cdots, iL'$ and $k=1, \cdots L_0$ using~\eqref{eq:beta_init}. \\
 Initialize $H_{k,m}$ for $m=(i-1)M'+1, \cdots, iM'$ \& $k=1, \cdots L_0$ using~\eqref{eq:H_init}. \vspace{-0.1in}
 } 

 \textbf{\textit{Calculating $\beta$ and $H$ coefficients:}} \\
 \While{$\hat{d}_1 \leq d_1 $ or $\hat{d}_2 \leq d_2$}{  
  \If{$\hat{d}_1 \leq d_1 $}{ 
  \For{$i=1, \cdots,N$, processor $i$}{
  Set $d_p=d_p+1$.
  Set $\hat{d}_1=\hat{d}_1+1$. \\
   Update $L$ using~\eqref{eq:L}.
    Update $L'$ as $L'=\mathtt{floor}(\frac{L}{N})$.\\
 Calculate $\beta_{k,j}$, $j=(i-1)L'+1, \cdots, iL'$ \& $k=1, \cdots L_0$ using~\eqref{eq:beta}. \vspace{-0.4in}} \vspace{-0.15in}
 } 
 \If{$\hat{d}_2 \leq d_2 $}{
  \For{$i=1, \cdots,N$, processor $i$}{
 \hspace*{-0.1in}  Set $d_{pa}=d_{pa}  + 1$ and $\hat{d}_2=\hat{d}_2  +  1$. \\
 Update $M$ using~\eqref{eq:M}. 
 Update $M'$ as $M'=\mathtt{floor}(\frac{M}{N})$.\\
Calculate $H_{k,m}$ for $m=(i-1)M'+1, \cdots, iM'$ and $k=1, \cdots L_0$. using~\eqref{eq:H}. \vspace{-0.1in}}  \vspace{-0.2in}
 }  \vspace{-0.2in}
 } 
\end{algorithm}

\begin{algorithm}
 \textbf{\textit{Calculating the SDP elements}:} \\
 \For{$i=1, \cdots,N$, processor $i$}{
Calculate the number of dual variables $K$ using~\eqref{eq:K}.\\
 Set $T' = \mathtt{floor}(\frac{L+M}{N})$.\\
 Calculate the blocks of the SDP element $C$ as 
\[
 \begin{cases}
 C_j \; \text{using~\eqref{eq:Cj}} & \text{for} \; j = (i-1)L'+1, \cdots, iL' \\
 C_j = 0_n & \hspace*{-0.2in} \text{for} \; j=L+(i-1)M'+1, \cdots, L+iM'.
 \end{cases} 
 \]
  Set the sub-blocks of the SDP element $C$ as 
  \begin{equation}
 \overline{\textbf{C}}_i = \text{diag}\left(C_{(i-1)T'+1}, \cdots, C_{iT'} \right). \label{eq:Cbar}  
  \end{equation}
 \For{$j=1, \cdots,K$}{
Calculate the blocks of the SDP elements $B_j$ as 
\[
 \begin{cases}
 B_{j,k}  \; \text{using~\eqref{eq:Aij}-\textit{I}} & \text{for} \; k=(i-1)L'+1, \cdots, iL' \\
 B_{j,k} \; \text{using~\eqref{eq:Aij}-\textit{II}} & \text{for} \; k=L+(i-1)M'+1, \cdots, L+iM'.
 \end{cases}
 \]
Set the sub-blocks of the SDP element $B_j$ as 
  \begin{equation}
\overline{\textbf{B}}_{j,i}= \text{diag} \left(  B_{j,(i-1)T'+1}, \cdots, B_{j,iT'}           \right). \label{eq:Bbar}   
  \end{equation}
 }
 }
\vspace{0.15in}
 \textbf{\textit{Outputs}:}\\  Sub-blocks  $\overline{\textbf{C}}_i$ and $\overline{\textbf{B}}_{j,i}$ of the SDP elements for $i=1,\cdots,N$ and $j=1,\cdots,K$.
\vspace{0.15in}\caption{A parallel set-up algorithm for robust stability analysis over the standard simplex} 
\label{alg:setup}
\end{algorithm}

\section{Complexity Analysis of the Set-up Algorithm}
\label{sec:comp_analysis_setupTAC}

Since verifying the positive definiteness of all representatives of a square matrix with entries on proper real intervals is intractable~(\cite{nemirovskii1993several}), the question of feasibility of~\eqref{eq:Lyap_LMI} is also intractable. To solve the problem of inherent intractability we establish a trade off between accuracy and complexity. In fact, we develop a sequence of decentralized polynomial-time algorithms whose solutions converge to the exact solution of the NP-hard problem. In other words, the translation of a polynomial optimization problem to an LMI problem is the main source of complexity. This high complexity is unavoidable and, in fact, is the reason we seek parallel algorithms.
Algorithm~\ref{alg:setup} distributes the computation and storage of coefficients $\{ \beta_{\langle h \rangle, \langle \gamma \rangle} \}$ and $\{H_{\langle h \rangle, \langle \gamma \rangle}\}$ among the processors and their dedicated memories, respectively. In an ideal case, where the number of available processors is sufficiently large (equal to the number of monomials in $P(\alpha) A(\alpha)$, i.e. $M$) only one monomial (that corresponds to $L_0$ of coefficients $\beta_{\langle h \rangle, \langle \gamma \rangle}$ and $L_0$ of coefficients $H_{\langle h \rangle, \langle \gamma \rangle}$) is assigned to each processor. \vspace*{0.05in}

\subsection{Computational Complexity Analysis}

The most computationally expensive part of the set-up algorithm is the calculation of the $B_{i,j}$ blocks in~\eqref{eq:Aij}. Considering that the cost of matrix-matrix multiplication is $\sim n^3$, the cost of calculating each $B_{i,j}$ block is
$
\sim \text{card}(W_{d_p}) \cdot n^3.
$
According to~\eqref{eq:Ai} and~\eqref{eq:Aij}, the total number of $B_{i,j}$ blocks is $K(L+M)$. Hence, as per Algorithm~\ref{alg:setup}, each processor processes $K \left( \mathtt{floor}(\frac{L}{N} ) + \mathtt{floor}(\frac{M}{N}) \right)$ of the $B_{i,j}$ blocks, where $N$ is the number of available processors. Therefore, the per processor computational cost of calculating the $B_{i,j}$ at each Polya's iteration is 
\begin{equation}
\sim \text{card}(W_{d_p}) \cdot n^3 \cdot  K \left( \mathtt{floor}\left(\frac{L}{N}\right) + \mathtt{floor}\left(\frac{M}{N} \right) \right). 
\end{equation}
By substituting for $K$ from~\eqref{eq:K}, card$(W_{d_p})$ from~\eqref{eq:L0}, $L$ from~\eqref{eq:L} and $M$ from~\eqref{eq:M}, the per processor computation cost at each iteration is
\begin{align*}
\sim \left( \frac{(d_p + l-1)!}{d_p! (l-1)!} \right)^2 \hspace*{-0.05in}  \frac{n^4(n+1)}{2} & \left( \mathtt{floor}  \left( \frac{ \dfrac{(d_p+d_1+l-1)!}{(d_p+d_1)!(l-1)!}}{N} \right) \right. \nonumber \\ 
& \hspace{0in} \left. + \mathtt{floor} \left( \frac {\dfrac{(d_{pa}+d_2+l-1)!}{(d_{pa}+d_2)!(l-1)!}} {N} \right) \right),
\end{align*}
assuming that $l > 0$ and $N \leq M$, i.e., the number of monomials in $\left( \sum_{i=1}^l \alpha_i \right)^{d_2}  \hspace{-0.1in} P(\alpha) A(\alpha)$ is at least as large as the number of available processors. Under the assumption that the dynamical systems has large numbers of states and uncertain parameters (large $n$  and $l$), Table~\eqref{tab:setup_complexity} presents the computational cost per processor of each Polya's iteration for three different numbers of available processors. For the case where $d_p \geq 3$, the number of operations grows more slowly in $n$ than in $l$.

\renewcommand{\arraystretch}{1}
\begin{table}\scalebox{0.85}{
 \begin{tabular}{|c||c|c|c|} 
\hline 
 {\small \pbox{60cm}{Number of processors}} & $L_0$ & $L$ & $M$ \\ 
\hline 
{\small \pbox{60cm}{Computational \vspace{-0.1in} \\ cost per processor}} & $\sim (l^{2d_p+d_1}+l^{2d_p+d_a+d_2}) n^5$ & $\sim (l^{2d_p+d_1}+l^{2d_p+d_a+d_2}) n^5$ & $\sim l^{2d_p+d_a+d_2-d_1}n^5$ \\ 
\hline 
\end{tabular}} \vspace{0.1in}
\caption{Per Processor, Per Iteration Computational Complexity of the Set-up Algorithm. $L_0$ is the Number of Monomials Is $P(\alpha)$; $L$ Is the Number of Monomials in $\left( \sum_{i=1}^l \alpha_i \right)^{d_1} P(\alpha)$; $M$ Is the Number of Monomials in $\left( \sum_{i=1}^l \alpha_i \right)^{d_2} P(\alpha) A(\alpha)$.}
\label{tab:setup_complexity}
\end{table}

\subsection{Communication Complexity Analysis}

Communication between processors can be modeled by a directed graph $G(V,E)$, where the set of nodes $V=\{ 1, \cdots, N\}$ is the set of indices of the available processors and the set of edges $E=\{ (i,j):i,j \in V \}$ is the set of all pairs of processors that communicate with each other. For every directed graph, we can define an adjacency matrix $T_G$ as follows. If processor $i$ communicates with processor $j$, then $[T_G]_{i,j} = 1$, otherwise $[T_G]_{i,j} = 0$.  Here we only define the adjacency matrix for the part of the algorithm that performs Polya's iterations on $P(\alpha)$. For Polya's iterations on $P(\alpha)A(\alpha)$, the adjacency matrix can be defined in a similar manner.
For simplicity, we assume that at each iteration, the number of available processors is equal to the number of monomials in $(\sum_{i=1}^l \alpha_i)^{d_1}P(\alpha)$. Using~\eqref{eq:L}, let us define $r_{d_1}$ and $r_{d_1+1}$ as the numbers of monomials in $(\sum_{i=1}^l \alpha_i)^{d_1}P(\alpha)$ and $(\sum_{i=1}^l \alpha_i)^{d_1+1}P(\alpha)$.
For $I=1, \cdots, r_{d_1}$, define
\begin{equation*}
\mathcal{E}_I := \{ \text{lexicographical indices of monomials in} \left( \sum_{i=1}^l \alpha_i \right) \alpha^{\gamma}: \gamma \in W_{d_p+d_1} \; \text{and} \; \langle \gamma \rangle = I \}. 
\end{equation*}
Then, for $i=1, \cdots, r_{d_1+1}$ and $j=1, \cdots, r_{d_1+1}$, the adjacency matrix of the communication graph is
\[
[T_G]_{i,j} :=
\begin{cases}
1 & \text{if} \quad i \leq r_{d_1} \; \text{and} \; j \in \mathcal{E}_i \; \text{and}   \; i \neq j \\
0 & \text{otherwise}.
\end{cases}
\]
Note that this definition implies that the communication graph of the set-up algorithm changes at every iteration.
To help visualize the graph, the adjacency matrix for the case where $\alpha \in \Delta^2$~is
\begin{small}
\[
T_G  :=  \hspace*{-0.05in} \left[ 
\begin{array}{cccccc|ccc}
0 & 1 & 0 & \cdots &  0 &                          0 & 0 & \cdots & 0 \\
0 & 0 & 1 & 0 & \cdots  &                          0 & \\
\vdots & \vdots  & \ddots & \ddots & \ddots & \vdots & \vdots & \ddots & \vdots \\
\vdots & \vdots &  & \ddots & \ddots &             0 &  \\
0 & 0 & \cdots & \cdots & 0 &                      1 & 0 & \cdots & 0 \\
\hline
0 &  &  & \cdots & &                               0 & 0 & \cdots & 0 \\
\vdots &  &  & \ddots & &   \vdots   &  \vdots & \ddots &  \vdots \\
0 &  &  & \cdots & &                               0 &  0 & \cdots & 0\\
\end{array} \hspace*{-0.025in}  \right] \hspace*{-0.025in} \in  \mathbb{R}^{r_{d_1+1} \times r_{d_1+1}},
\]
\end{small}
where the nonzero sub-block of $T_G$ lies in $\mathbb{R}^{r_{d_1} \times r_{d_1}}$.
We can also illustrate the communication graphs for the cases $\alpha \in \Delta^3$ and $\alpha \in \Delta^4$ with $d_p=2$ as seen in Figure~\ref{fig:graph1}.

\begin{figure}[t]
\centering
\subfigure[]{
  \includegraphics[scale=0.4]{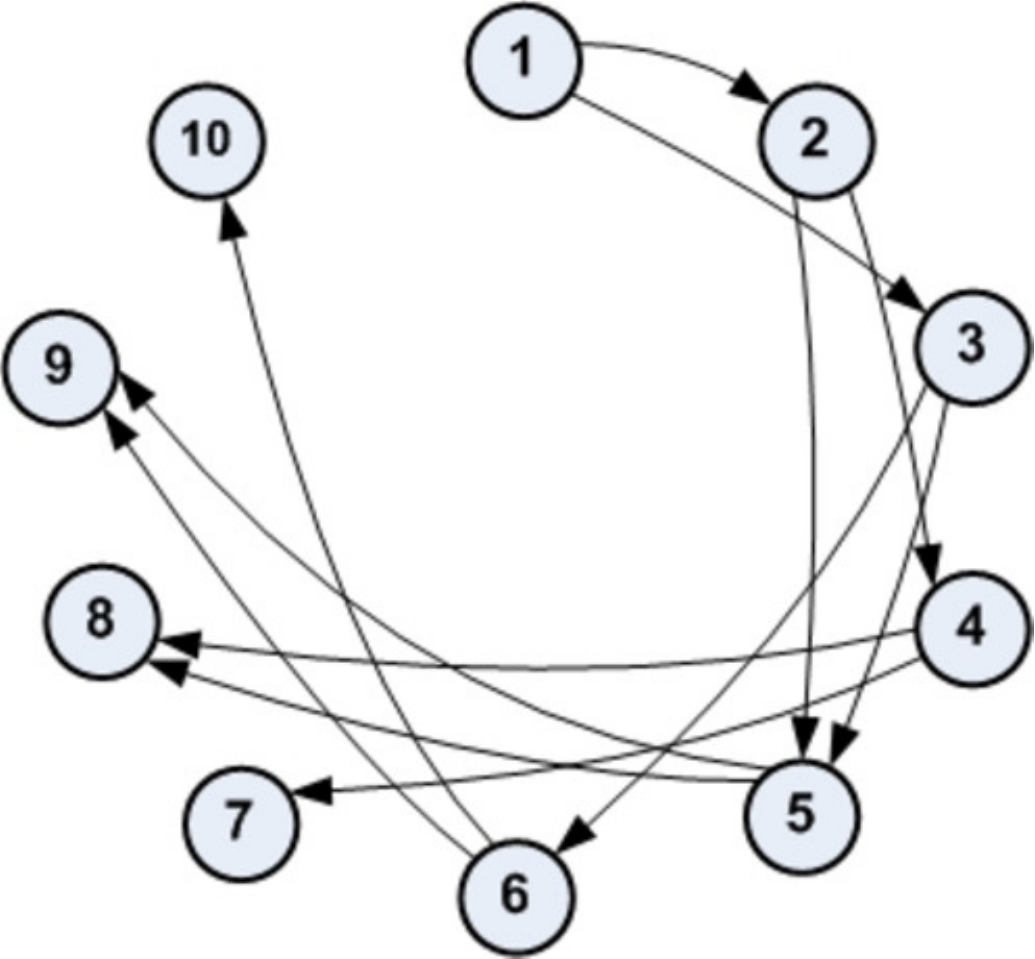}    
}  \hspace*{0.1in}
\subfigure[]{
\includegraphics[scale=0.4]{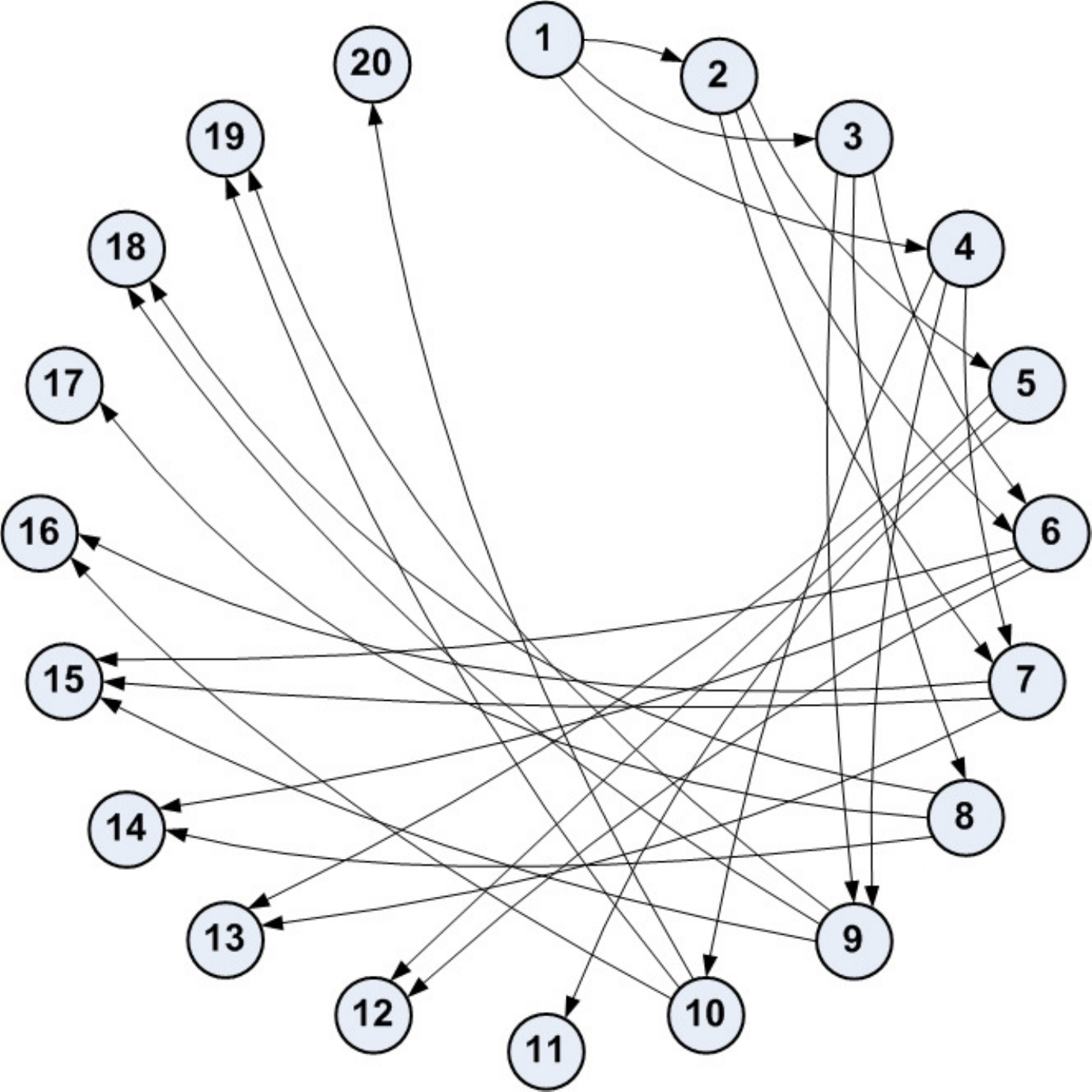}
}
\caption{Graph Representation of the Network Communication of the Set-up Algorithm. (a) Communication Directed Graph for the Case $\alpha \in \Delta^3,d_p=2$. (b) Communication Directed Graph for the Case $\alpha \in \Delta^4,d_p=2$.}
\label{fig:graph1} 
\end{figure}

For a given algorithm, the communication complexity is defined as the sum of the size of all communicated messages. For simplicity, let us consider the worst case scenario, where each processor is assigned more than one monomial and sends all of its assigned coefficients $\beta_{\langle h \rangle, \langle \gamma \rangle} $ and $H_{\langle h \rangle, \langle \gamma \rangle} $  to other processors. In this case, the algorithm assigns $\mathtt {floor}(\frac{L}{N})\cdot \text{card}(W_{d_p})$ of coefficients $\beta_{\langle h \rangle, \langle \gamma \rangle}$, each of size 1, and $\left( \mathtt {floor}(\frac{L}{N}) + \mathtt {floor}(\frac{M}{N}) \right) \cdot \text{card}(W_{d_p})$ of coefficients $H_{\langle h \rangle, \langle \gamma \rangle}$, each of size $n^2$, to each processor. Thus, the communication complexity of the algorithm per processor and per iteration becomes  
\begin{equation}
\text{card}(W_{d_p}) \left( \mathtt {floor}\left(\frac{L}{N}\right) + \mathtt {floor}\left(\frac{M}{N}\right) n^2 \right). 
\label{eq:setup_comm}
\end{equation}
This indicates that increasing the number of processors (up to $M$) actually leads to less communication overhead per processor and improves the scalability of the algorithm. By substituting for card$(W_{d_p})$ from~\eqref{eq:L0}, $L$ from~\eqref{eq:L} and $M$ from~\eqref{eq:M} and considering large $l$ and $n$, the communication complexity per processor of each Polya's iteration can be represented as in Table~\ref{tab:setup_comm_complexity}.

\renewcommand{\arraystretch}{1}
\begin{table} 
\begin{center}
\begin{tabular}{|c||c|c|c|}
\hline 
Number of processors & $L_0$ & $L$ & $M$ \\ 
\hline 
communication cost per processor & $\sim  l^{d_{pa}+d_2} n^2 $ & $\sim l^{d_{pa}+d_2-d_1} n^2$ & $\sim l^{d_p} n^2$ \\ 
\hline 
\end{tabular}
\end{center}
\caption{Per Processor, Per Iteration Communication Complexity of the Set-up Algorithm. $L_0$ is the Number of Monomials Is $P(\alpha)$; $L$ Is the Number of Monomials in $\left( \sum_{i=1}^l \alpha_i \right)^{d_1} P(\alpha)$; $M$ Is the Number of Monomials in $\left( \sum_{i=1}^l \alpha_i \right)^{d_2} P(\alpha) A(\alpha)$.}
\label{tab:setup_comm_complexity}
\end{table}

\section{A Parallel SDP Solver}
\label{sec:SDPSOLVER}

Current state-of-the-art interior-point algorithms for solving linear and semi-definite programs are: dual-scaling, primal-dual, cutting-plane/spectral bundle method. Although we found it possible to use dual-scaling algorithms, we chose to pursue a central-path following primal-dual algorithm. One reason that we prefer primal-dual algorithms is because in general, primal-dual algorithms converge faster than dual-scaling algorithms. This assertion is motivated by experience as well as bounds on the convergence rate, such as those found in the literature~(\cite{helmberg,benson2000solving}). More importantly, we prefer primal-dual algorithms because they have the property of preserving the structure (see~\eqref{eq:structure}) of the solution at each iteration. We will elaborate on this property in Theorem~\ref{thm:structure}. 

We prefer primal-dual algorithms over cutting plane/spectral bundle algorithms because, as we show in Section~\ref{sec:complexity_TAC}, the centralized part of our primal-dual algorithm consists of solving a symmetric system of linear equations (see~\eqref{eq:SAE1}), whereas for the cutting plane/spectral bundle algorithm, the centralized computation would consist of solving a constrained quadratic program (see~\cite{krishnan},~\cite{mahdu}) with the number of variables equal to the size of the system of linear equations. Because centralized computation is a limiting factor in a parallel algorithm (Amdahl's law), and because solving symmetric linear equations is simpler than solving a quadratic programming problem, we chose the primal-dual approach.

The choice of a central path-following primal-dual algorithm as in~\cite{helmberg} and~\cite{Alizadeh_method} was motivated by results in~\cite{alizadeh} which demonstrated better convergence, accuracy and robustness over the other types of primal-dual algorithms. More specifically, we chose the approach in~\cite{helmberg} over~\cite{Alizadeh_method} because unlike the \textit{Schur Complement Matrix} (SCM) approach of the algorithm in~\cite{Alizadeh_method}, the SCM of~\cite{helmberg} is symmetric and only the upper-triangular elements need to be sent/received by the processors. This leads to less communication overhead. The other reason for choosing~\cite{helmberg} is that the symmetric SCM of the algorithm in~\cite{helmberg} can be factorized using Cholesky factorization, whereas the non-symmetric SCM of~\cite{Alizadeh_method} must be factorized by LU factorization (LU factorization is roughly twice as expensive as Cholesky  factorization). Since factorization of SCM comprises the main portion of the centralized computation in our algorithm, it is crucial for us to use computationally-cheaper factorization methods to achieve a better scalability.

Recall from Section~\ref{sec:primal_dual} that in the primal-dual algorithm, both primal and dual problems are solved by iteratively calculating primal and dual search directions and step sizes, and applying these to the primal and dual variables. Let $X$ be the primal variable and $y$ and $Z$ be the dual variables. At each iteration, the variables are updated as
\begin{align}
X_{k+1} & = X_k + t_p \Delta X  \label{eq:X} \\
y_{k+1} & = y_k + t_d \Delta y  \label{eq:y} \\
Z_{k+1} & = Z_k + t_d \Delta Z  \label{eq:Z},
\end{align}
where $\Delta X$, $\Delta y$, and $\Delta Z$ are the search directions defined in~\eqref{eq:primal-dual_step} and $t_p$ and $t_d$ are primal and dual step sizes. For the SDPs associated with Polya's theorem (see~\eqref{eq:primal_TAC} and~\eqref{eq:dual_TAC}), because the map $G$ (defined in~\eqref{eq:B_G}) is zero, the predictor search directions defined in~\eqref{eq:deltay_deltat}-\eqref{eq:deltaZ} reduce to the following:
\begin{align}
&\Delta \widehat{y} = \Omega^{-1}\left(-a + B( Z^{-1} G X )\right)   \\
&\Delta \widehat{X} = -X + Z^{-1} G  \left( \sum_{i=1}^K B_i \Delta \widehat{y}_i \right) X      \label{eq:delX_hat_TAC} \\
&\Delta \widehat{Z} = \left( \sum_{i=1}^K B_i y_i \right) - Z - C + \left( \sum_{i=1}^K B_i \Delta \widehat{y}_i \right),   \label{eq:delZ_hat_TAC}
\end{align}
where
\begin{equation}
 G = -\sum_{i=1}^K B_i y_i  + Z + C,  
\label{eq:G} 
\end{equation}
 and $ \Omega = \left[ B( Z^{-1} B_1 X ) \; \cdots \; B( Z^{-1} B_K X ) \right]$.
Similarly, the corrector search directions defined in~\eqref{eq:deltay_deltat_bar}-\eqref{eq:deltaZbar} reduce to
\begin{align}
& \Delta \overline{y} = \Omega^{-1}\left(  B( \mu Z^{-1} ) - B( Z^{-1} \Delta \widehat{Z} \Delta \widehat{X} ) \right)  \\
& \Delta \overline{X}   =   \mu Z^{-1}-Z^{-1} \Delta \widehat{Z} \Delta \widehat{X} - Z^{-1} \Delta \overline{Z} X \label{eq:delx_bar_TAC}  \\
& \Delta \overline{Z}   =  \sum_{i=1}^K B_i \Delta \overline {y}_i. \label{eq:delz_bar_TAC}   
\end{align}
In the following section, we discuss the structure of the decision variables of the SDP defined by the Elements~\eqref{eq:C},~\eqref{eq:Ai} and~\eqref{eq:a}.

\subsection{Structure of the SDP Variables}

The key algorithmic insight of this study which allows us to use the primal-dual approach presented in Algorithm~\ref{alg:primal_dual_centeralized} is that by choosing an initial value for the primal variable with a certain block structure corresponding to the distributed structure of the processors, the algorithm will preserve this structure on the primal and dual variables at every iteration. Specifically, we define the following structured block-diagonal subspace, where each block corresponds to a single processor. 
\begin{equation}
S_{l,m,n} := \left\lbrace Y \subset \mathbb{R}^{(l+m)n \times (l+m)n} :  Y = \text{diag}(Y_1,\cdots Y_l, Y_{l+1}, \cdots Y_{l+m}) \; \text{for} \; Y_i \in \mathbb{R}^{n \times n} \right\rbrace
 \label{eq:structure}
\end{equation}
According to the following theorem, the subspace $S_{l,m,n}$ is invariant under the predictor and corrector iterations in the sense that when Algorithm~\ref{alg:primal_dual_centeralized} is applied to the SDP problem defined by the Elements~\eqref{eq:C},~\eqref{eq:Ai} and~\eqref{eq:a} with a primal starting point $X_0 \in S_{l,m,n}$, then the primal and dual variables remain in the subspace at every iteration.

\begin{mythm}\label{thm:structure}
Consider the SDP problem defined in~\eqref{eq:primal_TAC} and~\eqref{eq:dual_TAC} with elements given by~\eqref{eq:C},~\eqref{eq:Ai} and~\eqref{eq:a}. Suppose $L$ and $M$ are the cardinalities of $W_{d_p+d_1}$ and $W_{d_{pa} +d_2}$ as defined in~\eqref{eq:L} and~\eqref{eq:M}. If~\eqref{eq:X},~\eqref{eq:y} and~\eqref{eq:Z} are initialized by 
\[
X_0 \in S_{L,M,n}, \quad y_0 \in \mathbb{R}^{K}, \quad Z_0 \in S_{L,M,n}, 
\]
then for all $ k \in \mathbb{N} $, 
\[
 X_k \in S_{L,M,n}, \quad Z_k \in S_{L,M,n}. 
\]
\end{mythm}

\begin{proof}
We proceed by induction. First, suppose for some $k \in \mathbb{N}, $ 
\begin{equation}
 X_k \in S_{L,M,n} \quad \text{and}\quad Z_k \in S_{L,M,n}. \label{eq:assume} 
\end{equation}
We would like to show that this implies $X_{k+1}, Z_{k+1} \in S_{L,M,n}$. To see this, observe that according to~\eqref{eq:X}, $X_{k+1} = X_k + t_p \Delta X_k$ for all $k \in \mathbb{N}$.
From~\eqref{eq:primal-dual_step}, $\Delta X_k$ can be written as  
\begin{equation}
\Delta X_k = \Delta \widehat{X}_k + \Delta \overline{X}_k \quad \text{for all} \;\, k \in \mathbb{N}.  
\label{eq:DelX_k}
\end{equation}
To find the structure of $\Delta X_k$, we focus on the structures of $\Delta \widehat{X}_k$ and $\Delta \overline{X}_k$ individually. Using~\eqref{eq:delX_hat_TAC}, $\Delta \widehat{X}_k$ is 
\begin{equation}
\Delta \widehat{X}_k = -X_k + Z_k^{-1} G_k  \left( \sum_{i=1}^K B_i \Delta \widehat{y}_k \right) X_k \quad \text{for all} \;\, k \in \mathbb{N},
\label{eq:Dxhat_k} 
\end{equation}
where according to~\eqref{eq:G}, $G_k$ is
\begin{equation}
G_k = C-\sum_{i=1}^K B_i y_i  + Z_k \quad \text{for all} \;\, k \in \mathbb{N}.  
\label{eq:G_k1}
\end{equation}
First, we examine the structure of $G_k$. According to the definition of $C$ and $B_i$ in~\eqref{eq:C} and~\eqref{eq:Ai}, we know that 
\begin{equation}
C \in S_{L,M,n} \quad \text{and} \quad \sum_{i=1}^K B_i y_i \in S_{L,M,n} \;\; \text{ for any } y \in \mathbb{R}^K.
\label{eq:CA} 
\end{equation}
Since all the terms on the right hand side of~\eqref{eq:G_k1} are in $S_{L,M,n}$ and $S_{L,M,n}$ is a subspace, we conclude 
\begin{equation}
G_k \in S_{L,M,n}. 
\label{eq:G_k} 
\end{equation}
Returning to~\eqref{eq:Dxhat_k}, using our assumption in~\eqref{eq:assume} and noting that the structure of the matrices in $S_{L,M,n}$ is also preserved through multiplication and inversion, we conclude 
\begin{equation}
\Delta \widehat{X}_k \in S_{L,M,n}.
\label{eq:delxhat}
\end{equation}
According to~\eqref{eq:delx_bar_TAC}, the second term in~\eqref{eq:DelX_k} is 
\begin{equation}
\Delta \overline{X}_k = \mu Z_k^{-1}-Z_k^{-1} \Delta \widehat{Z}_k \Delta \widehat{X}_k - Z_k^{-1} \Delta \overline{Z}_k X_k \quad \text{for all} \;\, k \in \mathbb{N}. 
\label{Delxbar} 
\end{equation}
To determine the structure of $\Delta \overline{X}_k$, first we investigate the structure of $\Delta \widehat{Z}_k$ and $\Delta \overline{Z}_k$. According to~\eqref{eq:delZ_hat_TAC} and~\eqref{eq:delz_bar_TAC} we have 
\begin{align}
&\Delta \widehat{Z}_k =  \sum_{i=1}^K B_i y_{k_i} - Z_k - C +  \sum_{i=1}^K B_i  \Delta \widehat{y}_{k_i}  &&  \text{for all} \;\, k \in \mathbb{N}\label{eq:delZ1}\\
&\Delta \overline{Z}_k = \sum_{i=1}^K B_i \Delta \overline{y}_{k_i} &&  \text{for all} \;\, k \in \mathbb{N}. 
\label{eq:delZ2}
\end{align}
Because all the terms in the right hand side of~\eqref{eq:delZ1} and~\eqref{eq:delZ2} are in $S_{L,M,n}$, it follows that  
\begin{equation}
\Delta \widehat{Z}_k \in S_{L,M,n}, \quad \Delta \overline{Z}_k \in S_{L,M,n}. \label{eq:delZ} 
\end{equation}
Recalling~\eqref{eq:delxhat},~\eqref{Delxbar} and our assumption in~\eqref{eq:assume}, we have 
\begin{equation}
\Delta \overline{X}_k \in S_{L,M,n}.
\label{eq:delxbar}
\end{equation}
According to~\eqref{eq:delxhat},~\eqref{eq:delZ} and~\eqref{eq:delxbar}, the total step directions are in $S_{L,M,n}$, 
\begin{align*}
&\Delta X_k = \Delta \widehat{X}_k + \Delta \overline{X}_k \in S_{L,M,n}\\
&\Delta Z_k = \Delta \widehat{Z}_k + \Delta \overline{Z}_k \in S_{L,M,n},
\end{align*}
and it follows that 
\begin{align*}
& X_{k+1} = X_k + t_p \Delta X_k \in S_{L,M,n}\\
& Z_{k+1} = Z_k + t_p \Delta Z_k \in S_{L,M,n}.
\end{align*}
Thus, for any $y \in \mathbb{R}^{K}$ and $k \in \mathbb{N} $, if $X_k,Z_k \in S_{L,M,n}$, we have $X_{k+1},Z_{k+1} \in S_{L,M,n}$. Since we have assumed that the initial values $X_0, Z_0 \in S_{L,M,n}$, we conclude by induction that $X_k \in S_{L,M,n}$ and $Z_k \in S_{L,M,n}$ for all $k \in \mathbb{N}$.
\end{proof}

\subsection{A Parallel Implementation for the SDP Solver}

In this section, we propose a parallel algorithm for solving the SDP problems associated with Polya's algorithm. In particular, we show how to map the block-diagonal structure of the primal variable and the primal-dual search directions described in Section~\ref{sec:SDPSOLVER} to a parallel computing structure consisting of a central root processor with $N$ slave processors. Note that processor steps are simultaneous and transitions between root and processor steps are synchronous. Processors are idle when root is active and vice-versa. A C++ implementation of this algorithm using MPI and numerical linear algebra libraries CBLAS and CLAPACK is provided at:
\url{www.sites.google.com/a/asu.edu/kamyar/Software}.
Let $N$ be the number of available processors and $J := \mathtt {floor}\left( \frac{L+M}{N} \right)$. As per Algorithm 6, we assume processor $i$ has access to the sub-blocks $ \mathbf{\overline{C}}_i $ and $ \mathbf{\overline{\textbf{B}}}_{j,i} $ defined in~\eqref{eq:Cbar} and~\eqref{eq:Bbar} for $j=1, \cdots, K$. Be aware that minor parts of Algorithm 6 have been abridged in order to simplify the presentation.

\begin{algorithm}
\textbf{\textit{Inputs}:} $\overline{\textbf{C}}_i,\overline{\textbf{B}}_{j,i}$ for $i=1,\cdots,N$ and $j=1,\cdots,K$: the sub-blocks of the SDP elements  provided to processor $i$ by the set-up algorithm; Stopping criterion $\epsilon$. 

\textbf{\textit{Processors Initialization step:}}\\
\For{$i = 1, \cdots, N$, processor $i$}{
Initialize primal and dual variables $\mathbf{X}^0_i$, $\mathbf{Z}^0_i$ and $y^0$ as 
\begin{equation*}
\mathbf{X}^0_i =
\begin{cases}
    I_{(J +1) n },  &   0 \le i < L+M-N J  \\
	I_{J n },  &   L+M-N J \le  i  < N,
\end{cases},
\end{equation*}
\begin{equation*}
  \mathbf{Z}^0_i = \mathbf{X}^0_i \quad \text{and} \quad y^0 = \vec{0} \in \mathbb{R}^K, \label{eq:init}
\end{equation*}
Calculate the complementary slackness as $S_{i} = tr( \mathbf{Z}_i^0 \mathbf{X}_i^0 )$.\\
 Send $S_{i}$ to the processor root.
}

\textbf{\textit{Root Initialization step:}}\\
Root processor \textbf{do}\\ 

\hspace*{0.05in} Calculate the barrier parameter $\mu = \frac{1}{3} \sum\limits_{i=1}^{N} S_{i}$. Set SDP element $a= \vec{1} \in \mathbb{R}^K.$ \\

\textbf{\textit{Processors step 1:}}\\ 
\For{$i=1, \cdots,N$, processor $i$}{ 
\For{$k=1, \cdots,K$}{
Calculate the elements of $\Omega_1$ (R-H-S of System~\eqref{eq:SAE1}) 
\begin{small}
\[
\omega_{i,k}  =  \text{tr} \left( \mathbf{\overline{\textbf{B}}}_{k,i} (\mathbf{{Z}}_i)^{-1} \left(-\sum_{j=1}^K y_j \mathbf{\overline{\textbf{B}}}_{j,i} + \mathbf {Z}_i + \mathbf{\overline{C}}_i \right) \mathbf{X}_i \right) \vspace{-0.15in}  
\]
\end{small}
\For{$l=1, \cdots,K$}{
Calculate the elements of the SCM as  \vspace{-0.15in}
\begin{equation}
 \hspace*{-0.3in} \lambda_{i,k,l} =  \text{tr} \left( \mathbf{\overline{\textbf{B}}}_{k,i} (\mathbf{{Z}}_i)^{-1} \mathbf{\overline{\textbf{B}}}_{l,i} \mathbf{X}_i \right) \label{eq:T2}   \vspace{-0.2in}
\end{equation}
} \vspace{-0.2in} }
Send $\omega_{i,k}$ and $\lambda_{i,k,l}$, $k=1, \cdots,K$ and $l=1, \cdots,K$ to the root processor.\\
}
\label{alg:primal_dual_decenteralized}
\end{algorithm}
\begin{algorithm}

\textbf{\textit{Root step 1:}} \vspace{-0.1in}\\ 
Root processor \textbf{do} \\ 
\hspace*{0.05in} Construct the R-H-S of System~\eqref{eq:SAE1} and the SCM as 
\begin{small}
\begin{equation*}
\hspace*{0.1in} \Omega_1 = \left( \begin{array}{ccc}
  \sum_{i=1}^{N} \omega_{i,1} \\
  \sum_{i=1}^{N} \omega_{i,2} \\
  \vdots\\
  \sum_{i=1}^{N} \omega_{i,K}
\end{array} \right)-a \;\,\, \text{and} \;
\Lambda= \left[ \left(\begin{array}{ccc}
  \sum_{i=1}^{N} \lambda_{i,1,1} \\
  \sum_{i=1}^{N} \lambda_{i,2,1} \\
  \vdots\\
  \sum_{i=1}^{N} \lambda_{i,K,1} \end{array} \right) , \cdots,
  \left( \begin{array}{ccc}
  \sum_{i=1}^{N} \lambda_{i,1,K} \\
  \sum_{i=1}^{N} \lambda_{i,2,K} \\
  \vdots \\
  \sum_{i=1}^{N} \lambda_{i,K,K} \end{array}
  \right) \right]
\end{equation*}
\end{small}
\hspace*{0.1in} Solve the following system of equations for the predictor dual step ${\Delta \widehat{y}}$. \vspace{-0.2in}
\begin{equation}
 \hspace*{0.05in} \Lambda {\Delta \widehat{y}} = \Omega_1   \label{eq:SAE1}
 \vspace{-0.2in}
\end{equation}
\hspace{0.1in} Send ${\Delta \widehat{y}}$ to all processors.

 \textbf{\textit{Processors step 2:}} \vspace{-0.1in}\\ 
\For{$i=1, \cdots,N$, processor $i$}{
Calculate the predictor step directions \vspace{-0.1in}
\begin{align*}
& \Delta \mathbf{\widehat{X}}_i = -\mathbf{X}_i 
 + (\mathbf{{Z}}_i)^{-1}  \left( -\sum_{j=1}^K y_j \mathbf{\overline{\textbf{B}}}_{j,i}  + \mathbf {Z}_i + \mathbf{\overline{C}}_i \right) \sum_{j=1}^K \Delta \widehat{y}_j \, \mathbf{\overline{\textbf{B}}}_{j,i} \; \mathbf{X}_i, \\
& \Delta \mathbf{\widehat{Z}}_i = \sum_{j=1}^K y_j \mathbf{\overline{\textbf{B}}}_{j,i} - \mathbf{Z}_i - \mathbf{\overline{C}}_i + \sum_{j=1}^K \Delta \widehat{y}_j \mathbf{\overline{\textbf{B}}}_{j,i}.
\end{align*}
\vskip-0.2in
\For{$k=1, \cdots,K$}{
Calculate the elements of $\Omega_2$ (R-H-S of~\eqref{eq:SAE2})  \vspace{-0.1in}
\[
 \hspace*{-0.1in}\delta_{i,k} = \text{tr}(\mathbf{\overline{\textbf{B}}}_{k,i} (\mathbf{{Z}}_i)^{-1})  ,
 \tau_{i,k} = tr(\mathbf{\overline{\textbf{B}}}_{k,i} (\mathbf{{Z}}_i)^{-1} \Delta \mathbf{\widehat{Z}}_i \Delta \mathbf{ \widehat{X}}_i)  \vspace{-0.2in}
\]
} \vskip-0.1in
Send $\delta_{i,k}$ and $\tau_{i,k}$, $k=1, \cdots,K$ to the root processor. \vspace{-0.1in}
}
 \textbf{\textbf{\textit{Root step 2:}}} \vspace{-0.1in}  \\ 
\hspace*{0.05in} Construct the R-H-S of~\eqref{eq:SAE2} as \vspace*{-0.05in}
\begin{align*}
 \Omega_2 =
 \mu \begin{bmatrix}
\sum\limits_{i=1}^{N} \delta_{i,1} & \sum\limits_{i=1}^{N} \delta_{i,2} & \cdots & \sum\limits_{i=1}^{N} \delta_{i,K}
\end{bmatrix}^T
- 
  \begin{bmatrix}
  \sum\limits_{i=1}^{N} \tau_{i,1} & \sum\limits_{i=1}^{N} \tau_{i,2} &
  \cdots &
  \sum\limits_{i=1}^{N} \tau_{i,K} \end{bmatrix}^T  
\end{align*}

\hspace*{0.05in} Solve the following system of equations for the corrector dual variable $\Delta \overline{y}$. \vspace{-0.2in}
\begin{equation}
\Lambda \Delta \overline{y} = \Omega_2 \vspace{-0.2in}
  \label{eq:SAE2} 
\end{equation}
\hspace*{0.05in} Send $\Delta \overline{y}$ to all processors.
\end{algorithm}

\begin{algorithm}
\textbf{\textit{Processors step 3:}}  \vspace{-0.1in} \\ 
\For{$i=1, \cdots,N$, processor $i$}{
Calculate the corrector step directions as follows. \vspace{-0.15in} 
\begin{equation*}
\Delta \mathbf{ \overline{Z}}_i = \sum_{j=1}^K \Delta \overline {y}_j \mathbf{\overline{\textbf{B}}}_{j,i}  \vspace*{-0.05in}
\end{equation*}
\begin{equation*}
\Delta \mathbf{ \overline{X}}_i = -(\mathbf{{Z}}_i)^{-1} ( \Delta \mathbf{ \overline{Z}}_i \mathbf{X}_i + \Delta \mathbf{ \widehat{Z}}_i \Delta \mathbf{ \widehat{X}}_i) + \mu (\mathbf{{Z}}_i)^{-1}  \vspace{-0.15in}
\end{equation*}
Calculate the primal and dual total search directions as   \vspace{-0.15in} 
\begin{equation*}
\Delta \mathbf{ X}_i  = \Delta \mathbf{\widehat{X}}_i + \Delta \mathbf{ \overline{X}}_i, \quad
\Delta \mathbf{ Z}_i = \Delta \mathbf{\widehat{Z}}_i + \Delta \mathbf{ \overline{Z}}_i, \quad
\Delta {y} = \Delta \widehat {y} + \Delta \overline {y}.   \vspace{-0.15in}  
\end{equation*}
Set the primal step size $t_p$ and dual step size $t_d$ using a line search method.\\ 
Update the primal and dual variables as   \vspace{-0.15in} 
\[
\mathbf{X}_i \equiv \mathbf{X}_i + t_p \Delta \mathbf{X}_i, \quad \mathbf{Z}_i \equiv \mathbf{Z}_i + t_d \Delta \mathbf{Z}_i , \quad y \equiv y  + t_d \Delta {y}   \vspace{-0.25in} 
\]
}  \vspace{-0.1in}

\textbf{\textit{Processors step 4:}} \vspace{-0.1in}\\ 
\For{$i=1, \cdots,N$, processor $i$}{
Calculate the contribution of $X_i$ to primal cost and complementary slackness as  \vspace{-0.1in} 
\[ \tilde{\phi}_{i}  = tr \left( \mathbf{\overline{C}}_i \mathbf{X}_i \right) \quad \text{and} \quad S_{i}  = tr \left( \mathbf{Z}_i \mathbf{X}_i \right).
 \vspace{-0.1in} 
\]
Send $S_{i}$ and $\tilde{\phi}_{i}$ to the root processor. \vspace{-0.1in}
} \vspace{-0.1in}

\textbf{\textit{Root step 4:}} \vspace{-0.1in}\\ 
\hspace*{0.05in} Update the barrier parameter as $\mu = \frac{1}{3} \sum_{i=1}^{N} S_{i}$.\\
\hspace*{0.05in} Calculate the primal and dual costs as  $\phi = \sum_{i=1}^{N} \tilde{\phi}_{i} \; \text{and} \; \psi= a^Ty $.\\
  \If{$ | \phi - \psi | > \epsilon$}{ 
  go to Processors step 1  \vspace{-0.1in} 
}  \vspace{-0.15in} 
\Else{
	Calculate the coefficients of $P(\alpha)$ as $P_i= \sum_{j=1}^{\tilde{N}} E_j y_{(j + \tilde{N}i-1))}$ for $i=1, \cdots, L_0$.  \vspace{-0.1in} 
	} \vspace{-0.1in}
	
\textbf{\textit{Output:}}
Coefficients $P_i$ of a polynomial $P(\alpha)$ such that $P(\alpha) > 0$ for all $\alpha \in \Delta^l$ and satisfies the Lyapunov inequalities in~\eqref{eq:Lyap_LMI}.
\vspace{0.15in}\caption{A parallel SDP algorithm}
\end{algorithm}

\section {Computational complexity analysis of the SDP algorithm}
\label{sec:complexity_TAC}

Complexity theory for parallel computation has been studied in some depth~(\cite{limits}). The class NC $\subset$ P is often considered to be the class of problems that can be parallelized efficiently. More precisely, a problem is in NC if there exist integers $c$ and $d$ such that the problem can be solved in $\mathcal{O}(\log(n)^c)$ steps using $\mathcal{O}(n^d)$ processors.
On the other hand, the class P-complete is a class of problems which are equivalent up to an NC reduction, but contains no problem in NC and is thought to be the simplest class of ``inherently sequential" problems. It has been proven that Linear Programming (LP) is P-complete~\cite{limits} and SDP is P-hard (at least as hard as any P-complete problem) and thus is unlikely to admit a general-purpose parallel solution. Given this fact and given the observation that the problem we are trying to solve is NP-hard, it is important to thoroughly understand the complexity of the algorithms we are proposing and how this complexity scales with various parameters which define the size of the stability analysis problem. To better understand these issues, we have broken our complexity analysis down into several cases which should be of interest to the control community. Note that the cases below do not discuss memory complexity. This is because in the cases when a sufficient number of processors are available, for a system with $n$ states, the memory requirements per block are simply proportional to $n^2$.

\subsection{Complexity Analysis for Systems with Large Number of States}
\label{sec:comp_analysis_SDP}

Suppose we are considering a problem with $n$ states. For this case, the most computationally expensive part of the algorithm is the calculation of the Schur complement matrix $\Lambda$ by the processors in Processors step 1 (and summed by the root in Root step 1, although we neglect this part). In particular, the computational complexity of the algorithm is determined by the number of operations required to calculate~\eqref{eq:T2}, restated here.
\begin{equation*}
\lambda_{i,k,l} =  tr \left( \mathbf{\overline{\textbf{B}}}_{k,i} (\mathbf{{Z}}_i)^{-1} \mathbf{\overline{\textbf{B}}}_{l,i} \mathbf{X}_i \right) 
\quad \text{for} \quad k= 1, \cdots, K \; \text{and} \quad l=1, \cdots, K. 
\end{equation*}
Since the cost of $n \times n$ matrix-matrix multiplication requires $\mathcal{O}(n^3)$ steps and each of $\mathbf{X}_i, \mathbf{{Z}}_i, \mathbf{\overline{\textbf{B}}}_{l,i}$ has $\mathtt{floor}(\frac{L+M}{N})$ number of blocks in $\mathbb{R}^{n \times n}$, the number of operations performed by the $i^{th}$ processor to calculate  $\lambda_{i,k,l}$ for $k= 1, \cdots, K$ and $l=1, \cdots, K$ is proportional to
\begin{equation}
\begin{cases}
 \mathtt{floor} \left( \dfrac{L+M}{N} \right)  K^2 n^3 \; & N < L+M \\
 K^2 n^3 \; & N \geq L+M
\end{cases} 
\label{eq:flop1}
\end{equation}
at each iteration, where $i=1, \cdots,N$. By substituting $K$ in~\eqref{eq:flop1} from~\eqref{eq:K}, for $N \geq L+M$, each processor performs 
\begin{equation}
\sim \dfrac{((d_p+l-1)!)^2}{(d_p!)^2((l-1)!)^2}n^7 
\label{eq:flop_proc} 
\end{equation}
operations per iteration. Therefore, for systems with large number $n$ of states and fixed degree $d_p$ of $P(\alpha)$ and number $l$ of uncertain parameters, the number of operations per processor required to solve the SDP associated with parameter-dependent feasibility problem
$
A(\alpha)^TP(\alpha)+P(\alpha)A(\alpha) < 0,
$
is proportional to $n^7$. Solving the LMI associated with the parameter-independent problem
$
\mathrm{A}^TP+P\mathrm{A} < 0
$
using our algorithm or most of the SDP solvers such as~\cite{sedumi,csdp,sdpara} also requires $ O(n^7)$ operations per processor. Therefore, if we have a sufficient number of available processors (at least $L+M$), the proposed algorithm solves both the stability and robust stability problems by performing $O(n^7)$ operations per processor.

\subsection{Complexity of Increasing Accuracy/Decreasing Conservativeness}
\label{sec:accuracy_TAC}

We now consider the effect of raising Polya's exponent. Consider the definition of simplex as follows.
\begin{equation*}
\tilde{\Delta}^l_r=\left\lbrace \alpha\in \mathbb{R}^l : \sum_{i=1}^{l} \alpha_i=r, \alpha_i\geqslant 0 \right\rbrace
\end{equation*}
Suppose we now define the accuracy of the algorithm as the largest value of $r$ found by the algorithm (if it exists) such that if the uncertain parameters lie inside the corresponding simplex, the stability of the system is verified. Typically, increasing Polya's exponent $d$ in~\eqref{eq:polya_product_simplex} improves the accuracy of the algorithm. If we again only consider Processor step 1, according to~\eqref{eq:flop_proc}, the number of processor operations is independent of the Polya's exponent $d_1$ and $d_2$. Because this part of the algorithm does not vary with Polya's exponent, we look at the root processing requirements associated with solving the systems of equations in~\eqref{eq:SAE1} and~\eqref{eq:SAE2} in Root step 1 using Cholesky factorization. Each of these systems consists of $K$ equations. The computational complexity of Cholesky factorization is $O(K^3)$. Thus, the number of operations performed by the root processor is proportional to 
\begin{equation} 
K^3 = \dfrac{((d_p+l-1)!)^3}{(d_p!)^3((l-1)!)^3}n^6. 
\label{eq:flop_root}
\end{equation}
In terms of communication complexity, the most significant operation between the root and other processors is sending and receiving  $\lambda_{i,k,l}$ for $i=1, \cdots, N$, $k=1, \cdots,K$ and $l=1, \cdots,K$ in Processors step 1 and Root step 1. Thus, the total communication cost for $N$ processors per iteration is 
\begin{equation}
\sim N \cdot K^2 = N \dfrac{((d_p+l-1)!)^2}{(d_p!)^2((l-1)!)^2}n^4.  
\label{eq:com_cost}
\end{equation}
From~\eqref{eq:flop_proc},~\eqref{eq:flop_root} and~\eqref{eq:com_cost} it is observed that the number of processors operations, root operations and communication operations are independent of Polya's exponent $d_1$ and $d_2$. Therefore, we conclude that for a fixed $d_p$ and sufficiently large number of processors $N$ ($N \geq  L+M$), improving the accuracy by increasing $d_1$ and $d_2$ does not add any computation per processor or communication overhead.

\subsection{Analysis of Scalability/Speed-up}
\label{sec:speedup_TAC}

The speed-up of a parallel algorithm is defined as $\textit{\text{SP}}_N=\dfrac{T_s}{T_N}, $
where $T_s$ is the execution time of the algorithm on a single processor and $T_N$ is the execution time of the parallel algorithm using $N$ processors. The speed-up is governed by
\begin{equation}
\textit{\text{SP}}_N=\dfrac{N}{D+NS}, 
\label{eq:speedup} 
\end{equation}
where $D$ is the \textit{decentralization ratio} and is defined as the ratio of the total operations performed by all processors except the root to total operations performed by all processors and root. $S$ is the \textit{centralization ratio} and is defined as the ratio of the operations performed by the root processor to total operations performed by all processors and the root. Suppose that the number of available processors is equal to the number of sub-blocks in $C$ defined in~\eqref{eq:C}, i.e, equal to $L+M$. Using the above definitions for $D$ and $S$, Equation~\eqref{eq:flop_proc} as the decentralized computation and~\eqref{eq:flop_root} as the centralized computation, $D$ and $S$ can be approximated as
\begin{equation}
D \simeq \dfrac{N \dfrac{((d_p+l-1)!)^2}{(d_p!)^2((l-1)!)^2}n^7}{N \dfrac{((d_p+l-1)!)^2}{(d_p!)^2((l-1)!)^2}n^7+\dfrac{((d_p+l-1)!)^3}{(d_p!)^3((l-1)!)^3}n^6} \;\;
\end{equation}
and
\begin{equation}
S \simeq \dfrac{\dfrac{((d_p+l-1)!)^3}{(d_p!)^3((l-1)!)^3}n^6}{N \dfrac{((d_p+l-1)!)^2}{(d_p!)^2((l-1)!)^2}n^7+\dfrac{((d_p+l-1)!)^3}{(d_p!)^3((l-1)!)^3}n^6}.
\end{equation}
According to~\eqref{eq:L} and~\eqref{eq:M} the number of processors $N=L+M$ is independent of $n$. Therefore,
\[
\lim_{n\to\infty} D = 1 \quad \text{and} \quad \lim_{n\to\infty} S = 0. 
\]
By substituting $D$ and $S$ in~\eqref{eq:speedup} with their limit values, we have $\lim_{n\to\infty} \textit{\text{SP}}_N = N$. Thus, for large $n$, by using $L+M$ processors, the presented decentralized algorithm solves large robust stability problems $L+M$ times faster than the sequential algorithms. For different values of the state-space dimension $n$, the theoretical speed-up of the algorithm versus the number of processors is illustrated in Figure~\ref{fig:theo_speedup}. As shown in Figure~\ref{fig:theo_speedup}, for problems with large $n$, by using $N \leq L+M$ processors the parallel algorithm solves the robust stability problems approximately $N$ times faster than the sequential algorithm. As $n$ increases, the trend of speed-up becomes increasingly linear. Therefore, for problems with a large number of states, our algorithm becomes increasingly efficient in terms of processor utilization.

\begin{figure}[t]
   \centering
   \includegraphics[scale=0.4]{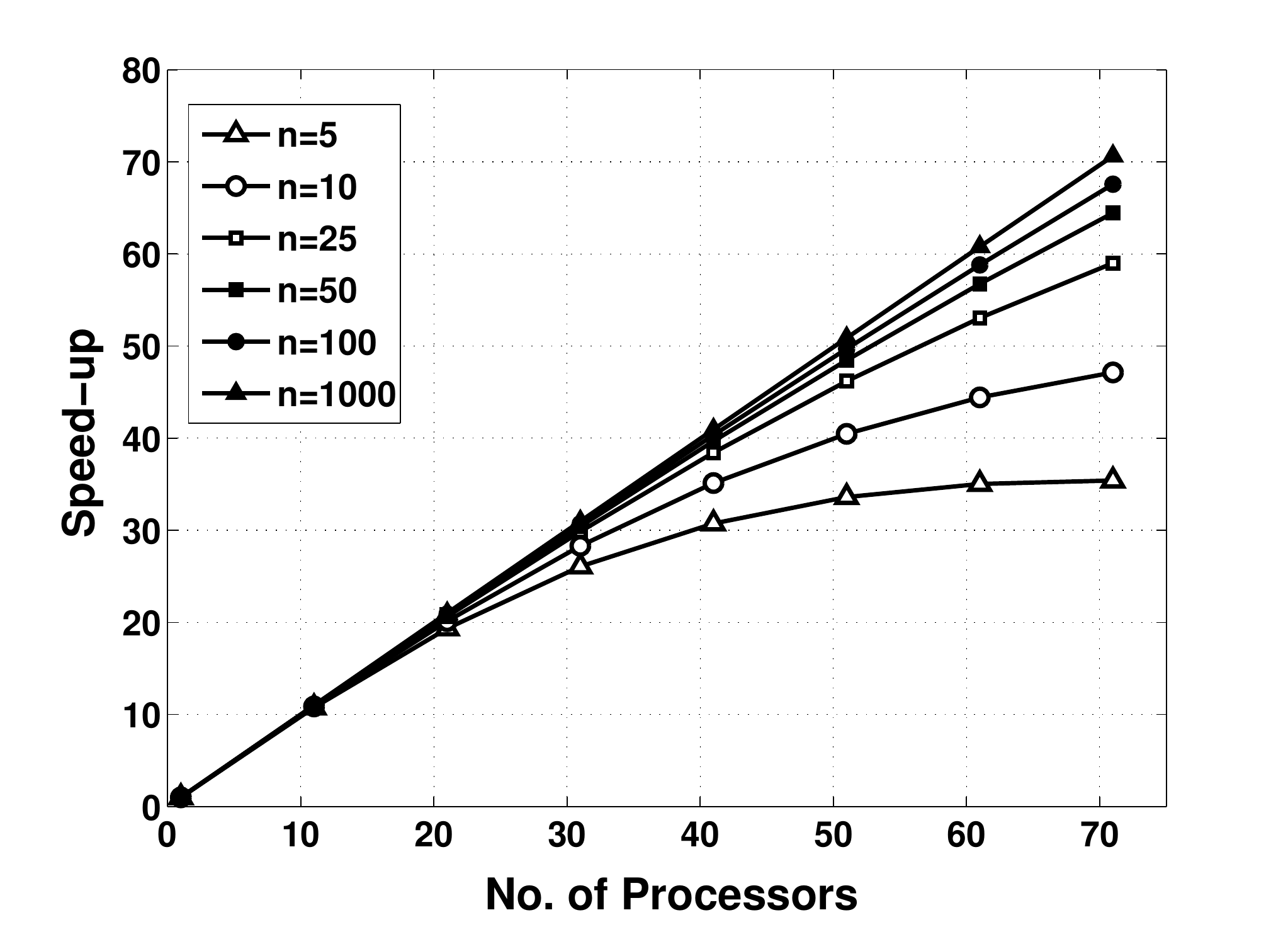}
   \caption{Theoretical Speed-up vs. No. of Processors for Different System Dimensions $n$ for $l=10$, $d_p=2$, $d_a=3$
and $d_1=d_2=4$, Where $L+M=53625$}
   \label{fig:theo_speedup} 
\end{figure}

\subsection{Synchronization and Load Balancing Analysis}

The proposed algorithm is synchronous in that all processors must return values before the centralized step can proceed. However, in the case where we have fewer processors than blocks, some processors may be assigned one block more than other processors. In this case, some processors may remain idle while waiting for the more heavily loaded blocks to complete. In the worst case, this can result in a 50\% decrease in the execution speed. We have addressed this issue in the following manner:
\begin{enumerate}
\item We allocate almost the same number ($\pm 1$) of blocks of the SDP elements $C$ and $B_{i}$ to all processors, i.e., $\mathtt{floor}(\frac{L+M}{N})+1$ blocks to $r$ processors and $\mathtt{floor}(\frac{L+M}{N})$ blocks to the other $N-r$ processors, where $r$ is the remainder of dividing $L+M$ by $N$.
\item We assign the same routine to all of the processors in the Processors steps of Algorithm 6.
\end{enumerate}
If $L+M$ is a multiple of $N$, then the algorithm assigns the same amount of data, i.e., $\frac{L+M}{N}$ blocks of $C$ and $B_{i}$ to each processor. In this case, the processors are perfectly synchronized. If $L+M$ is not a multiple of $N$, then according to~\eqref{eq:flop1}, $r$ of the $N$ processors perform $K^2n^3$ extra operations per iteration. This fraction is $\dfrac{1}{1+\mathtt{floor}(\frac{L+M}{N})} \leq 0.5$ of the operations per iteration performed by each of $r$ processors. Thus in the worst case, we have a 50\% reduction, although this situation is rare. As an example, the load balancing (distribution of data and calculation) for the case of solving an SDP of the size $L+M=24$ using different numbers of available processors $N$ is demonstrated in Figure~\ref{fig:load_balance}. This figure shows the number of blocks that are allocated to each processor. According to this figure, for $N=2,12$ and 24, the processors are perfectly balanced, whereas for the case where $N=18$, twelve processors perform 50$\%$ fewer calculations.

\begin{figure}[t]
   \centering
 \hspace*{-0.65in}  \includegraphics[scale=0.45]{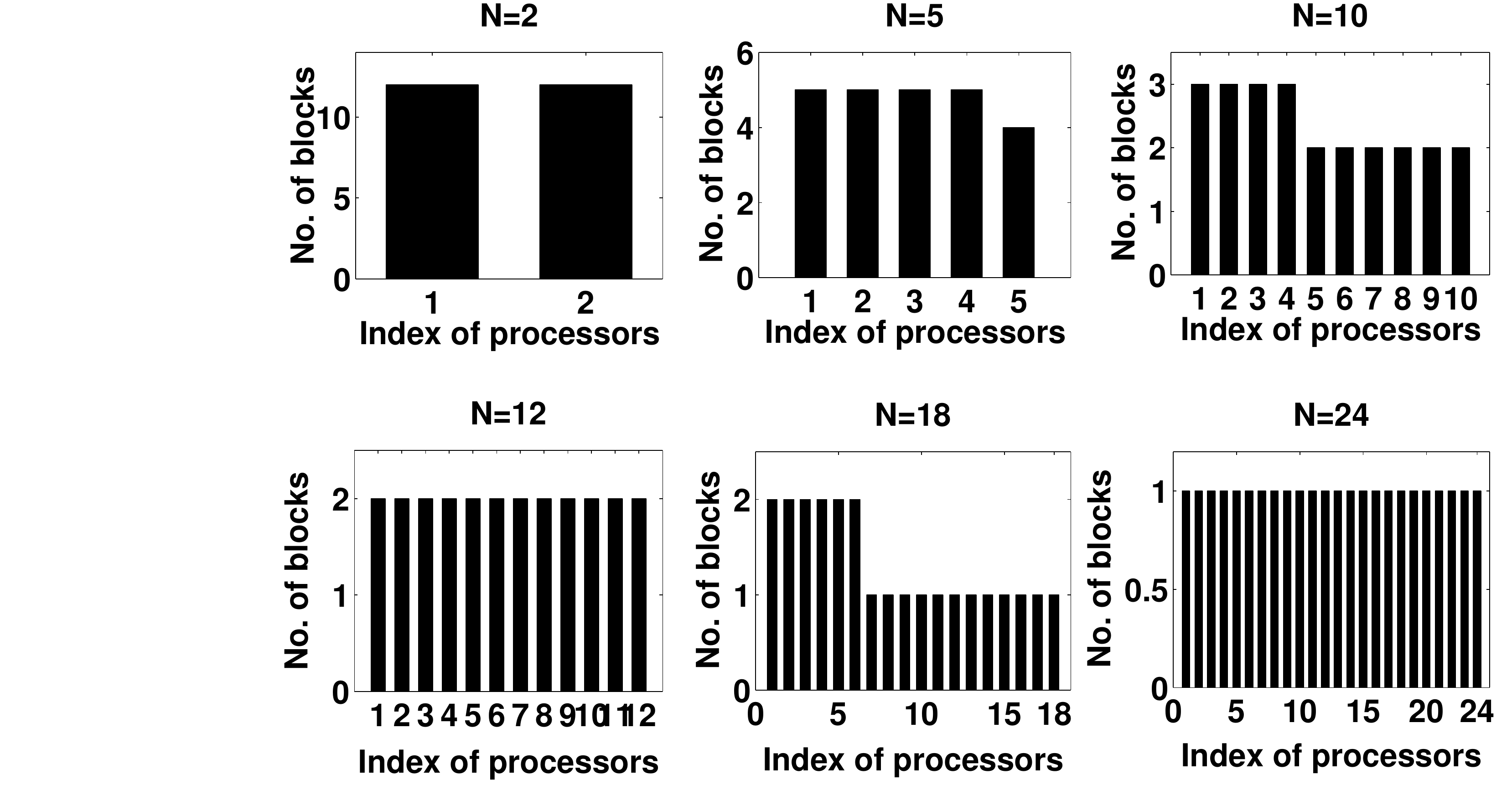} \vspace{0.1in}
   \caption{The Number of Blocks of the SDP Elements Assigned to Each Processor. An Illustration of Load Balancing.} 
   \label{fig:load_balance} 
\end{figure}

\subsection{Communication Graph of the Algorithm}

The communication directed graph of the SDP algorithm (see Figure~\ref{fig:SDP_graph}) is static (fixed for all iterations). At each iteration, root sends messages (dual predictor and corrector search directions $\Delta \widehat{y}$ and $\Delta \overline{y}$) to all of the processors and receives messages (elements $\lambda_{i,k,l}$ of the SCM defined in~\eqref{eq:T2}) from all of the processors. The adjacency matrix of the communication directed graph is defined as follows. For $i=1,\cdots,N$ and $j=1, \cdots, N$,
\[
[T_G]_{i,j} :=
\begin{cases}
1 \quad & \text{if} \; \big( i=1 \; \text{or} \; j=1 \big) \; \text{and} \; \big( i \neq j \big)     \\
0 \quad & \text{Otherwise}.
\end{cases} 
\]

\begin{figure}[t]
\vspace{0.2in}
\centering
\includegraphics[scale=0.55]{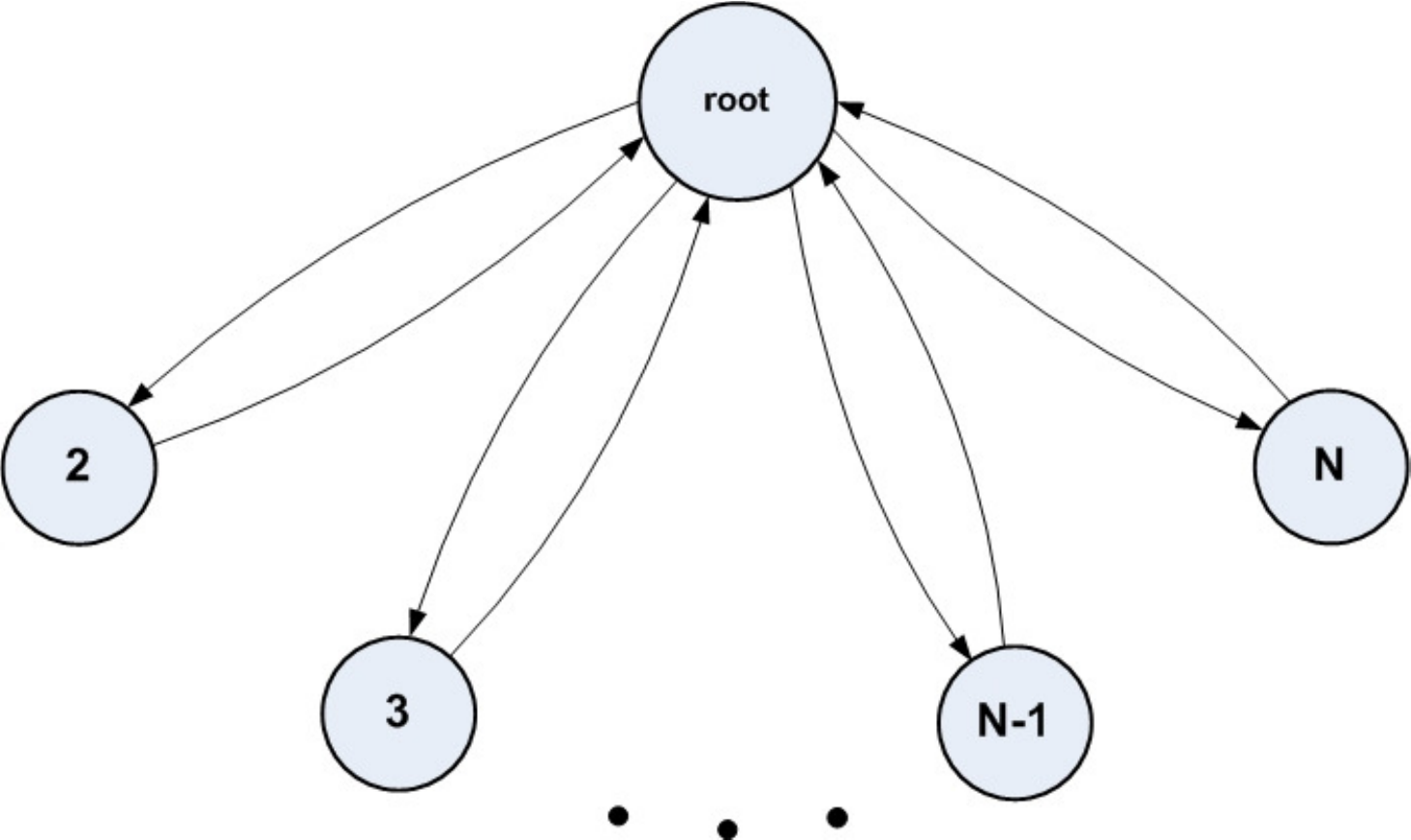} \vspace{0.15in}
\caption{The Communication Graph of the SDP Algorithm}
\label{fig:SDP_graph} 
\end{figure}

\section{Testing and Validation}
\label{sec:RESULTS_TAC}

In this section, we present validation data in 4 key areas. First, we present analysis results for a realistic large-scale model of Tokamak operation using a discretized PDE model. Next, we present accuracy and convergence data and compare our algorithm to the SOS approach. Next, we analyze scalability and speed-up of our algorithm as we increase the number of processors and compare our results to the state-of-the-art general-purpose parallel SDP solver SDPARA. Finally, we explore the limits of the algorithm in terms of the size of the problem, when implemented on a moderately powerful cluster computer and using a moderate processor allocation on the IBM Blue Gene supercomputer at Argonne National Laboratory.

\subsection{Example 1: Application to Control of a Discretized PDE Model in Fusion Research}

The goal of this example is to use the proposed algorithm to solve a real-world stability problem. A simplified model for the poloidal magnetic flux gradient in a Tokamak reactor~(\cite{Tokamak}) is
\begin{equation}
\dfrac{\partial \psi_x (x,t)}{\partial t} = \dfrac{1}{\mu_0 a^2} \dfrac{\partial}{\partial x} \left( \dfrac{\eta(x)}{x} \dfrac{\partial}{\partial x} \left( x \psi_x (x,t) \right) \right) 
\end{equation}
with Dirichlet boundary conditions $\psi_x(0,t)=0$ and $\psi_x(1,t)=0$ for all $t \in \mathbb{R}^+$, where $\psi_x$ is the deviation of the flux gradient from a reference flux gradient profile, $\mu_0$ is the permeability of free space, $\eta(x)$ is the plasma resistivity and $a$ is the radius of the Last Closed Magnetic Surface (LCMS). To obtain the finite-dimensional state-space representation of the PDE, we discretize the PDE in the spatial domain $[0,1]$ at $N=7$ points. The state-space model is then 
\begin{equation}
\dot{\psi}_x(t) = A(\eta(x)) \psi_x(t),  \label{eq:discrete_sys} 
\end{equation}
where $A(\eta(x)) \in \mathbb{R}^{N \times N}$ has the following non-zero entries.
\begin{align*}
& a_{11}= \dfrac{-4}{3  \mu_0 \Delta x^2  a^2} \left(  \dfrac{\eta(x_{\frac{3}{2}})}{x_{\frac{3}{2}}} + \dfrac{2 \eta(x_{\frac{3}{4}})}{x_{\frac{3}{4}}} \right), \\
&a_{12}= \dfrac{4}{3  \mu_0 \Delta x^2  a^2} \left(\dfrac{\eta(x_{\frac{3}{2}})x_2}{x_{\frac{3}{2}}} \right), \\
& a_{j,j-1}= \dfrac{1}{\Delta x^2 \mu_0 a^2} \left( \dfrac{\eta(x_{j-\frac{1}{2}})}{x_{j-\frac{1}{2}}}x_{j-1}  \right) && \text{for} \;  j=2, \cdots, N-1, \\
&  a_{j,j}=\dfrac{-1}{\Delta x^2 \mu_0 a^2}  \left(\dfrac{\eta(x_{j+\frac{1}{2}})}{x_{j+\frac{1}{2}}}+  \dfrac{\eta(x_{j-\frac{1}{2}})}{x_{j-\frac{1}{2}}}\right)x_j && \text{for} \;  j=2, \cdots, N-1, \\
& a_{j,j+1}=  \dfrac{1}{\Delta x^2 \mu_0 a^2} \left( \dfrac{\eta(x_{j+\frac{1}{2}})}{x_{j+\frac{1}{2}}}x_{j+1} \right)  && \text{for} \;  j=2, \cdots, N-1, \\
& a_{N,N-1}= \dfrac{4}{3 \Delta x \mu_0 a^2} \dfrac{\eta(x_{N-\frac{1}{2}})x_{N-1}}{x_{N-\frac{1}{2}} \Delta x}, \\
& a_{N,N}= \dfrac{-4}{3 \Delta x \mu_0 a^2} \left(\dfrac{2\eta(x_{N+\frac{1}{4}})x_N}{x_{N+\frac{1}{4}} \Delta x} + \dfrac{\eta(x_{N-\frac{1}{2}})x_N}{x_{N-\frac{1}{2}} \Delta x}\right),
\end{align*}
where $\Delta x=\dfrac{1}{N}$ and $x_j:=(j-\frac{1}{2}) \Delta x$. Typically $\eta(x_j)$ are not precisely known (they depend on other state variables), so we substitute for $\eta(x_j)$ in~\eqref{eq:discrete_sys} with $\widehat{\eta}(x_j)+\alpha_k$, where $\widehat{\eta}(x_j)$ are the nominal values of $\eta(x_j)$ and $\alpha_k$ for $k=1,\cdots,8$ are the uncertain parameters. The values for $x_1, \cdots, x_j$ and their corresponding values of $\widehat{\eta}(x_j)$ are presented in Table~\ref{tab:ToreSupra}. Note that we have used data from the Tore Supra reactor to estimate the nominal values $\widehat{\eta}(x_j)$.

\renewcommand{\tabcolsep}{2pt} 
\begin{small}
\begin{table}\scalebox{0.85}{
\begin{tabular}{|c|c|c|c|c|c|c|c|c|}
\hline 
$j$ & ${\frac{1}{2}}$ & ${\frac{3}{2}}$ & ${\frac{5}{2}}$ & ${\frac{7}{2}}$ & ${\frac{9}{2}}$ & ${\frac{11}{2}}$ & ${\frac{13}{2}}$ & ${\frac{15}{2}}$ \\
\hline 
$x_j$ & 0.036 & 0.143 & 0.286 & 0.429 & 0.571 & 0.714 & 0.857 & 0.964 \\ 
\hline 
$\widehat{\eta}(x_j)$ & $1.775e-8$ & $2.703e-8$ & $5.676 e-8$ & $1.182e-7$ & $2.058e-7$ & $3.655e-7$ & $1.076e-6$ & $8.419e-6$ \\
 \hline
\end{tabular}} \vspace{0.1in}
\caption{Data for Example 1: Nominal Values of the Plasma Resistivity}
\label{tab:ToreSupra}
\end{table}
\end{small}

The uncertain system is then written as   
\begin{equation}
\dot{\psi}_x(t) = A(\alpha) \psi_x(t), 
\label{eq:discrete_uncertain} 
\end{equation}
where $A$ is affine, $A(\alpha) = A_0 + \sum_{i=1}^8 A_i \alpha_i$, where 
\arraycolsep=1.4pt\def\arraystretch{2.2}
\begin{tiny}
\begin{align*}
&A_1=
\left[ \begin{array}{ccccccc}
\hspace{-0.1in}  -14.09 \hspace{-0.1in}&  6.47 \hspace{-0.1in}&     0    \hspace{-0.1in}&    0     \hspace{-0.1in}&    0      \hspace{-0.1in}&   0     \hspace{-0.1in}&   0  \\
\hspace{-0.1in}   0.41   \hspace{-0.1in}& -3.54 \hspace{-0.1in}&   3.83   \hspace{-0.1in}&    0     \hspace{-0.1in}&    0      \hspace{-0.1in}&   0     \hspace{-0.1in}&   0  \\
\hspace{-0.1in}    0     \hspace{-0.1in}&   2.30 \hspace{-0.1in}&  -10.24 \hspace{-0.1in}&   8.97 \hspace{-0.1in}&   0      \hspace{-0.1in}&   0     \hspace{-0.1in}&   0  \\
\hspace{-0.1in}    0     \hspace{-0.1in}&   0   \hspace{-0.1in}&   6.41   \hspace{-0.1in}&  -21.46  \hspace{-0.1in}&  16.06     \hspace{-0.1in}&   0     \hspace{-0.1in}&   0  \\
\hspace{-0.1in}    0     \hspace{-0.1in}&   0   \hspace{-0.1in}&     0    \hspace{-0.1in}&   12.49  \hspace{-0.1in}& -39.71    \hspace{-0.1in}&  28.91 \hspace{-0.1in}&   0  \\
\hspace{-0.1in}    0     \hspace{-0.1in}&   0   \hspace{-0.1in}&     0    \hspace{-0.1in}&    0     \hspace{-0.1in}&  23.66     \hspace{-0.1in}& -101.96 \hspace{-0.1in}&  86.33 \\
\hspace{-0.1in}    0     \hspace{-0.1in}&   0   \hspace{-0.1in}&     0    \hspace{-0.1in}&    0     \hspace{-0.1in}&    0    \hspace{-0.1in}& 97.40  \hspace{-0.1in}&  -1.74e3
\end{array} \right]
\quad \quad A_2=
\left[ \begin{array}{ccccccc}
\hspace{-0.1in} -2.86  \hspace{-0.1in}&  1.66 \hspace{-0.1in}&    0   \hspace{-0.1in}&   0     \hspace{-0.1in}&   0 \hspace{-0.1in}&    0    \hspace{-0.1in}&     0 \\
\hspace{-0.1in}  1.62  \hspace{-0.1in}& -7.15  \hspace{-0.1in}&  3.83  \hspace{-0.1in}&   0     \hspace{-0.1in}&   0 \hspace{-0.1in}&    0    \hspace{-0.1in}&     0 \\
\hspace{-0.1in}     0    \hspace{-0.1in}&   2.30 \hspace{-0.1in}& -10.24 \hspace{-0.1in}&  8.97 \hspace{-0.1in}&   0 \hspace{-0.1in}&    0    \hspace{-0.1in}&     0 \\
\hspace{-0.1in}     0    \hspace{-0.1in}&     0    \hspace{-0.1in}&  6.41  \hspace{-0.1in}& -21.46 \hspace{-0.1in}&  16.06  \hspace{-0.1in}&    0    \hspace{-0.1in}&     0 \\
\hspace{-0.1in}     0    \hspace{-0.1in}&     0    \hspace{-0.1in}&   0    \hspace{-0.1in}&  12.49 \hspace{-0.1in}& -39.72 \hspace{-0.1in}&  28.91  \hspace{-0.1in}&     0 \\
\hspace{-0.1in}     0    \hspace{-0.1in}&     0    \hspace{-0.1in}&   0    \hspace{-0.1in}&    0    \hspace{-0.1in}&  23.66 \hspace{-0.1in}& -101.96 \hspace{-0.1in}&   86.33 \\
\hspace{-0.1in}     0    \hspace{-0.1in}&     0    \hspace{-0.1in}&   0    \hspace{-0.1in}&    0    \hspace{-0.1in}&   0  \hspace{-0.1in}& 97.40 \hspace{-0.1in}& -1.74e3 \\
\end{array} \right]
\end{align*}
\begin{align*}
A_3=
\left[ \begin{array}{ccccccc}
\hspace{-0.1in}  -1.25  \hspace{-0.1in}&   6.47 \hspace{-0.1in}&    0   \hspace{-0.1in}&   0     \hspace{-0.1in}&   0 \hspace{-0.1in}&    0    \hspace{-0.1in}&     0 \\
\hspace{-0.1in}   0.41  \hspace{-0.1in}&   -5.35 \hspace{-0.1in}&  6.84  \hspace{-0.1in}&   0     \hspace{-0.1in}&   0 \hspace{-0.1in}&    0    \hspace{-0.1in}&     0 \\
\hspace{-0.1in}     0    \hspace{-0.1in}&    4.10 \hspace{-0.1in}&  -13.25 \hspace{-0.1in}&  8.97 \hspace{-0.1in}&   0 \hspace{-0.1in}&    0    \hspace{-0.1in}&    0 \\
\hspace{-0.1in}     0    \hspace{-0.1in}&     0    \hspace{-0.1in}&  6.41  \hspace{-0.1in}& -21.46 \hspace{-0.1in}&  16.06  \hspace{-0.1in}&    0    \hspace{-0.1in}&    0 \\
\hspace{-0.1in}     0    \hspace{-0.1in}&     0    \hspace{-0.1in}&   0    \hspace{-0.1in}&  12.49 \hspace{-0.1in}& -39.71 \hspace{-0.1in}&  28.91  \hspace{-0.1in}&     0 \\
\hspace{-0.1in}     0    \hspace{-0.1in}&     0    \hspace{-0.1in}&   0    \hspace{-0.1in}&    0    \hspace{-0.1in}&  23.66 \hspace{-0.1in}& -101.96 \hspace{-0.1in}&   86.33 \\
\hspace{-0.1in}     0    \hspace{-0.1in}&     0    \hspace{-0.1in}&   0    \hspace{-0.1in}&    0    \hspace{-0.1in}&   0  \hspace{-0.1in}& 97.40 \hspace{-0.1in}& -1.74e3 \\
\end{array} \right]
\quad \quad A_4=
\left[ \begin{array}{ccccccc}
\hspace{-0.1in}  -1.25  \hspace{-0.1in}&   6.47 \hspace{-0.1in}&    0   \hspace{-0.1in}&    0     \hspace{-0.1in}&   0 \hspace{-0.1in}&    0    \hspace{-0.1in}&     0 \\
\hspace{-0.1in}   0.41  \hspace{-0.1in}&   -3.54 \hspace{-0.1in}&  3.83  \hspace{-0.1in}&   0     \hspace{-0.1in}&   0 \hspace{-0.1in}&    0    \hspace{-0.1in}&     0 \\
\hspace{-0.1in}     0    \hspace{-0.1in}&   2.30   \hspace{-0.1in}& -12.25 \hspace{-0.1in}&  11.77 \hspace{-0.1in}&   0 \hspace{-0.1in}&    0    \hspace{-0.1in}&     0 \\
\hspace{-0.1in}     0    \hspace{-0.1in}&     0    \hspace{-0.1in}&  8.41  \hspace{-0.1in}& -24.27 \hspace{-0.1in}& 16.06  \hspace{-0.1in}&    0    \hspace{-0.1in}&     0 \\
\hspace{-0.1in}     0    \hspace{-0.1in}&     0    \hspace{-0.1in}&   0    \hspace{-0.1in}&  12.49 \hspace{-0.1in}& -39.72 \hspace{-0.1in}&  28.91  \hspace{-0.1in}&     0 \\
\hspace{-0.1in}     0    \hspace{-0.1in}&     0    \hspace{-0.1in}&   0    \hspace{-0.1in}&    0    \hspace{-0.1in}&  23.66 \hspace{-0.1in}& -101.96 \hspace{-0.1in}&   86.33 \\
\hspace{-0.1in}     0    \hspace{-0.1in}&     0    \hspace{-0.1in}&   0    \hspace{-0.1in}&    0    \hspace{-0.1in}&   0  \hspace{-0.1in}& 97.40 \hspace{-0.1in}& -1.74e3 \\
\end{array} \right]
\end{align*}
\begin{align*}
&A_5=
\left[ \begin{array}{ccccccc}
\hspace{-0.1in}  -1.25  \hspace{-0.1in}&   6.47 \hspace{-0.1in}&    0   \hspace{-0.1in}&    0     \hspace{-0.1in}&   0 \hspace{-0.1in}&    0    \hspace{-0.1in}&     0 \\
\hspace{-0.1in}   0.41  \hspace{-0.1in}&   -3.54 \hspace{-0.1in}&  3.83  \hspace{-0.1in}&   0     \hspace{-0.1in}&   0 \hspace{-0.1in}&    0    \hspace{-0.1in}&     0 \\
\hspace{-0.1in}     0    \hspace{-0.1in}&   2.30   \hspace{-0.1in}& -10.24 \hspace{-0.1in}&  8.97 \hspace{-0.1in}&   0 \hspace{-0.1in}&    0    \hspace{-0.1in}&     0 \\
\hspace{-0.1in}     0    \hspace{-0.1in}&     0    \hspace{-0.1in}&  6.41  \hspace{-0.1in}& -23.57 \hspace{-0.1in}& 18.76  \hspace{-0.1in}&    0    \hspace{-0.1in}&     0 \\
\hspace{-0.1in}     0    \hspace{-0.1in}&     0    \hspace{-0.1in}&   0    \hspace{-0.1in}&  14.59 \hspace{-0.1in}& -42.42 \hspace{-0.1in}&  28.91  \hspace{-0.1in}&     0 \\
\hspace{-0.1in}     0    \hspace{-0.1in}&     0    \hspace{-0.1in}&   0    \hspace{-0.1in}&    0    \hspace{-0.1in}&  23.66 \hspace{-0.1in}& -101.96 \hspace{-0.1in}&   86.33 \\
\hspace{-0.1in}     0    \hspace{-0.1in}&     0    \hspace{-0.1in}&   0    \hspace{-0.1in}&    0    \hspace{-0.1in}&   0  \hspace{-0.1in}& 97.40 \hspace{-0.1in}& -1.74e3 \\
\end{array} \right]
\quad \quad A_6=
\left[ \begin{array}{ccccccc}
\hspace{-0.1in}  -1.25  \hspace{-0.1in}&   6.47 \hspace{-0.1in}&    0   \hspace{-0.1in}&    0     \hspace{-0.1in}&   0 \hspace{-0.1in}&    0    \hspace{-0.1in}&     0 \\
\hspace{-0.1in}   0.41  \hspace{-0.1in}&   -3.54 \hspace{-0.1in}&  3.83  \hspace{-0.1in}&   0     \hspace{-0.1in}&   0 \hspace{-0.1in}&    0    \hspace{-0.1in}&     0 \\
\hspace{-0.1in}     0    \hspace{-0.1in}&   2.30   \hspace{-0.1in}& -10.24 \hspace{-0.1in}&  8.97 \hspace{-0.1in}&   0 \hspace{-0.1in}&    0    \hspace{-0.1in}&     0 \\
\hspace{-0.1in}     0    \hspace{-0.1in}&     0    \hspace{-0.1in}&  6.41  \hspace{-0.1in}& -21.46 \hspace{-0.1in}& 16.06  \hspace{-0.1in}&    0    \hspace{-0.1in}&     0 \\
\hspace{-0.1in}     0    \hspace{-0.1in}&     0    \hspace{-0.1in}&   0    \hspace{-0.1in}&  12.49 \hspace{-0.1in}& -41.88 \hspace{-0.1in}&  31.56  \hspace{-0.1in}&     0 \\
\hspace{-0.1in}     0    \hspace{-0.1in}&     0    \hspace{-0.1in}&   0    \hspace{-0.1in}&    0    \hspace{-0.1in}&  25.82 \hspace{-0.1in}& -104.61 \hspace{-0.1in}&   86.33 \\
\hspace{-0.1in}     0    \hspace{-0.1in}&     0    \hspace{-0.1in}&   0    \hspace{-0.1in}&    0    \hspace{-0.1in}&   0  \hspace{-0.1in}& 97.40 \hspace{-0.1in}& -1.74e3 \\
\end{array} \right]
\end{align*}
\begin{align*}
&A_7=
\left[ \begin{array}{ccccccc}
\hspace{-0.1in}  -1.25  \hspace{-0.1in}&   6.47 \hspace{-0.1in}&    0   \hspace{-0.1in}&    0     \hspace{-0.1in}&   0 \hspace{-0.1in}&    0    \hspace{-0.1in}&     0 \\
\hspace{-0.1in}   0.41  \hspace{-0.1in}&   -3.54 \hspace{-0.1in}&  3.83  \hspace{-0.1in}&   0     \hspace{-0.1in}&   0 \hspace{-0.1in}&    0    \hspace{-0.1in}&     0 \\
\hspace{-0.1in}     0    \hspace{-0.1in}&   2.30   \hspace{-0.1in}& -10.24 \hspace{-0.1in}&  8.97 \hspace{-0.1in}&   0 \hspace{-0.1in}&    0    \hspace{-0.1in}&     0 \\
\hspace{-0.1in}     0    \hspace{-0.1in}&     0    \hspace{-0.1in}&  6.41  \hspace{-0.1in}& -21.46 \hspace{-0.1in}& 16.06  \hspace{-0.1in}&    0    \hspace{-0.1in}&     0 \\
\hspace{-0.1in}     0    \hspace{-0.1in}&     0    \hspace{-0.1in}&   0    \hspace{-0.1in}&  12.49 \hspace{-0.1in}& -39.71 \hspace{-0.1in}&  28.91  \hspace{-0.1in}&     0 \\
\hspace{-0.1in}     0    \hspace{-0.1in}&     0    \hspace{-0.1in}&   0    \hspace{-0.1in}&    0    \hspace{-0.1in}&  23.66 \hspace{-0.1in}& -104.17 \hspace{-0.1in}&   88.94 \\
\hspace{-0.1in}     0    \hspace{-0.1in}&     0    \hspace{-0.1in}&   0    \hspace{-0.1in}&    0    \hspace{-0.1in}&   0  \hspace{-0.1in}& 100.34 \hspace{-0.1in}& -1.74e3 \\
\end{array} \right]
\quad \quad A_8=
\left[ \begin{array}{ccccccc}
\hspace{-0.1in}  -1.25  \hspace{-0.1in}&   6.47 \hspace{-0.1in}&    0   \hspace{-0.1in}&    0     \hspace{-0.1in}&   0 \hspace{-0.1in}&    0    \hspace{-0.1in}&     0 \\
\hspace{-0.1in}   0.41  \hspace{-0.1in}&   -3.54 \hspace{-0.1in}&  3.83  \hspace{-0.1in}&   0     \hspace{-0.1in}&   0 \hspace{-0.1in}&    0    \hspace{-0.1in}&     0 \\
\hspace{-0.1in}     0    \hspace{-0.1in}&   2.30   \hspace{-0.1in}& -10.24 \hspace{-0.1in}&  8.97 \hspace{-0.1in}&   0 \hspace{-0.1in}&    0    \hspace{-0.1in}&     0 \\
\hspace{-0.1in}     0    \hspace{-0.1in}&     0    \hspace{-0.1in}&  6.41  \hspace{-0.1in}& -21.46 \hspace{-0.1in}& 16.06  \hspace{-0.1in}&    0    \hspace{-0.1in}&     0 \\
\hspace{-0.1in}     0    \hspace{-0.1in}&     0    \hspace{-0.1in}&   0    \hspace{-0.1in}&  12.49 \hspace{-0.1in}& -39.71 \hspace{-0.1in}&  28.91  \hspace{-0.1in}&     0 \\
\hspace{-0.1in}     0    \hspace{-0.1in}&     0    \hspace{-0.1in}&   0    \hspace{-0.1in}&    0    \hspace{-0.1in}&  23.66 \hspace{-0.1in}& -101.96 \hspace{-0.1in}&   86.33 \\
\hspace{-0.1in}     0    \hspace{-0.1in}&     0    \hspace{-0.1in}&   0    \hspace{-0.1in}&    0    \hspace{-0.1in}&   0  \hspace{-0.1in}& 97.40 \hspace{-0.1in}& -1.74e3 \\
\end{array} \right].
\end{align*}
\end{tiny}

For a given $\rho$, we restrict the uncertain parameters $\alpha_k$ to $S_\rho$, defined as
\begin{equation*}
S_\rho := \{ \alpha \in \mathbb{R}^8 : \sum_{i=1}^8 \alpha_i = -6|\rho|, -|\rho| \leq \alpha_i \leq |\rho| \},
\end{equation*}
which is a simplex translated to the origin. We would like to determine the maximum value of $\rho$ such that the system is stable by solving the following optimization problem. 
\begin{equation*}
\hspace*{-2.05in} \rho^* := \max \quad \rho   \vspace*{-0.05in}
\end{equation*}
\begin{equation}
\hspace*{0.5in} \text{subject to } \quad \text{System~\eqref{eq:discrete_uncertain} is stable for all} \; \alpha \in S_\rho. 
\label{eq:optim22} 
\end{equation}
To represent $S_\rho$ using the standard unit simplex defined in~\eqref{eq:simplex}, we define the invertible map $g: \Delta^8 \rightarrow S_\rho$ as
\begin{equation}
g(\alpha)= \left[ g_1(\alpha) \; \cdots \; g_8(\alpha) \right], \; g_i(\alpha):= 2|\rho|(\alpha_i-0.5).
\end{equation}
Then, if we let $A'(\alpha) = A(g(\alpha))$, since $g$ is one-to-one, 
\[
\{ A(\alpha') :  \alpha' \in S_\rho\}  =  \{ A(g(\alpha)):  \alpha \in \Delta^8\}  = \{ A'(\alpha) :  \alpha \in \Delta^8\}. 
\] 
Thus, stability of $ \dot{\psi}_x(t) = A'(\alpha) \psi_x(t), \text{ for all }\alpha \in \Delta^l$ is equivalent to stability of Equation~\eqref{eq:discrete_uncertain} for all $\alpha \in S_\rho$.

We solve the optimization problem in~\eqref{eq:optim22} using bisection. For each trial value of $\rho$, we use the proposed parallel SDP solver in Algorithm 6 to solve the associated SDP obtained by our parallel set-up Algorithm~\ref{alg:setup}. The SDP problems have 224 constraints with the primal variable $X \in \mathbb{R}^{1092 \times 1092}$. The normalized value of $\rho^*$, i.e., $\frac{\rho^*}{\widehat{\eta}(x_{15/2})}$ is found to be $0.0019$, where $\widehat{\eta}(x_{15/2})=8.419 \cdot 10^{-6}$ from Table~\ref{tab:ToreSupra}. In this particular example, the optimal value of $\rho$ does not change with the degrees of $P(\alpha)$ and Polya's exponents~$d_1$ and $d_2$, primarily because the model is affine.
The SDPs are constructed and solved on a parallel Linux-based cluster Cosmea at Argonne National Laboratory. Figure~\ref{fig:Cosmea} shows the algorithm speed-up vs. the number of processors.
Note that solving this problem by SOSTOOLS~(\cite{sostools2013}) on the same machine is impossible due to the lack of unallocated memory. 

\begin{figure}[t]
   \centering
   \includegraphics[scale=0.45]{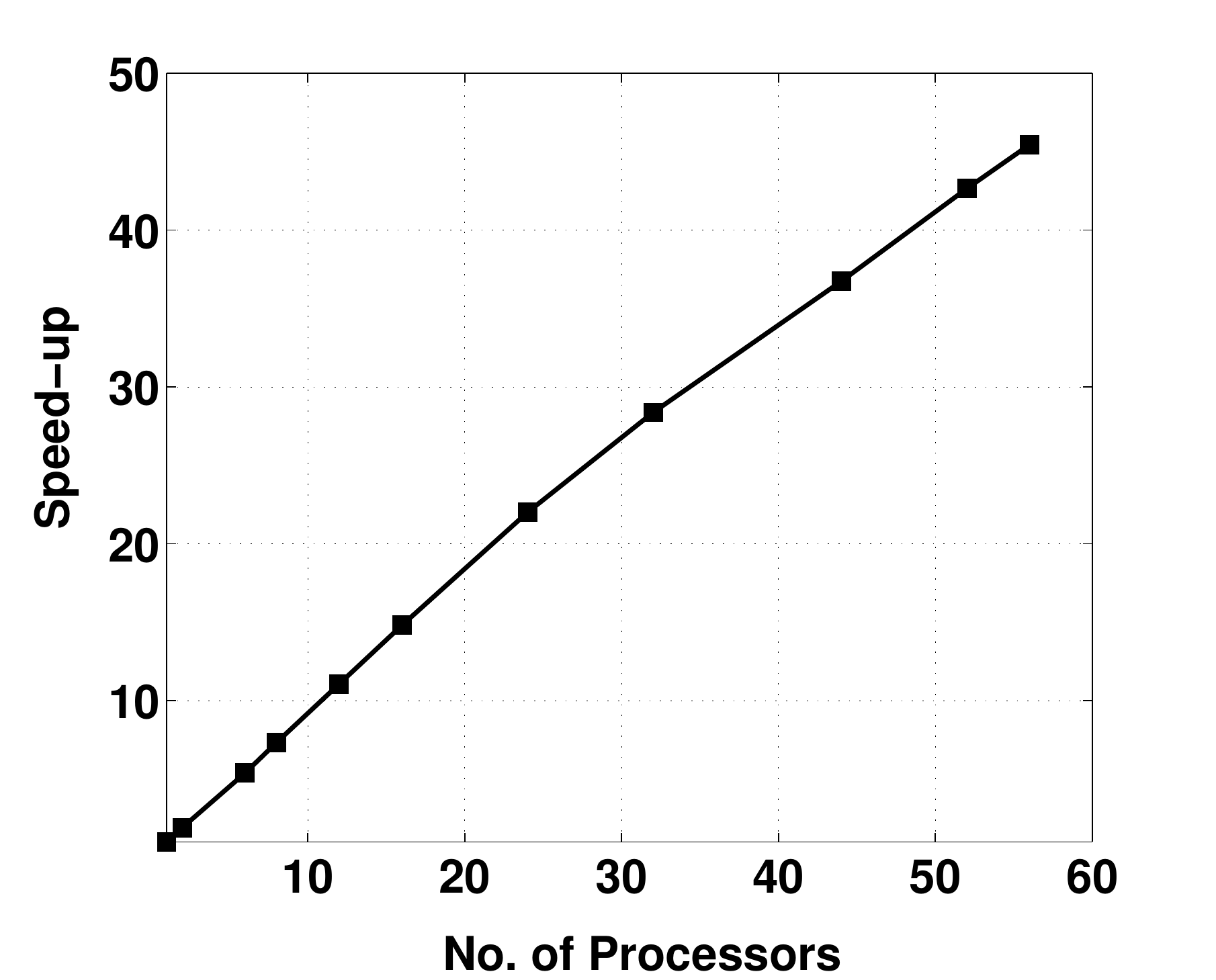} \vspace{0.15in}
   \caption{Speed-up of Set-up and SDP Algorithms vs. Number of Processors for a Discretized Model of Magnetic Flux in Tokamak}
   \label{fig:Cosmea}   
\end{figure}

\subsection{Example 2: Accuracy and Convergence}
The goal of this example is to investigate the effect of the degree $d_p$ of $P(\alpha)$ and the Polya's exponents, $d_1,d_2$ on the accuracy of our algorithms. Given a computer with a fixed amount of RAM, we compare the accuracy (as we defined in Section~\ref{sec:accuracy_TAC}) of the proposed algorithms to the SOS algorithm. Consider the system $\dot x(t)=A(\alpha)x(t)$ where $A$ is a polynomial degree 3 defined as
\begin{equation}
A(\alpha)=A_1\alpha_1^3+A_2\alpha_1^2\alpha_2+A_3\alpha_1\alpha_2\alpha_3+A_4\alpha_1\alpha_3^2+A_5\alpha_2^3+A_{6}\alpha_3^3
\vspace*{-0.1in} \end{equation}
with the constraint
\[
\alpha \in S_L := \left\lbrace \alpha \in \mathbb{R}^3: \sum_{i=1}^3 \alpha_i = 2L+1, L \leq \alpha_i \leq 1 \right\rbrace,
\]
where $A_i$ matrices are defined as
\arraycolsep=1.4pt\def\arraystretch{1}
\begin{footnotesize}
\begin{align*}
&A_1  = 
\begin{bmatrix}
-0.61 & -0.56 & 0.402 \\
-0.48 & -0.550 & 0.671 \\
-1.01 & -0.918 & 0.029
\end{bmatrix},
\quad A_2 =
\begin{bmatrix}
-0.484 & -0.86 & 1.5 \\
-0.732 & -0.841 & -0.126 \\
 0.685 &  0.305 &  0.106
\end{bmatrix}, \quad 
 A_3 = 
\begin{bmatrix}
-0.357 &  0.344 & -0.661\\
-0.210 & -0.505 &  0.588\\
0.268 &  0.487 & -0.846 \\
\end{bmatrix}, \\
& A_4 = 
\begin{bmatrix}
-0.881 & -0.436 & 0.228\\
0.503 & -0.812 & 0.249\\
-0.012 &  0.542 & -0.536
\end{bmatrix}, \;
 A_5= 
\begin{bmatrix}
-0.703 & -0.298 & -0.178\\
0.402 & -0.761 & -0.300\\
-0.010 &  0.461 & -0.588
\end{bmatrix}, \;\,
 A_6 = 
\begin{bmatrix}
-0.201 & -0.182 & -0.557\\
0.803 & -0.412 & -0.203\\
-0.440 &  0.011 & -0.881
\end{bmatrix}.
\end{align*}
\end{footnotesize}

\noindent Defining $g$ as in Example 1, the problem is
\begin{equation*}
\hspace*{-2.6in} \min \quad L 
\end{equation*}
\begin{equation}
\hspace*{0.2in} \text{s.t.} \quad \dot{x}(t)=A(g(\alpha))x(t) \; \text{is stable for all} \; \alpha \in \Delta^3. \label{eq:optim} 
\end{equation}
Using bisection in $L$, as in Example 1, we varied the parameters $d_p$, $d_1$ and $d_2$. The cluster computer Karlin at Illinois Institute of Technology with 24 Gbytes/node of RAM (216 Gbytes total memory) was used to run our algorithm. The upper bounds on the optimal $L$ are shown in Figure~\ref{fig:conserve1} in terms of $d_1$ and $d_2$ and for different $d_p$. Considering the optimal value of $L$ to be $L_{\text{opt}}=-0.111$, Figure~\ref{fig:conserve1} shows how increasing $d_p$ and/or $d_1,d_2$ - when they are still relatively small - improves the accuracy of the algorithm. Figure~\ref{fig:conserve2} demonstrates how the error in our upper bound for $L_{\text{opt}}$ decreases by increasing $d_p$ and/or $d_1,d_2$.

For comparison, we solved the same stability problem using the SOS algorithm~(\cite{sostools2013}) using only a single node of the same cluster computer and 24 Gbytes of RAM. We used Putinar's Positivstellensatz (see Section~\ref{sec:optim_semialg}) to impose the constraints $\sum_{i=1}^3 \alpha_i = 2L+1$ and $ L \leq \alpha_i \leq 1$. Table~\ref{tab:SOS_TAC} shows the upper bounds on $L$ given by the SOS algorithm using different degrees for $x$ and $\alpha$. By considering a Lyapunov function of degree two in $x$ and degree one in $\alpha$, the SOS algorithm gives $-0.102$ as an upper bound on $L_{opt}$ as compared with our value of $-0.111$. Increasing the degree of $\alpha$ in the Lyapunov function beyond two resulted in a failure due to lack of memory.

\begin{footnotesize}
\renewcommand{\arraystretch}{0.8}
\begin{table}[t]
\caption{Upper Bounds Found for $L_{opt}$ by the SOS Algorithm Using Different Degrees for $x$ and $\alpha$ (inf: Infeasible, O.M.: Out of Memory)} 
\label{tab:SOS_TAC}
\begin{center}
\begin{tabular}{|c|c|c|c|}
\hline
\backslashbox{  Degree in $x$}{Degree in $\alpha$}
  & 0 & 1 & 2 \\
\hline
1 & Infeasible & Infeasible & Infeasible \\
\hline
2 & Infeasible & -0.102 & Out of Memory \\
\hline
3 & Infeasible & Out of Memory &  Out of Memory \\
\hline
\end{tabular} 
\end{center}
\end{table} 
\end{footnotesize}

\begin{figure}[htbp]
\centering
  \includegraphics[scale=0.45]{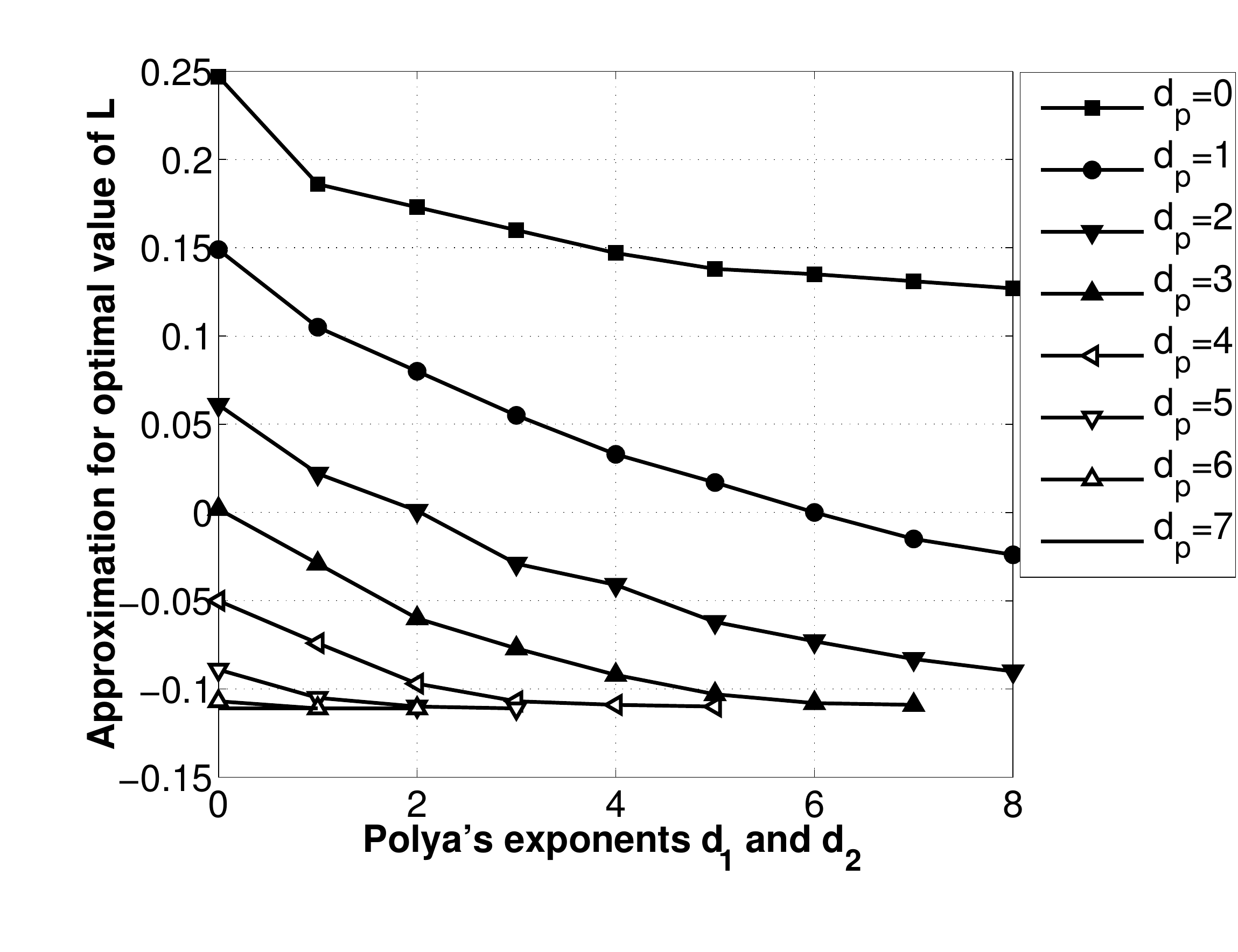} 
   \caption{Upper Bound on Optimal $L$ vs. Polya's Exponents $d_1$ and $d_2$, for Different Degrees of $P(\alpha)$. ($d_1=d_2$).}
   \label{fig:conserve1}
\centering 
  \hspace*{-0.35in} \includegraphics[scale=0.45]{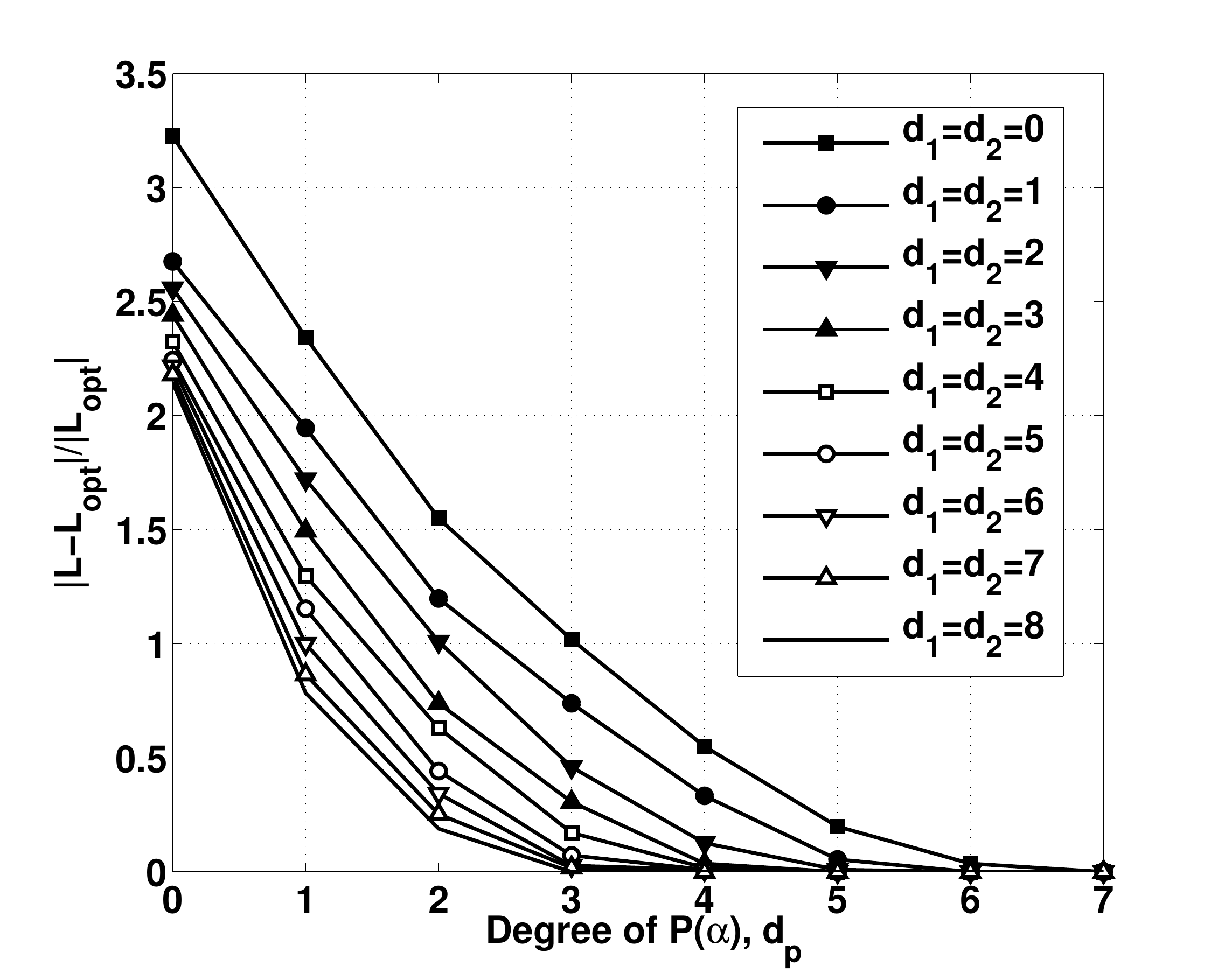} \vspace{0.1in}
   \caption{Error of the Approximation for the Optimal Value of $L$ vs. Degrees of $P(\alpha)$, for Different Polya's Exponents}
   \label{fig:conserve2}   
\end{figure}

\subsection{Example 3: Evaluating Speed-up}

In this example, we evaluate the efficiency of the algorithm in using additional processors to decrease computation time. As mentioned in Section~\ref{sec:complexity_TAC} on computational complexity, the measure of this efficiency is termed speed-up and in Section~\ref{sec:speedup_TAC}, we gave a formula for this number. To evaluate the true speed-up, we first ran the set-up algorithm on the Blue Gene supercomputer at Argonne National Laboratory using three random linear systems with different state-space dimensions and numbers of uncertain parameters. Figure~\ref{fig:speedup1} shows a log-log plot of the computation time of the set-up algorithm vs. the number of processors. One can be observed that the scalability of the algorithm is practically ideal for several different state-space dimensions and numbers of uncertain parameters.

To evaluate the speed-up of the SDP portion of the algorithm, we solved three random SDP problems with different dimensions using the Karlin cluster computer. Figure~\ref{fig:speedup2} gives a log-log plot of the computation time of the SDP algorithm vs. the number of processors for three different dimensions of the primal variable $X$ and the dual variable $y$. As indicated in the figure, the three dimensions of the primal variable $X$ are $ 200,\; 385$ and 1092,
 and the dimensions of the dual variable $y$ are $K=50, \; 90$ and 224, respectively.
In all cases, $d_p=2$ and $d_1=d_2=1$. The linearity of the Time vs. Number of Processors curves in all three cases demonstrates the scalability of the SDP algorithm.

For comparison, we plot the speed-up of our algorithm vs. that of the general-purpose parallel SDP solver SDPARA 7.3.1 as illustrated in Figure~\ref{fig:speedup_sdpara}. Although similar for a small number of processors, for a larger number of processors, SDPARA saturates, while our algorithm remains approximately linear.

\begin{figure}[htbp]
\centering  
\includegraphics[scale=0.45]{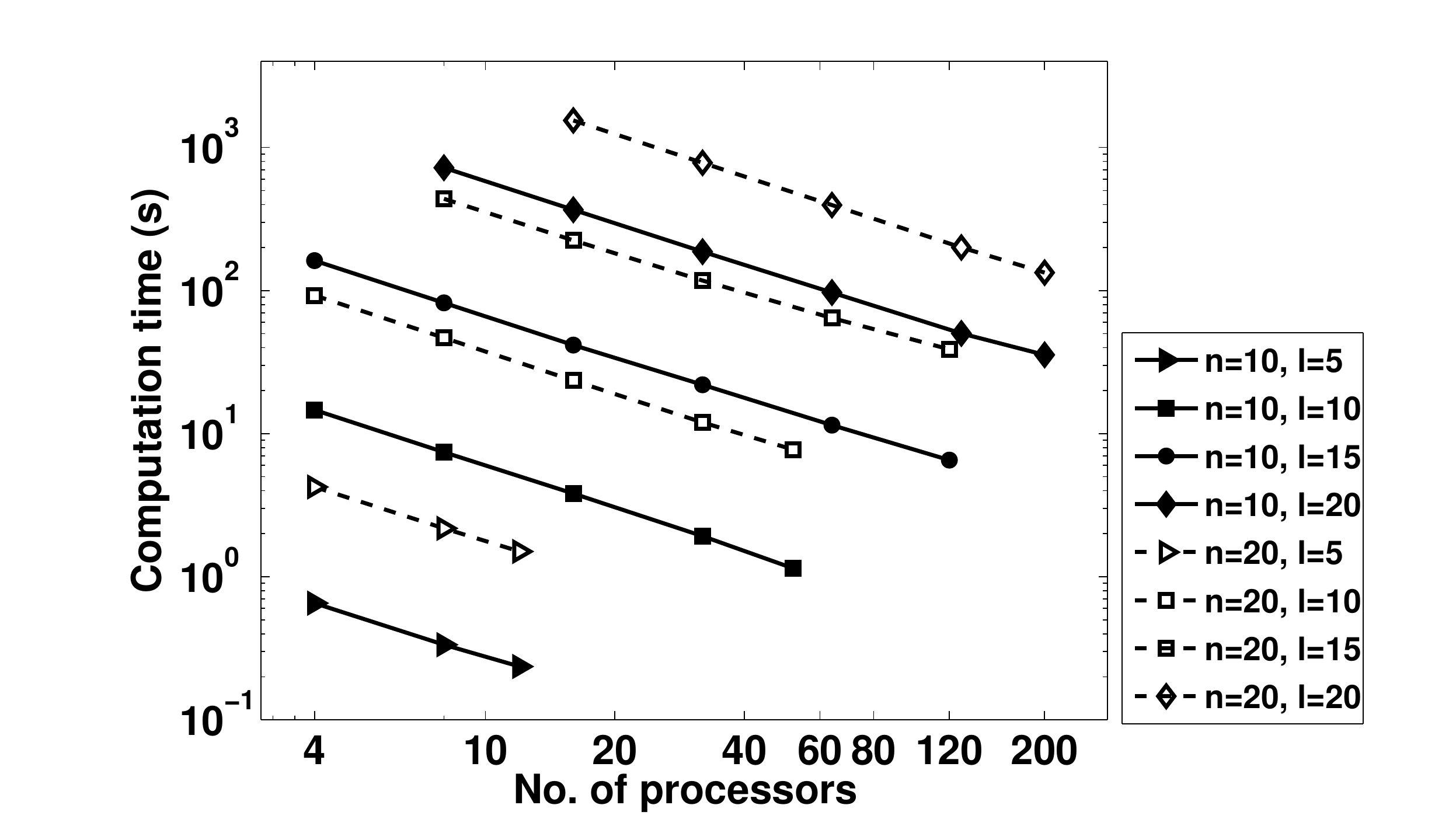}  
 \caption{Computation Time of the Parallel Set-up Algorithm vs. Number of Processors for Different Dimensions of Linear System $n$ and Numbers of Uncertain Parameters $l$- Executed on Blue Gene Supercomputer of Argonne National Labratory}
\label{fig:speedup1}
\end{figure}

\begin{figure}[h]
\centering
\includegraphics[scale=0.45]{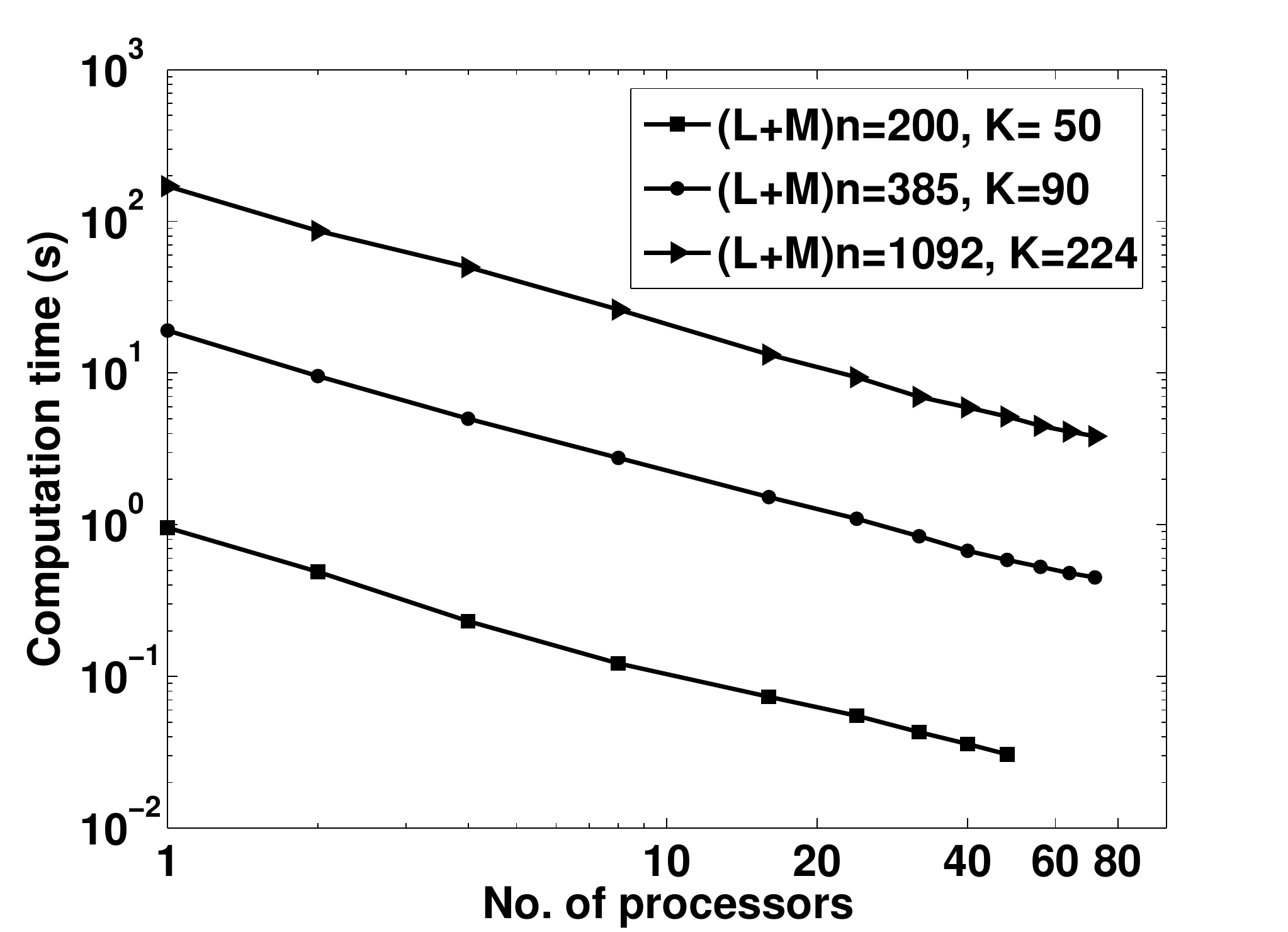}
 \caption{Computation Time of the Parallel SDP Algorithm vs. Number of Processors for Different Dimensions of Primal Variable $(L+M)n$ and of Dual Variable $K$- Executed on Karlin Cluster Computer}
\label{fig:speedup2} 
\end{figure}

\begin{figure}[h]
 \centering
 \includegraphics[scale=0.5]{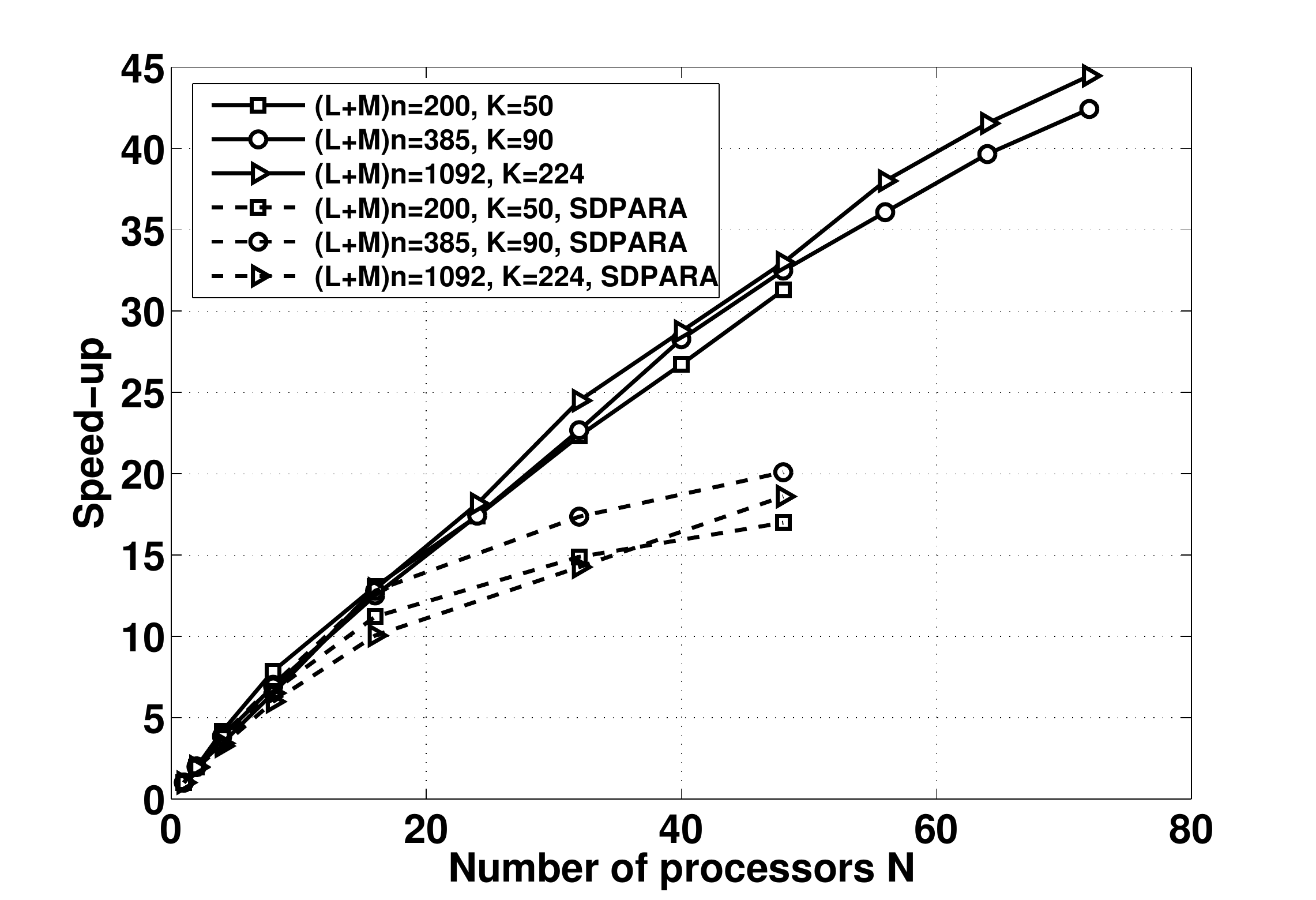} \vspace{0.1in}
 \caption{Comparison Between the Speed-up of the Present SDP Solver and SDPARA 7.3.1, Executed on Karlin Cluster Computer} 
 \label{fig:speedup_sdpara} 
 \end{figure}

\pagebreak
\newpage

\subsection{Example 4: Maximum State-space and Parameter Dimensions for a 9-Node Linux-based Cluster Computer}

The goal of this example is to show that given moderate computational resources, the proposed decentralized algorithms can solve robust stability problems for systems with 100+ states. We used the Karlin cluster computer with 24 Gbytes/node of RAM and nine nodes. We ran the set-up and the SDP algorithms to solve the robust stability problem with dimension $n$ and $l$ uncertain parameters on one and nine nodes of Karlin cluster computer. Thus, the total accessible memory was 24 Gbytes and 216 Gbytes, respectively.
Using trial and error, for different $n$ and $d_1,d_2$ we found the largest $l$ for which the algorithms do not terminate due to insufficient memory (Figure~\ref{fig:size_setup}). In all of the runs $d_a=d_p=1$. Figure~\ref{fig:size_setup} shows that by using 216 Gbytes of RAM, the algorithms can solve the stability problem of size $n=100$ with 4 uncertain parameters in $d_1=d_2=1$ Polya's iteration and with 3 uncertain parameters in $d_1=d_2=4$ Polya's iterations.

\begin{figure}[bt] 
\centering
 \includegraphics[scale=0.37]{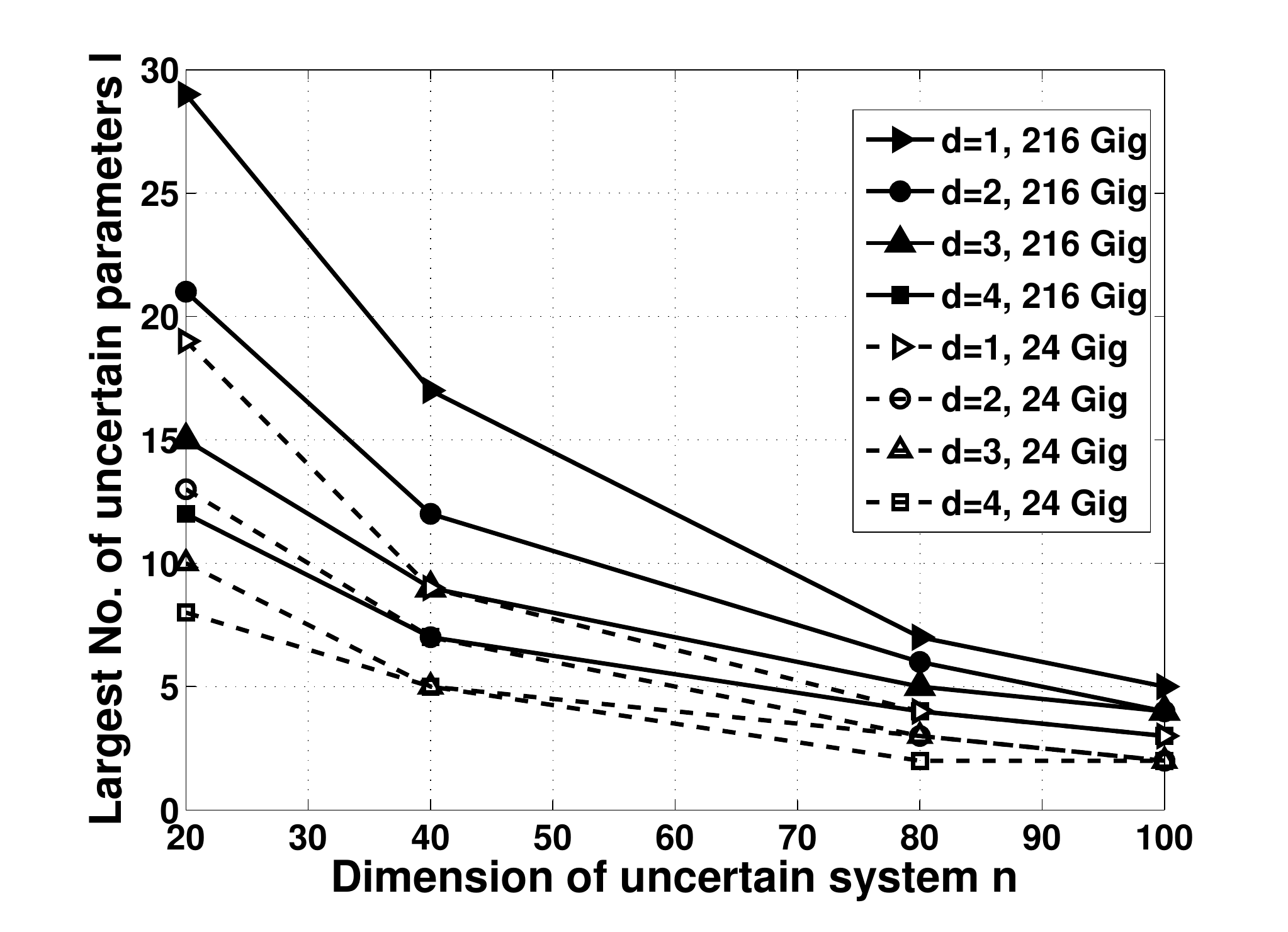} \hspace{-0.2in}
 \includegraphics[scale=0.37]{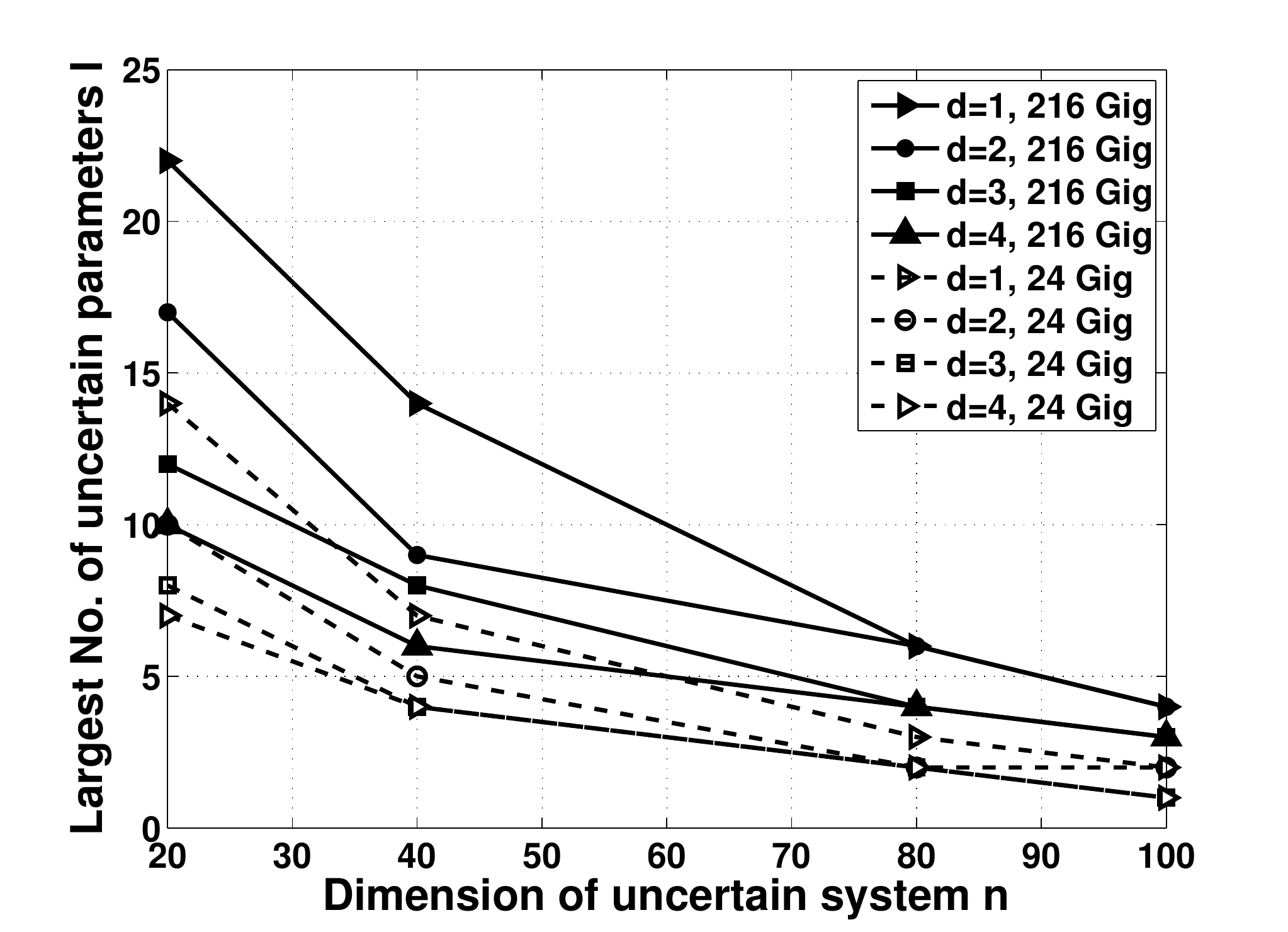} \vspace{0.15in}
\caption{Largest Number of Uncertain Parameters of $n$-Dimensional Systems for Which the Set-up Algorithm (Left) and SDP Solver (Right) Can Solve the Robust Stability Problem of the System Using 24 and 216 GB of RAM}
\label{fig:size_setup} 
\end{figure}


\chapter{PARALLEL ALGORITHMS FOR ROBUST STABILITY ANALYSIS OVER HYPERCUBES}
\label{chp:multisim}

\section{Background and Motivation}

In Chapter~\ref{chp:linear}, we proposed a distributed parallel algorithm for stability analysis over a simplex. Unfortunately, simplices are rather restrictive forms of uncertainty set in that they do not allow for parameters which take values on intervals or polytopes. Additionally, we hope to eventually extend our algorithms to the problem of nonlinear stability, which requires search over positive polynomials defined over a set which
contains the origin. Since simplicies do not include the origin, our algorithms cannot be readily applied to such problems.

In this chapter, our goal is to generalize our algorithms in Chapter~\ref{chp:linear} in order to perform robust stability analysis on linear systems with uncertain parameters defined over hypercubes. Several recent papers such as~\cite{chesi2005polynomially} and~\cite{bliman2004convex}, have proposed LMI-based techniques to construct parameter-dependent quadratic-in-state Lyapunov functions for this class of systems. In particular, researchers~(\cite{chesi_hypercube_2005}) have recently turned to SOS methods and the Positivstellensatz results (see Section~\ref{sec:history}) to construct increasingly accurate and increasingly complex LMI-based tests for stability over hypercubes. Unfortunately, due to the inherent intractability of the problem of polynomial optimization, SOS based algorithms typically
run out of memory for even relatively small-sized problems; see \textit{e.g.,} Table~\ref{tab:SOS_TAC} of Section~\ref{sec:RESULTS_TAC}. This makes it difficult to solve SOS-based algorithms on desktop computers. In this chapter, we seek for a parallel methodology to distribute the required memory and computation among hundreds of processors - each processor possessing a dedicated memory.

\subsection{Our Contributions}
We start by proposing an extension to Polya's theorem. This new result parameterizes every multi-homogeneous polynomial which is positive over a given multi-simplex/hypercube. Based on this result, we propose a parallel algorithm to set-up a sequence of block-structured LMIs (similar to the case of a single simplex). Solutions to these LMIs define parameter-dependent Lyapunov functions for the system. Finally, we use our parallel SDP solver in Section~\ref{sec:SDPSOLVER} to efficiently solve these structured LMIs.
Similar to Algorithm 7, the proposed set-up algorithm in this chapter has no centralized computation, memory or communication, hence resulting in a near-ideal speed-up. Specifically, we show that the communication operations per processor is proportional to $\frac{1}{N_c}$, where $N_c$ is the number of processors used by the algorithm. This implies that by increasing the number of processors, we actually decrease the communication overhead per processor and improve the speed-up. Naturally, there exists an upper-bound for the number of processors which can be used by the algorithm, beyond which, no speed-up is gained. This upper-bound is proportional to the number of uncertain parameters in the system and for practical problems will be far larger than the number of available processors.

\section{Notation and Preliminaries on Multi-homogeneous Polynomials}
\label{sec:notation_multi-homog}

Recall from Section~\ref{sec:notation_simplex} that we denote a monomial by $\alpha^\gamma = \prod_{i=1}^l \alpha_i^{\gamma_i}$, where $\alpha \in \mathbb{R}^l$ is the vector of variables and $\gamma \in W_d$ is the vector of exponents, were $W_d$ is the set of exponents defined in~\eqref{eq:W_d}. Now consider the case where $\alpha = [ \alpha_1, \cdots, \alpha_n ]$ with $\alpha_i \in \mathbb{R}^{l_i}$, and $h=[ h_1, \cdots , h_n ]$, where $h_i \in W_{d_{p_i}}$. Then, we define the set of $n$-variate
\textit{multi-homogeneous} polynomials of degree vector $D = \left[ d_1, \cdots, d_n \right] \in \mathbb{N}^n$ as (a generalization of~\eqref{eq:multi-homog_poly})
\begin{equation}
\left\lbrace P \in \mathbb{R}[\alpha_1, \cdots, \alpha_n] : P(\alpha) = \sum_{h_1 \in W_{d_1}} \cdots \sum_{h \in W_{d_n}} P_{\{ h_1, \cdots, h_n \}} \alpha_1^{h_1} \cdots \alpha_n^{h_n}   \right\rbrace.
\label{eq:multisim_notation}
\end{equation}
Note that for any $i \in \{1, \cdots,n \}$, the element $d_i$ of the degree vector $D$ is the degree of $\alpha_i^{h_i}$ in $P$. For brevity, we denote the index set $\{h_1, \cdots , h_n \}$ by $\mathcal{H}_n$ and $\{h_{1_j}, \cdots , h_{n_j} \}$ by $\mathcal{H}_{n,j}$, where $h_{i_j}$ is defined as  the $j^{th}$ element of $h_i \in W_{d_i}$ using lexicographical ordering. 
We define the unit multi-simplex $\tilde{\Delta}^{\{l_1, \cdots, l_N \}}$ as the Cartesian product of $N$ unit simplexes; i.e., $\tilde{\Delta}^{ \{l_1, \cdots, l_N \}} := \Delta^{l_1} \times \cdots \times \Delta^{l_N}$. Given $r_i \in \mathbb{R}$, let us define the hypercube $\Phi^n \subset \mathbb{R}^n$ as
\[
\Phi^n := \{ \alpha \in \mathbb{R}^n :  \vert \alpha_i \vert \leq r_i, \, i = 1, \cdots, n \}.
\]

 \noindent \textbf{Claim 1:} \textit{For every non-homogeneous polynomial $F(\alpha)$ with $\alpha \in \tilde{\Delta}^{\{ l_1, \cdots, l_n \}}$, there exists a multi-homogeneous polynomial $P$ such that}
\[
\left\lbrace F(\alpha) \in \mathbb{R}: \alpha \in \tilde{\Delta}^{\{ l_1, \cdots, l_n \}} \right\rbrace 
 = 
\left\lbrace P(\beta) \in \mathbb{R} : \beta \in \tilde{\Delta}^{\{ l_1, \cdots, l_n \}} \right\rbrace. \vspace{0.1in}
\] 
To construct $P$, first let $N_F$ be the number of monomials in $F$. Define $t^{(k)} := \left[ t_1^{(k)}, \cdots, t_n^{(k)} \right]$ for $k=1, \cdots, N_F$, where $t_i^{(k)}$ is the sum of the exponents of the variables inside $\Delta^{l_i}$, in the $k^{th}$ monomial of $F$. Then, one can construct $P$ by multiplying the $k^{th}$ monomial of $F$ (according to lexicographical ordering) for $k=1, \cdots, N_F$ by 
\[
\prod_{i=1}^n \left( \sum_{j=1}^{l_i} \alpha_{i_j} \right)^{ T_i - t_i^{(k)}}, \quad T_i := \max_{k \in \{ 1, \cdots, N_F \}} t_i^{(k)}.
\] 
For more clarification, we provide the following example of constructing the multi-homogeneous polynomial $P$.

\noindent \textit{Example:} Consider the non-homogeneous polynomial
\[
F(\alpha) = F_1 (\alpha_{1,1} + \alpha_{12}) \alpha_{2,1} + F_2 \alpha_{1,2}^2 + F_3 \alpha_{2,2},
\]
where $(\alpha_{1_1},\alpha_{1_2}), (\alpha_{2_1},\alpha_{2_2})  \in \Delta^2$,  $t^{(1)}= t^{(2)} =[1,1], t^{(3)} = [2,0]$ and $t^{(4)} = [0,1]$. Then, the multi-homogeneous polynomial $P(\alpha)$ is
\begin{align*}
P(\alpha) & = F_1 (\alpha_{1_1} + \alpha_{1_2})^2 \alpha_{2_1} + F_2 \alpha_{1_2}^2 (\alpha_{2_1}+\alpha_{2_2}) + F_3(\alpha_{1_1}+\alpha_{1_2})^2 \alpha_{2_2} \\
& = F_{\{(2,0),(1,0)\}} \alpha_{1_1}^2 \alpha_{2_1}+ F_{\{(2,0),(0,1) \}}\alpha_{1_1}^2 \alpha_{2_2} + F_{\{(1,1),(1,0) \}} \alpha_{1_1} \alpha_{1_2} \alpha_{2_1}  \\
& + F_{\{(1,1),(0,1) \}} \alpha_{1_1} \alpha_{1_2} \alpha_{2_2} + F_{\{(0,2),(1,0) \}} \alpha_{1_2}^2 \alpha_{2_1} + F_{\{(0,2),(0,1) \}} \alpha_{1_2}^2 \alpha_{2_2}.
\end{align*}
Thus, the coefficients of the multi-homogeneous polynomial $P$ are
\begin{align*}
&P_{\{(2,0),(1,0)\}} = F_1, && P_{\{(2,0),(0,1) \}} = F_3, && P_{\{(1,1),(1,0) \}} = 2F_1 \\
&P_{\{(1,1),(0,1) \}} = 2F_3, && P_{\{(0,2),(1,0) \}} = F_1+F_2, && P_{\{(0,2),(0,1) \}} = F_2+F_3. 
\end{align*}

 \noindent \textbf{Claim 2:} \textit{For every polynomial $F(x)$ with $x \in \Phi^n$, there exists a multi-homogeneous polynomial $P$ such that}
\begin{equation}
\left\lbrace F(x) \in \mathbb{R} : x \in \Phi^n \right\rbrace = \left\lbrace P(\alpha, \beta) \in \mathbb{R} : \, \alpha, \beta \in \mathbb{R}^n \text{ and } (\alpha_i, \beta_i) \in \Delta^2 \text{ for } i=1, \cdots,n \right\rbrace. 
\label{eq:fz_multi}
\end{equation}
To construct $P$, we propose the following steps.
\begin{enumerate}
\item Define new variables $\alpha_i := \frac{x_i+r_i}{2\,r_i} \in [0,1]$ for $i=1, \cdots,n$. 
\item Define $Q(\alpha_1, \cdots, \alpha_n) := F(2\, r_1 \alpha_1-r_1, \cdots, 2\, r_n \alpha_n-r_n)$.
\item Define a new set of variables $\beta_i := 1-\alpha_i$ for $i=1, \cdots,n$.
\item Let $N_Q$ be the number of monomials in $Q$. Define $t^{(k)} := \left[ t_1^{(k)}, \cdots, t_n^{(k)} \right]$ for $k=1, \cdots, N_Q$, where $t_i^{(k)}$ is the sum of the exponents of the variables inside $\Delta^{2}$, in the $k^{th}$ monomial of $Q$. Then, for $k=1, \cdots, N_Q$,  multiply the $k^{th}$ monomial of $Q$ (according to lexicographical ordering) by 
\[
\prod_{i=1}^n \left( \alpha_i + \beta_i \right)^{ T_i - t_i^{(k)}}, \quad T_i := \max_{k \in \{ 1, \cdots, N_Q \}} t_i^{(k)}.
\] 
\end{enumerate}
We provide the following example to further clarify this procedure.

\noindent \textit{Example:} Suppose $F(x_1,x_2)=x_1^2+x_2$, with $x_1 \in [-2,2]$ and $x_2 \in [-1,1]$. Define $\alpha_1:=\frac{x_1+2}{4} \in [0,1]$ and $\alpha_2:=\frac{x_2+1}{2}\in [0,1]$. Then, define
\[
Q(\alpha_1,\alpha_2) := f(4\alpha_1-2,2\alpha_2-1) =  16\alpha_1^2-16\alpha_1+2\alpha_2+3
\]
By homogenizing $Q$ we obtain the multi-homogeneous polynomial
\begin{align*}
P(\alpha,\beta)= & 16\alpha_1^2(\alpha_2+\beta_2)-16\alpha_1(\alpha_1+\beta_1)(\alpha_2+\beta_2)+2 \alpha_2(\alpha_1+\beta_1)^2 \\
&+3(\alpha_1+\beta_1)^2(\alpha_2+\beta_2),\;\; (\alpha_1,\beta_1), (\alpha_2,\beta_2) \in \Delta^2
\end{align*}
with the degree vector $D=[2,1]$, where $d_1=2$ is the sum of exponents of $\alpha_1$ and $\beta_1$ in every monomial of $P$, and $d_2=1$ is the sum of exponents of $\alpha_2$ and $\beta_2$ in every monomial of $P$.

In the following theorem~(\cite{kamyar_CDC2012}), we parameterize all of the multi-homogeneous polynomials which are positive over a multi-simplex.

\begin{mythm}(Polya's theorem, multi-simplex version)
A matrix-valued multi-homogeneous polynomial $F$ satisfies $F(\alpha,\beta) > 0$ for all $(\alpha_i,\beta_i) \in \Delta^2, i=1, \cdots, n$, if there exist $e \geq 0$ such that all the coefficients of
\[
\left( \prod_{i=1}^n \left( \alpha_i+\beta_i \right)^e \right) F(\alpha,\beta)
\]
are positive definite.
\label{thm:polya_multi-simplex2}
\end{mythm}
\begin{proof}
We use induction as follows. 

\underline{\textit{Basis step:}} Suppose $n=1$. Then, from the simplex version of Polya's theorem (Theorem~\ref{thm:polya}) it follows that for every $F(\alpha,\beta) > 0$ with $(\alpha,\beta) \in \Delta^2$, there exists some $e \geq 0$ such that all of the coefficients of 
$
(\alpha_1+\beta_1)^e F(\alpha,\beta)
$
are positive definite.

\underline{\textit{Induction hypothesis:}} Suppose for every $F(\alpha,\beta) > 0 $ with $(\alpha_i,\beta_i) \in \Delta^2, \, i=1, \cdots,k$ there exists some $e \geq 0$ such that all of the coefficients of
\[
\left( \prod_{i=1}^k \left( \alpha_i+ \beta_i \right)^e \right) F(\alpha,\beta)
\]
are positive definite. \\
We need to prove that for every $F(\alpha,\beta) > 0 $ with $(\alpha_i,\beta_i) \in \Delta^2, \, i=1, \cdots,k+1$ there exists some $e^* \geq 0$ such that all of the coefficients of
\[
\left( \prod_{i=1}^{k+1} \left( \alpha_i+\beta_i \right)^{e^*} \right) F(\alpha,\beta)
\]
are positive definite. From the induction hypothesis it follows that for any fixed $(\hat{\alpha},\hat{\beta}) \in \Delta^2$, if $F(\alpha_1, \cdots, \alpha_k, \hat{\alpha}, \beta_1, \cdots, \beta_k, \hat{\beta}) > 0$ for all $(\alpha_i,\beta_i) \in \Delta^2, \, i=1, \cdots,k$, then there exists some $e \geq 0$ such that all of the coefficients of
\begin{equation}
\left( \prod_{i=1}^k \left( \alpha_i+\beta_i \right)^e \right) F(\alpha_1, \cdots, \alpha_k, \hat{\alpha}, \beta_1, \cdots, \beta_k, \hat{\beta})
\label{eq:dummy1}
\end{equation}
are positive definite. Using our notation in~\eqref{eq:multi-homog_poly}, we can expand~\eqref{eq:dummy1} as
\begin{equation}
\left( \prod_{i=1}^k \left( \alpha_i+ \beta_i \right)^e \right) F(\alpha_1, \cdots, \alpha_k, \hat{\alpha}, \beta_1, \cdots, \beta_k, \hat{\beta}) = \hspace{-0.075in} \sum_{\substack{h,g \in \mathbb{N}^k \\ h+g=d+e \cdot \mathbf{1}_k }} \hspace{-0.075in} f_{h,g}(\hat{\alpha},\hat{\beta}) \alpha_1^{h_1} \beta_1^{g_1} \cdots \alpha_k^{h_k} \beta_{k}^{g_k},
\label{eq:dummy2}
\end{equation}
in which we have denoted the coefficients of Product~\eqref{eq:dummy1} by $f_{h,g}(\hat{\alpha},\hat{\beta})$ and we have denoted the degree vector of $F$ by $d$. Also $\mathbf{1}_k \in \mathbb{N}^k$ denotes the vector of ones. Because $F$ is a homogeneous polynomial, $f_{h,g}$ are also homogeneous polynomials. Since $f_{h,g}(\hat{\alpha},\hat{\beta}) > 0$ for all $(h,g) \in M_{d,e}:= \{(h,g) \in \mathbb{N}^k \times \mathbb{N}^k: \, h+g = d+e \cdot \mathbf{1} \}$, Polya's theorem implies that there exist $l_{g,h} \geq 0$ for any $h,g \in M_{d,e}$ such that all of the coefficients of 
$(\hat{\alpha}+\hat{\beta})^{l_{g,h}} f_{g,h}(\hat{\alpha},\hat{\beta})$
are positive definite. Let us define
\[
e^* := \max \left\lbrace \max_{h,g \in M_{d,e}} \{l_{g,h} \},e \right\rbrace.
\] 
Then, clearly all of the coefficients in $(\hat{\alpha}+\hat{\beta})^{e^*} f_{g,h}(\hat{\alpha},\hat{\beta})$ are also positive definite. By multiplying both sides of~\eqref{eq:dummy2} by $(\hat{\alpha}+\hat{\beta})^{e^*}$ we have
\begin{align}
&  \hspace{-0.45in} \left( \prod_{i=1}^k \left( \alpha_i+\beta_i \right)^e \right) (\hat{\alpha}+\hat{\beta})^{e^*} F(\alpha_1, \cdots, \alpha_k, \hat{\alpha}, \beta_1, \cdots,\beta_k, \hat{\beta}) \nonumber \\
& \hspace{1.5in} =  \sum_{\substack{h,g \in \mathbb{N}^k \\ h+g=d+e \cdot \mathbf{1}_k }}(\hat{\alpha}+\hat{\beta})^{e^*} f_{h,g}(\hat{\alpha},\hat{\beta}) \alpha_1^{h_1}\beta_1^{g_1} \cdots \alpha_k^{h_k} \beta_{k}^{g_k}.
\label{eq:dummy4}
\end{align}
Since all of the coefficients of $(\hat{\alpha}+\hat{\beta})^{e^*}f_{h,g}(\hat{\alpha},\hat{\beta})$ are positive definite, all of the coefficients of the monomials on the right hand side of~\eqref{eq:dummy4} are positive definite. Moreover, because $e^* \geq e$
\begin{equation}
\left( \prod_{i=1}^k \left( \alpha_i+ \beta_i \right)^{e^*} \right) (\hat{\alpha}+\hat{\beta})^{e^*}  F(\alpha_1, \cdots, \alpha_k, \hat{\alpha}, \beta_1, \cdots,\beta_k, \hat{\beta}) 
\label{eq:dummy3}
\end{equation}
will also have all positive definite coefficients. Since we chose $(\hat{\alpha},\hat{\beta})$ arbitrarily from the simplex $\Delta^2$, by replacing $\hat{\alpha}$ and $\hat{\beta}$ with $\alpha_{k+1}$ and $\beta_{k+1}$ in~\eqref{eq:dummy3}, 
\[
\left( \prod_{i=1}^{k+1} \left( \alpha_i+ \beta_i \right)^{e^*} \right) F(\alpha,\beta) \; \text{ with } \; (\alpha_i,\beta_i) \in \Delta^2,\, i=1,\cdots, k+1
\]
will have all positive definite coefficients.
\end{proof}

\section{Setting-up the Problem of Robust Stability Analysis over Multi-simplex}
\label{sec:setup_multi}

In this section, we focus on the problem of robust stability of a system the form
\begin{equation}
\dot{x}(t) = A(\alpha) x(t),
\label{eq:sys_multi}
\end{equation} 
where $A(\alpha) \in \mathbb{R}^{n \times n}$ is a multi-homogeneous polynomial of degree vector $D_a$ and $\alpha \in \tilde{\Delta}^{\{l_1, \cdots,l_N \}}$ denotes the parametric uncertainty in the system. 
Note that if $A$ is not homogeneous, one can use Claim 1 to find a multi-homogeneous representation for $A$ over the multi-simplex. Furthermore, if $\alpha \in \Phi^N$, then one can use Claim 2 to find an equivalent representation for $A$ over the multi-simplex $\tilde{\Delta}^{\{l_1, \cdots,l_N \}}$.

The following theorem gives necessary and sufficient conditions for asymptotic stability of System~\eqref{eq:sys_multi}.
\begin{mythm}
\label{thm:Lyap_multisim}
The linear system~\eqref{eq:sys_multi} is stable if and only if there exists a polynomial matrix $P(\alpha)$ such that $P(\alpha)>0$ and 
\begin{equation}
\mathrm{A}^T(\alpha)P(\alpha)+P(\alpha)\mathrm{A}(\alpha) < 0 \quad \text{for all } \; \alpha \in \tilde{\Delta}^{\{l_1, \cdots, l_N \}}.  
\label{eq:Lyap_LMI}
\end{equation}
\end{mythm}
Unfortunately, the question of feasibility of the inequalities in Theorem~\ref{thm:Lyap_multisim} is NP-hard. In this section, we show that applying Theorem~\ref{thm:polya_multi-simplex2} yields a sequence of SDPs of increasing size (and precision) whose solutions converge to a solution of the inequalities in Theorem~\ref{thm:Lyap_multisim}. Motivated by the result in~\cite{bliman2006existence}, we consider $P(\alpha)$ to be homogeneous. In particular, let $P$ be a multi-homogeneous matrix-valued polynomial of form
\begin{equation}
P(\alpha) = \sum_{h_N \in W_{d_{p_N}}} \cdots \sum_{h_1 \in W_{d_{p_1}}} P_{ \mathcal{H}_N} \alpha_1^{h_1} \cdots \alpha_N^{h_N}, 
\label{eq:P_multi} 
\end{equation} 
with degree $d_p=\sum_{i=1}^{N} d_{p_i}$ and unknown coefficients $P_{ \mathcal{H}_N} \in \mathbb{S}^n$. Moreover, let $A(\alpha)$ be of the form
\begin{equation} 
A(\alpha) = \sum_{h_1 \in W_{d_{a_1}}} \cdots \sum_{h_N \in W_{d_{a_N}}} A_{\mathcal{H}_N} \alpha_1^{h_1} \cdots \alpha_N^{h_N},  
\label{eq:A_multi} 
\end{equation} 
with degree $d_a = \sum_{i=1}^{N} d_{a_i}$. It follows from Theorem~\ref{thm:polya_multi-simplex2} that the Lyapunov inequalities in Theorem~\ref{thm:Lyap_multisim} hold for all $\alpha \in \tilde{\Delta}^{\{ l_1, \cdots, l_N \}}$ if there exist some $d_1 \geq 0$ and $d_2 \geq 0$ such that
\begin{equation}
\prod_{i=1}^N \left(\sum_{j=1}^{l_i} \alpha_{i_j} \right)^{d_1}  P(\alpha) \quad \text{and} 
\label{eq:dummy_multi1}
\end{equation}
\begin{equation}
-\prod_{i=1}^N \left(\sum_{j=1}^{l_i} \alpha_{i_j} \right)^{d_2} \left( \mathrm{A}^T(\alpha)P(\alpha)+P(\alpha)\mathrm{A}(\alpha) \right)
\label{eq:dummy_multi2}
\end{equation}
have all positive definite coefficients.
By substituting for $A(\alpha)$ and $P(\alpha)$ in~\eqref{eq:dummy_multi1} and~\eqref{eq:dummy_multi2} from~\eqref{eq:A_multi} and~\eqref{eq:P_multi}, we find that the inequalities of Theorem~\ref{thm:Lyap_multisim} hold if there exists $d_1,d_2 \geq 0$ such that 
\begin{equation}
\sum_{h_1 \in W_{d_1}} \cdots \sum_{h_N \in W_{d_N}} \beta_{\{ \mathcal{H}_N, \Gamma_N \}} P_{ \mathcal{H}_N} > 0  \label{eq:LMI1_multisim} 
\end{equation}
for all $\gamma_1 \in W_{d_{p_1} + d_1}, \cdots, \gamma_N \in W_{d_{p_N} + d_1}$,
and 
\begin{equation}
\sum_{h_1 \in W_{d_1}} \cdots \sum_{h_N \in W_{d_N}} \left( H_{\{ \mathcal{H}_N,\Gamma_N \}}^T P_{ \mathcal{H}_N } + P_{\mathcal{H}_N} H^{\hspace{0.1mm}}_{ \{ \mathcal{H}_N,\Gamma_N \}} \right) < 0 \label{eq:LMI2_multisim} 
\end{equation}
for all $\gamma_1 \in W_{d_{pa_1} + d_2}, \cdots, \gamma_N \in W_{d_{pa_N} + d_2}$, where recall that $\mathcal{H}_N$ denotes $\{h_1, \cdots, h_N \}$, $\Gamma_N$ denotes $\{ \gamma_1, \cdots, \gamma_N \}$ and $d_{pa_i}= d_{p_i}+d_{a_i}$ for $i=1, \cdots,N$.

\subsection{General Formulae for Calculating Coefficients $\beta$ and $H$}
\label{sec:betaH_multisim}

 To calculate the $ \left\lbrace \beta_{ \{ \mathcal{H}_N ,  \Gamma_N \} } \right\rbrace$ coefficients and $\left\lbrace H_{\{ \mathcal{H}_N ,  \Gamma_N \}} \right\rbrace$ we provide the following recursive formulae. These formulae are generalization of the recursive formulae in~\ref{sec:betaH_simplex} for the case of a single simplex. First, for all $\gamma_1 \in W_{d_{p_1}}, \cdots, \gamma_N \in W_{d_{p_N}},$ and for all $ h_1 \in W_{d_{p_1}}, \cdots, h_N \in W_{d_{p_N}}$ set
\begin{equation}
\beta^{(0)}_{ \{ \mathcal{H}_N ,  \Gamma_N \}} = \begin{cases}1 & \text{if } h_1=\gamma_1, \cdots, h_N=\gamma_N \\ 0& \text{otherwise}. \end{cases} 
\label{eq:beta_init_multi}
\end{equation}
Then, for $i=1, \cdots, d_1$, for all $\gamma_1 \in W_{d_{p_1}} + i, \cdots, \gamma_N \in W_{d_{p_N}}+i$ and for all $h_1 \in W_{d_{p_1}}, \cdots, h_N \in W_{d_{p_N}}$, $\beta^{(i)}_{ \{ \mathcal{H}_N ,  \Gamma_N \} }$ can be calculated using
\begin{equation}
\beta^{(i)}_{ \{ \mathcal{H}_N ,  \Gamma_N \} } =  \sum_{\lambda_N \in W_{1} } \cdots  \sum_{\lambda_1 \in W_{1}}   \beta^{(i-1)}_{ \{ \mathcal{H}_N , \{ \gamma_1-\lambda_1, \cdots, \gamma_N-\lambda_N \}\} }.
\label{eq:beta_i_multi} 
\end{equation}
Finally, set $\beta_{ \{ \mathcal{H}_N ,  \Gamma_N \} } = \beta^{(d_1)}_{ \{ \mathcal{H}_N ,  \Gamma_N \} }$, where $ \gamma \in W_{d_p+d_1} $.

To calculate $ \left\lbrace H_{ \{ \mathcal{H}_N ,  \Gamma_N \}} \right\rbrace$, first let 
\begin{equation}
H^{(0)}_{\{ \mathcal{H}_N ,  \Gamma_N \} } = \sum_{ \substack{\lambda_N \in W_{d_{a_N}}: \\ \lambda_N + h_N = \gamma_N}} \cdots \sum_{ \substack{ \lambda_1 \in W_{d_{a_1}}: \\ \lambda_1 + h_1 = \gamma_1}} A_{ \{ \lambda_1, \cdots, \lambda_N \}}. 
\label{eq:H_init_multi}
\end{equation}
for $\gamma_1 \in W_{d_{pa_1}}, \cdots, \gamma_N \in W_{d_{pa_N}}$ and $ h_1 \in W_{d_{p_1}}, \cdots, h_N \in W_{d_{p_N}}$. Then, for $i=1,\ldots, d$, $ \gamma_1 \in W_{d_{pa_1} + i}, \cdots, \gamma_1 \in W_{d_{pa_N} + i}$ and $ h_1 \in W_{d_{p_1}}, \cdots, h_N \in W_{d_{p_N}} $ we have 
\begin{equation}
H^{(i)}_{\{ \mathcal{H}_N ,  \Gamma_N \}}=  
\sum_{\lambda_N \in W_1} \cdots \sum_{\lambda_1 \in W_1} H^{(i-1)}_{\{ \mathcal{H}_N , \{ \gamma_1- \lambda_1, \cdots, \gamma_N-\lambda_N \}\}}. \label{eq:H_i_multi} \vspace*{-0.05in}
\end{equation}
Finally, set $H_{\{ \mathcal{H}_N ,  \Gamma_N \}} = H^{(d_2)}_{\{ \mathcal{H}_N ,  \Gamma_N \}}$, where $ \gamma_1 \in W_{d_{pa_1}+d_2}, \cdots, \gamma_N \in W_{d_{pa_N}+d_2} $.

\subsection{The SDP Elements Associated with the Multi-simplex Version of Polya's Theorem}

To solve the LMI conditions in~\eqref{eq:LMI1_multisim} and~\eqref{eq:LMI2_multisim}, we express them in the form of a dual Semi-Definite Programming (SDP) problem with a block-diagonal structure that is suitable for parallel computation.  Define the element $C$ of the SDP formulation of Conditions~\eqref{eq:LMI1_multisim} and~\eqref{eq:LMI2_multisim} as
\begin{equation}
C := \text{diag}(C_1, \cdots C_L, C_{L+1}, \cdots C_{L+M}), 
\label{eq:C_multi} 
\end{equation}
where for given Polya's exponents $d_1$ and $d_2$,
\begin{equation}
L = \prod_{i=1}^{N} \dfrac{(d_{p_i}+d_1+l_i-1)!}{(d_{p_i}+d_1)!(l_i-1)!} \label{eq:L_multi} 
\end{equation}
is the number of monomials in $ \prod\limits_{i=1}^N \left(\sum\limits_{j=1}^{l_i} \alpha_{i,j} \right)^{d_1}  P(\alpha )$ and 
\begin{equation}
M = \prod_{i=1}^{N} \frac{(d_{p_i}+d_{a_i}+d_2+l_i - 1)!}{(d_{p_i}+d_{a_i}+d_2)!(l_i-1)!}  
\label{eq:M_multi}
\end{equation}
is the number of monomials in
\[
\prod_{i=1}^N \left(\sum_{j=1}^{l_i} \alpha_{i,j} \right)^{d_2} (A^T(\alpha)P(\alpha)+P(\alpha)A(\alpha)),
\]
and for $j=1, \cdots, L+M$, 
\begin{equation}
C_j :=
\begin{cases}
	\epsilon I_n \zeta^{(j)} , & \text{if } 1 \le j \le L \\
	0_n, &  \text{if }   L+1 \le  j  \le L+M,
\end{cases}  
\label{eq:Cj_multi} 
\end{equation}
where $\epsilon > 0$. In~\eqref{eq:Cj_multi}, we define $\zeta^{(j)} \in \mathbb{N}^{L}$ recursively as follows. First, let
\[
\zeta^{(0)} = \begin{bmatrix}
\dfrac{ (d_{p_{N}}+d_1)! }{ \prod\limits_{i=1}^{l_N} h_{(N,i,1)}! }, & \cdots, & \dfrac{ (d_{p_{N} }+d_1)!}{ \prod\limits_{i=1}^{l_N} h_{(N,{i,f(l_N,d_{P_N}+d_1))}}! }
\end{bmatrix}, 
\]
where we have denoted the exponent of the $i^{th}$ variable in the $j^{th}$ (according to lexicographical ordering) element of $W_{d_{P_N}}$ by $h_{(N,i,j)}$. Recall that
\[
f(l,g) :=\dfrac{(l+g-1)!}{g!(l-1)!}
\]
is the number of monomials in a polynomial of degree $g$ with $l$ variables. Then, for $k=1, \cdots, N$, define
\begin{equation*}
\zeta^{(k)} := \zeta^{(k-1)} \otimes
 \begin{bmatrix}
\dfrac{ (d_{p_{r(k)}}+d_1)!}{\prod\limits_{i=1}^{l_{r(k)}} h_{(r(k),i,1)}! }, \cdots, \dfrac{ (d_{p_{r(k)}}+d_1)!}{ \prod\limits_{i=1}^{l_{r(k)}} h_{(r(k),i,s(k))}!}
\end{bmatrix}, 
\end{equation*}
where $r(k):= N-k+1$ and $s(k):=f(l_{r(k)},d_{p_{r(k)}}+d_1)$. Finally, set $\zeta = \zeta^{(N)}$.

For $i=1, \cdots, K$, define the elements $B_i$ of the SDP as
\begin{equation}
B_i = \text{diag} \left(B_{i,1}, \cdots, B_{i,L}, B_{i,L+1}, \cdots, B_{i,L+M} \right), 
\label{eq:Ai_multi}
\end{equation}
where 
\begin{equation}
K = \frac{n(n+1)}{2} \prod_{i=1}^N \frac{(d_{p_i} + l_i-1)!}{d_{p_i}! (l_i-1)!}, \label{eq:K}
\end{equation}
is the total number of dual variables in the SDP problem (i.e., the total number of upper-triangular elements in all of the coefficients of $P(\alpha)$) and where
\begin{small}
\begin{equation}
 B_{i,j} \hspace{-0.025in} = \hspace{-0.05in}
\begin{cases}
\sum\limits_{h_N \in W_{d_{p_N}}} \hspace{-0.1in}  \cdots \hspace{-0.1in} \sum\limits_{h_1 \in W_{d_{p_1}}} \hspace{-0.1in} \beta_{\{\mathcal{H}_N, \Gamma_{N,j}\}} V_{\mathcal{H}_N}(e_i), & \hspace{-0.05in} \text{if } 1 \le j \le L   \\
 - \hspace{-0.2in} \sum\limits_{h_N \in W_{d_{p_N}}} \hspace{-0.1in} \cdots  \hspace{-0.1in} \sum\limits_{h_1 \in W_{d_{p_1}}} \hspace{-0.1in} H_{\{ \mathcal{H}_N , \Gamma_{N,j-L}\}}^T V_{\mathcal{H}_N}  (e_i) + V_{\mathcal{H}_N}(e_i) H_{ \{\mathcal{H}_N, \Gamma_{N,j-L} \}} & \hspace{-0.05in} \text{if }
 L+1 \le j \le L+M,
\end{cases} 
\label{eq:Aij}
\end{equation}
\end{small}
where recall from Section~\ref{sec:notation_multi-homog} that $\Gamma_{N,j}=\{ \gamma_{1_j}, \cdots, \gamma_{N_j} \}$, where $\gamma_{i_j}$ is the $j^{th}$ element of $W_{d_{p_i}+d_1}$ using lexicographical ordering, and 
\[
V_{ \mathcal{H}_N }(x)= \sum_{k=1}^{\tilde{N}} E_k \; x_{k+N(I_{\mathcal{H}_N}-1)},
\]
where $E_k$ is the canonical basis for $ \mathbb{S}^n $ defined in~\eqref{eq:Ej_basis}, $I_{\mathcal{H}_N}$ is the lexicographical index of monomial $\alpha_1^{h_1} \cdots \alpha_N^{h_N}$, and $\tilde{N}:=\frac{n(n+1)}{2}$.
Finally, we complete the definition of the SDP problem by setting
$
a=\vec{1} \in \mathbb{R}^K.
$
In the following section, we propose a parallel set-up algorithm to calculate the SDP elements defined in this section.

\subsection{A Parallel Algorithm for Setting-up the SDP}

In this section, we propose a parallel set-up algorithm for computing the SDP elements in~\eqref{eq:C_multi} and~\eqref{eq:Ai_multi}. An abridged description of the algorithm is presented in Algorithm 7, wherein we suppose the algorithm is executed on $N_c$ number of processors. A C++ parallel implementation of the algorithm is available at:\\ \url{www.sites.google.com/a/asu.edu/kamyar/Software}.

\begin{algorithm}
\begin{small}
\noindent \textbf{\textit{Inputs:}} \vspace{-0.1in} \\
$N:$ dimension of multi-simplex; $l_1,\cdots,l_N:$ dimensions of simplexes; $D_p,D_a:$ degree vectors of $P$ and $A$; coefficients of $A$; $\hat{d}_1,\hat{d}_2:$ Polya's exponents.

\noindent \textbf{\textit{Initialization:}} \vspace{-0.1in}  \\ 
\For{$i=1, \cdots, N_c$, processor $i$}{
 Set $d_1=d_2=0$ and $d_{pa}=d_p+d_a$.  \\ 
 Calculate the number of monomials in $P(\alpha)$, i.e., $L$ using~\eqref{eq:L_multi}.\\
 Calculate the number of monomials in $P(\alpha)A(\alpha)$, i.e., $M$ using~\eqref{eq:M_multi}.\\
 Calculate the per-processor number of monomials in $P(\alpha)$ and $P(\alpha)A(\alpha)$, i.e., \vspace{-0.15in} 
\begin{equation}
L'=\mathtt{floor}\left(L/N_c\right) \quad \text{and} \quad M'=\mathtt{floor}\left(M/N_c \right).   \vspace{-0.1in} 
\label{eq:Lp_multi}
\end{equation}
 \For{$\gamma_1, h_1 \in W_{d_{p_1}}, \cdots, \gamma_N,h_N \in W_{d_{p_N}}$}{
	 Calculate $\beta_{\{\mathcal{H}_N,\Gamma_N\}}$ using~\eqref{eq:beta_init_multi}.
 } \vspace{-0.1in}
 \For{$\gamma_1 \in W_{d_{pa_1}}, \cdots, \gamma_N \in W_{d_{pa_N}}$ and $h_1 \in W_{d_{p_1}}, \cdots, h_N \in W_{d_{p_N}}$}{ \vspace{0.05in}
	 Calculate $H_{\{\mathcal{H}_N,\Gamma_N\}}$ using~\eqref{eq:H_init_multi}. 
 } \vspace{-0.1in}
} \vspace{-0.1in}

\noindent \textbf{\textit{Polya's iterations:}} \vspace{-0.1in}  \\

\For{$i=1, \cdots, N_c$, processor $i$}{ \vspace{-0.1in}
	\For{$d_1=1,\cdots,\hat{d}_1$}{
		Set $d_p = d_p + 1$. Update $L$ and $L'$ according to~\eqref{eq:L_multi} and~\eqref{eq:Lp_multi}.\\		
		\For{$h_1 \in W_{d_{p_1}}, \cdots, h_N \in W_{d_{p_N}}$}{ \vspace{0.05in}
			Update $\beta_{\{\mathcal{H}_N,\Gamma_N\}}$ for $\gamma_{(i-1)L'+1} \in W_{d_{p_{(i-1)L'+1}}}, \cdots, \gamma_{iL'} \in W_{d_{p_{iL'}}}$ as in~\eqref{eq:beta_i_multi}. \vspace{-0.1in}
		}		
\vspace{-0.1in}
	} \vspace{-0.1in}
	\For{$d_2=1,\cdots,\hat{d}_2$}{
		Set $d_{pa} = d_{pa}+1$. Update $M$ and $M'$ according to~\eqref{eq:M_multi} and~\eqref{eq:Lp_multi}.\\
		\For{$h_1 \in W_{d_{p_1}}, \cdots, h_N \in W_{d_{p_N}}$}{ \vspace{0.05in}
		Update $H_{\{\mathcal{H}_N,\Gamma_N\}}$ for $\gamma_{(i-1)M'+1} \in W_{d_{{pa}_{(i-1)M'+1}}}, \cdots, \gamma_{iM'} \in W_{d_{{pa}_{iM'}}}$ using~\eqref{eq:H_i_multi}.
		} \vspace{-0.1in}
	} \vspace{-0.1in}
}
\end{small} 
\label{alg:setup_multisim}
\end{algorithm}

\begin{algorithm}
\noindent \textbf{\textit{Calculating the SDP elements:}} \\
\For{$i=1, \cdots, N_c$, processor $i$}{
	\For{$j=(i-1)L'+1, \cdots, iL', L+(i-1)M'+1, \cdots, iM'$}{
 	Calculate $C_j$	using~\eqref{eq:Cj_multi}.\\
	Calculate $B_j$ for using~\eqref{eq:Ai_multi}.
	}
}

\noindent \textbf{\textit{Outputs:}}\\
The SDP elements $C$ and $B_i$ for $i=1, \cdots, K$. \vspace{0.1in}
\caption{A parallel set-up algorithm for robust stability analysis over the multi-simplex}
\end{algorithm}

\section{Computational Complexity Analysis of the Set-up Algorithm}
\label{sec:compexity_setup_multi}

In this section, we discuss the performance of Algorithm 7 in terms of speed-up, computation cost, communication cost and memory requirement.

\subsection{Computational Cost of the Set-up Algorithm:}

The most computationally expensive part of the algorithm is calculation of the elements $B_{i,j}$ elements for $i=1,\cdots, K$ and $j=1, \cdots, L+M$. If the number of available processors is 
\[
N_c=L_0 := \prod_{i=1}^{N} \dfrac{(d_{p_i}+l_i-1)!}{(d_{p_i})!(l_i-1)!},
\]
then the number of operations per processor at each Polya's iteration of Algorithm 7 is 
\begin{equation}
\thicksim K \cdot L_0 \left( \mathtt{floor} \left( \frac{L}{N_c}\right)+ \mathtt{floor} \left( \frac{M}{N_c} \right)  \right) n^3 \thicksim n^5\prod_{i=1}^N l_i^{2d_{p_i}+d_{a_i}+d_2}, 
\label{eq:comp_multi}
\end{equation}
where recall that $d_{p_i}$ is the degree of $\alpha_i^{h_i}$ in polynomial $P(\alpha)$ (see~\eqref{eq:P_multi}), $d_{a_i}$ is the degree of the variable $\alpha_i^{h_i}$ in the polynomial $A(\alpha)$ (see~\eqref{eq:A_multi}), and $d_2$ is the Polya's exponent. Note that for the case of systems with uncertain parameters inside a simplex,~\eqref{eq:comp_multi} reduces to our results in Table~\ref{tab:setup_comm_complexity}.
The number of operations versus the dimension of hypercube $N$ is plotted in Figure~\ref{fig:comp} for different Polya's exponents $d:=d_1=d_2$. The figure shows that for the case of analysis over a hypercube, the number of operations grows exponentially with the dimension of the hypercube, whereas in analysis over a simplex, the number of operations grows polynomially. This is due to the fact that an $N$-dimensional hypercube is represented by the Cartesian product of $N$ two-dimensional simplices.

\begin{figure}[thpb]
   \centering
   \includegraphics[scale=0.4]{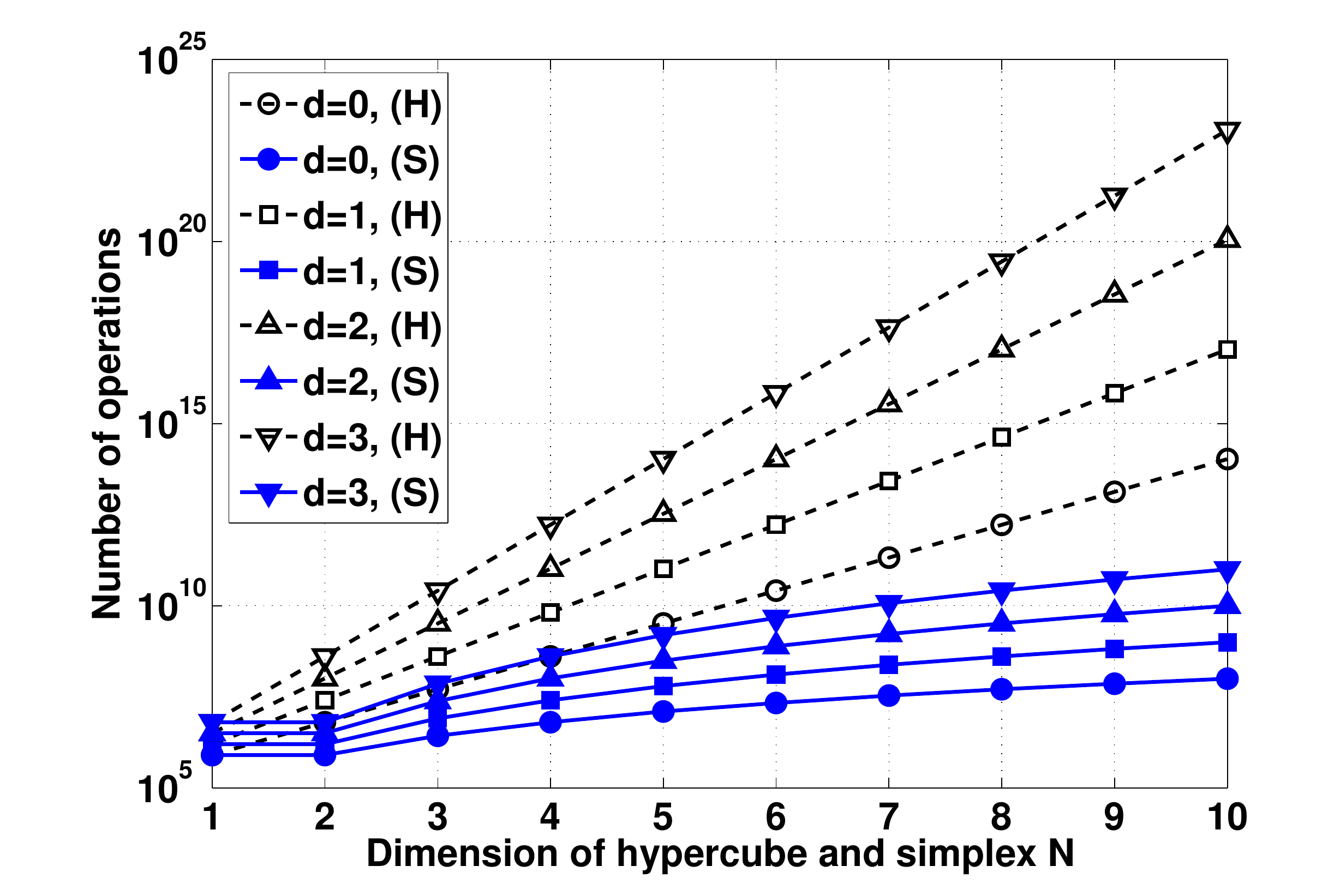} 
   \caption{Number of Operations vs. Dimension of the Hypercube, for Different Polya's Exponents $d_1=d_2=d$. (H): Hypercube and (S): Simplex.}
   \label{fig:comp} 
\end{figure}

\subsection{Communication Cost of the Set-up Algorithm:}

In the worst case scenario, where each processor sends all of its assigned coefficients $\{ H_{\{\mathcal{H}_N, \gamma_N \}} \}$ to other processors (a very rare situation), the communication cost per processor at each Polya's iteration is 
\begin{equation}
\thicksim L_0\left( \mathtt{floor} \left( \frac{L}{N_c} \right) + \mathtt{floor} \left( \frac{M}{N_c} \right)n^2 \right)   \thicksim { n^2} \prod_{i=1}^N l_i^{d_{pa_i}+d_2},
\label{eq:comm_multi}
\end{equation}
assuming the number of processors $N_c=L_0$. Therefore, in this case, by increasing the number of processors, the communication cost per processor decreases and the scalability of the algorithm improves. for the case where the uncertain parameters belong to a simplex,~\eqref{eq:comm_multi} reduces to our results in Table~\ref{tab:setup_comm_complexity}. Again, it can be shown that the communication cost increases exponentially with the dimension of the hypercube, whereas in analysis over a simplex, the communication cost increases polynomially.

\begin{figure}[thpb]
   \centering
   \hspace*{-0.5in}\includegraphics[scale=0.4]{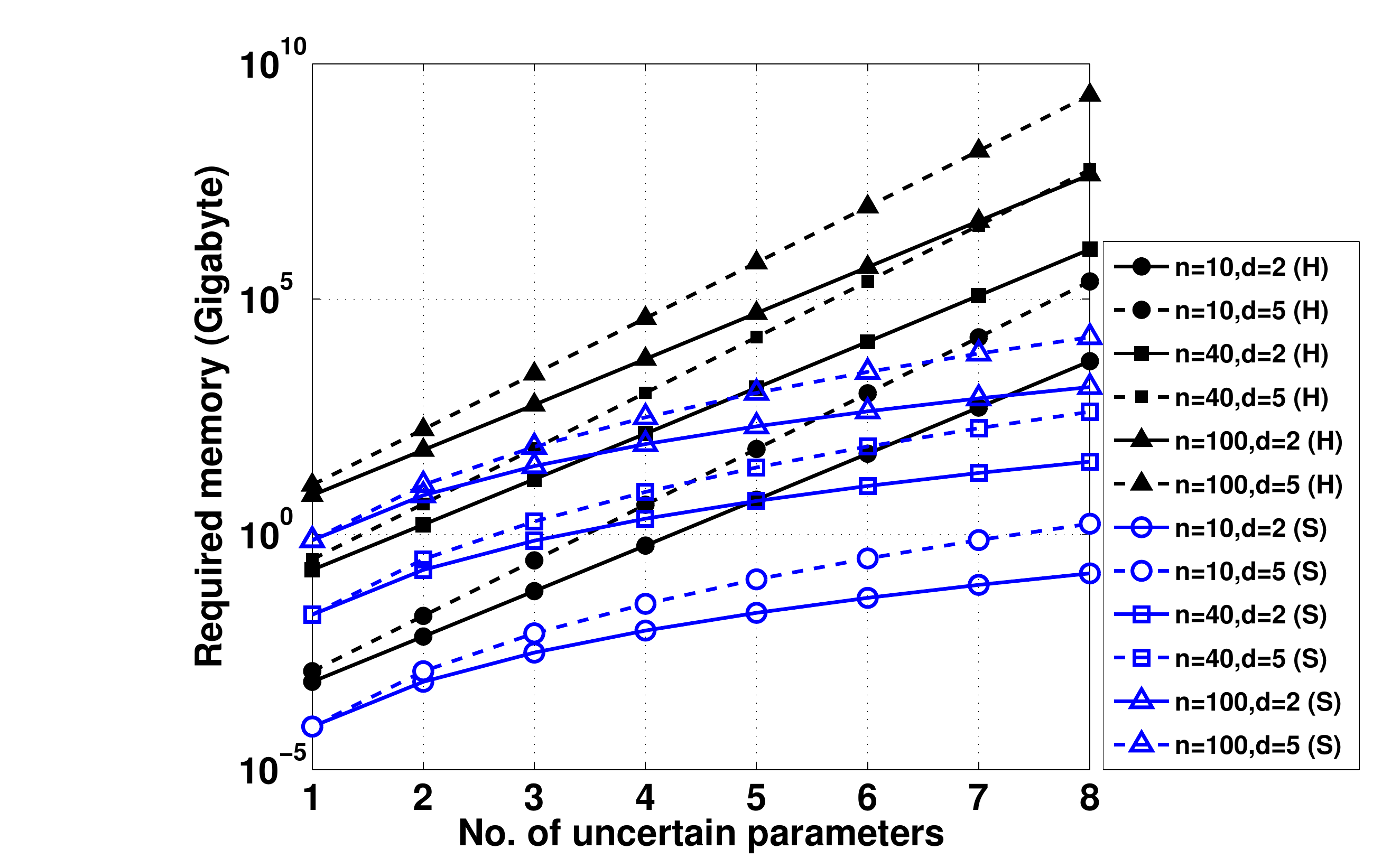} 
   \caption{Required Memory for the Calculation of SDP Elements vs. Number of Uncertain Parameters in Hypercube and Simplex, for Different State-space Dimensions and Polya's Exponents $d_1=d_2$. (H): Hypercube, (S): Simplex.} 
   \label{fig:memory}
\end{figure}

\subsection{Speed-up and Memory Requirement of the Set-up Algorithm:}

In the proposed set-up algorithm (Algorithm 7), calculation of the coefficients $\{ \beta \}$ and $\{ H \}$ is distributed among all of the available processors such that there exists no centralized computation. As a result, the algorithm can theoretically achieve ideal (linear) speed-up. In other words, the speed-up 
\[
\textit{\text{SP}}_N=\dfrac{N}{D+NS} = \dfrac{N}{1+0} = N,
\]
Where $D = 1$ is the ratio of the operations performed by all processors simultaneously to the total operations performed simultaneously and sequentially, and $S$ is the ratio of the operations performed sequentially to the total operations performed simultaneously and sequentially. 

In Figure~\ref{fig:memory}, we have shown the amount of memory required for storing the SDP elements versus the number of uncertain parameters in the unit hypercube and the unit simplex. The figure shows the required memory for  different dimensions of the state-space $n$ and Polya's exponents $d$. In all of the cases, we use $d_{p_i}=d_{a_i}=1$ for $i=1, \cdots, N$. The figure shows that for the case analysis over the hypercube, the required memory increases exponentially with the number of uncertain parameters, whereas for the case of the standard simplex the required memory grows polynomially. 

\pagebreak

\hspace{-0.325in} This is again because an $N$-dimensional hypercube is the Cartesian product of $N$ two-dimensional simplices, i.e., $\underbrace{\Delta^2 \times \cdots \times \Delta^2}_\text{N times}$.


\section{Testing and Validation}
\label{examples_multi}

In this section, we evaluate the scalability and accuracy of our algorithm through numerical examples. In example 1, we evaluate the speed-up of our algorithm through numerical experiments. In examples 2 and 3, we evaluate the conservativeness of our algorithm and compare it to other methods in the literature.

\subsection{Example 1: Evaluating Speed-up}
A parallel algorithm is scalable, if by using $N_c$ processors it can solve a problem $N_c$ times faster than solving the same problem using one processor. Thus, the speed-up of the ideal scalable algorithm is linear.
To test the scalability of our algorithm, we run the algorithm using two random uncertain systems with state-space dimensions $n=5$ and $n=10$. The tests were performed on a linux-based Karlin cluster computer at Illinois Institute of Technology. In all of the runs, $D_p=[2,2,2,2], D_a=[1,1,1,1]$ and $\alpha \in \Phi^4$. Figure~\ref{fig:speedup_multi} shows the computation time of the algorithm versus the number of processors, for two different state-space dimensions and two different number of Polya's iterations (Polya's exponents $d=d_1=d_2$). The linearity of the curves in all of the executions implies near-perfect scalability of the algorithm.

\begin{figure}[ht]
\centering
 \includegraphics[scale=0.38]{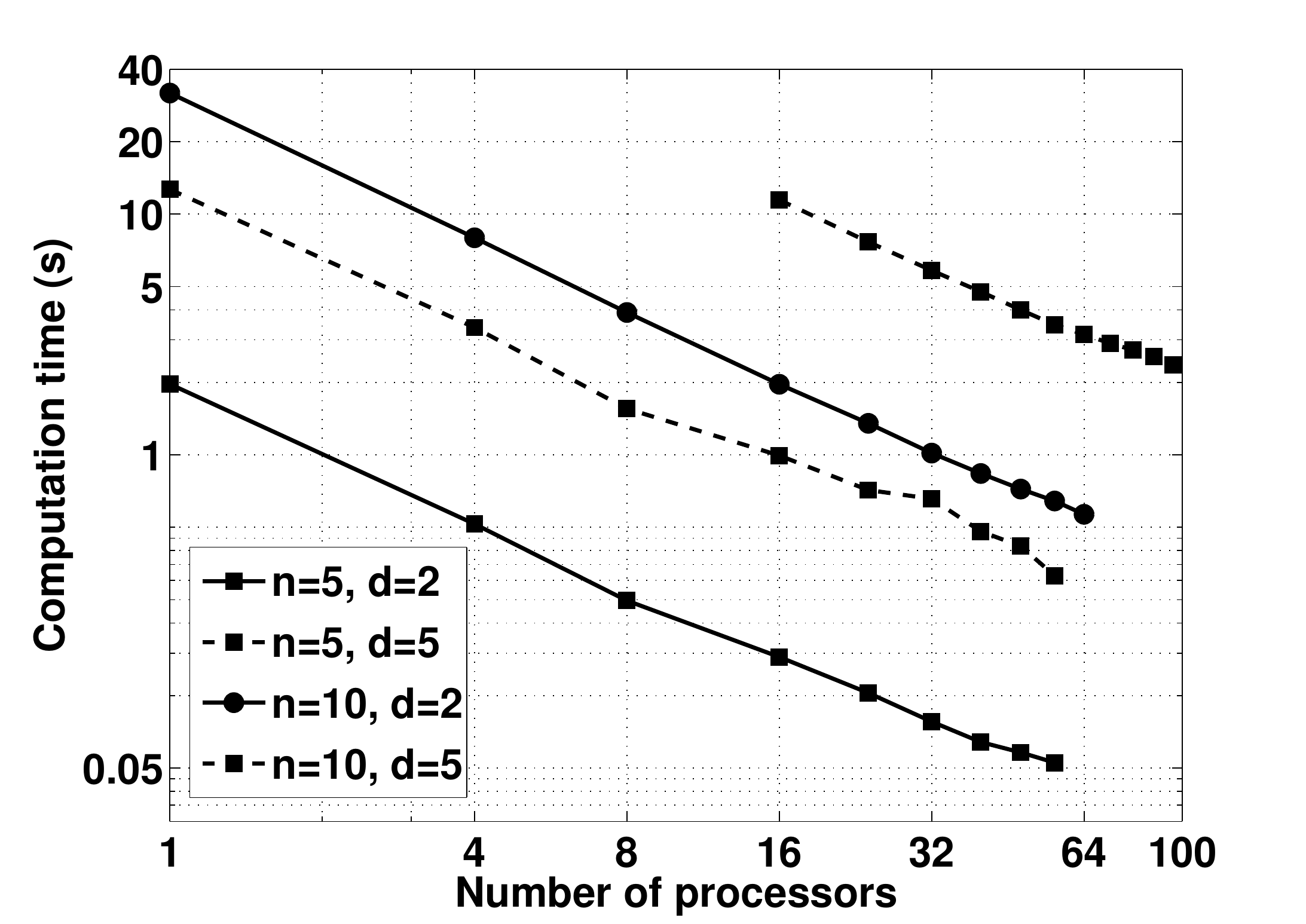} 
\caption{Execution Time of the Set-up Algorithm vs. Number of Processors, for Different State-space Dimensions $n$ and Polya's Exponents}
\label{fig:speedup_multi}
\end{figure}

\subsection{Example 2: Verifying Robust Stability over a Hypercube}
Consider the system $\dot{x}(t) = A(\alpha) x(t)$, where 
\[
A(\alpha)= A_0 + A_1 \alpha_1^2 + A_2 \alpha_1 \alpha_2 \alpha_3 + A_3 \alpha_1^2 \alpha_2 \alpha_3^2,  
\]
\[
\alpha_1 \in [-1,1], \, \alpha_2 \in [-0.5,0.5], \, \alpha_3 \in [-0.1,0.1],  
\]
where 
\renewcommand{\arraystretch}{0.75}
\begin{align*}
&A_0=\begin{bmatrix}
   -3.0  &     0 &  -1.7 &   3.0\\
   -0.2  & -2.9 & -1.7  & -2.60\\
    0.6  &  2.6 &  -5.8  & -2.60\\
   -0.7  &  2.9  & -3.3  & -2.10
\end{bmatrix} 
&&A_1=\begin{bmatrix}
    2.2 &  -5.4 &  -0.8  & -2.2\\
    4.4 &  1.4  & -3.0   & 0.8\\
   -2.4 &  -2.2 &   1.4  &  6.0\\
   -2.4 &  -4.4 &  -6.4  &  0.18
\end{bmatrix} \\
&A_2=\begin{bmatrix}
   -8.0 & -13.5 &  -0.5 &  -3.0\\
   18.0 &  -2.0  &  0.5 & -11.5\\
    5.5 & -10.0 &   3.5  &  9.0\\
   13.0 &   7.5 &   5.0 &  -4.0
\end{bmatrix}
&&A_3=\begin{bmatrix}
    3.0  &  7.5 &   2.5 &  -8.0\\
    1.0  &  0.5 &   1.0 &   1.5\\
   -0.5  & -1.0 &   1.0 &   6.0\\
   -2.5  & -6.0 &   8.5 &  14.25
\end{bmatrix}.    
\end{align*}
The problem is to investigate asymptotic stability of this system for all $\alpha$ in the given intervals using Algorithm 7 and our solver in Algorithm 6.
We first represented $A(\alpha)$ defined over the hypercube $[-1,1] \times [-0.5,0.5] \times [-0.1,0.1]$ by a multi-homogeneous polynomial $B(\beta,\eta)$ with $(\beta_i,\eta_i) \in \Delta^2$ and with the degree vector $D_b=[2,1,2]$. Then, in one Polya's iteration (i.e., $d_1=d_2=1$) our algorithm found the following Lyapunov function as a certificate for asymptotic stability of the system.
\begin{align*}
V(x,\beta,\eta)= x^T P(\beta,\eta)x  = x^T ( & \beta_1 (P_1 \beta_2 \beta_3 + P_2 \beta_2 \eta_3 +P_3 \eta_2 \beta_3 + P_4  \eta_2 \eta_3)  \\
 + & \eta_1 (P_5 \beta_2 \beta_3 + P_6 \beta_2 \eta_3 + P_7 \eta_2 \beta_3 + P_8  \eta_2 \eta_3 ) )x,
\end{align*}
where $\beta_1= 0.5 \alpha_1 +0.5, \beta_2=\alpha_2+0.5, \beta_3=5\alpha_3+0.5, \eta_1= 1- \beta_1, \eta_2 = 1- \beta_2, \eta_3 = 1- \beta_3$ and
\begin{align*}
&P_1=\begin{bmatrix}
   5.807  &     0.010 &  -0.187 &  -1.186\\
    0.010  &  5.042 & -0.369  &  0.227\\
    -0.187  & -0.369 &  8.227  & -1.824\\
   -1.186  &   0.227  & -1.824  &  8.127
\end{bmatrix} 
P_2=\begin{bmatrix}
    7.409 &  -0.803 &  1.804  & -1.594\\
    -0.803 &  6.016  &  0.042   & -0.538\\
   1.804 &   0.042 &   7.894  &  -1.118\\
   -1.594 &  -0.538 &  -1.118  &  8.590
\end{bmatrix} \\
&P_3=\begin{bmatrix}
    6.095 & -0.873 &   0.512 &  -1.125\\
   -0.873 &  5.934  &   -0.161 &  0.503\\
    0.512 &  -0.161 &   7.417  &  -0.538\\
   -1.125 &    0.503 &   -0.538 &  6.896
\end{bmatrix}
P_4=\begin{bmatrix}
    5.388  &  0.130 &   -0.363 &  -0.333\\
    0.130  &  5.044 &   -0.113 &   -0.117\\
   -0.363  & -0.113 &    6.156 &   -0.236\\
   -0.333  & -0.117 &   -0.236 &   5.653
\end{bmatrix}   \\
&P_5=\begin{bmatrix}
   7.410  &  -0.803 &  1.804 &  -1.594\\
   -0.803  &  6.016 &  0.042  &  -0.538\\
    1.804  &  0.042 &  7.894  & -1.118\\
   -1.594  &  -0.538  & -1.118  &  8.590
\end{bmatrix} 
P_6=\begin{bmatrix}
    5.807 &   0.010 &  -0.187  & -1.186\\
     0.010 &   5.042  & -0.369   &  0.227\\
   -0.187 &  -0.369 &    8.227  & -1.824\\
   -1.186 &   0.227 &  -1.824  &   8.127
\end{bmatrix} \\
&P_7=\begin{bmatrix}
    5.388 & 0.130 &   -0.363 &  -0.333\\
   0.130 &  5.044  & -0.113 & -0.117\\
   -0.363 & -0.113 &  6.156  &   -0.236\\
   -0.333 &  -0.117 &   -0.236 &  5.653
\end{bmatrix}
P_8=\begin{bmatrix}
    6.095  &   -0.873 &   0.512 &  -1.125\\
     -0.873  &  5.934 &  -0.161 &   0.503\\
   0.512  & -0.161 &   7.417 &   -0.538\\
    -1.125  &  0.503 &  -0.538 &   6.896
\end{bmatrix}    .
\end{align*}

\subsection{Example 2: Evaluating Accuracy}

In this example, we used our algorithm to find lower bounds on $r^*= \max \, r$ such that $\dot{x}(t) = A(\alpha) x(t)$ with
\[
A(\alpha) = A_0+\sum_{i=1}^4 A_i \alpha_i,
\]
\begin{small}
\begin{align*}
A_0 \hspace{-0.03in}  = \hspace{-0.03in}  \begin{bmatrix}
   -3.0  &  0   & -1.7 &   3.0\\
   -0.2  & -2.9 & -1.7  & -2.6\\
    0.6  &  2.6 & -5.8  & -2.6\\
   -0.7  &  2.9 & -3.3  & -2.4
\end{bmatrix},
A_1  \hspace{-0.03in}  =  \hspace{-0.03in} \begin{bmatrix}
    1.1 &  -2.7 &  -0.4  & -1.1\\
    2.2 &  0.7  &  -1.5   & 0.4\\
   -1.2 &  -1.1 &   0.7  &  3.0\\
   -1.2 &  -2.2 &  -3.2  &  -1.4
\end{bmatrix}, 
A_2  \hspace{-0.03in} = \hspace{-0.03in} \begin{bmatrix}
   1.6 & 2.7 &  0.1 &  0.6\\
   -3.6 &  0.4  &  -0.1 & 2.3\\
    -1.1 & 2 &  -0.7  &  -1.8\\
   -2.6 &  -1.5 &  -1.0 &  0.8
\end{bmatrix}, 
\end{align*}
\begin{align*}
A_3 =\begin{bmatrix}
    -0.6  & 1.5 &  0.5 & -1.6 \\
     0.2 & -0.1 & 0.2  &  0.3 \\
     -0.1 & -0.2  & -0.2  & 1.2  \\
    -0.5 & -1.2 & 1.7  & -0.1 
\end{bmatrix}, 
A_4=\begin{bmatrix}
    -0.4  & -0.1 &  -0.3 & 0.1 \\
     0.1 & 0.3 & 0.2  &  0.0 \\
     0.0 & 0.2  & -0.3  & 0.1  \\
     0.1 & -0.2 & -0.2  & 0.0 
\end{bmatrix}
\end{align*}
\end{small}
\hspace*{-0.09in} is asymptotically stable for all $\alpha \in \{ \alpha \in \mathbb{R}^4: \vert \alpha_i \vert \leq r \}$. In Table~\ref{tab:accuracy_multisim}, we have shown the computed lower bounds on $r^*$ for different degree vectors $D_p$ (degree vector of polynomial $P$ in Theorem~\ref{thm:Lyap_multisim}). In all of the cases, we set the Polya's exponents $d_1=d_2=0$. For comparison, we have also included the lower-bounds computed by the methods in~\cite{bliman2004convex} and~\cite{chesi_hypercube_2005} in Table~\ref{tab:accuracy_multisim}. \\

\renewcommand{\arraystretch}{1.1}
\renewcommand{\tabcolsep}{1.4pt} 
\begin{table}[h]
\caption{The Lower-bounds on $r^*$ Computed by Algorithm 7 Using Different Degree Vector $D_p$ and Using Methods in~\cite{bliman2004convex} and~\cite{chesi_hypercube_2005}.} 
\label{tab:accuracy_multisim}
\begin{center}
\begin{scriptsize}
\begin{tabular}{ c|c|c|c|c|c|c|c|c}
\cline{2-8}
& $D_p=$[0,0,0,0] & $D_p=$[0,1,0,1] & $D_p=$[1,0,1,0] & $D_p=$[1,1,1,1] &  $D_p=$[2,2,2,2] & \cite{bliman2004convex} & \cite{chesi_hypercube_2005} \\ 
\hline
\hspace*{-5.5pt} \vline Bound on $r^*$ & 0.494 & 0.508 & 0.615 & 0.731 & 0.840 & 0.4494 & 0.8739 \\ 
\hline 
\end{tabular} 
\end{scriptsize}
\end{center}
\end{table}


\chapter{PARALLEL ALGORITHMS FOR NONLINEAR STABILITY ANALYSIS}
\label{chp:Nonlinear}

\section{Background and Motivation}
\label{sec:intro_nonlin}

One approach to stability analysis of nonlinear systems is the search for a decreasing Lyapunov function.  For those systems with polynomial vector fields,~\cite{peet2009exponentially} has shown that searching for polynomial Lyapunov functions is necessary and sufficient for stability on any bounded set. However, searching for a polynomial Lyapunov function which proves local stability requires enforcing positivity on a neighborhood of the equilibrium. Unfortunately, while we do have necessary and sufficient conditions for positivity of a polynomial (e.g. Tarski-Seidenberg's algorithm in~\cite{tarski} and Artin's theorem in~\cite{artin}), it has been shown that the general problem of determining whether a polynomial is positive is NP-hard~(\cite{blum}).



Based on Artin's theorem, non-negativity of a polynomial is equivalent to existence of a representation in the form of sum of quotients of squared polynomials. If we leave off the quotient, the search for a Sum-of-Squares (SOS) is a common sufficient condition for positivity of a polynomial.
The advantage of the SOS approach is that verifying the existence of an SOS representation is a semidefinite programming problem. This approach was first articulated in~\cite{parillo_thesis}. SOS programming has been used extensively in stability analysis and control including stability analysis of nonlinear systems~(\cite{packard_ROA}), robust stability analysis of switched and hybrid systems~\cite{prajna_hybrid}, and stability analysis of time-delay systems~(\cite{papa2009analysis}). \\


The downside to the use of SOS (with Positivstellensatz multipliers) for stability analysis of nonlinear systems with many states is computational complexity. Specifically, this approach requires us to set up and solve large SDPs. As an example, applying SOS method to find a degree 8 Lyapunov function for a nonlinear system with 10 states requires at least 900 GB of memory and more than 116 days as computation time on a single-core 2.5 GHz processor. 
Although Polya's algorithm implies similar complexity to SOS, as we showed in Section~\ref{sec:SDPelements_TAC}, the SDPs associated with Polya's algorithm possess a block-diagonal structure. This allowed us to develop parallel algorithms (see Algorithms~\ref{alg:setup}, 6, and 7) for robust stability analysis of linear systems. 
Unfortunately, Polya's theorem cannot be used to represent polynomials which have zeros in the interior of the unit simplex (see~\cite{polya_corner} for an elementary proof of this). From the same reasoning as in~\cite{polya_corner} it follows that our multi-simplex version of Polya's theorem (See theorem~\ref{thm:polya_multi-simplex2}) cannot be used to represent polynomials which have zeros in the interior of a multi-simplex/hypercube. Our proposed solution to this problem is to reformulate the nonlinear stability problem using only strictly positive forms. Specifically, we consider Lyapunov functions of the form $V(x)=x^T P(x) x$, where $P$ is a strictly positive matrix-valued polynomial on the hypercube. This way, we can use our multi-simplex version of Polya's theorem to search for a polynomial $P(x)$ such that $P(x) > 0$ for all $x \in \Phi \setminus \{0\}$ and $ \langle \nabla x^T P(x) x, f(x) \rangle < 0 $ for all $x \in \Phi$ - hence proving asymptotic local stability of $\dot{x}(t) = f(x(t))$ for some $f \in \mathbb{R}[x]$.

Although Polya's algorithm has been generalized to positivity over simplices and hypercubes; as yet no further generalization to arbitrary convex polytopes exists. In order to perform analysis on more complicated geometries such as arbitrary convex polytopes, in this chapter, we look into Handelman's theorem (see Theorem~\ref{thm:Handelman}). Some preliminary work on the use of Handelman's theorem and interval evaluation for Lyapunov functions on the hypercube has been suggested in~\cite{sankaranarayanan2013lyapunov} and has also been applied to robust stability of positive linear systems in~\cite{briat}. One difficulty in using Handelman's theorem in stability analysis is that then theorem cannot be readily used to represent polynomials which have zeros in the interior of a given polytope. To see this, suppose a polynomial $g$ ($g$ is not identically zero) is zero at $x=a$, where $a$ is in the interior of a polytope 
\[
\Gamma^K:= \{ x \in \mathbb{R}^n: w_i^T x + u_i \geq 0, i=1, \cdots, K \}.
\]
Suppose there exist $b_\alpha \geq 0, \alpha \in \mathbb{N}^K$ such that for some $d \in \mathbb{N}$,
\[
g(x) = \sum _{\substack{\alpha \in \mathbb{N}^K : \Vert \alpha_i \Vert_1 \leq d}} b_\alpha (w_1^T x+u_1)^{\alpha_1} \cdots (w_K^T x+u_K)^{\alpha_K}.
\]
Then, 
\[
g(a) = \sum _{\substack{\alpha \in \mathbb{N}^K : \Vert \alpha_i \Vert_1 \leq d}} b_\alpha (w_1^T a+u_1)^{\alpha_1} \cdots (w_K^T a+u_K)^{\alpha_K} = 0.
\]
 From the assumption $a \in \text{int}(\Gamma^K)$ it follows that $w_i^Ta+u_i > 0$ for $i=1, \cdots,K$. Hence $b_\alpha < 0$ for at least one $\alpha \in \{ \alpha \in \mathbb{N}^K : \Vert \alpha \Vert_1 \leq d \} $. This contradicts with the assumption that all $b_\alpha \geq 0$. Based on the this reasoning, one cannot readily use Handelman's theorem to search for a polynomial $V$ such that $V(x) > 0$ for all $x \in \Gamma^K \setminus \{ 0 \}$ and $V(0) = 0$.

\subsection{Our Contributions}

In this chapter, we consider a new approach to the use of Handelman's theorem for computing regions of attraction of stable equilibria by constructing piecewise-polynomial Lyapunov functions over arbitrary convex polytopes. Specifically, we decompose a given convex polytope into a set of convex sub-polytopes that share a common vertex at the origin. Then, on each sub-polytope, we use Handelman's conditions to define linear programming constraints. Additional constraints are then proposed which ensure continuity of the Lyapunov function over the entire polytope. We then show the resulting algorithm has polynomial complexity in the number of states and compare this complexity with algorithms based on SOS and Polya's theorem. Finally, we evaluate the accuracy of our algorithm by numerically approximating the domain of attraction of two nonlinear dynamical systems.

\section{Definitions and Notation}
\label{Sec:Notation}

In this section, we present/review notations and definitions of convex polytopes, facets of polytopes, decompositions and Handelman bases. \vspace{0.05in}

\begin{mydef} (Convex Polytope)
\label{df:polytope1}
Given the set of vertices $P := \{ p_i \in \mathbb{R}^n, i=1, \cdots, K \}$, define the \textbf{convex polytope} $\Gamma_P$ as 
\[
\Gamma_P := \{ x \in \mathbb{R}^n: x = \sum_{i=1}^K \mu_i p_i : \mu_i \in [0,1]
\text{ and } \sum_{i=1}^K \mu_i = 1 \}.
\]
\end{mydef}
\noindent Every convex polytope can be represented as
\[
\Gamma^K := \{ x \in \mathbb{R}^n: w_i^T x+u_i \geq 0, i=1, \cdots, K\},
\]
for some $w_i \in \mathbb{R}^n, u_i \in \mathbb{R}, i=1, \cdots, K$. Throughout the chapter, every polytope that we use contains the origin. Moreover, for brevity, we will drop the superscript $K$ in $\Gamma^K$. \vspace{0.05in}

\begin{mydef}
Given a bounded polytope of the form $\Gamma := \{ x \in \mathbb{R}^n: w_i^T x + u_i \geq 0, i=1, \cdots, K \}$, we call
\begin{align*}
\zeta^i(\Gamma):= \left\lbrace x \in \mathbb{R}^n: w_i^Tx+u_i = 0 \text{ and } w_j^Tx+u_j \geq 0 \quad \text{ for } j \in \{1, \cdots,K \} \right\rbrace
\end{align*}
the \textbf{$i-$th facet} of $\Gamma$ if $\zeta^i(\Gamma) \neq \emptyset$.
\end{mydef}

\vspace{0.05in}

\begin{mydef}($D-$decomposition)
\label{df:decomposition}
Given a bounded polytope of the form $\Gamma:= \{ x \in \mathbb{R}^n: w_i^T x + u_i \geq 0, i=1, \cdots, K \}$, we call $D_{\Gamma}:=\{D_i\}_{i=1,\cdots,L}$ a $D-$\textbf{decomposition} of $\Gamma$ if
\begin{equation*}
D_i := \left\lbrace x \in \mathbb{R}^n: h^T_{i,j} x + g_{i,j} \geq 0, j=1, \cdots,m_i \right\rbrace
\label{eq:decomposition}
\end{equation*}
$\text{ for some } h_{i,j} \in \mathbb{R}^n, \, g_{i,j} \in \mathbb{R}$, such that $\cup_{i=1}^L D_i = \Gamma,\, \cap_{i=1}^L D_i = \{0\}$ and $\text{int}(D_i) \cap \text{int}(D_j) = \emptyset$.
\end{mydef} \vspace{0.05in}

\noindent In Figure~\ref{fig:subpolytopes}, we have illustrated a $D$-decomposition of a two-dimensional polytope.

\begin{figure}[t]
\centering
\includegraphics[scale=0.18]{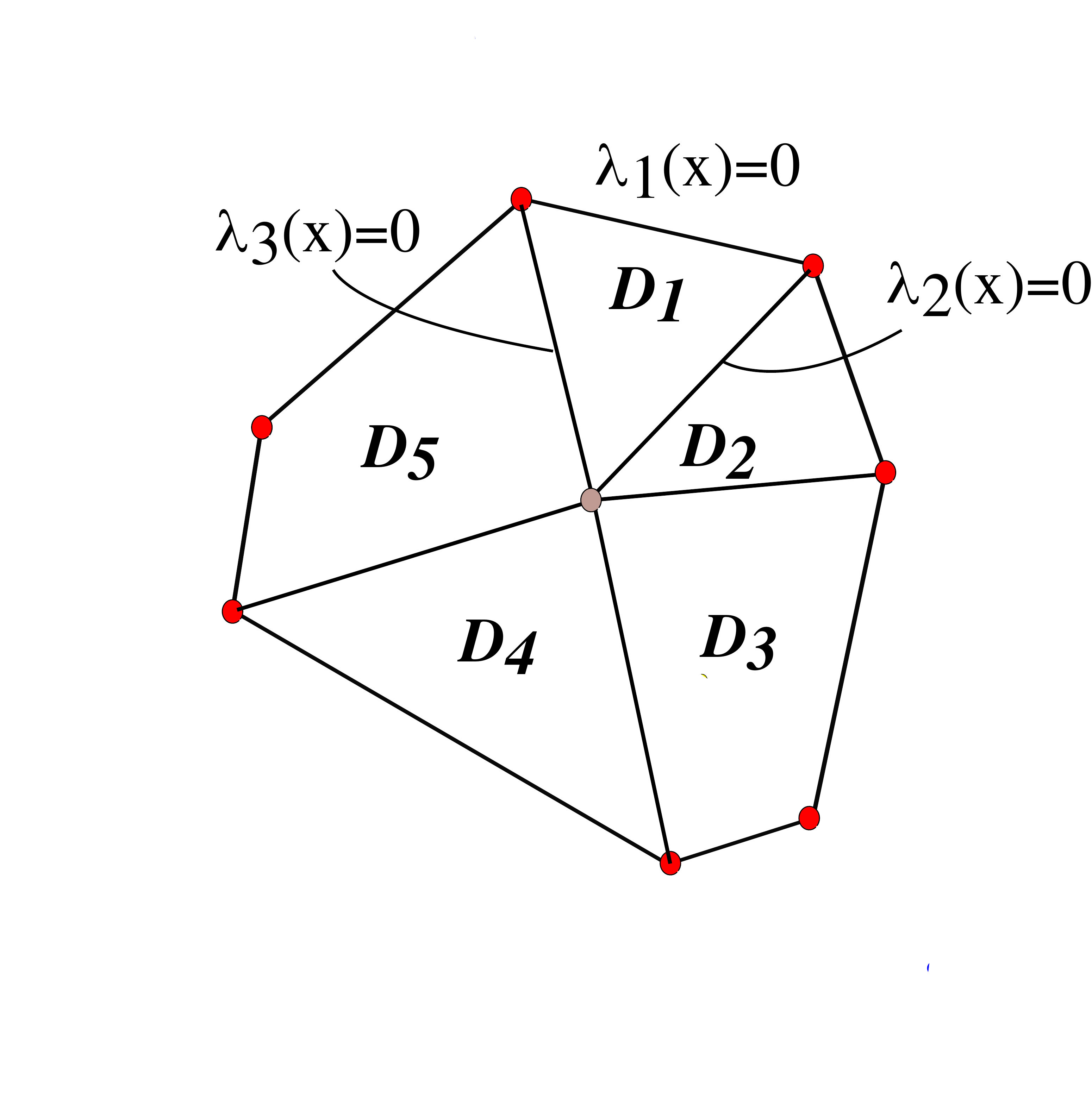} \vspace{-0.25in}
\caption{An Illustration of a D-decomposition of a 2D Polytope. $\lambda_i(x) := h_{i,j}^T x + g_{i,j}$ for $j=1, \cdots,m_i$.}
\label{fig:subpolytopes}
\end{figure}

\begin{mydef} (The Handelman basis associated with a polytope)
\label{df:Handelman_basis}
Given a polytope of the form
\[\Gamma:=\left\lbrace x\in\mathbb{R}^n: w_i^Tx+u_i \geq 0,\;i = 1,\cdots,K \right\rbrace,\]
we define the set of \textbf{Handelman bases}, indexed by
\begin{equation}\label{eq:E_d}
\alpha\in E_{d,K}:=\left\lbrace  \alpha\in\mathbb{N}^K:|\alpha|_1\leq d \right\rbrace \end{equation} 
as 
\[
\Theta_d(\Gamma):=\left\lbrace \rho_\alpha(x): \rho_\alpha(x) = \prod_{i=1}^K(w_i^Tx+u_i)^{\alpha_i},\;\alpha\in E_{d,K} \right\rbrace.
\]

\end{mydef}

\begin{mydef}(Restriction of a polynomial to a facet)\label{df:rstr_poly}
Given a polytope of the form $\Gamma:=\{x\in\mathbb{R}^n:w_i^Tx+u_i,\;i=1,\cdots, K\}$, and a polynomial $P(x)$ of the form
\[
P(x) = \sum_{\alpha\in E_{d,K}}b_{\alpha}\prod_{i=1}^K(w_i^Tx+u_i)^{\alpha_i},
\]
define the \textbf{restriction of $P(x)$ to the $k$-th facet of} $\Gamma$ as the function
\[
P|_{k}(x) := \sum_{\substack{\alpha\in E_d:\alpha_k=0}}b_{\alpha}\prod_{i=1}^K(w_i^Tx+u_i)^{\alpha_i}. \vspace{-0.07in}
\]
\end{mydef} \vspace{0.05in}

\noindent We will use the maps defined below in future sections.
\begin{mydef}
\label{df:maps}
 Given $w_i, h_{i,j} \in \mathbb{R}^n$ and $u_i, g_{i,j} \in \mathbb{R}$, let $\Gamma$ be a convex polytope as defined in Definition~\ref{df:polytope1} with $D-$decomposition $D_{\Gamma}:=\{D_i\}_{i=1,\cdots,L}$ as defined in Definition~\ref{df:decomposition}, and let $\lambda^{(k)},\, k=1, \cdots,B$ be the elements of $E_{d,n}$, as defined in \eqref{eq:E_d}, for some $d,n,\in \mathbb{N}$.  For any $\lambda^{(k)} \in E_{d,n}$, let $p_{\{\lambda^{(k)},\alpha,i\}}$ be the coefficient of $b_{i,\alpha} x^{\lambda^{(k)}}$ in 
\begin{equation}
P_i(x) := \sum_{\alpha \in E_{d,m_i}} b_{i,\alpha} \prod_{j=1}^{m_i} (h^T_{i,j}x+g_{i,j})^{\alpha_j}. 
\label{eq:V_i} \vspace{0.1in}
\end{equation}
Let $N_i$ be the cardinality of $E_{d,m_i}$, and denote by $b_i \in\mathbb{R}^{N_i}$ the vector of all coefficients $b_{i,\alpha}$. \vspace{0.1in}

\noindent Define $F_{i}: \mathbb{R}^{N_i} \times \mathbb{N} \rightarrow \mathbb{R}^{B}$ as 
\begin{equation}\label{eq:Fdef}
F_{i}(b_{i},d) := \left[
 \sum_{\alpha \in E_{d,m_i}}  p_{\{\lambda^{(1)},\alpha,i\}} b_{i,\alpha} , \; \cdots \; ,  \sum_{\alpha \in E_{d,m_i}}  p_{\{\lambda^{(
B)},\alpha,i\}} b_{i,\alpha} \right]^T
\end{equation}
for  $i=1, \cdots,L$. In other words, $F_i(b_i,d)$ is the vector of the coefficients of $P_i(x)$ after expansion. \vspace{0.1in}

\noindent Define $H_{i}: \mathbb{R}^{N_i} \times \mathbb{N} \rightarrow \mathbb{R}^{Q} $ as 
\begin{equation}\label{eq:Hdef}
H_{i}(b_i,d) := \left[
 \sum_{\alpha \in E_{d,m_i}}  p_{\{\delta^{(1)},\alpha,i\}} b_{i,\alpha } \; , \; \cdots \; ,  \sum_{\alpha \in E_{d,m_i}}  p_{\{\delta^{(
Q)},\alpha,i\}} b_{i,\alpha} \right]^T
\end{equation}
for $i = 1, \cdots,L$, where we have denoted the elements of $\{ \delta \in \mathbb{N}^n: \delta = 2e_j \text{ for } j=1, \cdots,n \}$ by $\delta^{(k)}, k=1, \cdots, Q$, where $e_j$ are the canonical basis for $\mathbb{N}^n$. In other words, $H_i(b_i,d)$ is the vector of coefficients of square terms of $P_i(x)$ after expansion. \vspace{0.1in}

\noindent Define $J_{i}: \mathbb{R}^{N_i} \times \mathbb{N}\times \{1,\cdots,m_i\} \rightarrow \mathbb{R}^{B} $ as 
\begin{equation}\label{eq:Jref}
J_{i}(b_{i},d,k) := \left[
 \sum_{\substack{\alpha \in E_{d,m_i}\\\alpha_k=0 }}  p_{\{\lambda^{(1)},\alpha,i\}} b_{i,\alpha} \; \cdots \; ,  \sum_{\substack{\alpha \in E_{d,m_i}\\\alpha_k=0} }  p_{\{\lambda^{(
B)},\alpha,i\}} b_{i,\alpha} \right]^T
\end{equation}
for $i = 1, \cdots,L$. In other words, $J_i(b_i,d,k)$ is the vector of coefficients of $P_i|_k(x)$ after expansion. \vspace{0.1in}

\noindent Given a polynomial vector field $f(x)$ of degree $d_f$, define $G_{i}: \mathbb{R}^{N_i} \times \mathbb{N} \rightarrow \mathbb{R}^{Z} $ as 
\begin{equation}\label{eq:Gdef} 
G_{i}(b_{i},d) := \left[
 \sum_{\alpha \in E_{d,m_i}}  s_{\{\eta^{(1)},\alpha,i\}} b_{i,\alpha} \; , \; \cdots \; ,  \sum_{\alpha \in E_{d,m_i}}  s_{\{\eta^{(
P)},\alpha,i\}} b_{i,\alpha} \right]^T
\end{equation}
for $i=1, \cdots,L$,  and where we have denoted the elements of $E_{d+d_f-1,n}$ by $\eta^{(k)}$, $k=1,\cdots,Z$. For any $\eta^{(k)} \in E_{d+d_f-1,n}$, we define $s_{\{ \eta^{(k)}, \alpha, i \}}$ as the coefficient of $b_{i,\alpha} x^{\eta^{(k)}}$ in $\langle \nabla P_i(x),f(x) \rangle$, where $P_i(x)$ is defined in~\eqref{eq:V_i}. In other words, $G_{i}(b_{i},d)$ is the vector of coefficients of $\langle \nabla P_i(x), f(x) \rangle$. \vspace{0.1in}

\noindent Define $R_i(b_{i},d): \mathbb{R}^{N_i} \times \mathbb{N} \rightarrow \mathbb{R}^C$ as
\begin{equation}\label{eq:Rdef}
R_i(b_{i},d):=
\left[ b_{i,\beta^{(1)}} \; , \; \cdots \; , \; b_{i,\beta^{(C)}} \right]^T,
\end{equation}
for $i = 1, \cdots,L$, where we have denoted the elements of
\[
S_{d,m_i}:=\{ \beta \in E_{d,m_i} :  \beta_j = 0 \text{ for } j \in \{ j \in \mathbb{N}: g_{i,j} = 0 \} \}
\]
 by $\beta^{(k)}$, $k=1,\cdots,C$. Consider $P_i$  in the Handelman basis $\Theta_d(\Gamma)$. Then, $R_i(b_i,d)$ is the vector of coefficients of monomials of $P_i$ which are nonzero at the origin.\\
 It can be shown that the maps $F_i,\;H_i,\;J_i,\;G_i\text{ and }R_i$ are affine in $b_i$.
\end{mydef}

\begin{mydef}(Upper Dini Derivative)
\label{df:Dini} Let $f:\mathbb{R}^n\rightarrow \mathbb{R}^n$ be a continuous map. Then, define the upper Dini derivative of a function $V:\mathbb{R}^n\rightarrow\mathbb{R}$ in the direction $f(x)$
as
\[D^+(V(x),f(x)) = \limsup_{h\rightarrow 0^+}\frac{V(x+hf(x))-V(x)}{h}.\]
\end{mydef}
It can be shown that for a continuously differentiable $V(x)$,
\[D^+(V(x),f(x)) = \langle \nabla V(x),f(x)\rangle.
\]

\section{Statement of the Stability Problem}
\label{sec:Background_Handelman}

We address the problem of local stability of nonlinear systems of the form 
\begin{equation}
\dot{x}(t) = f(x(t)),
\label{eq:system_nonlin}
\end{equation}
about the zero equilibrium, where $f:\mathbb{R}^n \rightarrow \mathbb{R}^n$. We use the following Lyapunov stability condition.

\begin{mythm}
\label{thm:lyap_nonlin}
For any $\Omega \subset \mathbb{R}^n$ with $0 \in \Omega$, suppose there exists
a continuous function $V: \mathbb{R}^n \rightarrow \mathbb{R} $ and continuous positive definite functions $W_1, W_2, W_3$,
\begin{align*}
& W_1(x) \leq V(x) \leq W_2 (x) \text{ for } x \in  \Omega  \text{ and} \\
&  D^+(V(x),f(x)) \leq -W_3(x)  \text{ for } x \in \Omega,
\end{align*}
then System~\eqref{eq:system_nonlin} is asymptotically stable on $\{ x: \{ y: V(y) \leq V(x) \} \subset \Omega \}$.
\end{mythm}
In this paper, we construct piecewise-polynomial Lyapunov functions which may not have classical derivatives. As such, we use Dini derivatives which are known to exist for piecewise-polynomial functions. \vspace{0.05in}

\textit{\textbf{Problem statement}:}\label{pb:prob1}
Given the vertices $p_i \in \mathbb{R}^n, i=1, \cdots, K$, we would like to find 
the largest positive $s$ such that there exists a polynomial $V(x)$ where $V(x)$ satisfies the conditions of Theorem~\ref{thm:lyap_nonlin} on the convex polytope
\[
\left\lbrace x \in \mathbb{R}^n : x = \sum_{i=1}^K \mu_i p_i: \mu_i \in [0,s] \text{ and } \sum_{i=1}^K \mu_i = s \right\rbrace.
\]


Given a convex polytope, the following result~(\cite{handelman_1988}) parameterizes the set of polynomials which are positive on that polytope using the positive orthant.

\begin{mythm} (Handelman's Theorem)
\label{thm:Handelman} Given $w_i \in \mathbb{R}^n, u_i \in \mathbb{R}, i=1,\cdots,K$, let $\Gamma$ be a convex polytope as defined in definition~\ref{df:polytope1}. If polynomial $P(x) > 0$ for all $x \in \Gamma$, then there exist $b_\alpha \geq 0$, $\alpha \in \mathbb{N}^K$ such that for some $d \in \mathbb{N}$, \vspace{-0.05in}
\[
P(x) := \sum_{\alpha \in E_{d,K}} b_{\alpha} \prod_{ji=1}^{K} (w^T_{i}x+u_{i})^{\alpha_i}.
\]
\end{mythm}


Given a D-decomposition $D_\Gamma := \{ D_i \}_{i=1, \cdots,L}$ of the form
\[
D_i := \left\lbrace x \in \mathbb{R}^n: h_{i,j}^T x + g_{i,j} \geq 0, j=1, \cdots, m_i \right\rbrace
\]
of some polytope $\Gamma$, we parameterize a cone of piecewise-polynomial Lyapunov functions which are positive on $\Gamma$ as
\[
V(x) = V_i(x) := \sum_{\alpha \in E_{d,m_i}} b_{i,\alpha} \prod_{j=1}^{m_i} (h^T_{i,j}x+g_{i,j})^{\alpha_j} \;\; \text{ for } x \in D_i \text{ and } i=1,\cdots,L.
\]
We will use a similar parameterization of piecewise-polynomials which are negative on $\Gamma$ in order to enforce negativity of the derivative of the Lyapunov function. We will also use linear equality constraints to enforce continuity of the Lyapunov function.

\section{Expressing the Stability Problem as a Linear Program}
\label{sec:setup}
We first present some lemmas necessary for the proof of our main result. The following lemma provides a sufficient condition for a polynomial represented in the Handelman basis to vanish at the origin ($V(0)=0$).

\begin{mylem}
\label{Zero Origin}
 Let $D_{\Gamma} := \{ D_i \}_{i=1,\cdots,L}$ be a D-decomposition of a convex polytope $\Gamma$, where
\[
D_i := \left\lbrace x \in \mathbb{R}^n: h_{i,j}^T x + g_{i,j} \geq 0, j=1, \cdots, m_i \right\rbrace.
\]
For each $i\in \{1\cdots,L\}$, let 
\[
P_i(x) := \sum_{\alpha \in E_{d,m_i}} b_{i,\alpha} \prod_{j=1}^{m_i} (h^T_{i,j}x+g_{i,j})^{\alpha_j}, 
\]
 $N_i$ be the cardinality $E_{d,m_i}$ as defined in \eqref{eq:E_d}, and let $b_i\in\mathbb{R}^{N_i}$ be the vector of the coefficients $b_{i,\alpha}$ . Consider $R_i:\mathbb{R}^{N_i}\times \mathbb{N}\rightarrow \mathbb{R}^C$ as defined in \eqref{eq:Rdef}. If $R_i(b_i,d)=\mathbf{0}$, then $P_i(x) = 0$ for all $i\in \{1\cdots,L\}$.
\end{mylem}
\begin{proof}
We can write
\[
P_i(x) = \sum_{\alpha \in E_{d,m_i}\backslash S_{d,m_i}}  b_{i,\alpha}\prod_{j=1}^{m_i} (h_{i,j}^Tx+g_{i,j})^{\alpha_i} + \sum_{\alpha \in S_{d,m_i}}   b_{i,\alpha} \prod_{j=1}^{m_i} (h_{i,j}^Tx+g_{i,j})^{\alpha_i} ,
\]
where
\[
S_{d,m_i} :=\{ \alpha \in E_{d,m_i} :  \alpha_j = 0 \text{ for } j \in \{ j \in \mathbb{N}: g_{i,j} = 0 \} \}.
\]
By the definitions of $E_{d,m_i}$ and $S_{d,m_i}$, we know that for each $\alpha \in \ E_{d,m_i}\backslash S_{d,m_i}$ for $i\in\{1,\cdots,L\}$, there exists at least one $j\in \{1,\cdots,m_i\}$ such that $g_{i,j} = 0$ and $\alpha_k>0$. Thus, at $x=0$,
\[
\sum_{\alpha \in E_{d,m_i}\backslash S_{d,m_i}}  b_{i,\alpha}\prod_{j=1}^{m_i} (h_{i,j}^Tx+g_{i,j})^{\alpha_i} = 0\;\;\;\text{for all }i\in\{1,\cdots,L\}.
\]
Recall the definition of the map $R_i$ from \eqref{eq:Rdef}. Since $R_i(b_i,d)=\mathbf{0}$ for each $i\in\{1,\cdots,L\}$, it follows from that $b_{i,\alpha} = 0$ for each  $\alpha \in S_{d,m_i}$ and $i\in\{1,\cdots,L\}$. Thus,
\[
\sum_{\alpha \in S_{d,m_i}}  b_{i,\alpha} \prod_{j=1}^{m_i} (h_{i,j}^Tx+g_{i,j})^{\alpha_i} = 0\;\;\;\text{for all }i\in\{1,\cdots,L\}.
\]
Thus, $P_i(0) = 0$ for all $i\in\{1,\cdots,L\}$.
\end{proof}

This Lemma provides a condition which ensures that a piecewise-polynomial function on a D-decomposition is continuous.
\begin{mylem}
\label{Continuity Lemma}
Let $D_{\Gamma} := \{ D_i \}_{i=1,\cdots,L}$ be a D-decomposition of a polytope $\Gamma$, where
\[
D_i := \{ x \in \mathbb{R}^n: h_{i,j}^T x + g_{i,j} \geq 0, j=1, \cdots, m_i \}.
\]
For each $i\in \{1\cdots,L\}$, let
\[P_i(x) := \sum_{\alpha \in E_{d,m_i}} b_{i,\alpha} \prod_{j=1}^{m_i} (h^T_{i,j}x+g_{i,j})^{\alpha_j},\]
 $N_i$ be the cardinality of $E_{d,m_i}$ as defined in \eqref{eq:E_d}, and let $b_i\in \mathbb{R}^{N_i}$ be the vector of the coefficients $b_{i,\alpha}$. Given $i,j \in \{1, \cdots,L \},i \neq j$, let 
\begin{align}
&\Lambda_{i,j} (D_\Gamma):= \nonumber \\ 
&\left\{ k,l \in \mathbb{N} : k \in \{ 1, \cdots, m_i \}, l \in \{ 1, \cdots, m_j \}: \zeta^k(D_i) \neq \emptyset\text{ and } \zeta^k(D_i) = \zeta^l (D_j) \right\}.
\label{eq:lambdaij}
\end{align}
Consider $J_{i}:\mathbb{R}^{N_i}\times \mathbb{N}\times\{1\cdots,m_i\}\rightarrow \mathbb{R}^B$ as defined in \eqref{eq:Jref}. If
\[
J_{i}(b_i,d,k) = J_{j}(b_j,d,l)
\]
for all $i,j \in \{1, \cdots,L \},\;i\neq j$ and $k,l\in\Lambda_{i,j}(D_\Gamma)$, then the piecewise-polynomial function
\[
P(x)=P_i(x),\quad \text{ for } x\in D_i,\;i=1,\cdots,L   
\]
is continuous for all $x\in \Gamma$.
\end{mylem}

\begin{proof}
 From \eqref{eq:Jref},  $J_{i}(b_i,d,k)$ is the vector of coefficients of $P_i |_k(x)$ after expansion. Therefore, if 
\[
 J_{i}(b_i,d,k) = J_{j}(b_j,d,l) \; \text{ for all } \; i,j\in \{1,\cdots,L\},\;i\neq j
 \]
and $(k,l)\in\Lambda_{i,j}(D_\Gamma)$, then
\begin{equation}
P_i |_k(x) = P_j |_l(x)\text{ for all }i,j\in \{1,\cdots,L\},\;i\neq j\text{ and }(k,l)\in\Lambda_{i,j}(D_\Gamma).
\label{eq:PiPj}
\end{equation}
On the other hand, from definition \ref{df:rstr_poly}, it follows that for any $i\in\{1,\cdots ,L\}$ and $k\in\{1,\cdots,m_i\}$,  \begin{align}\label{eq:restr_equal}
P_i|_k(x) = P_i(x)\text{ for all }x\in\zeta^k(D_i). 
\end{align}
Furthermore, from the definition of $\Lambda_{i,j}(D_\Gamma)$, we know that
 \begin{equation}\label{eq:interface}
\zeta^k(D_i)= \zeta^l(D_j)\subset D_i\cap D_j\end{equation} for any $i,j\in\{1\cdots,L\},\;i\neq j$ and any $(k,l)\in\Lambda_{i,j}(D_\Gamma)$.
 Thus, from \eqref{eq:PiPj}, \eqref{eq:restr_equal} and \eqref{eq:interface}, it follows that for any $i,j\in\{1,\cdots,L\},\;i\neq j$, we have $P_i(x)=P_j(x)$ for all $x\in D_i\cap D_j$ . Since for each $i\in\{1,\cdots,L\}$, $P_i(x)$ is continuous on $D_i$ and for any $i,j\in\{1\cdots,L\},\;i\neq j$, $P_i(x) = P_j(x)$ for all $x\in D_i\cap D_j$, we conclude that the piecewise polynomial function
 \[P(x) = P_i(x)\quad x\in D_i,i=1,\cdots,L\]
 is continuous for all $x\in \Gamma$.
\end{proof}

\begin{mythm}(Main Result)
\label{thm:main_theorem}
Let $d_f$ be the degree of the polynomial vector field $f(x)$ of System~\eqref{eq:system_nonlin}. Given $w_i,\;h_{i,j}\in\mathbb{R}^n$ and $u_i,\;g_{i,j}\in\mathbb{R}$, define the polytope
\[
\Gamma := \{ x \in \mathbb{R}^n: w_i^T x+u_i \geq 0, i=1, \cdots, K\},
\]
with D-decomposition $D_{\Gamma} := \{ D_i \}_{i=1,\cdots,L}$, where
\[
D_i := \{ x \in \mathbb{R}^n: h_{i,j}^T x + g_{i,j} \geq 0, j=1, \cdots, m_i \}.
\]
 Let $N_i$ be the cardinality of $E_{d,m_i}$, as defined in \eqref{eq:E_d} and let $M_i$ be the cardinality of $E_{d+d_f-1,m_i}$. Consider the maps $R_i$, $H_i$, $F_i$, $G_i$, and $J_i$ as defined in definition \ref{df:maps}, and $\Lambda_{i,j}(D_\Gamma)$ as defined in \eqref{eq:lambdaij} for $i,j\in\{1,\cdots,L\}$. If there exists $d\in\mathbb{N}$ such that $\max \gamma$ in the linear program (LP),
\begin{align}
&\max_{\gamma \in \mathbb{R} , b_i \in \mathbb{R}^{N_i}, c_i \in \mathbb{R}^{M_i}} \;\;\; \gamma  \nonumber  \\
\nonumber&  \text{subject to } \;\;
\\& b_{i} \geq \mathbf{0}   &&\hspace{-.3cm} \text{ for } i=1, \cdots, L \nonumber \\
&  c_{i} \leq \mathbf{0}           && \hspace{-.3cm}\text{ for } i=1, \cdots, L \nonumber  \\
&  R_i(b_{i},d) = \mathbf{0}       && \hspace{-.3cm}\text{ for } i=1, \cdots, L        \nonumber \\
&  H_i(b_i,d) \geq \mathbf{1}      && \hspace{-.3cm} \text{ for } i=1, \cdots, L \nonumber \\
&  H_i(c_i,d+d_f-1) \leq -\gamma \cdot \mathbf{1}  && \hspace{-.3cm}\text{ for } i=1, \cdots, L \nonumber \\
& G_i(b_i,d) = F_i(c_i,d+d_f-1)   &&\hspace{-.3cm}\text{ for } i=1, \cdots, L \nonumber \\
&  J_{i}(b_i,d,k) = J_{j}(b_j,d,l)  && \hspace{-.3cm} \text{ for } i,j=1, \cdots, L \text{ and } k,l \in \Lambda_{i,j}(D_\Gamma)
\label{eq:LP}
\end{align}
is positive, then the origin is an asymptotically stable equilibrium for System~\ref{eq:system_nonlin}. Furthermore,
\begin{equation*}
V(x) = V_i(x) = \sum_{\alpha\in E_{d,m_i}} b_{i,\alpha} \prod_{j=1}^{m_i} (h^T_{i,j}x+g_{i,j})^{\alpha_j} \text{ for } x \in D_i, \, i=1, \cdots,L 
\end{equation*}
with $b_{i,\alpha}$ as the elements of $b_i$, is a piecewise polynomial Lyapunov function proving stability of System~\eqref{eq:system_nonlin}.
\end{mythm}

\begin{proof}
 Let us choose 
\begin{equation*}
V(x) = V_i(x) = \hspace{-0.05in} \sum_{\alpha\in E_{d,m_i}} \hspace{-0.05in} b_{i,\alpha} \prod_{j=1}^{m_i} (h^T_{i,j}x+g_{i,j})^{\alpha_j} \text{ for } x \in D_i, \, i=1, \cdots,L
\end{equation*}
In order to show that $V(x)$ is a Lyapunov function for system~\ref{eq:system_nonlin}, we need to prove the following:
\begin{enumerate}
\item  $V_i(x)\geq x^Tx$ for all $x\in D_i,\;i=1,\cdots,L$,
\item  $D^+(V_i(x),f(x))\leq -\gamma \,x^Tx$ for all $x\in D_i,\;i=1,\cdots,L$ and for some $\gamma>0$,
\item $V(0) = 0$,
\item $V(x)$ is continuous on $\Gamma$.
\end{enumerate}
Then, by Theorem~\ref{thm:lyap_nonlin}, it follows that System~\eqref{eq:system_nonlin} is asymptotically stable at the origin. Now, let us prove items (1)-(4).
For some $d\in\mathbb{N}$, suppose $\gamma>0,\;b_i\text{ and }c_i\text{ for }i=1,\cdots,L$ is a solution to linear program \eqref{eq:LP}. \vspace{0.1in}

\noindent\textbf{Item 1.} First, we show that $V_i(x)\geq x^Tx$ for all $x\in D_i,\;i=1,\cdots,L$. From the definition of the D-decomposition in the theorem statement, $h_{i,j}^Tx+g_{i,j}\geq 0$,  for all $x\in D_i$, $j=1,\cdots,m_i$. Furthermore, $b_i\geq \mathbf{0}$. Thus,
\begin{equation}\label{eq:Vipos}
V_i(x) := \sum_{\alpha\in E_{d,m_i}} b_{i,\alpha} \prod_{j=1}^{m_i} (h^T_{i,j}x+g_{i,j})^{\alpha_j}\geq 0
\end{equation}
 for all $x\in D_i\backslash,\;i=1,\cdots,L$. From \eqref{eq:Hdef}, $H_i(b_i,d)\geq \mathbf{1}$ for each $i=1,\cdots,L$ implies that all the coefficients of the expansion of $x^Tx$ in $V_i(x)$ are greater than $1$ for $i=1,\cdots,L$. This, together with \eqref{eq:Vipos}, prove that $V_i(x)\geq x^Tx$ for all $x\in D_i,\;i=1,\cdots,L$. \vspace{0.1in}

\noindent\textbf{Item 2.} Next, we show that $D^+(V_i(x),f(x))\leq -\gamma x^Tx $ for all $x\in D_i,\;i=1,\cdots,L$. For $i=1,\cdots,L$, let us refer the elements of $c_i$ as $c_{i,\beta}$, where $\beta\in E_{d+d_f-1,m_i}$ . From \eqref{eq:LP}, $c_{i}\leq \mathbf{0}$ for $i=1,\cdots,L$. Furthermore, since  $h_{i,j}^Tx+g_{i,j}\geq 0$ for all $x\in D_i$, it follows that 
\begin{equation}\label{eq:Zineg}
Z_i(x) = \sum_{\beta \in E_{d+d_f-1}} c_{\beta,i} \prod_{j=1}^{m_i}  (h_{i,j}^Tx+g_{i,j})^{\beta_j}\leq 0 
\end{equation}
 for all $x\in D_i,\;i=1,\cdots,L$. From \eqref{eq:Hdef}, $H_i(c_i,d+d_f-1)\leq -\gamma\cdot\mathbf{1}$ for $i=1,\cdots,L$ implies that all the coefficients of the expansion of $x^Tx$ in $Z_i(x)$ are less than $-\gamma$ for $i=1,\cdots,L$. This, together with \eqref{eq:Zineg}, prove that $Z_i(x)\leq -\gamma x^Tx$ for all $x\in D_i$, for $i=1,\cdots,L$. Lastly, by the definitions of the maps $G_i$ and $F_i$ in \eqref{eq:Gdef} and \eqref{eq:Fdef}, if $G_i(b_i,d) = F_i(c_i,d+d_f-1)$, then 
 \[
 \langle \nabla V_i(x),f(x)\rangle = Z_i(x)\leq -\gamma x^Tx \; \text{ for all } x\in D_i \; \text{ and } \; i\in\{1\cdots,L\}.
 \]
 Since $D^+(V_i(x),f(x)) = \langle\nabla V_i(x),f(x)\rangle \text{ for all } x\in D_i$, it follows that 
\[
D^+(V_i(x),f(x))\leq -\gamma x^Tx  \text{ for all } x\in D_i,\;i\in\{1\cdots,L\}.  \vspace{0.1in}
\]

\noindent\textbf{Item 3.} Now, we show that $V(0) = 0$. By Lemma~\ref{Zero Origin}, $R_i(b_{i},d) = \mathbf{0}$ implies $V_i(0)=0$ for each $i\in\{1,\cdots,L\}$. \vspace{0.1in}

\noindent\textbf{Item 4.} Finally, we show that $V(x)$ is continuous for $x\in \Gamma$. By Lemma~\ref{Continuity Lemma}, $J_{i}(b_i,d,k) = J_{j}(b_j,d,l)$ for all $i,j\in\{1,\cdots,L\}$, $k,l\in\Lambda_{i,j}(D_\Gamma)$ implies that $V(x)$ is continuous for all $x\in\Gamma$.
\end{proof}
Using Theorem~\ref{thm:main_theorem}, we define Algorithm~\ref{alg:Handelman_LP} to search for piecewise-polynomial Lyapunov functions to verify local stability of system \eqref{eq:system_nonlin} on convex polytopes. We have provided a Matlab implementation for Algorithm~\ref{alg:Handelman_LP} at: \url{www.sites.google.com/a/asu.edu/kamyar/Software}.

\begin{algorithm}
\textbf{\textit{Inputs:}}\\
\hspace{0.1in} Vertices of the polytope: $p_i$ for $i=1, \cdots,K$.\\
\hspace{0.1in} $h_{i,j}$ and $g_{i,j}$ for $i=1,\cdots,K$ and $j=1,\cdots,m_i$.\\
\hspace{0.1in} Coefficients and degree of the polynomial vector field in~\eqref{eq:system_nonlin}.\\
\hspace{0.1in} Maximum degree of the Lyapunov function: $d_{max}$ \vspace{0.1in}

\While{ $d < d_{\text{max}}$}{	
  	\eIf{the LP defined in \eqref{eq:LP} is feasible}{		
  		Break the while loop. \\		
 	}{
 	 	Set $d=d+1$\.\	
 	}
} \vspace{0.1in}

\textbf{\textit{Outputs:}}\\
\hspace{0.1in} If the LP in \eqref{eq:LP} is feasible, then the output is the coefficients $b_{i,\alpha}$
of the\\
\hspace{0.1in} Lyapunov function
\[
V(x)=V_i(x) = \sum_{\alpha\in E_{d,m_i}} b_{i,\alpha} \prod_{j=1}^{m_i} (h^T_{i,j}x+g_{i,j})^{\alpha_j} \text{ for } x \in D_i,\;i=1, \cdots,L
\]
\caption{Search for piecewise polynomial Lyapunov functions using Handelman's theorem}
\label{alg:Handelman_LP}
\end{algorithm}

\section{Computational Complexity Analysis}
\label{sec:complexity_handelman}
In this section, we analyze and compare the complexity of the LP in \eqref{eq:LP} with the complexity of the SDPs associated with Polya's algorithm in~\cite{kamyar_CDC2013} and an SOS approach using Positivstellensatz multipliers. For simplicity, we consider Lyapunov functions defined on a hypercube centered at the origin. Note that we make frequent use of the formula 
\[
N_{vars} := \sum_{i=0}^d \dfrac{(i+K-1)!}{i!(K-1)!},
\]
which gives the number of basis functions in $\Theta_d(\Gamma)$ for a convex polytope $\Gamma$ with $K$ facets.

\subsection{Complexity of the LP Associated with Handelman's Representation}
Consider the following assumption on our $D-$decomposition.
\begin{myass}
We perform the analysis on an $n-$dimensional hypercube, centered at the origin. The hypercube is decomposed into $L=2n$ sub-polytopes such that the $i$-th sub-polytope has $m=2n-1$ facets. Figure~\ref{fig:polytope_complexity} shows the $1-$, $2-$ and $3-$dimensional decomposed hypercube.
\label{assum_decompose}
\end{myass}

\begin{figure}
\centering
\includegraphics[scale=2.4]{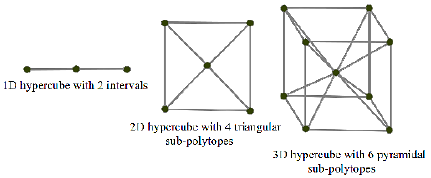}
\caption{Decomposition of the Hypercube in $1-$,$2-$ and $3-$Dimensions}
\label{fig:polytope_complexity}
\end{figure}

Let $n$ be the number of states in System~\eqref{eq:system_nonlin}. Let $d_f$ be the degree of the polynomial vector field in System~\eqref{eq:system_nonlin}. Suppose we use Algorithm 1 to search for a Lyapunov function of degree $d_V$. Then, the number of decision variables in the LP is 
\begin{align}
N^H_{vars}   = L \left(  \sum_{d=0}^{d_V} \dfrac{(d+m-1)!}{d!(m-1)!} + \sum_{d=0}^{d_V+d_f-1}  \dfrac{(d+m-1)!}{d!(m-1)!} -(d_V+1)  \right)
\label{eq:Nvar_H_general}
\end{align}
where the first term is the number of $b_{i,\alpha}$ coefficients, the second term is the number of $c_{i,\beta}$ coefficients and the third term is the dimension of $R_i(b_i,d)$ in \eqref{eq:LP}. By substituting for $L$ and $m$ in~\eqref{eq:Nvar_H_general}, from Assumption~\ref{assum_decompose} we have 
\[
N^H_{vars}=2n  \left( \sum_{d=0}^{d_V} \dfrac{(d+2n-2)!}{d!(2n-2)!} +  \sum_{d=0}^{d_V+d_f-1}  \dfrac{(d+2n-2)!}{d!(2n-2)!}-d_V-1  \right) .
\]
Then, for large number of states, i.e., large $n$,
\[
N^H_{vars} \sim 2n \left( (2n-2)^{d_V} + (2n-2)^{d_V+d_f-1} \right) \sim n^{d_V+d_f}.
\]
Meanwhile, the number of constraints in the LP is
\begin{equation}
\label{eq:Ncons_H_general}
N^H_{cons} = N^H_{vars} + L \left( \sum_{d=0}^{d_V} \dfrac{(d+n-1)!}{d!(n-1)!} + \sum_{d=0}^{d_V+d_f-1} \dfrac{(d+n-1)!}{d!(n-1)!} \right),
\end{equation}
where the first term is the total number of inequality constraints associated with the positivity of $b_{i}$ and negativity of $c_{i}$, the second term is the number of equality constraints on the coefficients of the Lyapunov function required to ensure continuity ($J_{i}(b_i,d,k) = J_j(b_j,d,l)$ in the LP \eqref{eq:LP}) and the third term is the number of equality constraints associated with negativity of the Lie derivative of the Lyapunov function ($G_i(b_i,d) = F_i(c_i,d+d_f-1)$ in the LP~\eqref{eq:LP}).
By substituting for $L$ in~\eqref{eq:Ncons_H_general}, from Assumption~\ref{assum_decompose} for large $n$ we get
\[
N^H_{cons} \sim n^{d_V+d_f} + 2n(n^{d_V}+n^{d_V+d_f-1}) \sim n^{d_V+d_f}.
\]
The complexity of an LP using interior-point algorithms is approximately $O(N_{vars}^2 N_{cons})$ (\cite{boyd2004convex}). Therefore, the computational cost of solving the LP~\eqref{eq:LP} is
\[
\sim n^{3(d_V+d_f)}.
\]

\subsection{Complexity of the SDP Associated with Polya's Algorithm}
Recall our approach in Section~\ref{sec:setup_multi} for applying Polya's algorithm to analyze stability over hypercubes. In~\cite{kamyar_CDC2013}, we used the same approach to construct Lyapunov functions for nonlinear ODEs with polynomial vector fields. In particular, this approach uses semi-definite programming to search for the coefficients of a matrix-valued polynomial $P(x)$ which defines a Lyapunov function as $V(x) = x^T P(x) x$. Using a similar complexity analysis as in~\ref{sec:compexity_setup_multi}, we determine that the number of decision variables in the associated SDP is
\[
N^P_{vars} = \dfrac{n(n+1)}{2} \sum_{d=0}^{d_V-2}\dfrac{(d+n-1)!}{d! (n-1)!}.
\]
The number of constraints in the SDP is
\[
N^P_{cons} =  \dfrac{n(n+1)}{2} \left( (d_V+e-1)^n + (d_V+d_f+e-2)^n \right),
\]
where here we have denoted Polya's exponent by $e$. Then, for large $n$, 
$
N^P_{vars} \sim n^{d_V}$ and $N^P_{cons} \sim (d_V+d_f+e-2)^n.
$
Since solving an SDP with an interior-point algorithm typically requires $ O(N_{cons}^3+N_{var}^3 N_{cons} + N_{var}^2 N_{cons}^2 )$ operations~(\cite{boyd2004convex}), the computational cost of solving the SDP associated with Polya's algorithm is estimated as
\[
\sim (d_V+d_f+e-2)^{3n}.
\]

\subsection{Complexity of the SDP Associated with SOS Algorithm}
To find a Lyapunov function for~\eqref{eq:system_nonlin} over the polytope 
\[
\Gamma= \left\lbrace x \in \mathbb{R}^n : w_i^Tx+u_i \geq 0, i\in \{1, \cdots, K \} \right\rbrace
\]
using the SOS approach with Positivstellensatz multipliers~\cite{stengle}, we search for a polynomial $V(x)$ and SOS polynomials $s_i(x)$ and $t_i(x)$ such that for any $\epsilon > 0$ 
\begin{align*}
V(x)- \epsilon x^Tx - \sum_{i=1}^K s_i(x) (w_i^Tx+u_i) \;\; \text{ is SOS}& \\
-  \langle \nabla V(x),f(x) \rangle - \epsilon x^T x -\sum_{i=1}^K t_i(x) (w_i^Tx+u_i) \;\; \text{ is SOS}&.
\end{align*}
Suppose we choose the degree of the $s_i(x)$ to be $d_V-2$ and the degree of the $t_i(x)$ to be $d_V+d_f-2$. Then, it can be shown that the total number of decision variables in the SDP associated with the SOS approach is
\begin{small}
\begin{equation}
N^S_{vars} = \dfrac{N_1(N_1+1)}{2} + K \dfrac{N_2(N_2+1)}{2} + K \dfrac{N_3(N_3+1)}{2},
\label{eq:NSvar}
\end{equation}
\end{small}
where $N_1$ is the number of monomials in a polynomial of degree $d_V/2$ , $N_2$ is the number of monomials in a polynomial of degree $(d_V-2)/2$ and $N_3$ is the number of monomials in a polynomial of degree $(d_V+d_f-2)/2$ calculated as 
\begin{small}
\[
N_1=  \sum_{d=1}^{d_V/2} \dfrac{(d+n-1)!}{(d)!(n-1)!}, \quad
N_2=  \sum_{d=0}^{(d_V-2)/2} \dfrac{(d+n-1)!}{(d)!(n-1)!} \quad \text{and} \quad
N_3= \sum_{d=0}^{(d_V+d_f-2)/2} \dfrac{(d+n-1)!}{(d)!(n-1)!}.
\]
\end{small}
The first terms in~\eqref{eq:NSvar} is the number of scalar decision variables associated with the polynomial $V(x)$. The second and third terms are the number of scalar variables in the polynomials $s_i$ and $t_i$, respectively. The number of constraints in the SDP is
\begin{equation}
\label{eq:NScons}
N^S_{cons} = N_1 + K \, N_2 + K \, N_3 + N_4,
\end{equation}
where
\[
N_4 = \sum_{d=0}^{(d_V+d_f)/2} \dfrac{(d+n-1)!}{(d)!(n-1)!}.
\]
The first term in~\eqref{eq:NScons} is the number of constraints associated with  positivity of $V(x)$, the second and third terms are the number of constraints associated with positivity of the polynomials $s_i$ and $t_i$, respectively. The fourth term is the number of constraints associated with negativity of the Lie derivative. By substituting $K=2n$ (For the case of a hypercube), for large $n$ we have
$
N^S_{vars} \sim N_3^2 \sim n^{d_V+d_f-1} \text{ and}
$
\[
\quad N^S_{cons} \sim KN_3+N_4 \sim n \, N_3+N_4 \sim n^{ 0.5(d_V+d_f)}.
\]
Finally, using an interior-point algorithm with complexity $O(N_{cons}^3+N_{var}^3 N_{cons} + N_{var}^2 N_{cons}^2 )$ to solve the SDP associated the SOS algorithm requires
$
\sim n^{3.5(d_V+d_f)-3}
$
operations. As an additional comparison, we also consider the SOS algorithm for global stability analysis, which does not use Positivstellensatz multipliers. For a large number of states, we have
$
N^S_{vars} \sim n^{ 0.5d_V}  \quad \text{and} \quad N^S_{cons} \sim n^{ 0.5(d_V+d_f)}.
$
In this case, the complexity of the SDP is \vspace{-0.05in}
\[
\sim n^{1.5(d_V+d_f)}+ n^{2d_V+d_f}. \vspace{-0.05in}
\]

\pagebreak

\subsection{Comparison of the Complexities}
We draw the following conclusions from our complexity analysis.

\begin{enumerate}
\item For large number of states, the complexity of the LP defined in~\eqref{eq:LP} and the SDP associated with SOS are both \textbf{polynomial} in the number of states, whereas the complexity of the SDP associated with Polya's algorithm grows \textbf{exponentially} in the number of states. For a large number of states and large degree of the Lyapunov polynomial, the LP has the least computational complexity.

\item The complexity of the LP defined in~\eqref{eq:LP} scales linearly with the number of sub-polytopes $L$.

\item In Figure~\ref{fig:complexity_comparison}, we show the number of decision variables and constraints for the LP and SDPs using different degrees of the Lyapunov function and different degrees of the vector field. The figure shows that in general, the SDP associated with Polya's algorithm has the least number of variables and the greatest number of constraints, whereas the SDP associated with SOS has the greatest number of variables and the least number of constraints. \vspace{0.5in}
\end{enumerate}

\begin{figure}[ht]
\hspace{-0.2in}
\includegraphics[scale=0.42]{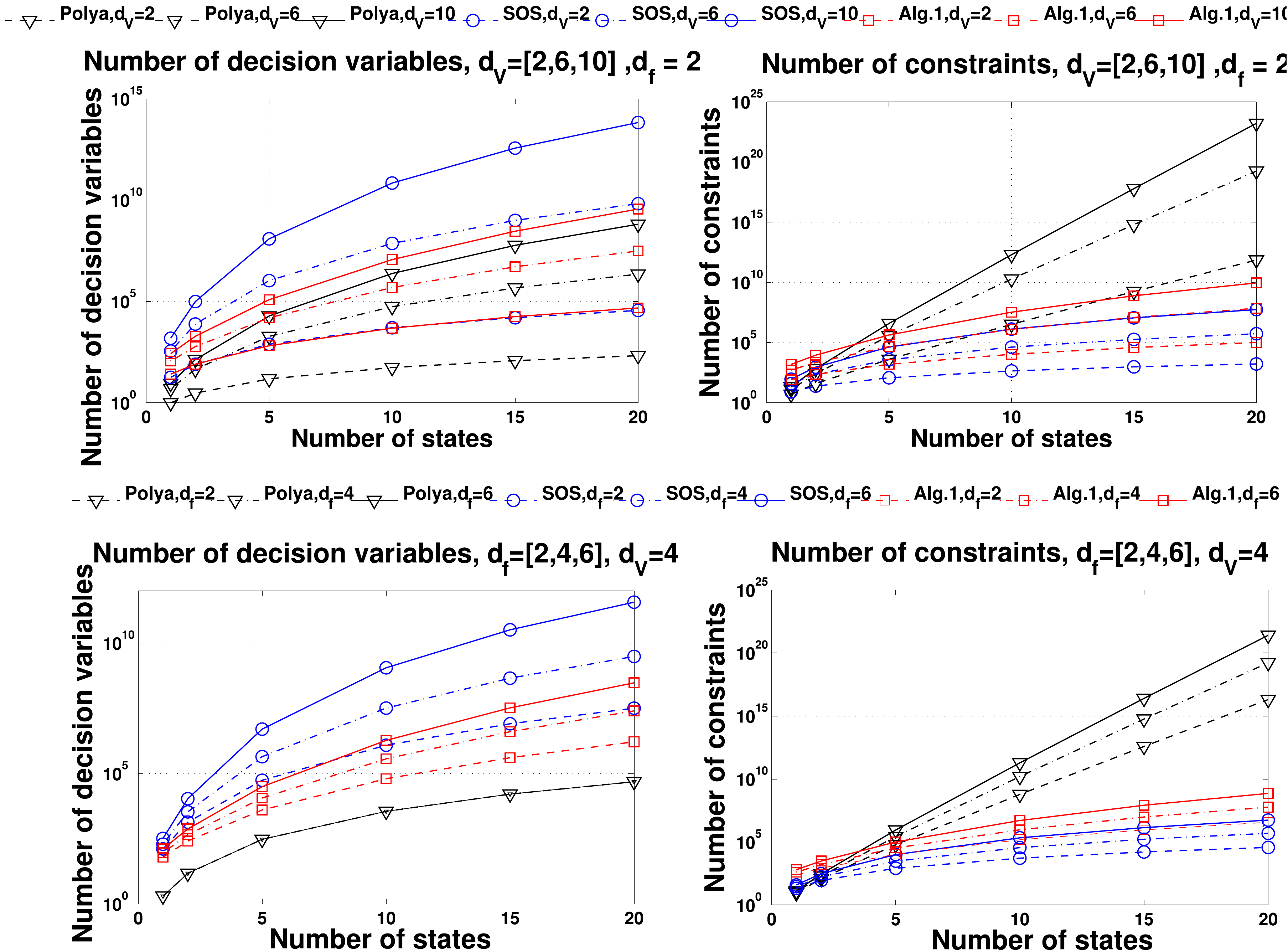}
\caption{Number of Decision Variables and Constraints of the Optimization Problems Associated with Algorithm 1, Polya's Algorithm and SOS Algorithm for Different Degrees of the Lyapunov Function and the Vector Field $f(x)$}
\label{fig:complexity_comparison} 
\end{figure}

\section{Numerical Results}

In this section, we first use our algorithm to construct a Lyapunov function for a nonlinear system. we then assess the accuracy of our algorithm in estimating the region of attraction of the equilibrium point using different types of convex polytopes.\\

\noindent \textbf{Numerical Example 1:}\\
 Consider the following nonlinear system~(\cite{chesi_2005}).
\begin{align*}
& \dot{x}_1 = x_2, \\
& \dot{x}_2 = -2x_1-x_2+x_1 x_2^2-x_1^5+x_1 x_2^4+x_2^5.
\end{align*}
Using the polytope
\begin{align}
\Gamma = \{  x_1,x_2 \in \mathbb{R}^2: \, &1.428 x_1 + x_2 - 0.625 \geq 0, -1.428x_1+x_2+0.625 \geq 0, \nonumber \\
              & 1.428 x_1 + x_2 + 0.625 \geq 0, -1.428x_1+x_2 - 0.625 \geq0 \},
\label{eq:polytope0}
\end{align}
and $D-$decomposition
\begin{align*}
&D_1:=\{x_1,x_2 \in \mathbb{R}^2: -x_1 \geq 0, x_2 \geq 0, -1.428x_1+x_2 - 0.625 \geq 0 \} \\
&D_2:=\{x_1,x_2 \in \mathbb{R}^2: x_1 \geq 0, x_2 \geq 0, 1.428 x_1 + x_2 + 0.625 \geq 0 \}\\
&D_3:=\{x_1,x_2 \in \mathbb{R}^2: x_1 \geq 0, -x_2 \geq 0, -1.428x_1+x_2+0.625 \geq 0  \}\\
&D_4:=\{x_1,x_2 \in \mathbb{R}^2: -x_1 \geq 0, -x_2 \geq 0, 1.428 x_1 + x_2 + 0.625 \geq 0  \},
\end{align*}
we set-up the LP in~\eqref{eq:LP} with $d=4$. The solution to the LP certified asymptotic stability of the origin and yielded the following piecewise polynomial Lyapunov function. Figure~\ref{fig:handelman_nonlin0} shows the largest level set of $V(x)$ inscribed in the polytope $\Gamma$. 
\begin{figure}[htbp]
\centering
\includegraphics[scale=0.48]{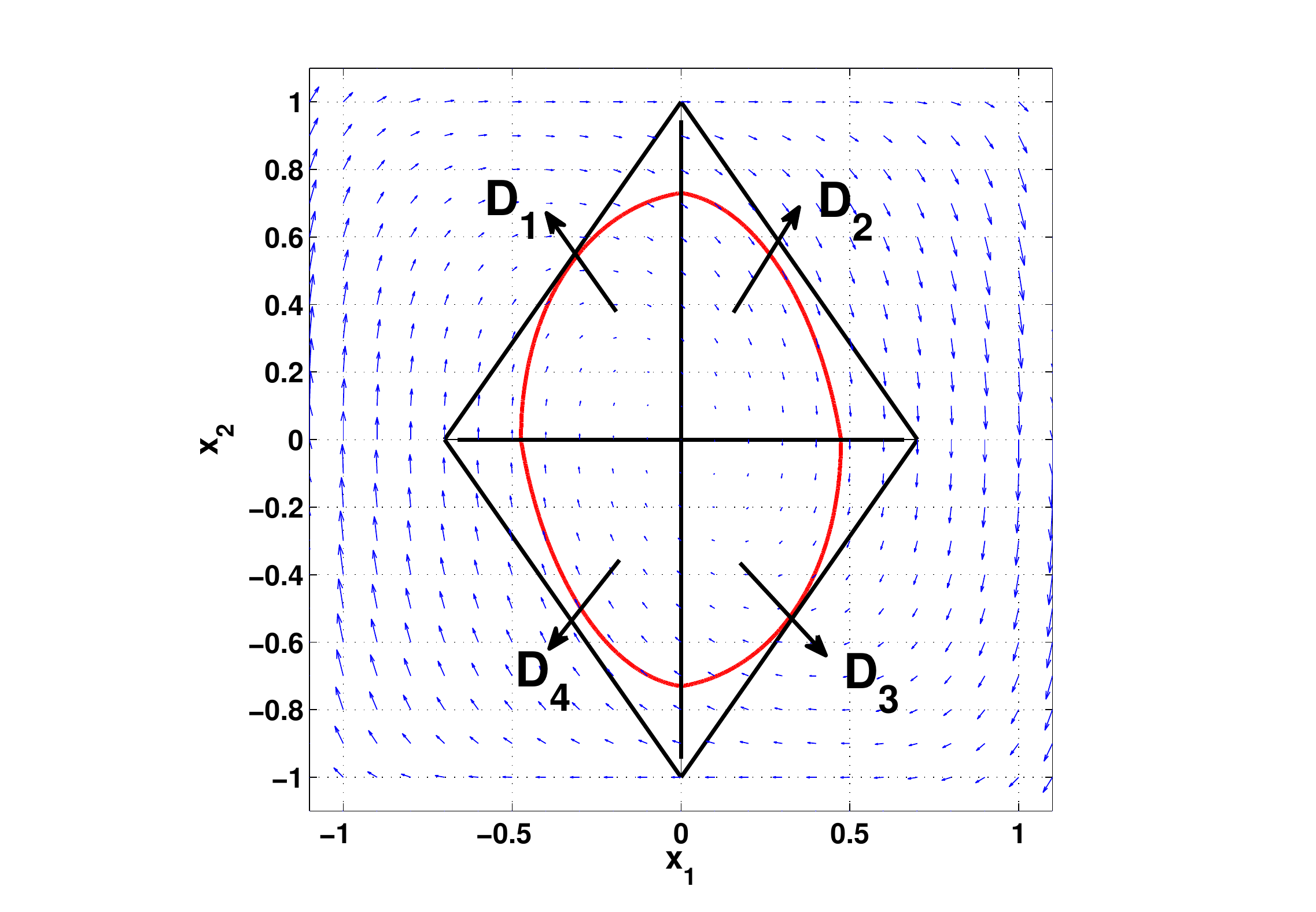}
\caption{The Largest Level-set of Lyapunov Function~\eqref{eq:handelman_PW_Lyap} Inscribed in Polytope~\eqref{eq:polytope0}}
\label{fig:handelman_nonlin0}
\end{figure}
\begin{align}
V(x) = 
\begin{cases} 
0.543 x_1^2  + 0.233 x_2^2 + 0.018 x_2^3 - 0.074 x_1 x_2^2  - 0.31 x_1^3 \\
  + 0.004 x_2^4 - 0.013 x_1 x_2^3 + 0.015 x_1^2 x_2^2 + 0.315 x_1^4  &\text{if } x \in D_1 \vspace*{0.1in} \\
0.543 x_1^2+ 0.329 x_1 x_2 + 0.233 x_2^2 + 0.018 x_2^3 +0.031 x_1 x_2^2 \\
 + 0.086 x_1^2 x_2 +  0.3 x_1^3 + 0.004 x_2^4 +  0.009 x_1 x_2^3 + 0.015 x_1^2 x_2^2\\
  + 0.008 x_1^3 x_2 + 0.315 x_1^4  &\text{if } x \in D_2  \vspace*{0.1in} \\
0.0543 x_1^2 + 0.0233 x_2^2  -0.0018 x_2^3 + 0.0074 x_1 x_2^2 + 0.03 x_1^3 \\
 + 0.004 x_2^4 -0.013 x_1 x_2^3 + 0.015 x_1^2 x_2^2 + 0.315 x_1^4 &\text{if } x\in D_3  \vspace*{0.1in} \\ 
0.543 x_1^2+ 0.329 x_1 x_2 + 0.233 x_2^2 - 0.018 x_2^3 - 0.031 x_1 x_2^2 \\
 - 0.086 x_1^2 x_2 - 0.3 x_1^3 + 0.004 x_2^4 +  0.009 x_1 x_2^3 + 0.015 x_1^2 x_2^2 \\
  +  0.008 x_1^3 x_2 + 0.315 x_1^4 &\text{if } x \in D_4  
  \label{eq:handelman_PW_Lyap}
\end{cases}\\
\end{align}

\noindent \textbf{Numerical Example 2:}

In this example, we test the accuracy of our algorithm in approximating the region of attraction of a locally-stable nonlinear system known as the reverse-time Van Der Pol oscillator. The system is defined as 
\begin{equation}
\label{eq:vanderpol}
\dot{x}_1 = -x_2, \; \dot{x}_2 =  x_1 + x_2 (x_1^2 -1).
\end{equation}
We considered the following convex polytopes:
\begin{enumerate}
\item Parallelogram $\Gamma_{P_s}$,
$P_s:=\{sp_i\}_{i=1,\cdots,4}$, where
\[
 p_1= [-1.31,0.18],\quad
p_2=[0.56,1.92],\quad
p_3=[-0.56,-1.92],\quad
p_4=[1.31,-0.18]
\]
\item Square $\Gamma_{Q_s}$, $Q_s:=\{sq_i\}_{i=1,\cdots,4}$, where
\[
q_1= [-1,1],\quad
q_2=[1,1],\quad
q_3=[1,-1],\quad
q_4=[-1,-1] 
\]
\item Diamond $\Gamma_{R_s}$, $R_s:=\{sr_i\}_{i=1,\cdots,4}$, where
\[
r_1=[-1.41,0], \quad
r_2=[0,1.41],\quad
r_3=[1.41,0],\quad
r_4=[0,-1.41]
\]
\end{enumerate}
where $s\in\mathbb{R}_+$ is a scaling factor. We decompose the parallelogram and the diamond into 4 triangles and decompose the square into 4 squares. We solved the following optimization problem for Lyapunov functions of degree $d=2,4,6,8$:
{\small
\begin{align*}
& \qquad \max_{s \in \mathbb{R}^+} \;\;\;\; s \\
& \text{subject to} \quad \text{max }\gamma\text{ in LP \eqref{eq:LP} is positive, where}  \\
& \Gamma=\Gamma_{P_s} := \{ x \in \mathbb{R}^2 : x = \sum_{i=1}^4 \mu_i sp_i: \mu_i \geq 0 \text{ and } \sum_{i=1}^K \mu_i = 1\}.
\end{align*}
}
To solve this problem, we use a bisection search on $s$ in an outer-loop and an LP solver in the inner loop. Figure~\ref{fig:quad_level_sets} illustrates the largest $\Gamma_{P_s}$, i.e.
\[
\Gamma_{P_{s^*}} := \{ x \in \mathbb{R}^n : x = \sum_{i=1}^4 \mu_is^* p_i: \mu_i\geq 0 \text{ and } \sum_{i=1}^4 \mu_i = 1 \}
\]
and the largest level-set of $V_i(x)$ inscribed in $\Gamma_{P_{s^*}}$, for different degrees of $V_i(x)$.
Similarly, we solved the same optimization problem replacing $\Gamma_{P_s}$ with the square $\Gamma_{Q_s}$ and diamond $\Gamma_{R_s}$. In all cases, increasing $d$ resulted in a larger maximum inscribed sub-level set of $V(x)$ (see Figure~\ref{fig:diam_level_sets}). We obtained the best results using the parallelogram $\Gamma_{P_s}$ which achieved the scaling factor $s^*=1.639$.
The maximum scaling factor for $\Gamma_{Q_s}$ was $s^*=1.800$ and the maximum scaling factor for $\Gamma_{R_s}$ was $s^*=1.666$. 

\begin{figure}[htbp]
\centering
\includegraphics[scale=0.36]{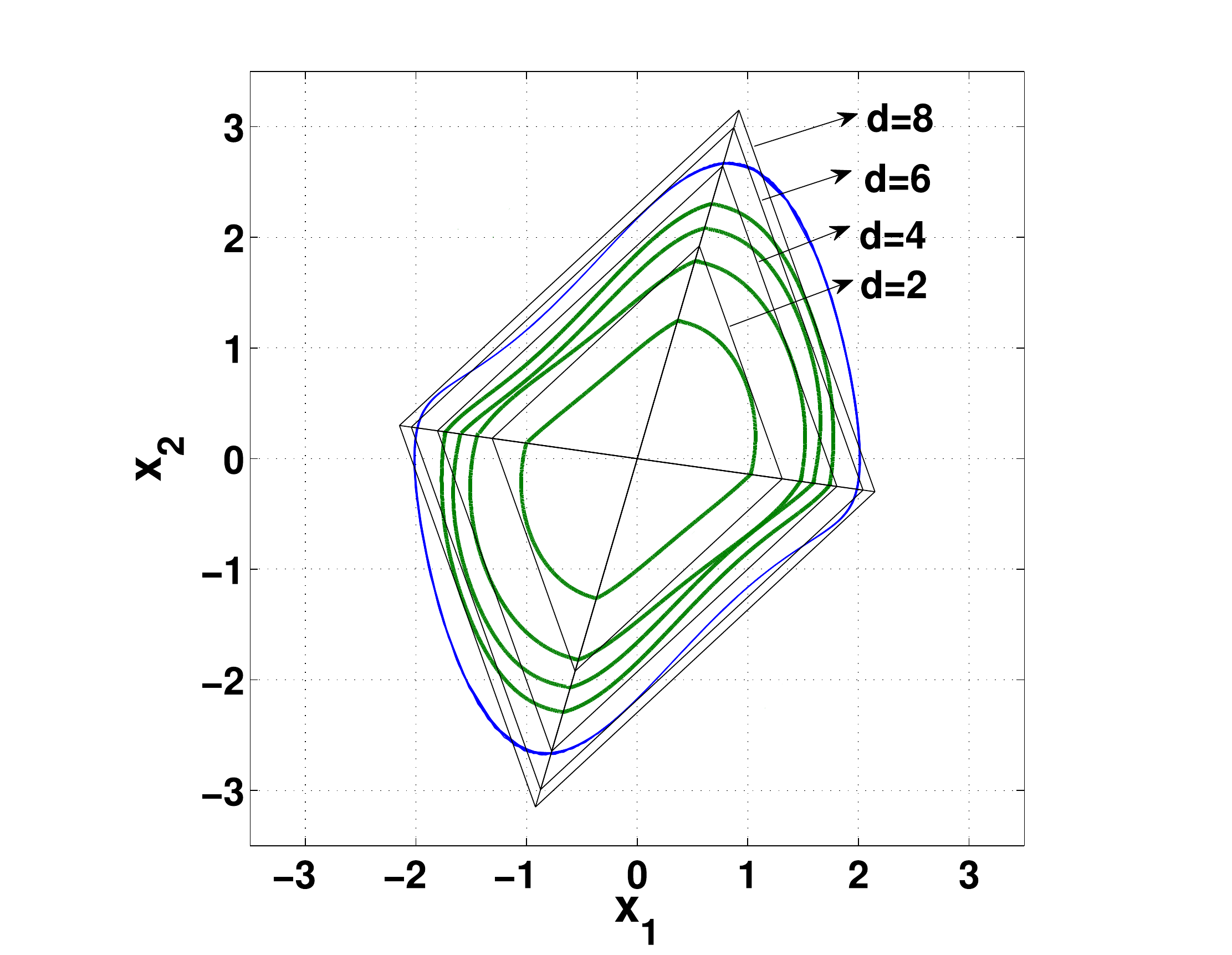} 
\caption{Largest Level-sets of Lyapunov Functions of Different Degrees and Their Associated Parallelograms}\vspace{-0.1in}
\label{fig:quad_level_sets}  
\centering
\begin{subfigure}[Square polytopes]
{\includegraphics[width=0.52\columnwidth ]{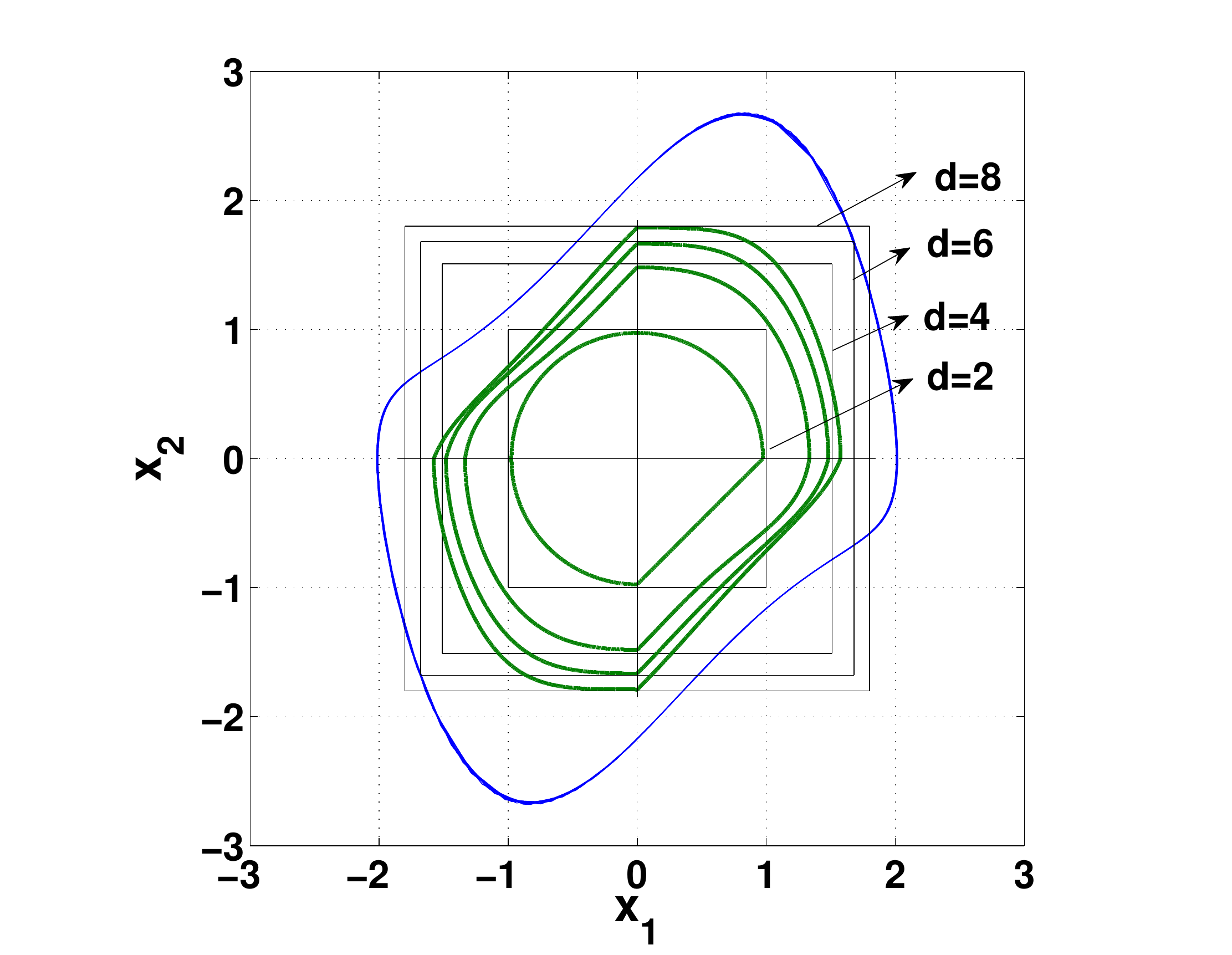}}
\end{subfigure}
~\hspace{-1.25cm}
\begin{subfigure}[{Diamond polytopes}]
{\includegraphics[width=0.52\columnwidth]{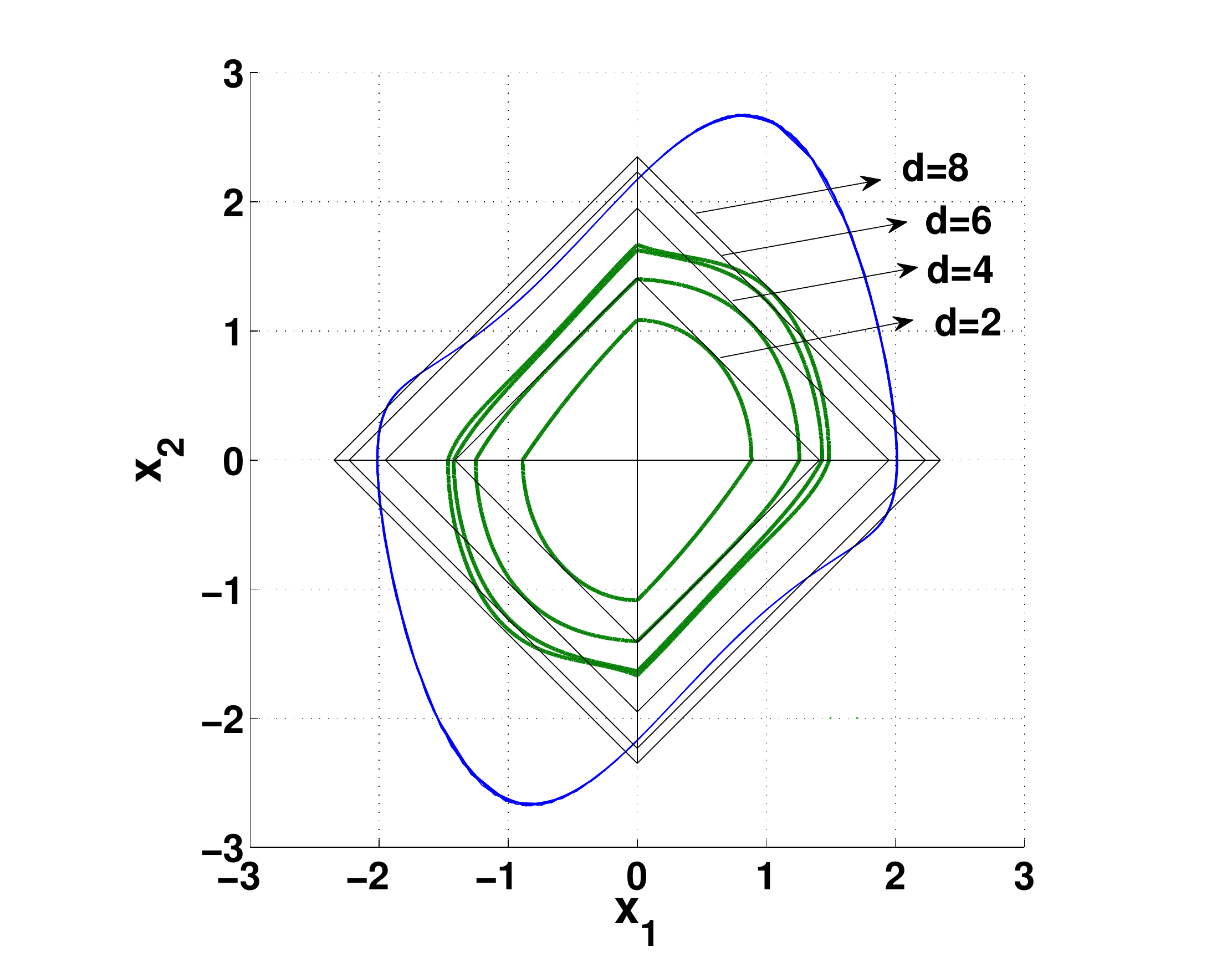}}
\end{subfigure}
 \caption{Largest Level-sets of Lyapunov Functions of Different Degrees and Their Associated Polytopes}
\label{fig:diam_level_sets}  
\end{figure}


\chapter{OPTIMIZATION OF SMART GRID OPERATION: OPTIMAL UTILITY PRICING AND DEMAND RESPONSE}
\label{chp:DP}

\section{Background and Motivation}
\label{motivation}

Reliable and efficient production and transmission of electricity are essential to the progress of modern industrial societies. Engineers have strived for years to operate power generating systems in a way to achieve the following objectives: 1) Reliability: maintaining an uninterrupted balance between the generated power and demand; 2) Minimizing the cost of generation and transmission of electricity; 3) Reducing the adverse effects of the system on the environment by increasingly the use of renewable sources such as solar energy. Unfortunately, the first two objectives are in conflict: increasing reliability (often by increasing the maximum capacity of generation) results in higher costs. Moreover, the dependence of reliability of power networks and costs on integration of renewables is not yet well-understood.

One concern of electric utilities is that rapid increase in distributed solar generation may change customers' consumption pattern in ways that current generating units cannot accommodate for these changes. One example of such a change is shown in Figure~\ref{fig:C7:summer_winter} (\cite{APS_solar_trend}). In this figure, we have compared the daily net demand profile of Arizona's customers in 2014 with its projection in 2029. Because of the misalignment between the solar generation peak (at noon) and the demand peak (at 6 PM), as the solar penetration increases, the resulting demand profile will reshape to a \textit{double-peak} curve (see Figure~\ref{fig:C7:summer_winter}).
To respond to such variability in the demand profile, utilities will be required to re-structure their generating capacity by installing generating units which possess a shorter start-up time and higher generation ramp rates. 
Moreover, as solar generation by users increases, the total energy provided by the utility will decrease - implying a reduction in revenue for utility companies which charge users based on their total energy consumption. 
This type of change in the demand can indeed already be seen in a report by the US Energy Information Administration (EIA) as a significant increase in the ratio of the annual demand peak to annual average demand (see Figuew~\ref{fig:C7:peak_trend}). Because utilities must pay to build and maintain generating capacity as determined by peak demand, the increasing use of solar will thus result in a decrease in revenue, yet no decrease in this form of cost. Ultimately, utilities might have a significant fraction of solar users with negative energy consumption (kWh) during the day and positive consumption during the evening and morning. Due to net metering, such users might pay nothing for electricity while contributing substantially to the costs of building and maintaining generating capacity. \\


\begin{figure}
\centering
\includegraphics[scale=0.4]{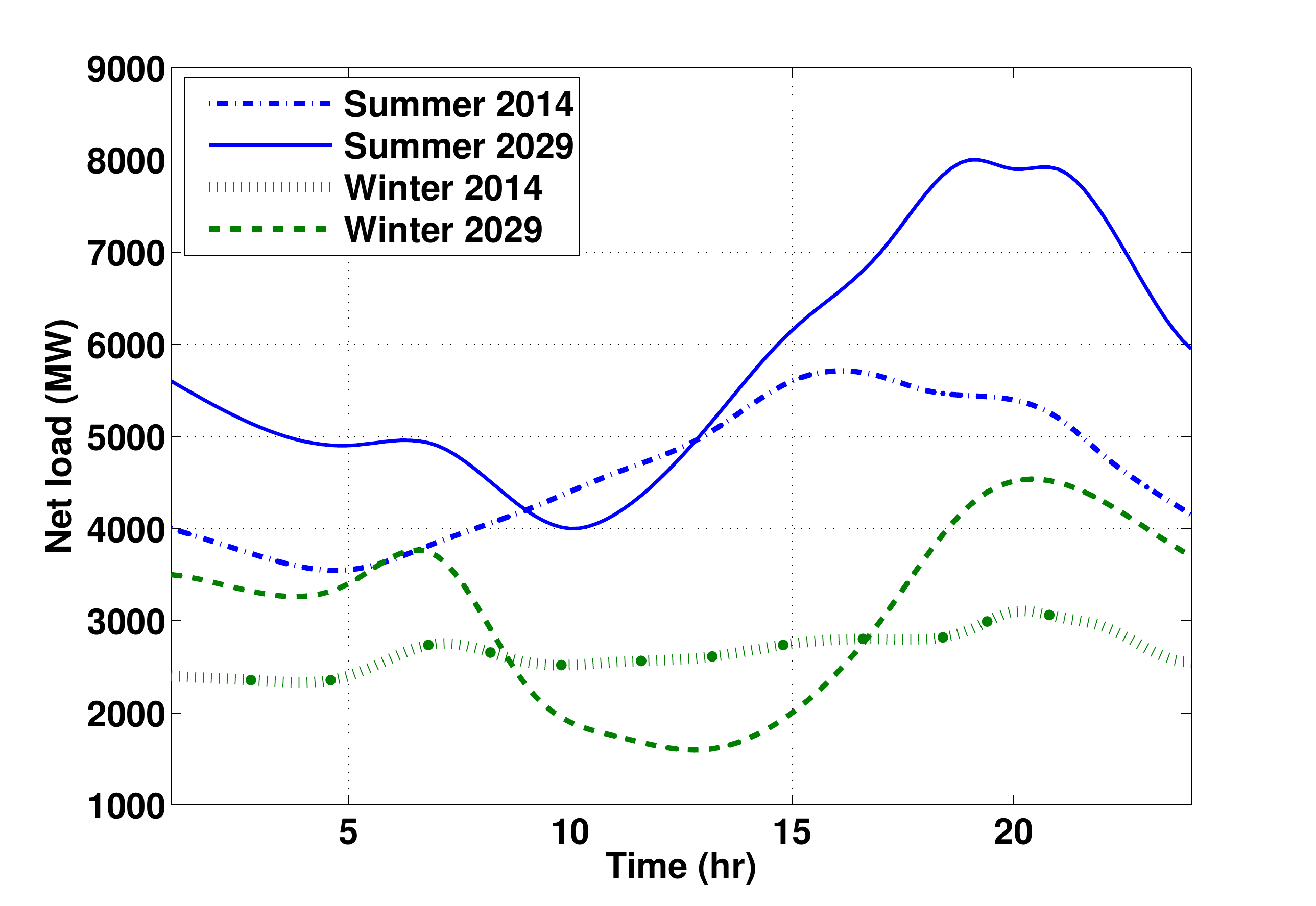}
\caption{Effect of Solar Power on Demand: Net Loads for Typical Summer and Winter Days in Arizona in 2014 and for 2029 (Projected), from~\cite{APS_solar_trend}}
\label{fig:C7:summer_winter}
\end{figure}

\begin{figure}
\centering
\includegraphics[scale=0.4]{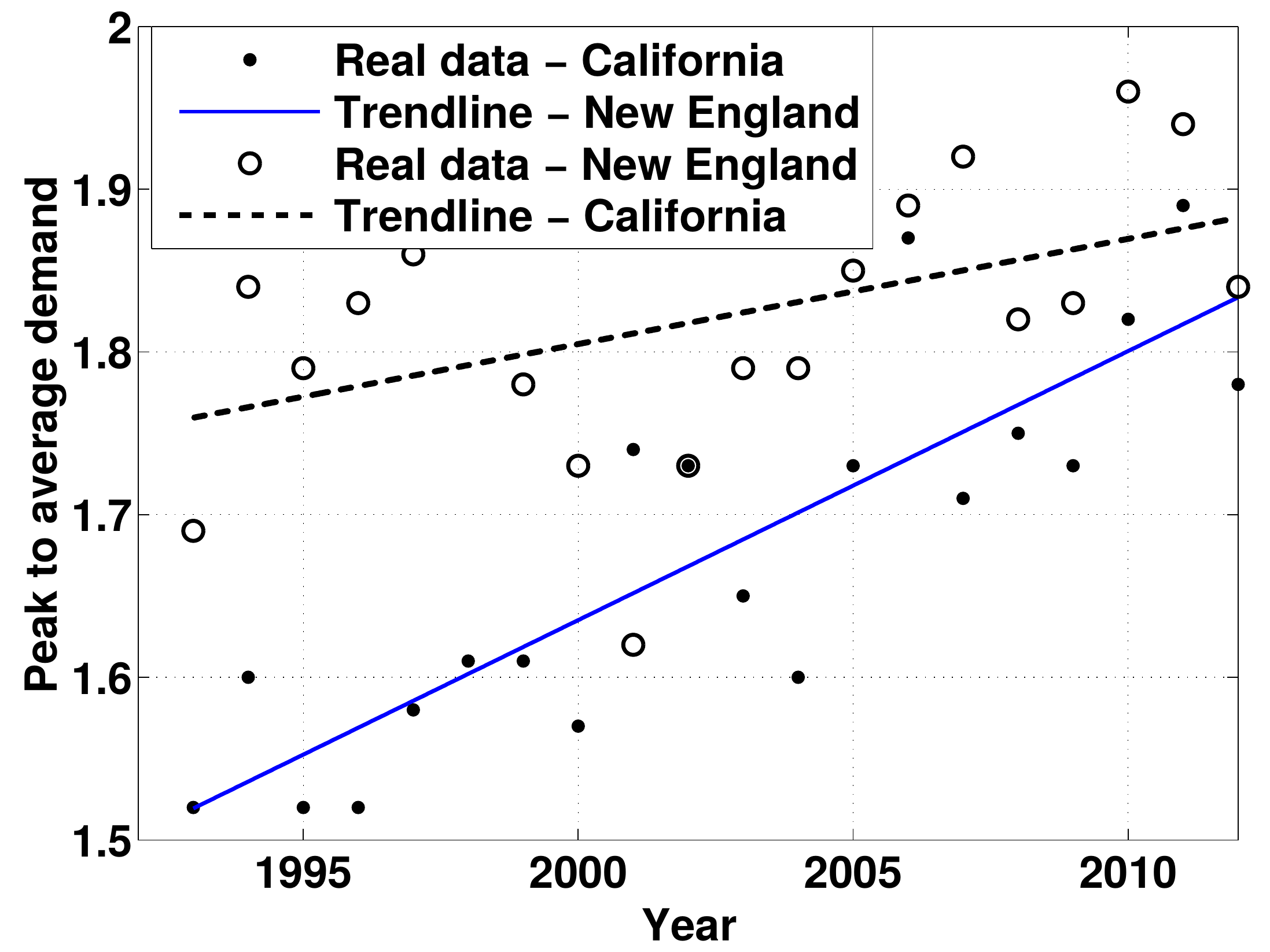}
\caption{Peak to Average Demand of Electricity and Its Trend-line in California
and New England from 1993 to 2012, Data Adopted from~\cite{peak_to_average}}
\label{fig:C7:peak_trend}
\end{figure}

Recently, there has been extensive research on how to exploit \textit{smart grid} features such as smart metering, energy storage, thermostat programming and variable/dynamic pricing in order to reduce peak demands and cost of generation, enhance monitoring and security of networks, and prevent unintended events such as cascade failures and blackouts. Smart metering enables two-way communications between consumers and utilities. It provides utilities with real-time data of consumption - hence enables them to directly control the load and/or apply prices as a function of consumption. Naturally, utilities have been studying this problem for some time and with the widespread adoption of smart-metering (95\% in Arizona), have begun to implement various pricing strategies at scale. Examples of this include on-peak, off-peak and super-peak pricing - rate plans wherein the energy price (\$/kWh) depends on the time of day. By charging more during peak hours, utilities encourage conservation or deferred consumption during hours of peak demand. Quite recently, some utilities have introduced demand charges for residential customers~(\cite{SRP_plan},\cite{APS_plan}). These charges are not based on energy consumption, but rather the maximum \textit{rate of consumption} (\$/kW) over a billing period. While such charges more accurately reflect the cost of generation for the utilities, in practice the effects of such charges on consumption are not well-understood - meaning that the magnitude of the demand charge must be set in an ad-hoc manner (typically proportional to marginal cost of adding generating capacity).


An alternative approach to reducing peaks in demand is to use energy storage. In this scenario, batteries, pumping and retained heat are used during periods of low demand to create reservoirs of energy which can then be tapped during periods of high demand - thus reducing the need to increase maximum generating capacity. Indeed, the optimal usage of energy storage in a smart-grid environment with dynamic pricing has been recently studied in, for example,~\cite{storage_benefit}. See~\cite{Topcu} for optimal distributed load scheduling in the presence of network capacity constraints. However, to date the high marginal costs of storage infrastructure relative to incentives/marginal cost of additional generating capacity have limited the widespread use of energy storage by consumers/utilities~(\cite{battery_usage}). As a cost-free alternative to direct energy storage, it has been demonstrated experimentally by~\cite{experiment1},~\cite{experiment2}, and in-silico by~\cite{simulation1} and~\cite{simulation2} that the interior structure of buildings and appliances can be exploited as a \textit{passive} thermal energy storage system to reduce the peak-load of HVAC. A typical strategy - known as \textit{pre-cooling} - is to artificially cool the interior thermal mass (e.g., walls and floor) during periods of low demand. Then, during periods of high demand, heat absorption by these cool interior structures supplements or replaces electricity which would otherwise be consumed by the HVAC. Quantitative assessments of the effect of pre-cooling on demand peak and electricity bills can be found in~\cite{Braun_2006} and~\cite{sun2013peak}. Furthermore, there is an extensive literature on thermostat programming for HVAC systems for on-peak/off-peak pricing~(\cite{lu2005global2, arguello1999nonlinear}) as well as \textit{real-time} pricing~(\cite{old_2010,henze2004evaluation,chen2001real}) using Model Predictive Control (MPC).~\cite{kintner1995optimal} consider optimal thermostat programming with passive thermal energy storage and on-peak/off-peak rates.~\cite{Braun_2006} use the concept of \textit{deep} and \textit{shallow} mass to create a simplified analogue circuit model of the thermal dynamics of the structure. By using this model and certain assumptions on the gains of the circuit elements,~\cite{Braun_2006} derive an analytical optimal temperature set-point for the demand limiting period which minimizes the demand peak. This scenario would be equivalent to minimizing the demand charge while ignoring on-peak or off-peak rates. Finally,~\cite{henze2004evaluation} use the heat equation to model the thermal energy storage in the walls and apply MPC to minimize monthly electricity bill in the presence of on-peak and off-peak charges.


\subsection{Our Contributions}

In this chapter, we design a computational framework to achieve the three objectives of a modern power network: reliability, cost minimization and integration of renewables to promote sustainability. This framework relies on smart metering, thermal-mass energy storage, distributed solar generation and   on-peak, off-peak and demand pricing. This framework consists of two nested optimization problems: 1) Optimal thermostat programming (\textit{user-level problem}); 2) Optimal utility pricing (\textit{utility-level problem}).
In the first problem, we consider optimal HVAC usage for a consumer with fixed on-peak, off-peak and demand charges and model passive thermal energy storage using the heat equation. We address both solar and non-solar consumers. For a given range of acceptable temperatures and using typical data for exterior temperature, we pose the optimal thermostat programming problem as a constrained optimization problem and present a Dynamic Programming (DP) algorithm which is guaranteed to converge to the solution. This yields the temperature set-points which minimize the monthly electricity bill for the consumer. For the benefit of the consumers who do not have access to continuously adjustable thermostats, we also develop thermostat programming solutions which include only four programming periods, where each period has a constant interior temperature.

After solving the thermostat programming problem, we use this solution as a model of user behaviour in order to quantify the consumer response to changes in on-peak rates, off-peak rates, and demand charges. Then in the second optimization problem, we apply a descent algorithm to this model in order to determine the prices which minimize the cost-of-generation for the utility. Through several case studies, we show that the optimal prices are NOT necessarily proportional to the marginal costs of generation - meaning that current pricing strategies may be inefficient. Furthermore, we show that in a network of solar and non-solar customers who use our optimal thermostat, the influence of solar generated power on the electricity bills of non-solar customers is NOT significant. Finally, we conclude that although the policy of calculating the demand charge based on the peak consumption over a full-day (rather than the on-peak hours) can substantially reduce the demand peak, it may not reduce optimal cost of production. Our study differs from existing literature (in particular~\cite{Braun_2006},~\cite{Braun_complex_storage},~\cite{henze2004evaluation} and~\cite{kintner1995optimal}) in that it: 1) Considers demand charges (demand charges are far more effective at reducing demand peaks than dynamic pricing) 2) Uses a PDE model for thermal storage (yields a more accurate model of thermal storage) 3) Uses a regulated model for the utility (although unregulated utility models are popular, the fact is that most US utilities remain regulated) 4) Considers the effect of solar generation on the electricity prices and cost of production.

\section{Problem Statement: User-level and Utility Level Problems}

In this section, we first define a model of the thermodynamics which govern heating and cooling of the interior structures of a building. We then use this model to pose the user-level (optimal thermostat programming) problem in Sections~\ref{sec:user} and~\ref{sec:user2} as minimization of a monthly electricity bill (with on/peak, off-peak and demand charges) subject to constraints on the interior temperature of the building. Finally, we use this map of on-peak, off-peak and demand prices to electricity consumption to define the utility-level problem in Section~\ref{sec:utility} as minimizing the cost of generation, transmission and distribution of electricity.

\subsection{A Model for the Building Thermodynamics}
\label{sec:thermal_mass}

In 1822, J. Fourier proposed a PDE to model the dynamics of temperature and energy in a solid mass. Now known as the classical one-dimensional unsteady heat conduction equation, this PDE can be applied to an interior wall as 
\begin{equation}
\dfrac{\partial T(t,x)}{\partial t} = \alpha \dfrac{\partial^2 T(t,x)}{\partial x^2},
\label{eq:PDE_conduction}
\end{equation}
where $T: \mathbb{R}^+ \times [0,L_{in}] \rightarrow \mathbb{R}$ represents the temperature distribution in the interior walls/floor with nominal width $L_{in}$,
and where $\alpha = \frac{k_{in}}{\rho C_p}$ is the coefficient of thermal diffusivity. Here $k_{in}$ is the coefficient of thermal conductivity, $\rho$ is the density and $C_p$ is the specific heat capacity. The wall is coupled to the interior air temperature using Dirichlet boundary conditions, i.e., $T(t,0) = T(t,L_{in}) = u(t) \;\; \text{for all } t \in \mathbb{R}^+$, where $u(t)$ represents the interior temperature which we assume can be controlled instantaneously by the thermostat. In the Fourier model, the heat/energy flux through the surface of the interior walls is modelled as
\begin{equation}
q_{in}(T(t,x)) := 2C_{in} \dfrac{\partial T}{\partial x} (t,0),
\end{equation}
where $C_{in}=k_{in}A_{in}$ is the thermal capacitance of the interior walls and $A_{in}$ is the nominal area of the interior walls.
We assume that all energy storage occurs in the interior walls and surfaces and that energy transport through exterior walls can be modelled using a steady-state version of the heat equation. This implies that the heat flux $q_{loss}$ through the exterior walls is the linear sink
\begin{equation}
q_{loss}(t,u(t)) := \dfrac{T_{e}(t)-u(t)}{R_{e}},
\label{eq:qloss}
\end{equation}
where $T_e(t)$ is the outside temperature and $R_{e} = L_{e}/(k_{e}A_{e})$ is the thermal resistance of the exterior walls, where $L_{e}$
is the nominal width of exterior walls, $k_{e}$ is the coefficient of thermal conductivity and $A_{e}$ is the nominal area of the exterior walls.  By conservation of energy, the power required from the HVAC to maintain the interior air temperature is
\begin{equation}
q(t,u(t),T(t,x)) = q_{loss}(u(t),T_e(t)) + q_{in}(T(x,t)).
\label{eq:q}
\end{equation}
See Figure~\ref{fig:building} for an illustration of the model.

Eqn.~\eqref{eq:PDE_conduction} is a PDE. For optimization purposes, we discretize~\eqref{eq:PDE_conduction} in space, using $T(t)\in \mathbb{R}^M$ to replace $T(t,x)\in \mathbb{R}$, where $T_i(t)$ denotes $T(t,i\,\Delta x)$, where $\Delta x := \frac{L_{in}}{M+1}$.
Then
\begin{equation}
\dot{T}(t) = A \, T(t)+ B \, u(t),
\label{eq:linearized}
\end{equation}
\[
\renewcommand{\arraystretch}{0.6}
\text{where }
A = \dfrac{\alpha}{\Delta x^2}
\begin{bmatrix}
-2 & 1 & 0 & 0  \\
1 & \ddots & \ddots & 0  \\
0 & \ddots & \ddots & 1  \\
0 & 0 & 1 & -2
\end{bmatrix}, \;
B= \dfrac{\alpha}{\Delta x^2}
\begin{bmatrix}
1 \\
0 \\
\vdots \\
0 \\
1
\end{bmatrix}\in \mathbb{R}^M.
\]
We then discretize in time, using
$
\dot{T}(t) \approx
(T(t+\Delta t) - T(t))/\Delta t 
$
to rewrite Equation~\eqref{eq:linearized} as a difference equation. 
\begin{equation}
T^{k+1}  = 
\begin{bmatrix}
T^{k+1}_{1} \\
\vdots \\
T^{k+1}_{M}
\end{bmatrix} = 
 f(T^{k},u_{k}) = \begin{bmatrix}
f_1(T^{k},u_{k}) \\
\vdots \\
f_M(T^{k},u_{k})
\end{bmatrix}  
 =  (I + A \, \Delta \, t )T^{k} + B \, \Delta t \, u_{k}
\label{eq:discrete_dyn}
\end{equation}
for $k=0, \cdots, N_f-1$, where $T^k=T(k\, \Delta t)$ and $u_k=u(k \, \Delta t)$. \vspace{-0.1in}
\begin{figure}[t]
\centering
\includegraphics[scale=0.45]{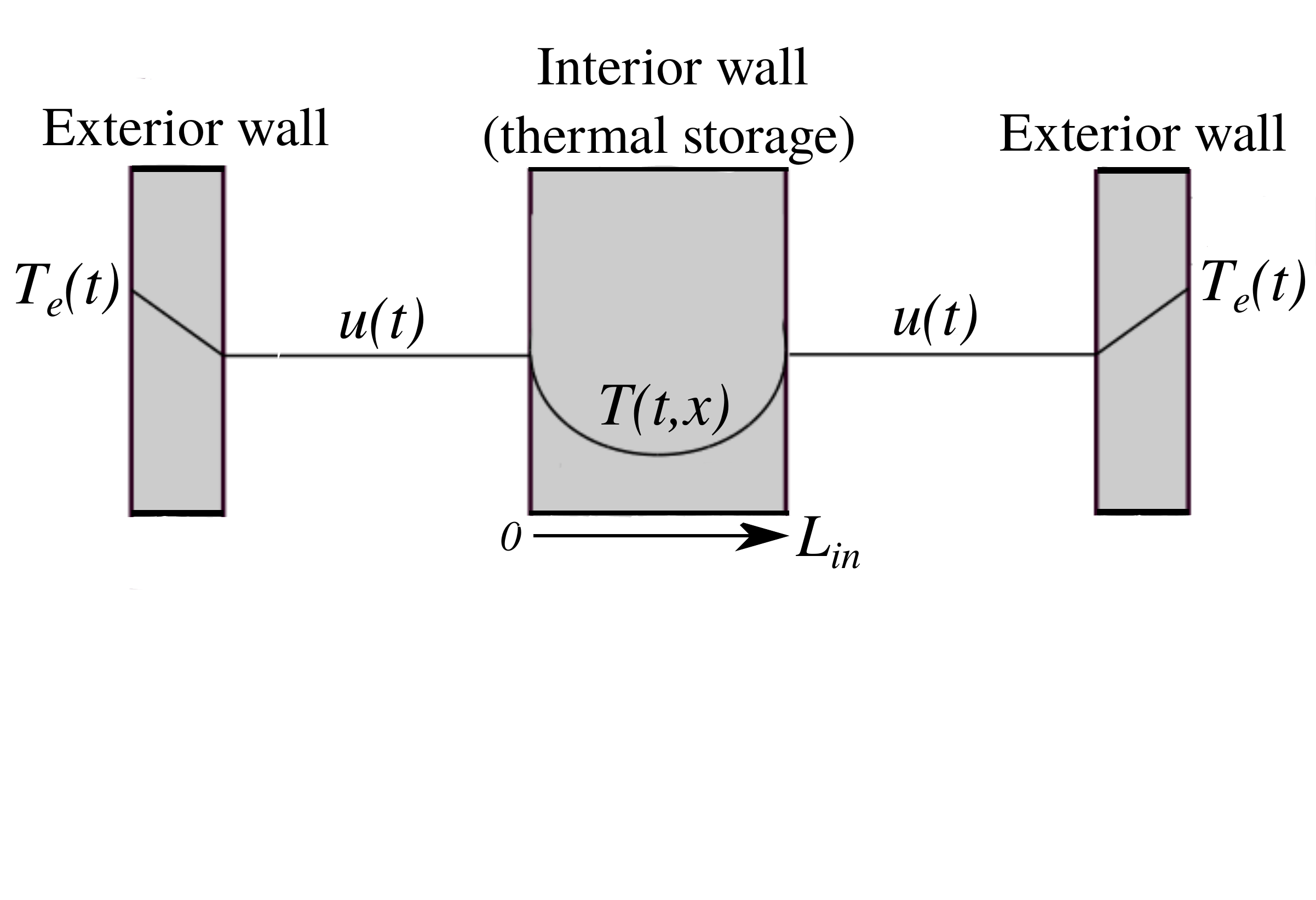} \vspace{-1.05in}
\caption{A Schematic View of Our Thermal Mass Model}
\label{fig:building}
\end{figure}

\subsection{Calibrating the Thermodynamics Model}
\label{sec:validation}

To find empirical values for the parameters $\alpha, C_{in}, R_e$ and $L_{in}$ in the thermodynamic model in Section~\ref{sec:thermal_mass}, we collected data from a 4600 sq ft residential building in Scottsdale, Arizona. The building was equipped with a 5 ton two-stage and three 2.5 ton single-stage RHEEM/RUUD heat pumps, 4-set point thermostats, and 5-min data metering for energy consumption and interior and exterior temperature. In this experiment, we applied two different thermostat programming sequences for two non-consecutive summer days. On the first day, we applied a pre-cooling strategy which lowers the interior temperature to 23.9$^\circ C$ during the off-peak hours and allows the temperature to increase to 27.8$^\circ C$ during the on-peak hours, i.e., 12:00 PM to 7:00 PM. In the second day, we applied the same pre-cooling strategy except that the temperature is again lowered to 23.9$^\circ C$ between 2:00 PM and 4:00 PM. We then used Matlab's least squares optimization algorithm to optimize the parameters such that the root-mean-squared error between the measured power consumption and the simulated power consumption during the entire two days is minimized. The result was the following values for the parameters: $L_{in} = 0.4 (m)$, $\alpha = 8.3 \times 10^{-7} (m^2/s)$, $R_e = 0.0015 (K/W)$ and $C_{in} = 45 (Wm/K)$. In Figure~\ref{fig:validation}, we have compared the resulting simulated and measured power consumption for the entire two days.\vspace{-0.1in}

\begin{figure}[t]%
 \vspace{-0.2in}
\begin{center}
\subfigure[Power Consumption Corresponding to a Pre-cooling Strategy for the Interior Temperature Setting]{\includegraphics[scale=0.313]{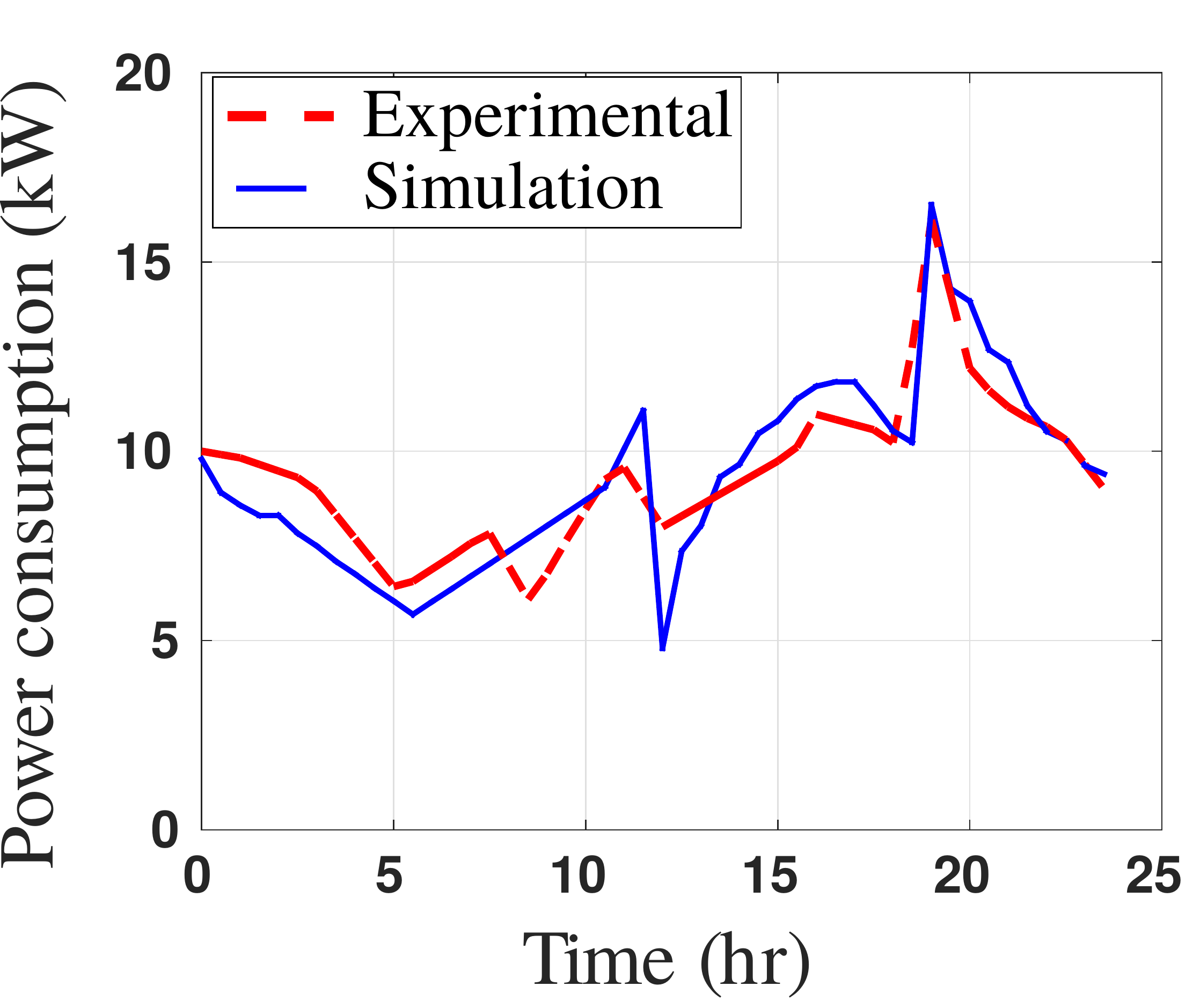}
}
\hspace{0.15in}
\subfigure[Power Consumption Corresponding to a Pre-cooling Strategy with Additional Cooling from 14:00-16:00]{\hspace{-0.12in}\includegraphics[scale=0.31]{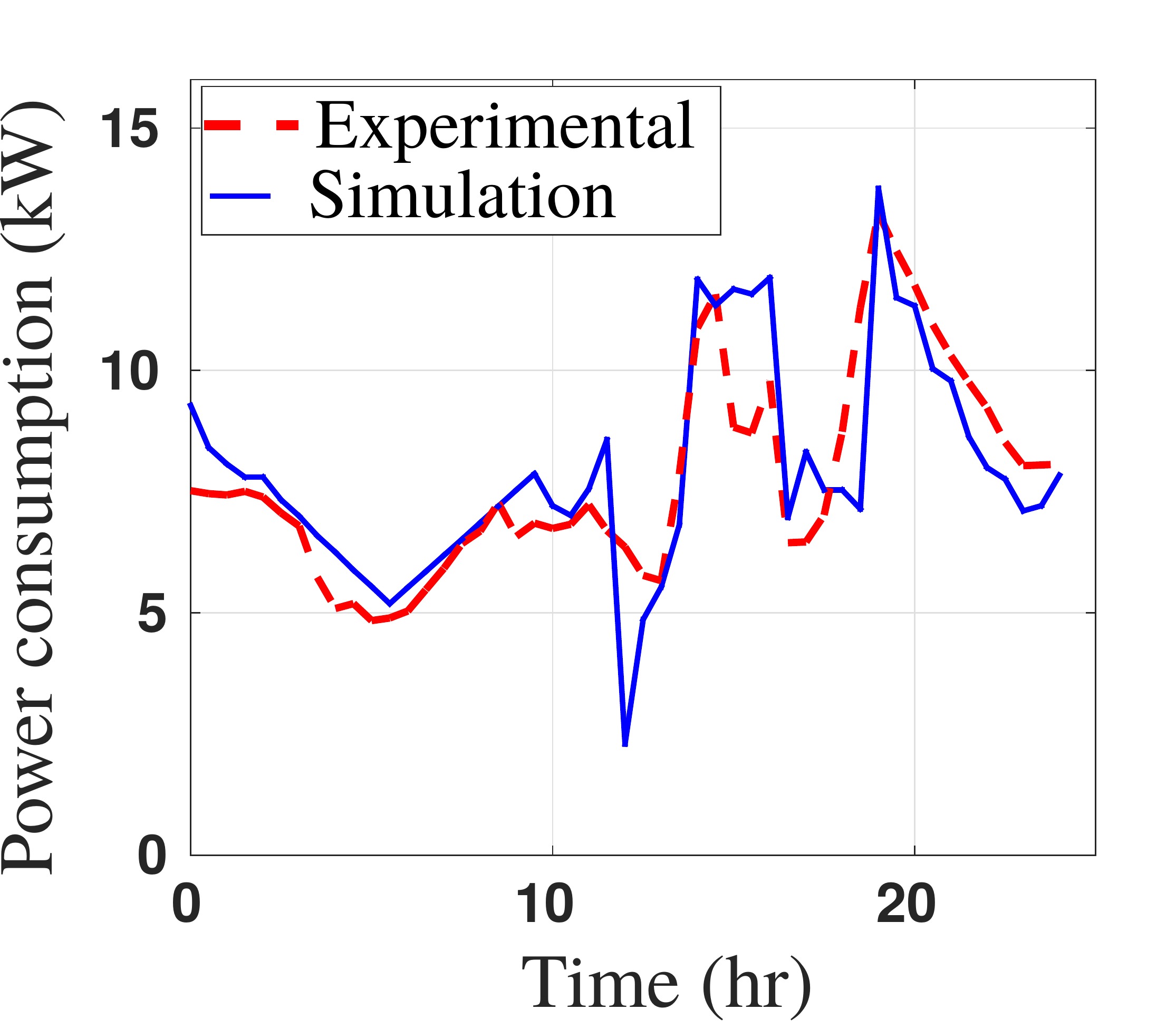}
}
\end{center} \vspace{-0.15in}
\caption{Simulated and Measured Power Consumptions}
\label{fig:validation} \vspace{-0.2in}
\end{figure}

\subsection{User-level Problem I: Optimal Thermostat Programming}
\label{sec:user}
In this section, we define the problem of optimal thermostat programming. We first divide each day into three periods: off-peak hours from 12 AM to $t_{\text{on}}$ with electricity price $p_{\text{off}} \, (\$/kWh)$; on-peak hours beginning at $t_{\text{on}}$ and ending at $t_{\text{off}} > t_{\text{on}}$ with electricity price $p_{\text{on}}  \,(\$/kWh)$; and off-peak hours from $t_{\text{off}}$ to 12 AM with electricity price $p_{\text{off}} \, (\$/kWh)$. In addition to the on-peak and off-peak charges, we consider a monthly charge which is proportional to the maximum rate of consumption during the peak hours. The proportionality constant is called the \textit{demand price} $p_d \, (\$/kW)$. Given $p:=[p_{\text{on}}, p_{\text{off}},p_d]$, the total cost of consumption (daily electricity bill) is divided as
\begin{equation}
J_t(\mathbf{u} ,T_1, p) = J_e(\mathbf{u} ,T_1,p_{\text{on}},p_{\text{off}}) + J_d(\mathbf{u} ,T_1, p_d), 
\label{eq:Jy}
\end{equation}
where $J_e$ is the energy cost, $J_d$ is the demand cost and \vspace{-0.05in}
\[
\mathbf{u} := [u_0, \cdots, u_{N_f-1}] \in \mathbb{R}^{N_f} \vspace{-0.05in}
\]
is the vector temperature settings. The energy cost is 
\begin{equation}
J_e(\mathbf{u}, T_1,p_{\text{on}},p_{\text{off}})  = \Big( p_{\text{off}} \sum_{k \in S_{\text{off}}} g(k,u_{k},T_1^k)  + p_{\text{on}}  \sum_{k \in S_{\text{on}}} g(k,u_{k},T_1^k) \Big) \Delta t,
\label{eq:Je}
\end{equation}
where $k \in S_{\text{on}}$ if $k \, \Delta t \in [t_{\text{on}},t_{\text{off}}]$
and $k\in S_{\text{off}}$ otherwise. That is, $S_{\text{on}}$ and $S_{\text{off}}$ correspond to the set of on-peak and off-peak sampling times, respectively. The function $g$ is a discretized version of $q$ (Eqn.~\eqref{eq:q}):
\begin{equation}
g(k,u_k, T_1^k) :=
 \dfrac{T^k_{e} - u_k}{R_e} + 2 \, C_{in} \dfrac{T_1^k - u_k}{\Delta x}.
 \label{eq:gk}
\end{equation}
i.e., $g$ the power consumed by the HVAC at time step $k$, where $T_e^k$ denotes the external temperature at time-step $k$. If demand charges are calculated monthly, the demand cost, $J_d$, for a single day can be considered as
\begin{equation}
J_d(\mathbf{u}, T_1, p_d) := \dfrac{p_d}{30} \max_{k \in S_{\text{on}}} g(k,u_{k},T_1^k). 
\label{eq:Jd}
\end{equation}

We now define the optimal thermostat programming (or user-level) problem as minimization of the total cost of consumption, $J_t$, as defined in~\eqref{eq:Jy}, subject to the building thermodynamics in~\eqref{eq:discrete_dyn} and interior temperature constraints:
\begin{align}
& J^\star(p) = \min_{u_{k}, \gamma \in \mathbb{R}, T^k \in \mathbb{R}^M} J_{e}(\mathbf{u}, T_1,p_{\text{on}},p_{\text{off}}) + \frac{p_d}{30} \, \gamma   \nonumber \\
& \text{subject to} \nonumber  \;\, g(k, u_{k}, T_1^k) \leq \gamma \hspace{0.92in} \text{ for } k \in S_{\text{on}} \nonumber \\
& \hspace{0.58in} T^{k+1} = f(T^k,u_k) \hspace{0.83in} \text{ for } k \in S_{\text{on}} \cup S_{\text{off}} \nonumber \\
& \hspace{0.58in}  T_{\min} \leq u_{k} \leq T_{\max} \hspace{0.85in}   \text{ for } k \in S_{\text{on}} \cup S_{\text{off}}  \nonumber \\
&  \hspace{0.58in} T^{0} = [T_{\text{init}}(\Delta x), \cdots, T_{\text{init}}(M \, \Delta x)]^T,
\label{eq:user_discrete}
\end{align}
where $T_{\min},T_{\max}$ are the acceptable bounds on interior temperature. 
Note that this optimization problem depends implicitly on the external temperature through the time-varying function $g$.

\subsection{User-level Problem II: 4-Setpoint Thermostat Program}
\label{sec:user2}
Most of the commercially available programmable thermostats only include four programming periods per-day, each period possessing a constant temperature. In this section, we account for this contraint. First, we partition the day into programming periods: $P_i := [t_{i-1},t_i], \, i=1, \cdots,4$ such that 
 \[
\bigcup_{i=1}^4 P_i = [0,24], \quad t_{i-1} \leq t_i, \quad t_0 = 0 \;\text{ and }\; t_4=24. \vspace{-0.05in}
\]
We call $t_0, \cdots, t_4$  \textit{switching} times. Similar to the previous model, $u_i \in [T_{\text{min}}, T_{\text{max}}]$ denotes the temperature setting corresponding to the programming period $P_i$.

To simplify the mathematical formulation of our problem, we introduce some additional notation. Define the set $S_i$ by $k\in S_i$ if $k \Delta t \in P_i$. Denote $L_i := \max_{k \in S_i}k$. For clarity, we have depicted $L_i$ in Figure~\ref{fig:4_setpoint}. Moreover, for each $P_i$, we define $\overline{\Delta t}_i$ as the period between the last time-step of $P_i$ and the end of $P_i$, i.e.,
$
\overline{\Delta t}_i := t_i - L_i \, \Delta t.
$
See Figure~\ref{fig:4_setpoint} for an illustration of $\overline{\Delta t}_i$.
In this framework, the daily consumption charge is 
\[
I_t(\mathbf{u}, T_1,p) =  I_e(\mathbf{u}, T_1,p_{\text{on}},p_{\text{off}}) + I_d(\mathbf{u}, T_1,p_d)
\]
where $I_e$ is the energy cost
\begin{equation}
I_e(\mathbf{u}, T_1,p_{\text{on}},p_{\text{off}}) = 
 \sum_{i=1}^{4} \left( \sum_{k \in S_i} \left(r(k) g(k,u_i,T_1^k) \Delta t \right) + r(L_{i}) g(k,u_i,T_1^{L_{i}}) \overline{\Delta t}_{i} \right),
\label{eq:4_setpoint_bill}
\end{equation}
and $I_d$ is the demand cost
$
I_d(\mathbf{u}, T_1,p_d) = \max_{k \in S_{\text{on}}}  g(k,u,T_1^k),
$
where $L_0 = 0$ and $r$ is defined as
\[
r(k) :=
\begin{cases}
p_{\text{on}} & t_{\text{off}} \leq k \, \Delta t  < t_{\text{on}} \\
p_{\text{off}} &\text{otherwise}.
\end{cases}
\]
Assuming that the demand cost for a single day is $ \frac{p_d}{30} \max_{k \in S_{\text{on}}}  g(k,u,T_1^k)$, we define the 4-setpoint thermostat programming problem as 
\begin{align}
& \hspace{-5mm}\min_{ \substack{ u_1, \cdots,u_4 \in \mathbb{R} \\ t_1,t_2,t_3,\gamma \in \mathbb{R}, \, T^k \in \mathbb{R}^M }} \; I_e(\mathbf{u}, T_1,p_{\text{on}},p_{\text{off}}) + \dfrac{p_d}{30} \gamma \qquad \text{subject to} \nonumber \\
& g(k, u_i, T_1^k) \leq \gamma \hspace{0.61in} \text{ for } k \in S_{\text{on}},\; i\in\{1,2,3,4\}  \nonumber \\
&  T^{k+1} = f(T^{k},u_i) \hspace{0.5in} \text{ for } k \in S_i \text{ and } i\in\{1,2,3,4\}    \nonumber \\
&   T_{\min} \leq u_i \leq T_{\max} \hspace{0.52in}   \text{ for } i\in\{1,2,3,4\}  \nonumber \\
& 0 \leq t_{i-1} \leq t_i \leq 24 \hspace{0.45in}   \text{ for } i\in\{1,2,3,4\} \nonumber\\
&  T^{0} = [T_{\text{init}}(\Delta x), \cdots, T_{\text{init}}(M \, \Delta x)]^T,
\label{eq:user_4setpoint}
\end{align}
where $t_0=0$ and $t_4=24$.

\begin{figure}[t]
\centering
 \includegraphics[scale=0.4]{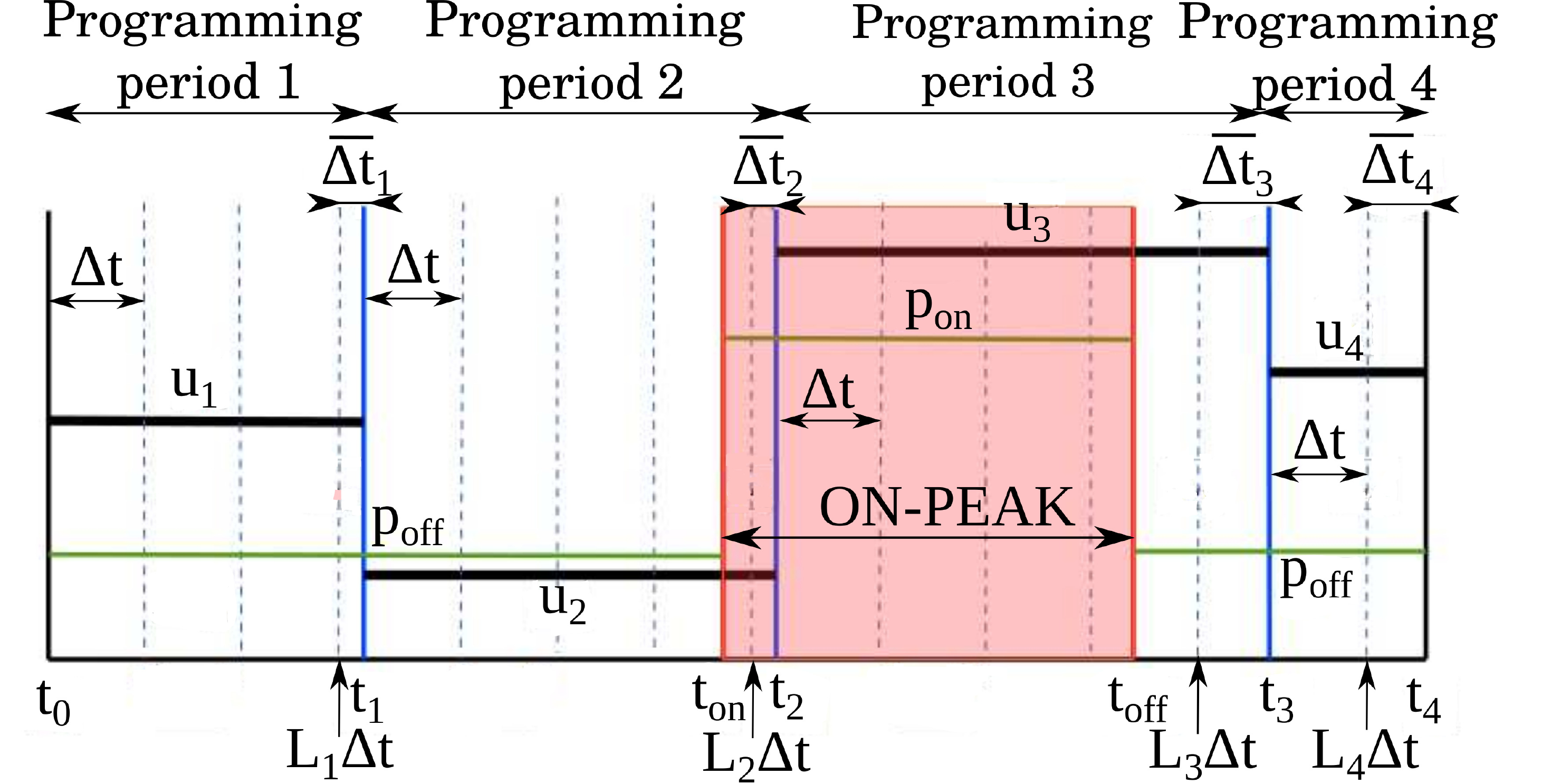}\vspace{0.1in}
\caption{An Illustration for the Programming Periods of the 4-Setpoint Thermostat Problem, Switching Times $t_i$, Pricing Function $r$, $L_i$ and $\overline{\Delta t}_i$.}
\label{fig:4_setpoint} 
\end{figure}

\subsection{Utility-level Optimization Problem}
\label{sec:utility}
Regulated utilities must meet expected load while maintaining a balance between revenue and costs. Therefore, we define the utility-level optimization problem as minimization of the total cost of generation, transmission and distribution of electricity such that generation is equal to consumption, and the total cost is a fixed percentage of the revenue of the utility company. Note that in this dissertation, we focus on vertically integrated utility companies - meaning that the company provides all aspects of electric services including generation, transmission, distribution, metering and billing services as a single firm. Let $s(t)$ be the amount of electricity produced as a function of time and let $\mathbf{s}:= [s_0, \cdots, s_{N_f-1}]$, where $s_k=s(k \, \Delta t)$.
The vector $\mathbf{s}$ is determined by the electricity consumed by the users, which we model as a small number of user groups which are lumped according to different building models, temperature limits, and solar generating capacity so that aggregate user group $i$ has $N_i$ members.
Next, we define $u_k^{\star,i}(p)$ to be the minimizing user temperature setting for user $i$ at time $k$ with prices $p$ and $T_j^{i,\star,k}(p)$ to be the minimizing interior wall temperatures for aggregate user $i$ at time $k$ and discretization point $j$ for prices $p$, where minimization is with respect to the user-level problem defined in~\eqref{eq:user_discrete}. Then the model of electricity consumption by the rational user $i$ at time step $k$ for prices $p$ is given by $g(k,u^{\star,i}_{k}(p),T_{1}^{\star,k,i}(p))$. Thus the constraint that production equals consumption at all time implies
\begin{equation}
s_k = \sum_i N_i g(k,u^{\star,i}_{k}(p),T_{1}^{i,\star,k}(p)) \; \text{for all } k=0, \cdots, N_f-1.  \vspace{-0.05in}
\label{eq:sk}
\end{equation}
Now, since utility's revenue equals the amount paid by the users, the model for revenue from rational user $i$ becomes $J_t(u^{\star,i}(p) ,T^{i,\star}_{1}(p), p)$, where $J_t$ is defined in~\eqref{eq:Jy}. We may now define the utility-level optimization problem as minimization of the total cost subject to equality of generation and consumption and proportionality of revenue and total costs.
\begin{align}
&  \min_{ p_{\text{on}},p_{\text{off}},p_d \in \mathbb{R}} \quad  C(\mathbf{s})    \nonumber\\
& \text{subject to } \quad s_k = \sum_i N_i \, g(k,u^{\star,i}_{k}(p),T_{1}^{i,\star,k}(p)) \quad k=0,\cdots, N_f-1  \nonumber\\
& \qquad\;\, \hspace{0.5in} C(\mathbf{s}) = \lambda \sum_i N_i \, J_t(u^{\star,i}(p) ,T^{i,\star}_1(p), p), 
\label{eq:utility_problem}
\end{align}
where $\lambda \leq 1$ is usually determined by the company's assets, accumulated depreciation and allowed rate of return. We refer to the minimizers $p^\star_{\text{on}},p^\star_{\text{off}},p^\star_d$ which solve Problem~\eqref{eq:utility_problem} as \textit{optimal electricity prices}.

\paragraph{Model of total cost, $C(s)$, to utility company} The algorithm defined in the following section was chosen so that only a black-box model of utility costs is required. However, for the case studies included in Section~\ref{sec:numerical}, we use two models of utility costs based on ongoing discussions and collaboration with Arizona's utility company SRP. In the first model, we consider a linear representation of both fuel and capacity costs.
\begin{equation}
C(\mathbf{s}) := a \hspace{-0.1in} \sum_{k \in S_{\text{on}} \cup S_{\text{off}}} \hspace{-0.1in} s_k \, \Delta t + b\max_{k\in S_{\text{on}}} s_k,
\label{eq:gen_cost}
\end{equation}
where $a  \, (\$/kWh)$ is the marginal cost of producing the next $kWh$ of energy and $b \, (\$/kW)$
is the marginal cost of installing and maintaining the next $kW$ of capacity. Estimated values of the coefficients $a$ and $b$ for SRP can be found in~\cite{SRP} as $a=0.0814 \$/kWh$ and $b=59.76 \$/kW$. According to~\cite{SRP}, these marginal costs include fuel, building, operation and maintenance of facilities, transmission and distribution costs. The advantage of this model is that the solution to the utility optimization problem does not depend on the number of users, but rather the fraction of users in each group.

Our second model for utility costs includes a quadratic term to represent fuel costs. The quadratic term reflects the increasing fuel costs associated with the required use of older, less-efficient generators when demand increases.
\begin{equation}
C(\mathbf{s}) := \tau \Big( \sum_{k \in S_{\text{on}} \cup S_{\text{off}}} \hspace{-0.1in} s_k \, \Delta t \Big)^2 + \nu \hspace{-0.1in} \sum_{k \in S_{\text{on}} \cup S_{\text{off}}} \hspace{-0.1in} s_k \, \Delta t + b\max_{k\in S_{\text{on}}} s_k
\label{eq:gen_costQuad}
\end{equation}
This model was calibrated using artificially modified fuel, operation and maintenance \\

\pagebreak

\hspace{-0.33in} data provided by SRP, yielding estimated coefficients  $\tau=$0.00401 \$/(MWh)$^2$ and $\nu=$4.54351 \$/(MWh).

\section{Solving User- and Utility-level Problems by Dynamic Programming}

First, we solve the optimal thermostat programming problem using a variant of dynamic programming. This yields consumption as a function of prices $p_{\text{on}},p_{\text{off}},p_d$. Next, we embed this implicit function in the Nelder-Mead simplex algorithm in order to find prices which minimize the production cost in the utility-level optimization problem as formulated in~\eqref{eq:utility_problem}. We start the user-level problem by fixing the variable $\gamma \in \mathbb{R}^+$ and defining a \textit{cost-to-go} function, $V_k$. At the final time $N_f \, \Delta t=24$, we have
\begin{equation}
V_{N_f}(x) := \dfrac{p_d}{30} \gamma.
\label{eq:V_Nf}
\end{equation}
Here for simplicity, we use $x=T \in \mathbb{R}^M$ to represent the discretized temperature distribution in the wall. We define the dilated vector of prices by $p_{j}=p_{\text{off}}$ if $j\in S_{\text{off}}$ and $p_{j}=p_{\text{on}}$ otherwise. Then, we construct the cost-to-go function inductively as
\begin{align}
\hspace{-0.1in} V_{j-1} (x) := \hspace{-0.05in}  \min_{u \in W_{\gamma,j-1}(x)}  \hspace{-0.05in}  \left( p_{j-1} \, g(j-1, u,x_1) \Delta t + V_j(f(x,u)) \right)
\label{eq:V0_on}
\end{align}
 for $j=1, \cdots,N_f$, where $W_{\gamma,j}(x)$ is the set of allowable inputs (interior air temperatures) at time $j$ and state $x$:
\begin{align*}
W_{\gamma,j}(x) := 
\begin{cases} \{u \in \mathbb{R}: T_{\min} \leq u \leq T_{\max}, g(j,u,x_1) \leq \gamma\}, & \hspace{-0.1in} j \in S_{\text{on}}\\
\{u \in \mathbb{R}: T_{\min} \leq u \leq T_{\max} \}, & \hspace{-0.1in} j \in S_{\text{off}}.
\end{cases}
\end{align*}
Now we present the main result. 
\pagebreak
\begin{mythm} 
Given $\gamma \in \mathbb{R}^+$, suppose that $V_i$ satisfies~\eqref{eq:V_Nf} and~\eqref{eq:V0_on}. Then $V_0(T^0)=J^*$, where
\begin{align}
& J^*(p) = \min_{u_{k}, T^k \in \mathbb{R}^M} J_{e}(\mathbf{u}, T_1,p_{\text{on}},p_{\text{off}}) + \frac{p_d}{30} \, \gamma   \nonumber \\
& \text{subject to }  \;\; g(k, u_{k}, T_1^k) \leq \gamma && \text{ for } k \in S_{\text{on}} \nonumber \\
& \hspace{0.78in} T^{k+1} = f(T^k,u_k) && \text{ for } k \in S_{\text{on}} \cup S_{\text{off}} \nonumber \\
& \hspace{0.78in}  T_{\min} \leq u_{k} \leq T_{\max} && \text{ for } k \in S_{\text{on}} \cup S_{\text{off}}  \nonumber \\
&  \hspace{0.78in} T^{0} = [T_{\text{init}}(\Delta x), \cdots, T_{\text{init}}(M \, \Delta x)]^T.
\label{eq:user_discrete2}
\end{align}
\vskip-0.3in
\label{thm:DP}
\end{mythm}
To prove Theorem~\ref{thm:DP}, we require the following definitions.  
\begin{mydef}
Given $p_{\text{off}}, p_{\text{on}}, p_d, \gamma \in \mathbb{R}^+$, $N_f \in \mathbb{N}^+$, and $t_{\text{off}},t_{\text{on}},\Delta t \in \mathbb{R}^+$ such that $\frac{t_{\text{on}}}{\Delta t}, \frac{t_{\text{off}}}{\Delta t} \in \mathbb{N}$, define the cost-to-go functions
\[
Q_j: \mathbb{R}^{N_f-j} \times \mathbb{R}^{N_f-j+1} \times \mathbb{R}^+ \times \mathbb{R}^+ \rightarrow \mathbb{R} \; \text{ for } j=0, \cdots, N_f \text{ as} \vspace{-0.1in}
\]
\begin{align}
& Q_j(x, y, p_{\text{on}}, p_{\text{off}}) := \nonumber \\
&\begin{cases}
    \!\begin{aligned}
       & \hspace{-0.03in} \left( p_{\text{off}} \hspace{-0.2in} \sum\limits_{\substack{k \in S_{\text{off}} \\ k \notin \{ 0, \cdots, j-1\}}} \hspace{-0.2in} g(k,x_{k},y_{k}) + p_{\text{on}} \sum\limits_{k \in S_{\text{on}}} \hspace{-0.05in} g(k,x_{k},y_{k}) \right) \Delta t + \sum_{i=1}^3 \Gamma_i & 
    \end{aligned}  & \hspace{-0.1in} \text{ if } 0 \leq j < N_{\text{on}}  \\
\Gamma_1 =  \left(  p_{\text{on}} \hspace{-0.2in} \sum\limits_{\substack{k \in S_{\text{on}} \\ k \notin \{ N_{\text{on}},\cdots,j-1\} }} \hspace{-0.28in} g(k,x_{k},y_{k}) + p_{\text{off}} \hspace{-0.27in} \sum\limits_{\substack{k \in S_{\text{off}} \\ k \notin \{0,\cdots,N_{\text{on}}-1 \}}} \hspace{-0.28in} g(k,x_{k},y_{k}) \hspace{-0.03in} \right) \Delta t + \Gamma_2 + \Gamma_3 & \hspace{-0.1in} \text{ if } N_{\text{on}} \leq j < N_{\text{off}} \\
\Gamma_2 =   p_{\text{off}} \hspace{-0.2in}  \sum\limits_{\substack{ k \in \{j,\cdots, N_f-1 \}}} \hspace{-0.2in} g(k,x_{k},y_{k})  \Delta t + \Gamma_3  & \hspace{-0.1in} \text{ if } N_{\text{off}} \leq j < N_f \\
\Gamma_3 = \dfrac{p_d}{30} \gamma  & \hspace{-0.1in} \text{ if } j = N_f,
\end{cases}
\label{eq:cost-to-go} 
\end{align}
where $g$ is defined as in~\eqref{eq:gk}, and $N_{\text{on}} := \frac{t_{\text{on}}}{\Delta t}$ and $N_{\text{off}} := \frac{t_{\text{off}}}{\Delta t}$ are the time-steps corresponding to start and end of the on-peak hours.
\label{def:cost-to-go}
\end{mydef}
Note that from~\eqref{eq:Je}, it is clear that $Q_0 = J_e + \dfrac{p_d}{30} \gamma$.

\begin{mydef}
Given $\gamma, T_{\min}, T_{\max} \in \mathbb{R}$ and $N_f,M \in \mathbb{N}^+$, define the set
\begin{align}
U_j(x) := \{ &(u_j,  \cdots,  u_{N_f-1}) \in \mathbb{R}^{N_f-j} : \nonumber \\
&  g(k,u_{k}, T_1^{k}) \leq \gamma \; &&\text{ for all }  k \in S_{\text{on}}, \nonumber \\
&  T^j = x \text{ and } T^{k+1} = f(T^{k},u_{k}) && \text{ for all } k \in \{ j, \cdots, N_f-1\}, \nonumber \\
&  T_{\min} \leq u_{k} \leq T_{\max}  && \text{ for all }  k \in S_{\text{on}} \cup S_{\text{off}} \}
\label{eq:Uj}
\end{align}
for any $x \in \mathbb{R}^M$ and for every $j \in \{ 0, \cdots, N_f-1 \}$, where $f$ and $g$ are defined as in~\eqref{eq:discrete_dyn} and~\eqref{eq:gk}.  \vspace{0.1in}
\label{def:Uj}
\end{mydef} 

\begin{mydef}
Given $N_f, M \in \mathbb{N}^+$, $j \in \{ 0, \cdots, N_f-1 \}$, let
\[
\overline{\mu}_{j} :=[ \mu_j, \cdots, \mu_{N_f-1} ]
\]
where $ \mu_k : \mathbb{R}^{M} \rightarrow \mathbb{R} \text{ for } k=j, \cdots, N_f-1.
$
Consider $U_j$ as defined in~\eqref{eq:Uj} and $f$ as defined in~\eqref{eq:discrete_dyn}. If
\[
\overline{\mu}_j(w) := [\mu_j(w), \mu_{j+1}(T^{j+1}) \cdots, \mu_{N_f-1}(T^{N_f-1})] \in U_j(T^j)
\]
for any $w \in \mathbb{R}^M$, where 
\[
T^{k+1} = f(T^k,\mu_k(T^k)), T^j=w \;\; \text{for } k=j ,\cdots, N_f-2,
\]
then we call $\overline{\mu}_{j}$  an \textit{admissible control law} for the system
\[T^{k+1} = f(T^k,\mu_k(T^k)), \; k=j ,\cdots, N_f-1
\]
for any $w \in \mathbb{R}^M$.
\label{def:maps}
\end{mydef}

We now present a proof for Theorem~\ref{thm:DP}.

\begin{proof}
Since the cost-to-go function $Q_0 = J_e + \dfrac{p_d}{30} \gamma$, if we show that 
\begin{equation}
\min_{ \overline{\mu}_j(T^j) \in U_j(T^j)} \; Q_j(\overline{\mu}_{j}(T^j),T_1, p_{\text{on}}, p_{\text{off}}) = V_{j}(T^j) \vspace{-0.05in}
\label{eq:proof1}
\end{equation}
for $j=0, \cdots, N_f$ and for any $T^j \in \mathbb{R}^{M}$, where
\[
T_1:=[T^j,f(T^j,\mu_j(T^j)), \cdots, f(T^{N_f-1},\mu_{N_f-1}(T^{N_f-1}))],
\] 
then it will follow that $J^* = V_0(T^0)$. For brevity, we denote $\overline{\mu}_j(T^j)$ by $\overline{\mu}_j$, $U_j(T^j)$ by $U_j$ and we drop the last two arguments of $Q_j$.
To show~\eqref{eq:proof1}, we use induction as follows. 

\noindent \textit{\underline{Basis step}}: If $j=N_f$, then from~\eqref{eq:V_Nf} and~\eqref{eq:cost-to-go} we have $V_{N_f}(T^{N_f}) = \frac{p_d}{30} \gamma$.\vspace*{0.05in}

\noindent \textit{\underline{Induction hypothesis}}: Suppose 
\[
\min_{\overline{\mu}_k \in U_k} \; Q_k(\overline{\mu}_{k}, T_1) = V_{k}(T^k) \vspace{-0.065in}
\]
for some $k \in \{0, \cdots, N_f \}$ and for any $T^k \in \mathbb{R}^{M}$. Then, we need to prove that 
\begin{equation}
\min_{ \overline{\mu}_{k-1} \in U_{k-1}} \; Q_{k-1}(\overline{\mu}_{k-1},T_1) = V_{k-1}(T^{k-1}) \vspace{-0.05in}
\label{eq:proof4}
\end{equation}
for any $T^k \in \mathbb{R}^{M}$. Here, we only prove~\eqref{eq:proof4} for the case which $N_{\text{off}} < k \leq N_f-1$.
The proofs for the cases $0 \leq k \leq N_{\text{on}}$ and $N_{\text{on}} < k \leq N_{\text{off}}$ follow the same exact logic.

Assume that $N_{\text{off}} < k \leq N_f-1$. Then, from Definition~\ref{def:cost-to-go} 
\begin{align}
& \min_{ \overline{\mu}_{k-1} \in U_{k-1}} \;  Q_{k-1}(\overline{\mu}_{k-1},T_1) \nonumber  \\
=&   \min_{\mu_{k-1}, \cdots, \mu_{N_f-1} \in R} \; p_{\text{off}} \left( \sum\limits_{j=k-1}^{N_f-2}  g\left(j,\mu_j,T_1^{j}\right) \right) \Delta t \nonumber  \\
=&  \min_{\mu_{k-1}, \cdots, \mu_{N_f-1} \in R}  p_{\text{off}} \left( g\left(k-1, \mu_{k-1},T_1^{k-1}\right) 
 +  \sum\limits_{j=k}^{N_{f}-2}  g\left(j, \mu_j ,T_1^{j}\right)\right) \Delta t,
\label{eq:proof5}
\end{align}
\hspace{-0.18in} where $R:= \{ x \in \mathbb{R}: T_{\min} \leq x  \leq T_{\max} \}$.
From the principle of optimality~(\cite{bellman}) it follows that 
\begin{align}
&\min_{\mu_{k-1}, \cdots, \mu_{N_f-1} \in R}  p_{\text{off}} \left(  g\left( k-1,\mu_{k-1},T_1^{k-1} \right) + \sum\limits_{j=k}^{N_f-1}   g\left(j, \mu_j ,T_1^{j}\right)  \right) \Delta t \nonumber \\
=&\min_{\mu_{k-1} \in R} \; \left( p_{\text{off}} \, g\left(k-1,\mu_{k-1},T_1^{k-1}\right)  \Delta t + \min_{\mu_{k}, \cdots, \mu_{N_f-1} \in R} \;  p_{\text{off}} \sum\limits_{j=k}^{N_f-1}  g\left( j,\mu_{j},T_1^{j}\right) \right) \Delta t,
 \label{eq:proof6}
\end{align}
By combining~\eqref{eq:proof5} and~\eqref{eq:proof6} we have
\begin{align}
& \min_{\overline{\mu}_{k-1} \in U_{k-1}} \; Q_{k-1}(\overline{\mu}_{k-1}, T_1) = \min_{\mu_{k-1} \in R} \; \left( p_{\text{off}} \, g\left(k-1,\mu_{k-1}),T_1^{k-1}\right)  \Delta t \right. \nonumber  \\
 &   \hspace{1.6in} \left. + \min_{\mu_{k}, \cdots, \mu_{N_f-1} \in R} \;  p_{\text{off}} \sum\limits_{j=k}^{N_f-1}  g\left(j, \mu_{j},T_1^{j}\right) \right) \Delta t.
 \label{eq:proof7}
\end{align}
From Definition~\ref{def:cost-to-go}, we can write
\begin{align}
\min_{\mu_{k}, \cdots, \mu_{N_f-1}} p_{\text{off}} \, \Delta t \sum\limits_{j=k}^{N_f-1}  g\left(j,\mu_{j},T_1^{j}\right)
=  \min_{\overline{\mu}_{k} \in U_k}   Q_{k}(\overline{\mu}_{k},T_1).
\label{eq:proof2}
\end{align}
Then, by combining~\eqref{eq:proof7} and~\eqref{eq:proof2} and using the induction hypothesis it follows that
\begin{align*}
 \min_{\overline{\mu}_{k-1} \in U_{k-1}}  \; Q_{k-1}(\overline{\mu}_{k-1},T_1) & = \min_{\mu_{k-1} \in R} \left( p_{\text{off}} \, g\left(k-1,\mu_{k-1},T_1^{k-1}\right) \Delta t  +  \min_{\overline{\mu}_{k} \in U_k} Q_{k}(\overline{\mu}_{k},T_1) \right) \\
& = \min_{\mu_{k-1} \in R} \; \left( p_{\text{off}} \, g\left(k-1,\mu_{k-1},T_1^{k-1}\right) \Delta t + V_{k}(T^k) \right)
\end{align*}
for any $T^k \in \mathbb{R}^{M}$.
By substituting for $T^k$ from~\eqref{eq:discrete_dyn} and using the definition of $V$ in~\eqref{eq:V0_on} we have
\begin{align*}
\min_{\overline{\mu}_{k-1} \in U_{k-1}} Q_{k-1} ( & \overline{\mu}_{k-1}, T_1)   = 
\min_{\mu_{k-1} \in R} \left( p_{\text{off}} \, g\left(k-1, \mu_{k-1},T_1^{k-1}\right) \Delta t \right. \\
&  \hspace{1.5in} \left. + V_{k}\left(f \left(T^{k-1},\mu_{k-1}\left(T^{k-1} \right) \right) \right) \right) = V_{k-1} \left(T^{k-1} \right)
\end{align*}
for any $T^{k-1} \in \mathbb{R}^M$.
By using the same logic it can be shown that
\[
\min_{\overline{\mu}_{k-1} \in U_{k-1}} Q_{k-1}(\overline{\mu}_{k-1},T_1) = V_{k-1}(T^{k-1})
\]
for any $k \in \{0, \cdots, N_{\text{off}}-1 \}$ and for any $T^{k-1} \in \mathbb{R}^{M}$.
Therefore, by induction,~\eqref{eq:proof1} is true. Thus, $J^* = V_0(T^0)$.
\end{proof}

The optimal temperature set-points for Problem~\eqref{eq:user_discrete2} can be found as the sequence of minimizing arguments in the value function~\eqref{eq:V0_on}. However, this is not a solution to the original user-level optimization problem in~\eqref{eq:user_discrete}, as the solution only applies for a fixed consumption bound, $\gamma$. However, as this consumption bound is scalar, we may apply a bisection on $\gamma$ to solve the original optimization problem as formulated in~\eqref{eq:user_discrete}. Details are presented in Algorithm~\ref{alg:thermostat}. The computational complexity of this algorithm is proportional to $N_f \cdot n_s^M \cdot  n_u$,
where $N_f$ is the number of discretization points in time, $M$ is the state-space dimension of the discretized system in~\eqref{eq:discrete_dyn}, $n_s$ is the number of possible discrete values for each state, $T$ and $n_u$ is the number of possible discrete values for the control input (interior air temperature). In all of the case studies in Section~\ref{sec:numerical}, we use $N_f=73, M=3, n_s=n_u=13$. The execution time of our Matlab implementation of Algorithm~\ref{alg:thermostat} for solving the three-day user-level problem on a Core i7 processor with 8 GB of RAM was less than 4.5 minutes.

Finding a solution to the 4-Setpoint thermostat programming problem~\eqref{eq:user_4setpoint} is significantly more difficult due to the presence of the switching times $t_1$, $t_2$, $t_3$ as decision variables. However, for this specific problem, a simple approach is to use Algorithm~\ref{alg:thermostat} as an inner loop for fixed $t_i$ and then use a Monte Carlo search over $t_i$. For fixed $t_i$, our Matlab implementation for Algorithm~\ref{alg:thermostat} solves the 4-Setpoint thermostat programming problem in less than 17 seconds on a Core i7 processor with 8 GB of RAM. Our experiments on the same machine show that the total execution time for a Monte Carlo search over 300 valid (i.e., $t_i \leq t_{i+1}$) random combinations of $t_1$, $t_2$, $t_3$ is less than 1.41 hours.

To solve the utility-level problem in~\eqref{eq:utility_problem}, we used Algorithm~\ref{alg:thermostat} as an inner loop for the Nelder-Mead simplex algorithm~(\cite{olsson_1975}). The Nelder-Mead simplex algorithm is a heuristic optimization algorithm which is typically applied to problems where the derivatives of the objective function and/or constraint functions are unknown. Each iteration is defined by a reflection step and possibly a contraction or expansion step. The reflection begins by evaluation of the inner loop (Algorithm~\ref{alg:thermostat}) at each of 4 vertices of a polytope. The polytope is then reflected about the hyperplane defined by the vertices with the best three objective values. The polytope is then either dilated or contracted depending on the objective value of the new vertex.
In all of our case studies in Section~\ref{sec:numerical}, this hybrid algorithm achieved an error convergence of $<10^{-4}$ in less than 15 iterations. Using a Core i7 machine with 8 GB of RAM, the execution time of the hybrid algorithm for solving the utility-level problem was less than 2.25 hours. \vspace{-0.075in}

\begin{algorithm}
\begin{footnotesize}
\textbf{Inputs:}
$p_{\text{on}}$, $p_{\text{off}}$, $p_{\text{d}}$, $T_e$, $t_{\text{on}}, t_{\text{off}}$, $R_e$, $C_{in}$, $T_{\text{init}}$, $\Delta t$, $\Delta x$, $T_{\min}, T_{\max}$, maximum number of bisection iterations $b_{\max}$, lower bound $\gamma_l$ and upper bound $\gamma_u$ for bisection search.

\textbf{Main loop:}\\
Set $k=0$. \\
\While{ $k < b_{\max} $}{
	Set $\gamma = \frac{\gamma_u+\gamma_l}{2}$. 	\\	
  	\eIf{ $V_0$ in~\eqref{eq:V0_on} exists }{		
  		Calculate $u_0, \cdots, u_{N_f-1}$ as the minimizers of the RHS of~\eqref{eq:V0_on} using a policy iteration technique. 	\\
  		Set $\gamma_u = \gamma$. Set $u_i^{\star} = u_i$ for $i=0, \cdots, N_{f-1}$.	 \\		
 	}
 	{
	 	 Set $\gamma_l = \gamma$. \\
 	}					
	Set $k = k+1$. \\
} 

\textbf{Outputs:}
Optimal interior temperature setting: $u_0^{\star}, \cdots, u_{N_f-1}^{\star}$. 
\end{footnotesize}
\caption{A bisection/dynamic programming algorithm for optimal thermostat programming} 
\label{alg:thermostat}
\end{algorithm}

\begin{algorithm}
\DontPrintSemicolon
\begin{footnotesize}
\textbf{Inputs:}
$a$, $b$, reflection parameter $\theta$, expansion parameter $\kappa$, contraction parameter $\zeta$, reduction parameter $\tau$, initial prices $p^{0}_{d_i}, p^{0}_{\text{on}_i}$ such that $p^{0}_{d_i}+ p^{0}_{\text{on}_i} < 1$ for $i=1, \cdots,4$, stopping threshold $\epsilon$ and inputs to Algorithm~\ref{alg:thermostat}.
\vspace{0.05in}

\textbf{Initialization:} \vspace{-0.1in} \\
Set $p_{d_i} = p^{0}_{d_i}, \, p_{\text{on}_i} = p^{0}_{\text{on}_i}, \, p_{\text{off}_i} = 1 - p_{\text{on}_i} - p_{d_i}$, $p_{\text{old}_i} = [10^{12},10^{12},10^{12}]$ for $i=1,\cdots,4$. \\

\textbf{Main loop:} \vspace{-0.1in} \\ 
\While{ $ \sum_{i=1}^4 \Vert p_{\text{old}_i} - p_i  \Vert_2^2  > \epsilon $}{	

	\For{ $i=1, \cdots, 4$}{
		Calculate opt. temp. setting $u^{\star}_{0,i}, \cdots, u^{\star}_{N_f-1,i}$ associated with prices $p_{i}$ using Alg.~\ref{alg:thermostat}.\\ 
		Calculate the cost $C_i$ associated with $u^{\star}_{0,i}, \cdots, u^{\star}_{N_f-1,i}$ using~\eqref{eq:gen_cost} and~\eqref{eq:sk}.	
	}
	Re-index the costs and their associated prices and temperature settings such that $ C_1 \leq C_2 \leq C_3 \leq C_4$. Set $p^{\star} = p_{1}$. \\ 
	Calculate the centroid of all the prices as $
	\bar{p} = \Big[\sum_{i=1}^4 p_{\text{off}_i}, \sum_{i=1}^4 p_{\text{on}_i}, \sum_{i=1}^4 p_{d_i}\Big]$.
	Calculate reflected prices as $p_{r} = \bar{p}+\theta(\bar{p}+p_{4})$.\\ 
	Calculate optimal temperature setting $u^{\star}_{0,r}, \cdots, u^{\star}_{N_f-1,r}$ and the cost $C_r$ associated with the reflected prices using Algorithm~\ref{alg:thermostat}.\\ \vspace{0.025in}

	\uIf{$C_1 \leq C_r < C_3$}{ \vspace{0.025in}
		Set $p_{\text{old}_4} = p_{4}$ and $p_{4} = p_{r}$. Go to the beginning of the loop. \vspace{0.05in} \\
		
	}
	\uElseIf{$C_r < C_1$}{
		Calculate expanded prices as $p_{e} = \bar{p}+\kappa(\bar{p}+p_{4})$. \\ \vspace{0.035in}
	Calculate optimal temperature setting $u^{\star}_{0,e}, \cdots, u^{\star}_{N_f-1,e}$ and the cost $C_e$ associated with expanded prices using Algorithm~\ref{alg:thermostat}.\\  \vspace{0.05in}
	\lIf{$C_e < C_r$}{
		Set $p_{\text{old}_4} = p_{4}$ and $p_{4} = p_{e}$. Set $p^{\star} = p_{e}$.       
	}\lElse
	{
		Set $p_{\text{old}_4} = p_{4}$ and $p_{4} = p_{r}$. Set $p^{\star} = p_{r}$.
	}
	Go back to While loop.
	}
	\Else{
		Calculate contraction prices $p_{c} = \bar{p}+\xi(\bar{p}+p_{4})$.
	Calculate optimal temp. setting $u^{\star}_{0,c}, \cdots, u^{\star}_{N_f-1,c}$ \& the cost $C_c$ associated with expanded prices using Algorithm~\ref{alg:thermostat}.\\  \vspace{0.025in}
	\lIf{$C_c < C_4$}{
		Set $p_{\text{old}_4} = p_{4}$ and $p_{4} = p_{c}$. 
		Go back to While loop.}
	\lElse{Set $p_{\text{old}_i} = p_i$ for $i=2,3,4$.
		Update prices as $p_i = p_1 + \tau (p_i - p_1)$ for $i=2,3,4$.
	}\vspace{-0.1in}
	}\vspace{-0.1in}
}\vspace{-0.2in}
\textbf{Outputs:} Optimal electricity prices $p^{\star}$.
\end{footnotesize}
\caption{An algorithm for computing optimal electricity prices}
\label{alg:utility}
\end{algorithm}

\section{Policy Implications and Analysis}
\label{sec:numerical}

In this section, we use Algorithms~\ref{alg:thermostat} and~\ref{alg:utility} in three case studies to assess the effects of passive thermal storage, solar power and various cooling strategies on utility prices, peak demand and cost to the utility company.

In Case I, we compare our optimal thermostat program with other HVAC programming strategies and analyze the resulting peak demands and electricity bills for a set of electricity prices.

In Case II, we apply the Nelder-Mead simplex and Algorithm~\ref{alg:thermostat} to the user-level problem in~\eqref{eq:user_discrete} and the utility-level problem in~\eqref{eq:utility_problem} to compute optimal electricity prices and optimal cost of production.

In Case III, we first define an optimal thermostat program for solar users. Then, we examine the effect of solar power generation on the electricity prices of non-solar users by solving a two-user single-utility optimization problem.
We ran all cases for three consecutive days prorated from a one month billing cycle with the time-step $\Delta t = 1 \; hr$, spatial-step $\Delta x = 0.1 \; m$ and with the building parameters in Table~\ref{tab:params}. These parameters were determined using the model calibration procedure described in Section~\ref{sec:validation}.
We used an external temperature profile measured for three typical summer days in Phoenix, Arizona (see Figure~\ref{fig:Te}). For each day, the on-peak period starts at 12 PM and ends at 7 PM. We used min and max interior temperatures as $T_{\min} = 22^{\circ}C$ and $T_{\max} = 28^{\circ}C$.

\begin{figure}[t]
\centering
\includegraphics[scale=0.42]{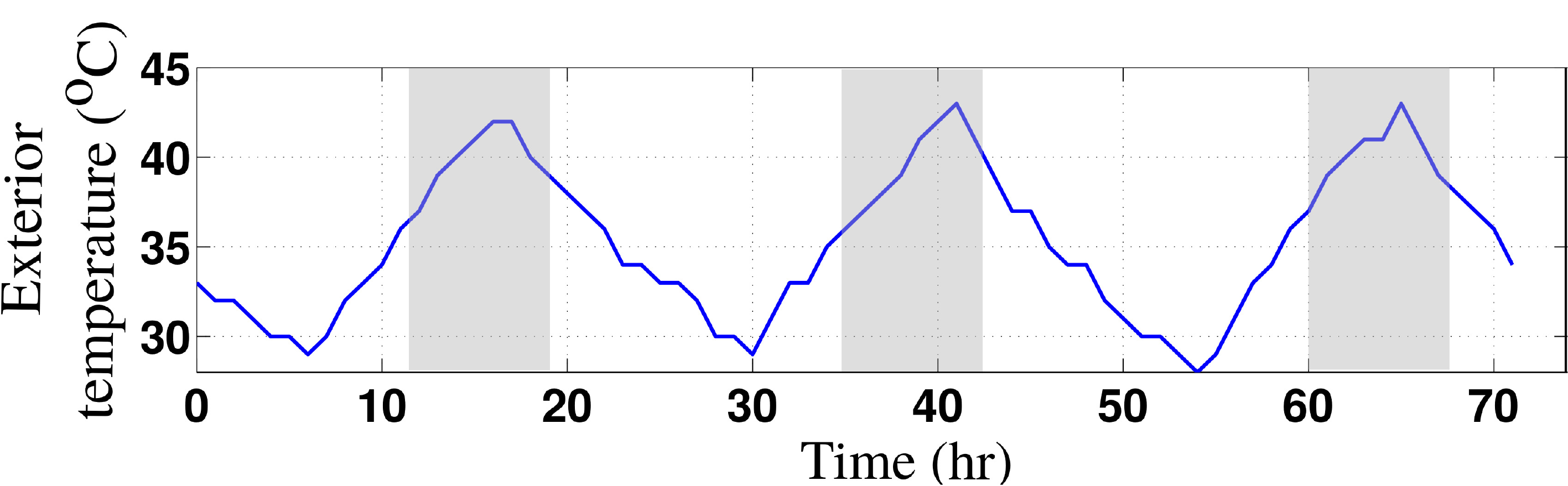} 
\caption{External Temperature of Three Typical Summer Days in Phoenix, Arizona. Shaded Areas Correspond to On-peak Hours.}
\label{fig:Te} 
\end{figure}

\begin{table}[t]
\renewcommand{\arraystretch}{0.8}
\renewcommand{\tabcolsep}{2pt}
\caption{Building's Parameters as Determined in Section~\ref{sec:thermal_mass}} \vspace{0.1in}
\centering
\begin{tabular}{|c|c|c|c|c|}
\hline
$L_{in} (m)$ & $\alpha (m^2/s)$ & $R_e (K/W)$ & $C_{in} (Wm/K)$ & $\Delta x (m)$ \\
\hline
0.4 & $8.3 \times 10^{-7}$ & 0.0015 & 45 & 0.1 \\
\hline
\end{tabular}
\label{tab:params} 
\end{table}

\begin{table}[t]
\renewcommand{\arraystretch}{0.8}
\renewcommand{\tabcolsep}{2pt}
\caption{On-peak, Off-peak \& Demand Prices of Arizona Utility APS}
\vspace{0.1in} 
\centering
\begin{tabular}{|c|c|c|c|}
\hline
 & On-peak $(\frac{\$}{kWh})$ & Off-peak $(\frac{\$}{kWh})$ & Demand $(\frac{\$}{kWh})$  \\
\hline
APS & 0.089  & 0.044  & 13.50  \\
\hline
\end{tabular}
\label{tab:rates} 
\end{table}

\subsection{Effect of Electricity Prices on Peak Demand and Production Costs}
\label{sec:case1}
In this case, we first applied Algorithm~\ref{alg:thermostat} to the optimal and 4-Setpoint thermostat programming problems (See \eqref{eq:user_discrete} and~\eqref{eq:user_4setpoint}) for a non-solar customer using the electricity prices determined by APS in Table~\ref{tab:rates}~(\cite{APS_plan}). The resulting electricity bills are given in Table~\ref{tab:APS} as the total cost paid for three days prorated from a one month billing cycle with the external temperature profile shown in Figure~\ref{fig:Te}.
Prorated in this case means that for a 30-day month, the bill is one-tenth of the monthly bill based on repetition of the three-day cycle ten times. Practically, what this means is that the period in Problems~\eqref{eq:user_discrete} and~\eqref{eq:user_4setpoint} is tripled while the demand charge in Problems~\eqref{eq:user_discrete} and~\eqref{eq:user_4setpoint} uses a demand price $\frac{1}{10} \, p_d=1.35 \frac{\$}{kW}$. For comparison, we have solved Problem~\eqref{eq:user_discrete} using the general-purpose optimization solver GPOPS~(\cite{gpops}). We have also compared our results with a naive strategy of setting the temperature to $T_{\max}$ (constant) and a pre-cooling strategy with the temperature setting: $u= 25^\circ$C from 12 AM to 8 AM; $u= T_{\min} = 22^\circ$C from 8 AM to 12 PM; $u=T_{\max}=28^\circ$C from 12 PM to 8 PM; $u=25^\circ$C from 8 PM to 12 AM. As can be seen from Table~\ref{tab:APS}, our algorithm offers significant improvement over heuristic approaches. The power consumption and the temperature setting as a function of time for each strategy can be found in Figure~\ref{fig:APS}. For convenience, the on-peak and off-peak intervals are indicated on the figure. 

To examine the impact of changes in electricity prices on peak demand, we next chose several different prices corresponding to high, medium and low penalties for peak demand. Again, in each case, our algorithms (optimal and 4-setpoint) are compared to GPOPS and the same pre-cooling strategy. The results are summarized in Table~\ref{tab:scen1}. Note that for brevity, in this section, we refer to the total cost of generation, transmission and distribution as simply production cost. For each price, the smallest computed production cost and associated demand peak are listed in bold. The power consumption and the temperature settings as a function of time for the optimal and 4-Setpoint strategies can be found in Figures~\ref{fig:scen1} and~\ref{fig:scen1.1}. For the optimal strategy, notice that by increasing the demand penalty, relative to the low-penalty case, the peak consumption is reduced by 14\% and 23\% in the medium and high penalty cases respectively. Furthermore, notice that by using the optimal strategy and the high demand-limiting prices, we have reduced the demand peak by 29\% with respect to the constant strategy in Table~\ref{tab:APS}.
Of course, a moderate reduction in peak demand at the expense of large additional production costs may not be desirable. Indeed, the question of optimal distribution of electricity prices for minimizing the production cost is discussed in Case II.

\begin{table}[t]
\renewcommand{\arraystretch}{0.7}
\renewcommand{\tabcolsep}{2pt}
\caption{CASE I: Electricity Bills (or Three Days) and Demand Peaks for 
Different Strategies. Electricity Prices Are from APS.} \vspace{0.1in}
\centering
\begin{tabular}{|c|c|c|}
\hline
Temperature setting  & Electricity bill ($\$$)   & Demand peak ($kW$) \\
\hline
\hline
Optimal (Theorem~\ref{thm:DP}) & \textbf{36.58} & 9.222  \\
\hline
GPOPS~(\cite{gpops})  & 37.03 & 9.155  \\
\hline
4-Setpoint (Theorem~\ref{thm:DP})  & 37.71 & 9.401  \\
\hline
Pre-cooling  & 39.23 &  \textbf{8.803}  \\
\hline
Constant  & 39.42 & 10.462   \\
\hline
\end{tabular} 
\label{tab:APS}
\end{table}

\begin{table}[h] 
\renewcommand{\arraystretch}{0.8}
\renewcommand{\tabcolsep}{3.5pt}
\caption{CASE I: Costs of Production (for Three Days) and Demand Peaks for Various Prices and Strategies. Prices Are Non-regulated and SRP's Coefficients of Utility Cost Are: $\tau=$0.00401 \$/(MWh)$^2$, $\nu=$4.54351 \$/(MWh)} \vspace{0.1in}
\centering
\begin{tabular}{|c|c|c|c|c|c|}
\hline
& \multicolumn{1}{c|}{Prices $[p_{\text{off}}, p_{\text{on}}, p_d]$} & \multicolumn{1}{c|}{Demand-limiting} & \multicolumn{1}{c|}{Production cost} & \multicolumn{1}{c|}{Demand peak} \\
\hline
 &    $[0.007,0.010,13.616]$    &  high    & \textbf{46.78}$\$$ (0.086$\frac{\$}{kWh}$)   & \textbf{7.4132} $kW$ \\
\cline{2-5}
\parbox[t]{3mm}{\multirow{-2}{*}{\rotatebox[origin=]{90}{$ \hspace{-0.2in} \text{Optimal}$}}} &    $[0.015,0.045,13.573]$    &  medium  & \textbf{51.56}$\$$ (0.116$\frac{\$}{kWh}$) & \textbf{8.2898} $kW$ \\
\cline{2-5}
 &    $[0.065,0.095,13.473]$    &  low     & \textbf{59.42}$\$$ (0.168$\frac{\$}{kWh}$)   &   9.6749 $kW$ \\
\hline
\end{tabular}

\begin{tabular}{|c|c|c|c|c|c|}
\hline
 & Prices  $[p_{\text{off}}, p_{\text{on}}, p_d]$  & Demand-limiting   & Production cost  & Demand peak   \\
\hline
\cline{2-5}
 &    $[0.007,0.010,13.616]$    &  high    & 53.47$\$$ (0.114$\frac{\$}{kWh}$)   &   8.5914 $kW$ \\
\cline{2-5}
\parbox[t]{3mm}{\multirow{-2}{*}{\rotatebox[origin=]{90}{$ \hspace{-0.15in} \text{4-Setpoint}$}}} &    $[0.015,0.045,13.573]$    &  medium  & 55.19$\$$ (0.130$\frac{\$}{kWh}$)   &  8.910   $kW$ \\
\cline{2-5}
 &    $[0.065,0.095,13.473]$    &  low     & 61.24$\$$ (0.169$\frac{\$}{kWh}$)   &  9.974  $kW$ \\
\hline
\end{tabular}

\begin{tabular}{|c|c|c|c|c|c|}
\hline
 & Prices  $[p_{\text{off}}, p_{\text{on}}, p_d]$  & Demand-limiting   & Production cost  & Demand peak  \\
\hline
 &     $[0.007,0.010,13.616]$    &  high    & 49.53$\$$  (0.109$\frac{\$}{kWh}$)  & 7.9440 $kW$\\
\cline{2-5}
\parbox[t]{3mm}{\multirow{-2}{*}{\rotatebox[origin=]{90}{$ \hspace{-0.2in} \text{GPOPS}$}}}  &    $[0.015,0.045,13.573]$    &  medium  & 56.48$\$$  (0.142$\frac{\$}{kWh}$)  &  9.1486 $kW$\\
\cline{2-5}
 &     $[0.065,0.095,13.473]$    &  low     & 59.19$\$$  (0.159$\frac{\$}{kWh}$)   & 9.6221 $kW$\\
\hline
\end{tabular}

\begin{tabular}{|c|c|c|c|c|c|}
\hline
 &  Prices  $[p_{\text{off}}, p_{\text{on}}, p_d]$  & Demand-limiting   & Production cost  & Demand peak  \\
\hline
 &     $[0.007,0.010,13.616]$    &  high    & 54.75$\$$  (0.116$\frac{\$}{kWh}$)    &  8.8031 $kW$ \\
\cline{2-5}
\parbox[t]{3mm}{\multirow{-2}{*}{\rotatebox[origin=]{90}{$ \hspace{-0.15in} \text{Precooling}$}}} &     $[0.015,0.045,13.573]$    &  medium  & 54.75$\$$  (0.116$\frac{\$}{kWh}$)   & 8.8031 $kW$\\
\cline{2-5}
 &     $[0.065,0.095,13.473]$    &  low     & \textbf{54.75}$\$$  (0.116$\frac{\$}{kWh}$)  & \textbf{8.8031} $kW$\\
\hline
\end{tabular}
\label{tab:scen1} 
\end{table}

\clearpage

\begin{figure}[t]
\vspace{0.2in}
\hspace{-0.6in} \includegraphics[scale=0.52]{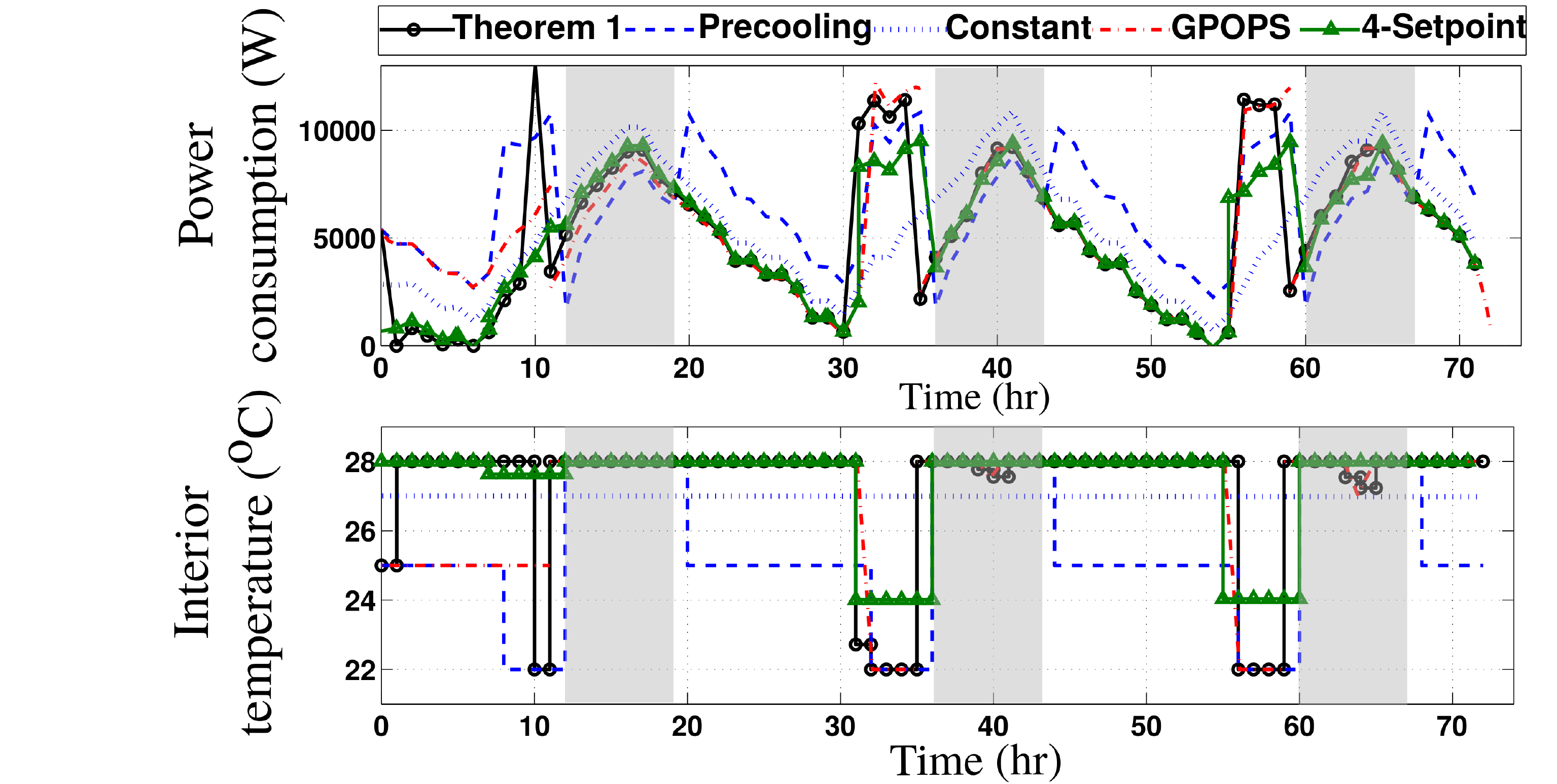} 
\caption{CASE I: Power Consumption and Temperature Settings for Various Programming Strategies Using APS's Rates.} 
\label{fig:APS}
\end{figure}

\begin{figure}[t]
\hspace{0.1in} \includegraphics[scale=0.49]{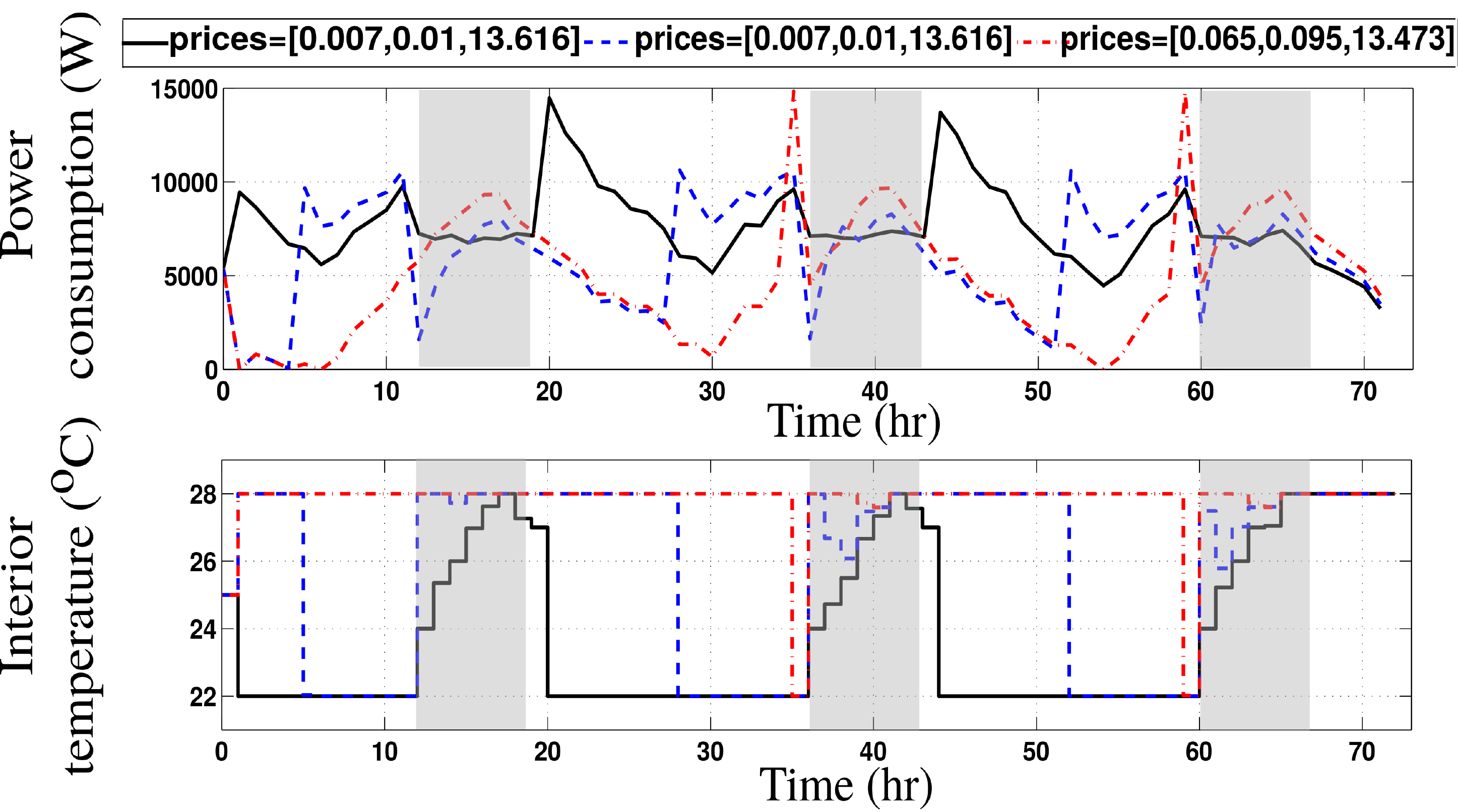} \vspace{0.1in}
\caption{CASE I: Power Consumption and Optimal Temperature Settings for High, Medium and Low Demand Penalties. Shaded Areas Correspond to On-peak Hours.}
\label{fig:scen1} 
\end{figure}

\clearpage

\begin{figure}[t]
\hspace{0.2in} \includegraphics[scale=0.46]{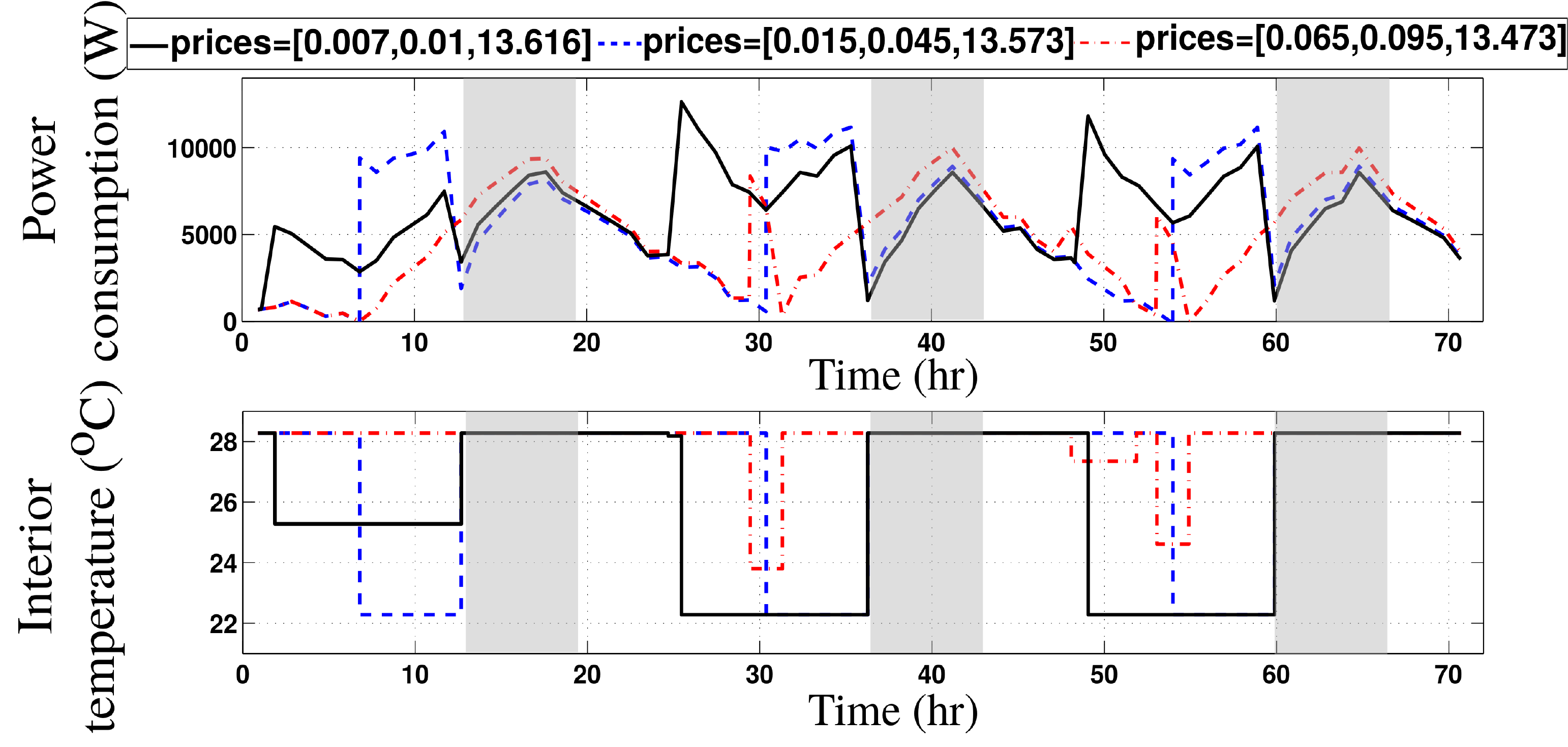}
\caption{CASE I: Power Consumption and Temperature Settings for High, Medium and Low Demand Penalties Using 4-Setpoint Thermostat Programming.} 
\label{fig:scen1.1} \vspace{0.2in}
\end{figure}

\subsection{Optimal Thermostat Programming with Optimal Electricity Prices}
\label{sec:case2}
In this case, we consider the quadratic model of fuel cost defined in Section~\eqref{eq:gen_costQuad}. A typical pricing strategy for SRP and other utilities is to set prices proportional to marginal production costs. SRP estimates the mean marginal fuel cost at $a=0.0814 \$/kWh$ (See~\eqref{eq:gen_cost}). Linearizing our quadratic model of fuel cost and equating to this estimate of the marginal cost yields an estimate of the mean load. Dividing this mean load by the aggregate user defined in Case I yields an estimate of the mean number of users of this class at $N=24,405$.

To compare the marginal pricing strategy with the optimal pricing strategy, we use this mean number of users in the utility optimization problem under the assumption that the building parameters in Section~\ref{sec:validation} represent a single aggregate rational user. The resulting optimal prices, associated production cost, and associated peak demand are listed in Table~\ref{tab:scen2}. For comparison, we also included in Table~\ref{tab:scen2} the production cost and demand peak for the same rational user subject to prices based solely on the marginal costs, where these prices are scaled so that revenue equals costs. In other words, we solved~\eqref{eq:utility_problem} for $\lambda = 1$, meaning that the regulated utility does not make any profit from generation, transmission and distribution.


From Table~\ref{tab:scen2}, optimal pricing results in a slight reduction (\$82,000) in production costs. The discrepancy between optimal prices and marginal costs may be surprising given that both the user and utility are trying to minimize the cost of electricity. However, there are several reasons for this difference. The first and most obvious reason is that the price structure for the user and the cost structure for the utility are not perfectly aligned. In the first place, the utility has a quadratic in consumption model for costs, where the user has a linear model. The second misalignment is that the capacity cost for the utility is calculated as a maximum over 24 hours and the demand charge for the user is calculated only during peak hours.
An additional reason that marginal costs will not always be optimal prices is nonlinearity of the cost function and heterogeneity of the users. To see this, suppose that cost function exactly equaled the price function for each user. The problem in this case is that the sum of the individual bills is NOT equal to the total production cost. This can be seen in the demand charge, where $\sup_x f(x) +\sup_x g(x) \neq \sup_x (f(x)+g(x))$.\\

\begin{table}[h]
\renewcommand{\arraystretch}{0.9}
\renewcommand{\tabcolsep}{3.5pt}
\caption{CASE II: Production Costs (for Three Days) and Demand Peaks Associated with Regulated Optimal Electricity Prices (Calculated by Algorithm~\ref{alg:utility}) and SRP's Electricity Prices. SRP's Marginal Costs: $a=0.0814  \frac{\$}{kWh}, b=59.76 \frac{\$}{kW}$} \vspace{0.1in}
\centering
\begin{tabular}{|c|c|c|c|}
\hline
Strategy &  $[p_{\text{off}} (\frac{\$}{kWh}), p_{\text{on}} (\frac{\$}{kWh}), p_d (\frac{\$}{kW})]$ & Production cost  &  Demand peak \\
 \hline
 Optimal & $[0.0564,0.0667, 51.1859]$  &  \textbf{1,595,309} $\$$   &  195.607 $MW$  \\   
 \hline
SRP & $[0.0668,0.0668,49.0018]$  &  1,677,516 $\$$   &  211.79 $MW$ \\ 
\hline
\end{tabular}
\label{tab:scen2}
\end{table}

\subsection{Optimal Thermostat Programming for Solar Customers - Impact of Distributed Solar Generation on Non-solar Customers}
\label{sec:case3}
We now evaluate the impact of solar power on the bills of non-solar users in a regulated electricity market. We consider a network consisting of a utility company and two aggregate users - one solar and one non-solar. For the non-solar user, we define optimal thermostat programming as in~\eqref{eq:user_discrete}. For the solar user, the optimal thermostat programming problem is as defined in~\eqref{eq:user_discrete}, where we have now redefined the consumption function as
\begin{equation}
g(k,u_k,T_1^k) :=  \dfrac{T^k_{e} - u_k}{R_e} + 2 \, C_{in} \dfrac{T_1^k - u_k}{\Delta x} - Q_k,
\label{eq:g_s}
\end{equation}
where $Q_k$ is the power supplied locally by solar panels. We assume that solar penetration is 50\%, so that both aggregate users contribute equally to revenue and costs to the utility. For $Q_k$, we used data generated on June 4-7 from a typical 13kW south-facing rooftop PV array in Scottsdale, AZ. We applied Algorithm~\ref{alg:utility} separately to each user, while considering~\eqref{eq:gen_cost} as the utility cost model. The results are presented in Table~\ref{tab:solar}. For comparison, we have also included optimal prices, prorated electricity bills over three days and demand peaks of both users. From Table~\ref{tab:solar} we observe that the difference between the electricity bill of a non-solar user in a single-user network and the  bill of a non-solar user in a two-user network (solar and non-solar) is $<$ 2\%. This increase in bill for the solar user is $<$ 8\%. The utility-generated power, solar-generated power and optimal temperature settings are shown in Figure~\ref{fig:solar}. 

\clearpage

\begin{figure}[t] 
\hspace{0.4in} \includegraphics[scale=0.52]{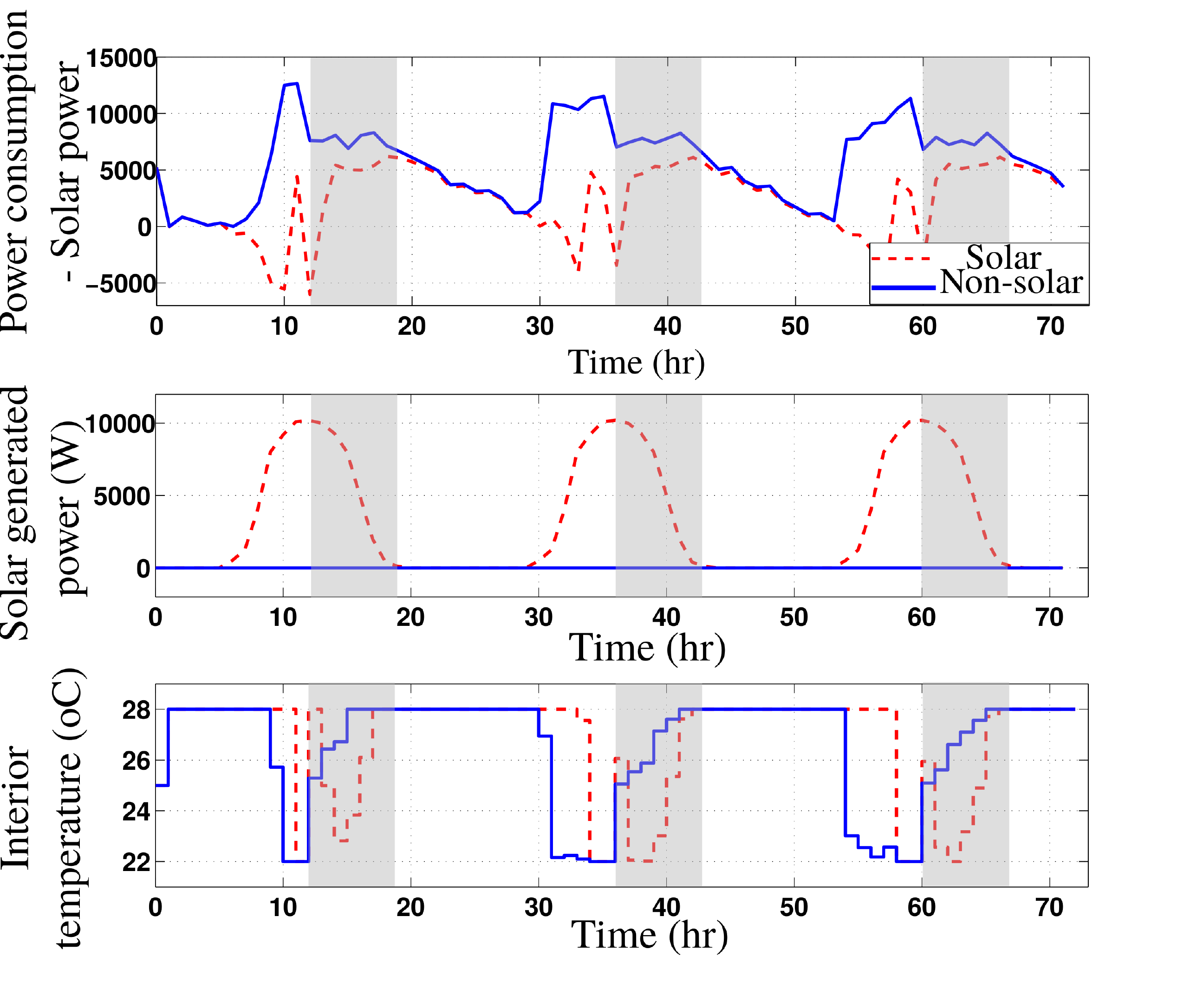} 
\caption{CASE III: Power Consumption, Solar Generated Power and Optimal Temperature Settings for the Non-solar and Solar Users.} 
\label{fig:solar}
\end{figure}

\begin{table}[h]
\renewcommand{\tabcolsep}{2pt}
\caption{CASE III: Optimal Electricity Prices, Bills (for Three Days) and Demand Peaks for Various Customers.
Marginal osts from SRP: $a=0.0814  \frac{\$}{kWh}, b=59.76 \frac{\$}{kW}$} \vspace{0.1in}
\centering
\begin{tabular}{|c|c|c|c|}
\hline
Customers &  $[p^\star_{\text{off}} (\frac{\$}{kWh}), p^\star_{\text{on}} (\frac{\$}{kWh}), p^\star_d (\frac{\$}{kW})]$  &  Elect. Bill &  Demand peak \\
\hline
\hline
Solar \&     &    \multirow{2}{*}{$[0.089,0.115,51.988]$}    & $\$$ 50.052   & 6.1947  $kW$ \\
 Non-solar &        &  $\$$  84.717   & 8.6787 $kW$ \\
\hline
Single Non-solar  &    $[0.081,0.108,54.004]$    & $\$$ 83.333  & 8.3008 $kW$ \\
\hline
Single Solar &    $[0.088,0.118,58.556]$      & $\$$ 54.311   &   6.1916 $kW$ \\
\hline
\end{tabular}
\label{tab:solar}
\end{table}

\chapter{SUMMARY, CONCLUSIONS AND FUTURE DIRECTIONS OF OUR RESEARCH}
\label{sec:conclusion}

\section{Summary and Conclusions}
Thanks to the development of converse Lyapunov theory, a broad class of problems in stability analysis and control can be formulated as optimization of polynomials. In this dissertation, we focus on design and implementation of parallel algorithms for optimization of polynomials. In Chapter~\ref{chp:introduction}, we provide a brief overview of the existing state-of-the-art algorithms for solving polynomial optimization and optimization of polynomials problems. As our contribution, we chose to design our optimization algorithms based on two well-known results in algebraic geometry, known as Polya's theorem and Handelman's theorem. Our motivation behind this choice is that these theorems define structured parameterizations\footnote{A detailed discussion on the structure of these parameterizations can be found in
Section~\ref{sec:SDPelements_TAC}} for positive polynomials - a property that our parallel algorithms exploit to achieve near-linear speed-up and scalability.

In Chapter~\ref{chp:background}, we discuss how variants of Polya's theorem, Handelman's theorem and the Positivstellensatz results can be applied to optimization of polynomials defined over various compact sets, e.g., simplex, hypercube, convex polytopes and semi-algebraic sets. We show that applying these theorems to an optimization of polynomials problem yields convex optimization problems in the form of LPs and/or SDPs. By solving these LPs and SDPs, one can find asymptotic solutions to the optimization of polynomials problems (as defined in~\eqref{eq:optimization_of_polynomials2}). Subsequently, by combining \\

\pagebreak

\hspace{-0.3in}  this method with a branch-and-bound algorithm, one can find solutions to polynomial optimization problems (as we define in~\eqref{eq:polynomial_optimization2}).

In Chapter~\ref{chp:convex_optim}, we briefly review Newton-based descent algorithms for constrained optimization of convex functions. In particular, we discuss a state-of-the-art primal-dual interior-point path-following algorithm~(\cite{helmberg2005interior}) for solving semi-definite programs. We later decentralize the computation of the search directions of this algorithm to design a parallel SDP solver in Chapter~\ref{chp:linear}.

In Chapter~\ref{chp:linear}, we propose a parallel-computing approach to stability analysis of large-scale linear systems of the form $\dot x(t) = A(\alpha)x(t)$, where $A$ is a real-valued polynomial, $\alpha \in \Delta^l \subset \mathbb{R}^l$ and $x \in \mathbb{R}^n$. This approach is based on mapping the structure of the SDPs associated with Polya's theorem to a parallel computing environment. We first design a parallel set-up algorithm with no centralized computation to construct the SDPs associated with Polya's theorem. We then show that by choosing a block-diagonal starting point for the SDP algorithm in~\cite{helmberg2005interior}, the primal and dual search directions will preserve their block-diagonal structure at every iteration. By exploiting this property, we decentralize the computation of the search directions - the most computationally expensive step of an SDP algorithm. The result is a parallel algorithm which under certain conditions, can solve the NP-hard problem of robust stability of linear systems at the same per-core computational cost as solving the Lyapunov inequality for a linear system with no parametric uncertainty. Theoretical and experimental results verify near-linear speed-up and scalability of our algorithm for up to 200 processors. In particular, our numerical tests on cluster computers show that our MPI/C++ implementation of the SDP algorithm outperforms the existing state-of-the-art SDP solvers such as SDPARA~(\cite{sdpara}) in terms of speed-up. Moreover, our experimental tests on a mid-size (9-node) Linux-based cluster computer demonstrate the ability of our algorithm in performing robust stability analysis of systems with 100+ states and several uncertain parameters. A comprehensive complexity analysis of both set-up and solver algorithms can be found in Sections~\ref{sec:comp_analysis_setupTAC} and~\ref{sec:comp_analysis_SDP}.

In Chapter~\ref{chp:multisim}, we further extend our analysis to consider linear systems with uncertain parameters inside hypercubes. We propose an extended version of Polya's theorem for positivity over a multi-simplex (Cartesian product of standard simplicies). We claim that every polynomial defined over a hypercube has an equivalent homogeneous representation over the multi-simplex. Therefore, our the multi-simplex version of Polya's theorem can be used to verify positivity over hypercubes. In the next step, we generalize our parallel set-up algorithm from Chapter~\ref{chp:linear} to construct the SDPs associated with our multi-simplex version of Polya's theorem. Our complexity analysis shows that for sufficiently large number of available processors, at each Polya's iteration, the per processor computation and communication cost of the algorithm scales polynomially with the number of states and uncertain parameters. Through numerical experiments on a large cluster computer, we show that the algorithm can achieve a near-perfect speed-up.

In Chapter~\ref{chp:Nonlinear}, we extend our approach to consider optimization of polynomials defined over more complicated geometries such as  convex polytopes. Specifically, we apply Handelman's theorem to construct piecewise polynomial Lyapunov functions for nonlinear dynamical systems defined by polynomial vector fields. Unfortunately, neither Polya's theorem nor Handelman's theorem can readily certify non-negativity of polynomials which have zeros in the interior of a simplex/polytope. Our proposed solution to this problem is to decompose the domain of analysis (in this case a polytope) into several convex sub-polytopes with a common vertex at the equilibrium. Then, by using Handelman's theorem, we derive a new set of affine feasibility conditions - solvable by linear programming - on each sub-polytope. Any solution to this feasibility problem yields a piecewise polynomial Lyapunov function on the entire polytope. In a computational complexity analysis, we show that
for large number of states and large degrees of the Lyapunov function, the complexity of the proposed feasibility problem is less than the complexity of certain semi-definite programs associated with alternative methods based on Sum-of-Squares and Polya's theorem.

Finally, in chapter~\ref{chp:DP}, we address a real-world optimization problem  in energy planning and smart grid control. We consider the coupled problems of optimal control of HVAC systems for residential customers and optimal pricing of electricity by utility companies. Our framework consists of multiple users (solar and non-solar customers) and a single regulated utility company. The utility company sets prices for the users, who pay for both total energy consumed (\$/kWh, including peak and off-peak rates) and the peak rate of consumption in a month (a demand charge) (\$/kW). The cost of electricity for the utility company is based on a combination of capacity costs (\$/kW) and fuel costs (\$/kWh). On the demand side, the users minimize the amount paid for electricity while staying within a pre-defined temperature range. The users have access to energy storage in the form of thermal capacitance of interior structures. Meanwhile, the utility sets prices designed to minimize the total cost of generation, transmission and distribution of electricity. To solve the user-level problem, we use a variant of dynamic programming. To solve the utility-level problem, we use the Nelder-Mead simplex algorithm coupled with our dynamic programming code - yielding optimal on-peak, off-peak and demand prices. We then apply our algorithms to a variety of scenarios in which show that: 1) Thermal storage and optimal thermostat programming can reduce electricity bills using current rates from utilities Arizona Public Service (APS) and Salt River Project (SRP). 2) Our optimal pricing can reduce the total cost to the utility companies. 3) In the presence of demand charges, \\

\pagebreak
\hspace{-0.3in} the impact of distributed solar generation on the electricity bills of the non-solar users is not significant ($\mathbf{< 2\%}$).

\section{Future Directions of Our Research}

In the following sections, we discuss how the proposed algorithms in this dissertation can be extended to solve three well-known problems in controls: 1) Robust stability analysis of nonlinear systems; 2) Synthesis of parameter-varying $\mathcal{H}_{\infty}$-optimal controller; 3) Computing value functions in approximate dynamic programming problems.

\subsection{A Parallel Algorithm for Nonlinear Stability Analysis Using Polya's Theorem}  

Consider the problem of local stability analysis of a nonlinear system of the form
\begin{equation}
\dot{x}(t) = A(x,\alpha) x(t),
\label{eq:sys_nonlin2}
\end{equation}
where $A: \mathbb{R}^n \times \mathbb{R}^m \rightarrow \mathbb{R}^{n \times n}$ is a matrix-valued polynomial and $A(0,0) \neq 0$. From converse Lyapunov theory, this problem can be expressed as a search for a  polynomial $V: \mathbb{R}^n \times \mathbb{R}^m \rightarrow \mathbb{R}$ which satisfies the Lyapunov inequalities
\[
W_1(x,\alpha) \leq V(x,\alpha) \leq W_2(x,\alpha)
\]
\[
\langle \nabla_x V,f \rangle \leq -W_3(x,\alpha)
\]
for all $x, \alpha \in \Omega \subset \mathbb{R}$, where $0 \in \Omega$.
However, as we discussed in Section~\ref{sec:intro_nonlin}, Polya's theorem (simplex and multi-simplex versions) cannot certify positivity of polynomials which have zeros in the interior of the unit- and/or multi-simplex. Moreover, if $F(x)$ in~\eqref{eq:fz_multi} has a zero in the interior of $\Phi^n$, then any multi-homogeneous polynomial $P(x,y)$ that satisfies~\eqref{eq:fz_multi} has a zero in the interior of the multi-simplex $\Delta^2 \times \cdots \times \Delta^2$ - hence cannot be parameterized by Polya's theorem. One way to enforce the condition $V(\textbf{0},\textbf{0}) = 0$ is to search for coefficients of a matrix-valued polynomial $P$ which defines a Lyapunov function of the form $V(x,\alpha)=x^T P(x,\alpha) x$. It can be shown that $V(x,\alpha)=x^T P(x,\alpha) x$ is a Lyapunov function for System~\eqref{eq:sys_nonlin2} if and only if $\gamma^*$ in the following optimization of polynomials problem is positive.
\begin{align}
&\gamma^* =  \max_{\gamma \in \mathbb{R,\alpha}, P \in \mathbb{R}[x,\alpha]} \;\; \gamma \nonumber \\
& \text{subject to}
\renewcommand\arraystretch{0.8} 
\begin{bmatrix}
P(x,\alpha) & 0 \\ 
0 & -Q(x,\alpha)
\end{bmatrix} - \gamma I \geq 0  \;\; \text{ for all } x \in \Phi^n, \, \alpha \in \Phi^m,
\label{eq:OOP_nonlin2}
\end{align}
where 
\begin{align*}
 Q(x,\alpha)=
 A^T(x,\alpha) P(x,\alpha) + & P(x,\alpha) A(x,\alpha) \\
 & + \frac{1}{2} \left( A^T(x,\alpha) 
 \renewcommand\arraystretch{0.8}
 \begin{bmatrix}
x^T \frac{\partial P(x,\alpha)}{\partial x_1} \\ 
\vdots \\ 
x^T \frac{\partial P(x,\alpha)}{\partial x_n}
\end{bmatrix}
+
 \begin{bmatrix}
x^T \frac{\partial P(x,\alpha)}{\partial x_1} \\ 
\vdots \\ 
x^T \frac{\partial P(x,\alpha)}{\partial x_n}
\end{bmatrix}^T A(x,\alpha)  \right).
\end{align*}
As we discussed in Section~\ref{sec:optim_hypercube}, by applying bisection search on $\gamma$ and using the multi-simplex version of Polya's theorem (Theorem~\ref{thm:polya_multi-simplex2}) as a test for feasibility of Constraint~\eqref{eq:OOP_nonlin2}, we can compute lower bounds on $\gamma^*$. Suppose there exists a matrix-valued multi-homogeneous  polynomial $S$ of degree vector $d_s \in \mathbb{N}^n$ ($d_{s}=[d_{s_1}, \cdots,d_{s_n},d_{s_{n+1}}, \cdots, d_{s_{n+m}}]$, where for $i\in \{1,\cdots,n\},\, d_{s_i}$ is the degree of $y_i$ and for $i\in \{n+1,\cdots,m\},\, d_{s_i}$ is the degree of $\beta_i$) such that
\begin{align*}
& \{ P(x,\alpha)\in \mathbb{S}^n: x \in \Phi^n, \, \alpha \in \Phi^m \} = \\
& \{ S(y,z,\beta,\eta) \in \mathbb{S}^n: (y_i,z_i),(\beta_j,\eta_j) \in \Delta^2, i=1,\cdots,n, \text{and } j = 1, \cdots, m \}.
\end{align*}
Likewise, suppose there exist matrix-valued multi-homogeneous polynomials $B$ and $C$ of degree vectors $d_b \in \mathbb{N}^n$ and $d_c=d_s \in \mathbb{N}^n$ such that
\begin{align*}
& \{ A(x,\alpha) \in \mathbb{R}^{n \times n}: x \in \Phi^n \} = \\
& \{ B(y,z, \beta, \eta) \in \mathbb{R}^{n \times n}: (y_i,z_i),(\beta_j,\eta_j) \in \Delta^2, i=1,\cdots,n, \, \text{and } j=1, \cdots, m \} 
\end{align*}
and
\begin{align*}
&\left\lbrace
 \renewcommand\arraystretch{0.8}
  \begin{bmatrix}
 \frac{\partial P(x,\alpha)}{\partial x_1} x, \cdots , \frac{\partial P(x,\alpha)}{\partial x_n}  x
\end{bmatrix} \in \mathbb{R}^{n \times n}: x \in \Phi^n, \, \text{and } \alpha \in \Phi^m \right\rbrace = \\ & \left\lbrace C(y,z, \beta, \eta) \in \mathbb{R}^{n \times n}: (y_i,z_i),(\beta_j,\eta_j) \in \Delta^2, i=1,\cdots,n, \, \text{and } j=1, \cdots, m \right\rbrace.
\end{align*}
Given $\gamma \in \mathbb{R}$, it follows from Theorem~\ref{thm:polya_multi-simplex2} that the inequality condition in~\eqref{eq:OOP_nonlin2} holds for all $\alpha \in \Phi^l$ if there exist $e \geq 0$ such that 
\begin{equation}
\left( \prod_{i=1}^n (y_i+z_i)^e \cdot \prod_{j=1}^m (\beta_j+\eta_j)^e \right) \left( S(y,z,\beta,\eta) - \gamma I \left( \prod_{i=1}^n (y_i+z_i)^{d_{p_i}} \cdot \prod_{j=1}^m (\beta_i+\eta_i)^{\hat{d}_{p_i}} \right) \right)
\label{eq:product1}
\end{equation}
and
\begin{align}
& \left( \prod_{i=1}^n (y_i+z_i)^e \cdot \prod_{j=1}^m (\beta_j+\eta_j)^e  \right) \left( B^T(y,z,\beta,\eta) S(y,z,\beta,\eta) + S(y,z,\beta,\eta) B(y,z,\beta,\eta) \right. \nonumber \\
 & \hspace{1.68in}  \left. + \frac{1}{2} \left(   B^T(y,z,\beta,\eta) C^T(y,z,\beta,\eta) +
C(y,z,\beta,\eta) B(y,z,\beta,\eta)  \right) \right. \nonumber \\
& \hspace{2.6in} \left. - \gamma I  \left( \prod_{i=1}^n (y_i+z_i)^{d_{q_i}} \cdot \prod_{j=1}^m (\beta_i+\eta_i)^{\hat{d}_{q_i}} \right)  \right)
\label{eq:product2}
\end{align}
have all positive coefficients, where $d_{p_i}$ and $\hat{d}_{p_i}$ are the degrees of $x_i$ and $\alpha_i$ in $P(x,\alpha)$, and $d_{q_i}$ and $\hat{d}_{q_i}$ are the degrees of $x_i$ and $\alpha_i$ in $Q(x,\alpha)$. Now, let $S,B$ and $C$ be of  the following forms. 
\begin{equation}
S(y,z,\beta,\eta) = \sum_{ \substack{ h,g \in \mathbb{N}^{n+m} \\ h+g=d_s }} S_{h,g} y_1^{h_1} z_1^{g_1} \cdots y_n^{h_n} z_n^{g_n}\beta_1^{h_{n+1}}\eta_1^{g_{n+1}} \cdots \beta_{m}^{h_{n+m}}\eta_m^{g_{n+m}}
\label{eq:S_polya_nonlin}
\end{equation}
\begin{equation}
B(y,z,\beta,\eta) = \sum_{ \substack{ h,g \in \mathbb{N}^{n+m} \\ h+g=d_b }} B_{h,g} y_1^{h_1} z_1^{g_1} \cdots y_n^{h_n} z_n^{g_n}\beta_1^{h_{n+1}}\eta_1^{g_{n+1}} \cdots \beta_{m}^{h_{n+m}}\eta_m^{g_{n+m}}
\label{eq:B_polya_nonlin}
\end{equation}
\begin{equation}
C(y,z,\beta,\eta) = \sum_{ \substack{ h,g \in \mathbb{N}^{n+m} \\ h+g=d_c }} C_{h,g} y_1^{h_1} z_1^{g_1} \cdots y_n^{h_n} z_n^{g_n}\beta_1^{h_{n+1}}\eta_1^{g_{n+1}} \cdots \beta_{m}^{h_{n+m}}\eta_m^{g_{n+m}}
\label{eq:C_polya_nonlin}
\end{equation}
Note that the coefficients $C_{h,g}$ can be written as linear combinations of $S_{h,g}$. For brevity we have denoted $C_{h,g}(S_{h,g})$ as $C_{h,g}$. By combining~\eqref{eq:S_polya_nonlin},~\eqref{eq:B_polya_nonlin} and~\eqref{eq:C_polya_nonlin} with~\eqref{eq:product1} and~\eqref{eq:product2} it follows that for a given $\gamma \in \mathbb{R}$, the inequality condition in~\eqref{eq:OOP_nonlin2} holds for all $\alpha \in \Phi^n$ if there exist some $e \geq 0$ such that
\begin{align}
& \sum_{\substack{ h,g \in \mathbb{N}^{n+m} \\ h+g=d_s }} f_{\{q,r\}, \{h,g \}} S_{h,g} > 0 \;\; \text{ for all } \; q,r \in \mathbb{N}^{n+m}: q+r=d_s+2 \, e \cdot \mathbf{1}_{n+m}
\label{eq:polya_nonlin_conditions1}  \\
& \text{ and} \nonumber \\
& \sum_{\substack{ h,g \in \mathbb{N}^{n+m} \\ h+g=d_s }} \hspace*{-0.09in}  M_{\{u,v\},\{h,g\}}^T S_{h,g} + S_{h,g} M_{\{u,v\},\{h,g\}}  +  N_{\{u,v\},\{h,g\}}^T C^T_{h,g} + C_{h,g} N_{\{u,v\},\{h,g\}} < 0
\label{eq:polya_nonlin_conditions2}
\end{align}
for all  $u,v \in \mathbb{N}^{n+m}: u+v=d_s+d_b+2 \, e \cdot \mathbf{1}_{n+m}$,
where we define $f_{\{ q,r \},\{h,g\}}$ to be the coefficient of
\[
S_{h,g} \, y_1^{q_1}z_1^{r_1} \cdots y_n^{q_n}z_n^{r_n} \beta_1^{q_{n+1}}\eta_1^{r_{n+1}} \cdots \beta_m^{q_{n+m}}\eta_m^{r_{n+m}}
\]
after substituting~\eqref{eq:S_polya_nonlin} into~\eqref{eq:product1}. Likewise, we define $M_{\{u,v\},\{h,g\}}$ to be the coefficient of
\[
S_{h,g} \, y_1^{u_1}z_1^{v_1} \cdots y_n^{u_n}z_n^{v_n} \beta_1^{u_{n+1}}\eta_1^{v_{n+1}} \cdots \beta_m^{u_{n+m}}\eta_m^{v_{n+m}}
\]
and $N_{\{u,v\},\{h,g\}}$ to be the coefficient of
\[
C_{h,g} \, y_1^{u_1}z_1^{v_1} \cdots y_n^{u_n}z_n^{v_n} \beta_1^{u_{n+1}}\eta_1^{v_{n+1}} \cdots \beta_m^{u_{n+m}}\eta_m^{v_{n+m}}
\] 
after substituting~\eqref{eq:B_polya_nonlin} and~\eqref{eq:C_polya_nonlin} into~\eqref{eq:product2}. For any $\gamma \in \mathbb{R}$, if there exist $e \geq 0$ and $\{ S_{h,g} \}$ such that Conditions~\eqref{eq:polya_nonlin_conditions1} and~\eqref{eq:polya_nonlin_conditions2} hold, then $\gamma$ is a lower bound for $\gamma^*$ as defined in~\eqref{eq:OOP_nonlin2}. Furthermore, if $\gamma$ is positive, then origin is an asymptotically stable equilibrium for System~\eqref{eq:sys_nonlin2}.
Fortunately, Conditions~\eqref{eq:polya_nonlin_conditions1} and~\eqref{eq:polya_nonlin_conditions2} form an SDP with a block-diagonal structure - hence an algorithm similar to Algorithm 7 can be developed to set-up the SDP in parallel. Furthermore, our parallel SDP solver in Section~\ref{sec:SDPSOLVER} can be used to efficiently solve the SDP.

\subsection{Parallel Computation for Parameter-varying $\mathcal{H}_{\infty}$-optimal Control Synthesis}  
  
\hspace*{-0.25in} Algorithm~\ref{alg:setup} can be generalized to consider a more general class of feasibility problems,~i.e.,
\[
\sum_{i=1}^N \left( A_i(\alpha)X(\alpha)B_i(\alpha) + B^T_i(\alpha)X(\alpha)A^T_i(\alpha) + R_i(\alpha) \right) < -\gamma I
\quad \text{for all } \;\alpha \in \Delta^l,
\]
where $A_i, B_i$ and $R_i$ are polynomials. Formulations such as this can be used to solve a wide variety of problem in systems analysis and control such as $\mathcal{H}_{\infty}$-optimal control synthesis for systems with parametric uncertainty. To see this, consider a plant $G$ with the state-space formulation
\[
\dot{x}(t) = A(\alpha) x(t) + 
\renewcommand\arraystretch{0.8}
\begin{bmatrix}
B_1(\alpha) & B_2(\alpha)
\end{bmatrix}
\begin{bmatrix}
\omega(t) \nonumber \\
u(t)
\end{bmatrix},
\]
\begin{equation}
\renewcommand\arraystretch{0.8}
\begin{bmatrix}
z(t) \\
y(t)
\end{bmatrix} =
\begin{bmatrix}
C_1(\alpha) \\
C_2(\alpha)
\end{bmatrix} x(t) +
\begin{bmatrix}
D_{11}(\alpha) & D_{12}(\alpha) \\
D_{21}(\alpha) & 0
\end{bmatrix}
\begin{bmatrix}
\omega(t) \\
u(t)
\end{bmatrix},
\label{eq:sys_G}
\end{equation}
where $\alpha \in Q \subset \mathbb{R}^l$, $x(t) \in \mathbb{R}^n$, $u(t) \in \mathbb{R}^m$ is the control input, $\omega(t) \in \mathbb{R}^p$ is the external input and $z(t) \in \mathbb{R}^q$ is the external output. Suppose $(A(\alpha),B_2(\alpha))$ is stabilizable and $(C_2(\alpha),A(\alpha))$ is detectable for all $\alpha \in Q$.
 According to~\cite{Gahinet_1994_LMI_hinf} there exists a state feedback gain $K(\alpha) \in \mathbb{R}^{m \times n}$ such that
\[
\| S (G,K(\alpha))  \|_{\mathcal{H}_{\infty}} \leq \gamma, \; \text{for all} \; \alpha \in Q,
\]
if and only if there exist $P(\alpha) > 0$ and $R(\alpha) \in \mathbb{R}^{m \times n}$ such that $K(\alpha)=R(\alpha)P^{-1}(\alpha)$ and
\begin{equation}
\renewcommand\arraystretch{0.8}
\begin{bmatrix}
\begin{bmatrix}
A(\alpha) \hspace*{-0.05in} &  B_2(\alpha)
\end{bmatrix} \hspace*{-0.05in} \begin{bmatrix}
P(\alpha) \\
R(\alpha)
\end{bmatrix}  \hspace*{-0.05in} + \hspace*{-0.05in} \begin{bmatrix}
P(\alpha) \hspace*{-0.05in} & R^T(\alpha)
\end{bmatrix} \hspace*{-0.05in} \begin{bmatrix}
A^T(\alpha) \\
B_2^T(\alpha)
\end{bmatrix} & \star & \star \\
B_1^T(\alpha) & -\gamma I & \star \\
\begin{bmatrix}
C_1(\alpha) & D_{12}(\alpha)
\end{bmatrix}\begin{bmatrix}
P(\alpha) \\
R(\alpha)
\end{bmatrix} & D_{11}(\alpha) & -\gamma I
\end{bmatrix} < 0,
\end{equation}
for all $\alpha \in Q$, where $\gamma > 0$ and $S(G,K(\alpha))$ is the map from the external input $\omega$ to the external output $z$ of the closed loop system with a static full state feedback controller. The symbol $\star$ denotes the symmetric blocks in the matrix inequality. 
To find a solution to the robust $H_\infty$-optimal static state-feedback controller problem with optimal feedback gain $K(\alpha)=P(\alpha)R^{-1}(\alpha)$, one can solve the following optimization of polynomials problem.
\begin{align}
&\gamma^* = \min_{P,R \in \mathbb{R}[\alpha],\gamma \in \mathbb{R}} \;\; \gamma \nonumber \\
& \text{subject to } \nonumber \\
& \hspace*{-0.07in} \renewcommand\arraystretch{0.7}
\begin{bmatrix}
-P(\alpha) & \star  & \hspace*{-0.1in} \star  & \star \\
0 & \hspace*{-0.05in}
\begin{bmatrix}
 A(\alpha) \hspace*{-0.09in} &  B_2(\alpha)
\end{bmatrix} \hspace*{-0.05in} \begin{bmatrix}
P(\alpha) \\
R(\alpha)
\end{bmatrix}  \hspace*{-0.05in} + \hspace*{-0.05in} \begin{bmatrix}
P(\alpha) \hspace*{-0.05in} & R^T(\alpha)
\end{bmatrix} \hspace*{-0.05in} \begin{bmatrix}
A^T(\alpha) \\
B_2^T(\alpha)
\end{bmatrix} & \hspace*{-0.1in} \star &  \star\\
0 & 
B_1^T(\alpha) & \hspace*{-0.1in} 0 & \star \\
0 & 
\begin{bmatrix}
C_1(\alpha) & D_{12}(\alpha)
\end{bmatrix}
\begin{bmatrix}
P(\alpha) \\
R(\alpha)
\end{bmatrix} & \hspace*{-0.1in} D_{11}(\alpha) & 0
\end{bmatrix} 
& \hspace*{-1.1in} -\gamma 
\renewcommand\arraystretch{0.7}
\begin{bmatrix}
0 & 0 & 0 & 0 \\ 
0 & 0 & 0 & 0 \\ 
0 & 0 & I & 0 \\ 
0 & 0 & 0 & I
\end{bmatrix} < 0 \nonumber \\
& \hspace*{4.5in} \text{ for all }  \alpha \in Q.
\label{eq:hinf_ineq}
\end{align}
In Problem~\eqref{eq:hinf_ineq}, if $Q = \Delta^l$ as defined in~\eqref{eq:simplex}, then we can apply Polya's theorem as described in  Section~\ref{sec:optim_simplex} to find a $\gamma \leq \gamma^*$ and $P$ and $R$ which satisfy the inequality in~\eqref{eq:hinf_ineq}. Suppose $P, A, B_1, B_2, C_1, D_{11}$ and $D_{12}$ are homogeneous polynomials (otherwise use the procedure in Section~\ref{sec:notation_simplex} to homogenize them). Let
\[
F(P(\alpha), R(\alpha)) :=
\renewcommand\arraystretch{0.8}
 \begin{bmatrix}
-P(\alpha) & \star  & \hspace*{-0.1in} \star  & \star \\
0 & \hspace*{-0.05in}
\begin{bmatrix}
 A(\alpha) \hspace*{-0.09in} &  B_2(\alpha)
\end{bmatrix} \hspace*{-0.05in}
\renewcommand\arraystretch{0.8}
 \begin{bmatrix}
P(\alpha) \\
R(\alpha)
\end{bmatrix}  \hspace*{-0.05in} + \hspace*{-0.05in} \begin{bmatrix}
P(\alpha) \hspace*{-0.05in} & R^T(\alpha)
\end{bmatrix} \hspace*{-0.05in} \begin{bmatrix}
A^T(\alpha) \\
B_2^T(\alpha)
\end{bmatrix} & \hspace*{-0.1in} \star &  \star\\
0 & 
B_1^T(\alpha) & \hspace*{-0.1in} 0 & \star \\
0 & 
\begin{bmatrix}
C_1(\alpha) & D_{12}(\alpha)
\end{bmatrix}
\begin{bmatrix}
P(\alpha) \\
R(\alpha)
\end{bmatrix} & \hspace*{-0.1in} D_{11}(\alpha) & 0
\end{bmatrix},
\]
\hspace*{-0.08in} and denote the degree of $F$ by $d_f$. Given $\gamma \in \mathbb{R}$, the inequality in~\eqref{eq:hinf_ineq} holds if there exist $e \geq 0$ such that all of the coefficients of the polynomial
\begin{equation}
\left( \sum_{i=1}^l \alpha_i \right)^e \left( F(P(\alpha), R(\alpha)) - \gamma 
\renewcommand\arraystretch{0.8}
\begin{bmatrix}
0 & 0 & 0 & 0 \\ 
0 & 0 & 0 & 0 \\ 
0 & 0 & I & 0 \\ 
0 & 0 & 0 & I
\end{bmatrix} \left( \sum_{i=1}^l \alpha_i \right)^{d_f} \right)
\label{eq:polya_hinf}
\end{equation}
are negative-definite. Let $P$ and $R$ be of the forms
\begin{equation}
P(\alpha) = \sum_{h \in W_{d_p}} P_h \alpha_1^{h_1} \cdots \alpha_l^{h_l}, P_h \in \mathbb{S}^n
\label{eq:PR1}
\end{equation}
and 
\begin{equation}
 R(\alpha) = \sum_{h \in W_{d_r}} R_h \alpha_1^{h_1} \cdots \alpha_l^{h_l}, R_h \in \mathbb{R}^{n \times n},
\label{eq:PR2}
\end{equation}
where $W_{d_p}$ and $W_{d_r}$ are the exponent sets defined in~\eqref{eq:W_d}. By combining~\eqref{eq:PR1} and~\eqref{eq:PR2} with~\eqref{eq:polya_hinf} it follows from Polya's theorem that for a given $\gamma$, the inequality in~\eqref{eq:hinf_ineq} holds, if there exist $e \geq 0$ such that
\begin{equation}
\sum_{h \in W_{d_p}}\left( M_{h,q}^T P_h+P_h M_{h,q} \right) + \sum_{h \in W_{d_r}}\left( N_{h,q}^T R_h^T + R_h N_{h,q} \right) < 0 \;\; \text{ for all }  q \in W_{d_f+e},
\label{eq:LMI_hinf}
\end{equation}
where we define $M_{h,q} \in \mathbb{R}^{n \times n}$ as the coefficient of $P_h \alpha_1^{q_1} \cdots \alpha_l^{q_l}$ after substituting~\eqref{eq:PR1} and~\eqref{eq:PR2} into~\eqref{eq:polya_hinf}. Likewise, $N_{h,q} \in \mathbb{R}^{n \times n}$ is the coefficient of $R_h \alpha_1^{q_1} \cdots \alpha_l^{q_l}$ after substituting~\eqref{eq:PR1} and~\eqref{eq:PR2} into~\eqref{eq:polya_hinf}. For given $\gamma > 0$, if there exist $e \geq 0$ such that LMI~\eqref{eq:LMI_hinf} has a solution, say $P_h, h\in W_{d_p}$ and $R_g, g\in W_{d_r}$, then 
\[
K(\alpha)= \left( \sum_{h \in W_{d_p}}  P_h \alpha_1^{h_1} \cdots \alpha_l^{h_l} \right) \left( \sum_{g \in W_{d_r}}  R_g \alpha_1^{g_1} \cdots \alpha_l^{g_l} \right)^{-1}
\]
is a feedback law of an $H_{\infty}$-suboptimal static state-feedback controller 
for System~\eqref{eq:sys_G}. By performing bisection search on $\gamma$ and solving~\eqref{eq:LMI_hinf} for each $\gamma$ of the bisection, one may find an $H_{\infty}$-optimal controller for System~\eqref{eq:sys_G}.

\subsection{Parallel Computation of Value Functions for Approximate Dynamic Programming}

Consider the discrete-time optimal control problem
\begin{align}
& J^* := \min_{u_k \in U}\;\; \sum_{k=0}^{\infty} \beta^k g(x_k,u_k) \nonumber \\
& \; \text{subject to } \; x_{k+1} = f(x_k, u_k)   && \text{for } k=1,2,3, \cdots \nonumber\\
& \hspace{0.75in} \; x_0 = z, \; x_k \in X  && \text{for } k=1,2,3, \cdots,
\label{eq:J}
\end{align}
where $f: \mathbb{R}^n \times \mathbb{R}^m \rightarrow \mathbb{R}^n$ and  $g: \mathbb{R}^n \times \mathbb{R}^m \rightarrow \mathbb{R}^n$ are given polynomials, $\beta \in (0, 1]$ is a discount factor, $U \subset \mathbb{R}^m, X \subset \mathbb{R}^n$, and $z \in \mathbb{R}^n$ is a given initial condition for the dynamical system. It is well-known that dynamic programming approach~(\cite{bertsekas1995dynamic}) provides sufficient conditions for existence of a solution to the optimal control problem in~\eqref{eq:J}. The key idea underlying dynamic programming is that optimization over-time can often be considered as optimization in stages. In such framework, optimal control is any decision which minimizes the sum of: 1. cost of transition from current stage $k$ to the next stage $k+1$; and 2. cost of all stages subsequent to $k+1$, incurred by the decision made at stage $k$. This is referred to as the
principle of optimality and was first formulated by Bellman~(\cite{bellman1965dynamic}) as
\begin{equation}
J^* = V^*(z) = ( \mathcal{P} V^*)(z) := \inf_{v \in U} \{ g(z,v) + \beta \, V^*(f(z,v)) \} \quad \text{ for all } z \in X.
\label{eq:Bellman}
\end{equation}
The unique solution to Bellman's equation is called the value function - can be thought of as the minimum cost-to-go from the current state. Existence of the value functions is a sufficient condition for existence of an optimal control. In fact, an optimal policy $\mu^*: X \rightarrow U$ can be expressed in terms of the value function $V^*$:
\[
\mu^*(z) = \argmin_{u \in U} \{ g(z,u) + \beta \, V^*(f(z,u)) \}
\]
for any $x_0 \in X$. Thus, Bellman's equation solves the optimal control problem by providing a closed-loop feedback law for \textit{every} initial condition. 

It is shown that the Bellman's operator $\mathcal{P}$ defined in~\eqref{eq:Bellman} possesses the following two properties:
\begin{enumerate}
\item Iteratively applying of Bellman's operator $\mathcal{P}$ on any function $h:X \rightarrow \mathbb{R}$ results in a pointwise convergence to a value function, i.e.,
\begin{equation}
V^*(x) = \lim_{k \rightarrow \infty} (\mathcal{P}^k h )(x) \quad \text{for all } \; x \in X.
\label{eq:V_convergence}
\end{equation}

\item Bellman's operator is \textit{monotonic}: If $V$ satisfies the Bellman's inequality $V(x) \leq (\mathcal{P}V)(x)$ for all $x \in X$, then $V(x) \leq (\mathcal{P}^k V)(x)$ for all $x \in X$ and for any $k \geq 1$.
\end{enumerate}
From these two properties one can conclude that
\[
V \leq \mathcal{P}^k V \; \text{ for some } k \geq 1 \; \Rightarrow \; V \leq V^*.
\]
Unfortunately, for $k>1$, the constraint $V \leq \mathcal{P}^k V$ is non-convex in the coefficients of polynomial $V$. A sufficient condition for $V \leq \mathcal{P}^k V$ is to search for polynomials $V$ and $W_i, \, i = 1, \cdots, k$ such that
\[
V \leq \mathcal{P}W_1, \; W_1 \leq \mathcal{P}W_2, \; \cdots \;, W_{k-1} \leq \mathcal{P}V.
\]
Note that all of these constraints are convex in the coefficients of $V$ and $W_i$.
Let $V$ and $W_i$ be polynomials of forms
\[
V(x)= \sum_{\alpha \in I(d_V)} V_\alpha x^\alpha \text{ and } W_i(x)= \sum_{ \alpha \in I(d_{W_i})} W_{i,\alpha} x^\alpha,
\]
where $I(d) := \{ \alpha \in \mathbb{N}^n : \Vert \alpha \Vert_1 \leq d \}$.
Then, any polynomial $V$ which solves the convex optimization problem 
\begin{align}
&  J_k := \max_{V_\alpha,W_{i,\alpha}} \;\; \sum_{\alpha \in I(d_V)}  V_\alpha z^\alpha \nonumber \\
& \text{subject to}\quad  \sum_{\alpha \in I(d_V)}  V_\alpha x^\alpha \leq \mathcal{P} \sum_{\alpha \in I(d_{W_1})} W_{1,\alpha} x^\alpha && \text{for all } x \in X \nonumber \\
& \hspace{0.82in} \sum_{\alpha \in I(d_{W_i})} W_{i,\alpha} x^\alpha \leq \mathcal{P} \sum_{\alpha \in I(d_{W_{i+1}})} W_{i+1,\alpha} x^\alpha && \text{for all } x \in X  \text{ and } i=1,\cdots, k-2 \nonumber \\
& \hspace{0.82in} \sum_{\alpha \in I(d_{W_{k-1}})} W_{k-1,\alpha} x^\alpha \leq \mathcal{P}\sum_{\alpha \in I(d_{V})}  V_\alpha x^\alpha && \text{for all } x \in X
\label{eq:optim1}
\end{align}
for any initial condition $z \in X$ and some $k \geq 1$, is an under-estimator for the value function $V^*$. Moreover, from monotonicity of $\mathcal{P}$ it follows that
\[
J_1 \leq J_2 \leq \cdots \leq J_k \leq \cdots \leq V^*(z).
\]
 In other words, by increasing $k$, the lower bound $J_k$ defined in~\eqref{eq:optim1} can only improve or remain constant. By substituting for $\mathcal{P}$ in~\eqref{eq:optim1} from~\eqref{eq:Bellman}, removing the infimum and enforcing the constraints of Problem~\ref{eq:optim1} for all control inputs $u \in U$, we get the following optimization of polynomials problem.
\begin{align}
& \max_{V_\alpha,W_{i,\alpha}} \;\; \sum_{\alpha \in I(d_V)}  V_\alpha z^\alpha \nonumber \\
&  \text{subject to}  \sum_{\alpha \in I(d_V)}  V_\alpha x^\alpha \leq g(x,u) + \beta \sum_{\alpha \in I(d_{W_1})} \left( W_{1,\alpha} f(x,u)^{\alpha} \right) && \hspace{-0.05in} \text{for } x \in X, u \in U \nonumber \\
& \hspace{0.6in} \sum_{\alpha \in I(d_{W_i})} W_{i,\alpha} x^\alpha \leq g(x,u) + \beta \hspace{-0.15in} \sum_{\alpha \in I(d_{W_{i+1}})} \hspace{-0.15in} \left( W_{i+1,\alpha} f(x,u)^{\alpha} \right) && \hspace{-0.05in} \text{for } x \in X, u \in U, i \in \Theta \nonumber \\
& \hspace{0.55in} \sum_{\alpha \in I(d_{W_{k-1}})} W_{k-1,\alpha} x^\alpha \leq g(x,u) + \beta \sum_{\alpha \in I(d_V)} \left( V_{\alpha} f(x,u)^{\alpha} \right) && \hspace{-0.05in} \text{for } x \in X, u \in U,
\label{eq:optim2}
\end{align}
where for brevity, we have denoted $f_1(x,u)^{\alpha_1} \cdots f_n(x,u)^{\alpha_n}$ by $f(x,u)^{\alpha}$ and we have defined $\Theta := \{ 1, \cdots, k-2 \}$.

Problem~\eqref{eq:optim2} has some interesting computational properties. Since all of the constraints in this problem have the same structure, if we choose the same degree for $V$ and $W_i$, it is then sufficient to set-up only one of the constraints in order to set-up the entire Problem~\eqref{eq:optim2}. If $X$ and $U$ are simplicies, Algorithm~\ref{alg:setup} can be used to perform Polya's iterations on the constraints of Problem~\eqref{eq:optim2}. The result is a linear program whose solution yields an under-estimator for the value function $V^*$ defined in~\eqref{eq:Bellman}. Likewise, if $X$ and $U$ are hypercubes (or polytopes), then Algorithm 7 (or Algorithm~\ref{alg:Handelman_polyoptim}) can be used to perform Polya's iterations (or Handelman's iterations) on the constraints of Problem~\eqref{eq:optim2}. Another interesting property of Problem~\eqref{eq:optim2} is that increasing the accuracy of the under-estimations (by increasing $k$) amounts to a \textit{linear} growth in the number of decision variables and number of constraints.

\clearpage

{\singlespace
\addcontentsline{toc}{part}{REFERENCES}
\bibliographystyle{asudis}
\bibliography{dis}}

\begin{thebibliography}{161}
\newcommand{\enquote}[1]{``#1''}
\expandafter\ifx\csname natexlab\endcsname\relax\def\natexlab#1{#1}\fi
\expandafter\ifx\csname url\endcsname\relax
  \def\url#1{\texttt{#1}}\fi
\expandafter\ifx\csname urlprefix\endcsname\relax\def\urlprefix{URL }\fi

\bibitem[SRP(2014)]{SRP}
\enquote{Appendix schedule {B} - summary of marginal costs}, available at:
  \url{http://www.srpnet.com/prices/priceprocess/pdfx/Unbundled.pdf}  (2014).

\bibitem[SRP(2015)]{SRP_plan}
\enquote{Standard electric price plans}, {S}alt {R}iver {P}roject Agricultural
  Improvement and Power District Corporate Pricing  (November 2015).

\bibitem[Ackermann \emph{et~al.}(2001)]{ackermann_2001}
Ackermann, J., A.~Bartlett, D.~Kaesbauer, W.~Sienel and R.~Steinhauser,
  \emph{Robust Control: Systems with Uncertain Physical Parameters}
  (Springer-Verlag New York, Inc., Secaucus, NJ, USA, 2001).

\bibitem[Adams and Loustaunau(1994)]{adams_groebner}
Adams, W. and P.~Loustaunau, \emph{An Introduction to Groebner Bases} (American
  Mathematical Society, 1994).

\bibitem[Alizadeh \emph{et~al.}(1998)]{alizadeh}
Alizadeh, F., J.~Haeberly and M.~Overton, \enquote{Primal-dual interior-point
  methods for semidefinite programming: Convergence rates, stability and
  numerical results}, SIAM Journal of Optimization \textbf{8}, 3, 746--768
  (1998).

\bibitem[Amdahl(1967)]{amdahl}
Amdahl, G.~M., \enquote{Validity of the single processor approach to achieving
  large-scale computing capabilities}, No.~30, pp. 483--485 (AFIPS Conference
  Proceedings, 1967).

\bibitem[Arguello-Serrano and Velez-Reyes(1999)]{arguello1999nonlinear}
Arguello-Serrano, B. and M.~Velez-Reyes, \enquote{Nonlinear control of a
  heating, ventilating, and air conditioning system with thermal load
  estimation}, IEEE Transactions on Control Systems Technology \textbf{7}, 1,
  56--63 (1999).

\bibitem[{A}rizona~{P}ublic {S}ervice(2014)]{APS_solar_trend}
{A}rizona~{P}ublic {S}ervice, \enquote{2014 integrated resource plan: Executive
  summary},   (2014).

\bibitem[Artin(1927)]{artin}
Artin, E., \enquote{Uber die zerlegung definiter funktionen in quadra},
  Quadrate, Abh. Math. Sem. Univ. Hamburg \textbf{5}, 85--99 (1927).

\bibitem[Barmish and DeMarco(1986)]{affine_orig}
Barmish, B.~R. and C.~L. DeMarco, \enquote{A new method for improvement of
  robustness bounds for linear state equations},  (in Proceedings Conf. Inform.
  Sci. Syst. Princeton University, 1986).

\bibitem[Bellman and Kalaba(1965)]{bellman1965dynamic}
Bellman, R. and R.~E. Kalaba, \emph{Dynamic programming and modern control
  theory} (Academic Press New York, 1965).

\bibitem[Bellman and Dreyfus(1962)]{bellman}
Bellman, R.~E. and S.~E. Dreyfus, \emph{Applied Dynamic Programming} (Princeton
  University Press, 1962).

\bibitem[Ben-Tal and Nemirovski(1998)]{Np_hard}
Ben-Tal, A. and A.~Nemirovski, \enquote{Robust convex optimization}, Math.
  Operat. Res. \textbf{23}, 4, 769--805 (1998).

\bibitem[Benson \emph{et~al.}(2000)]{benson2000solving}
Benson, S., Y.~Ye and X.~Zhang, \enquote{Solving large-scale sparse
  semidefinite programs for combinatorial optimization}, SIAM Journal on
  Optimization \textbf{10}, 2, 443--461 (2000).

\bibitem[Benson(2001)]{dual_scaling}
Benson, S.~J., \enquote{{DSDP}3: Dual scaling algorithm for general positive
  semidefinite programs}, Preprint ANL/MCS-P851-1000, Argonne National Labs
  (2001).

\bibitem[Benson \emph{et~al.}(1998)]{primal}
Benson, S.~J., Y.~Ye and X.~Zhang, \enquote{Solving large-scale sparse
  semidefinite programs for combinatorial optimization}, SIAM Journal on
  Optimization \textbf{10}, 443--461 (1998).

\bibitem[Bernstein(1915)]{bernstein_1915}
Bernstein, S., \enquote{Sur la repr sentation des polynomes positif}, Soobshch.
  Har'k. Mat. Obshch. \textbf{2}, 14, 227--228 (1915).

\bibitem[Bertsekas \emph{et~al.}(1995)]{bertsekas1995dynamic}
Bertsekas, D.~P., D.~P. Bertsekas, D.~P. Bertsekas and D.~P. Bertsekas,
  \emph{Dynamic programming and optimal control}, vol.~1 (Athena Scientific
  Belmont, MA, 1995).

\bibitem[Bhattacharyya \emph{et~al.}(1995)]{bhattacharyya_1995}
Bhattacharyya, S.~P., H.~Chapellat and L.~H. Keel, \emph{Robust Control: The
  Parametric Approach} (Prentice Hall, 1995).

\bibitem[Bliman(2004{\natexlab{a}})]{bliman2004convex}
Bliman, P.-A., \enquote{A convex approach to robust stability for linear
  systems with uncertain scalar parameters}, SIAM Journal on Control and
  Optimization \textbf{42}, 6, 2016--2042 (2004{\natexlab{a}}).

\bibitem[Bliman(2004{\natexlab{b}})]{Bliman_existence}
Bliman, P.~A., \enquote{An existence result for polynomial solutions of
  parameter dependent {LMI}s}, Systems \& Control Letters , 3-4, 165--169
  (2004{\natexlab{b}}).

\bibitem[Bliman \emph{et~al.}(2006)]{bliman2006existence}
Bliman, P.-A., R.~Oliveira, V.~Montagner and P.~Peres, \enquote{Existence of
  homogeneous polynomial solutions for parameter-dependent linear matrix
  inequalities with parameters in the simplex}, in \enquote{IEEE Conference on
  Decision and Control}, pp. 1486--1491 (2006).

\bibitem[Blondel and Tsitsiklis(2000)]{blondel2000survey}
Blondel, V. and J.~Tsitsiklis, \enquote{A survey of computational complexity
  results in systems and control}, Automatica \textbf{36}, 9, 1249--1274
  (2000).

\bibitem[Blondel and Tsitsiklis(1999)]{blondel1999complexity}
Blondel, V.~D. and J.~N. Tsitsiklis, \enquote{Complexity of stability and
  controllability of elementary hybrid systems}, Automatica \textbf{35}, 3,
  479--489 (1999).

\bibitem[Borchers and Young(2007)]{csdp}
Borchers, B. and J.~G. Young, \enquote{Implementation of a primal dual method
  for {SDP} on a shared memory parallel architecture}, Computational
  Optimization and Applications \textbf{37}, 3, 355--369 (2007).

\bibitem[Boudaoud \emph{et~al.}(2008)]{roy}
Boudaoud, F., F.~Caruso and M.~Roy, \enquote{Certificates of positivity in the
  bernstein basis}, Discrete and Computational Geometry \textbf{39}, 4,
  639--655 (2008).

\bibitem[Boukas(2006)]{boukas2006static}
Boukas, E.~K., \enquote{Static output feedback control for stochastic hybrid
  systems: {LMI} approach}, Automatica \textbf{42}, 1, 183--188 (2006).

\bibitem[Boyd and Vandenberghe(2004)]{boyd2004convex}
Boyd, S. and L.~Vandenberghe, \emph{Convex optimization} (Cambridge university
  press, 2004).

\bibitem[Braun(1990)]{Braun_complex_storage}
Braun, J.~E., \enquote{Reducing energy costs and peak electrical demand through
  optimal control of building thermal storage}, ASHRAE transactions
  \textbf{96}, 2, 876--888 (1990).

\bibitem[Braun(2003)]{experiment2}
Braun, J.~E., \enquote{Load control using building thermal mass}, Journal of
  solar energy engineering \textbf{125}, 3, 292--301 (2003).

\bibitem[Braun \emph{et~al.}(2002)]{experiment1}
Braun, J.~E., T.~Lawrence, C.~Klaassen and J.~House, \enquote{Demonstration of
  load shifting and peak load reduction with control of building thermal mass},
  Teaming for Efficiency: Commercial buildings: technologies, design,
  performance analysis, and building industry trends \textbf{3}, 55 (2002).

\bibitem[Braun and Lee(2006)]{Braun_2006}
Braun, J.~E. and K.~H. Lee, \enquote{Assessment of demand limiting using
  building thermal mass in small commercial buildings}, Transactions on
  American Society of Heating, Refrigerating and Air-Conditioning Engineers
  \textbf{112}, 1, 547--558 (2006).

\bibitem[Braun \emph{et~al.}(2001)]{simulation1}
Braun, J.~E., K.~W. Montgomery and N.~Chaturvedi, \enquote{Evaluating the
  performance of building thermal mass control strategies}, HVAC\&R Research
  \textbf{7}, 4, 403--428 (2001).

\bibitem[Briat(2013)]{briat}
Briat, C., \enquote{Robust stability and stabilization of uncertain linear
  positive systems via integral linear constraints: {L}1-gain and {L}2-gain
  characterization}, International Journal of Robust and Nonlinear Control
  \textbf{23}, 17, 1932--1954 (2013).

\bibitem[Brown(2003)]{QEPCAD}
Brown, C., \enquote{{QEPCAD B}: a program for computing with semi-algebraic
  sets using {CAD}s}, ACM SIGSAM Bulletin \textbf{37}, 4, 97--108 (2003).

\bibitem[Castle \emph{et~al.}(2011)]{polya_edge}
Castle, M., V.~Powers and B.~Reznick, \enquote{Polya's theorem with zeros},
  Journal of Symbolic Computation \textbf{46}, 9, 1039--1048 (2011).

\bibitem[Chang and Wah(1994)]{poly_groebner}
Chang, Y. and B.~Wah, \enquote{Polynomial programming using {G}roebner bases},
  IEEE Computer Software and Applications Conference pp. 236--241 (1994).

\bibitem[Chen(2001)]{chen2001real}
Chen, T.~Y., \enquote{Real-time predictive supervisory operation of building
  thermal systems with thermal mass}, Journal of Energy and Buildings
  \textbf{33}, 2, 141--150 (2001).

\bibitem[Chesi(2005)]{chesi_hypercube_2005}
Chesi, G., \enquote{Establishing stability and instability of matrix
  hypercubes}, System and Control Letters \textbf{54}, 381--388 (2005).

\bibitem[Chesi \emph{et~al.}(2005)]{chesi2005polynomially}
Chesi, G., A.~Garulli, A.~Tesi and A.~Vicino, \enquote{Polynomially
  parameter-dependent lyapunov functions for robust stability of polytopic
  systems: an lmi approach}, IEEE Transactions on Automatic Control
  \textbf{50}, 3, 365--370 (2005).

\bibitem[Collins and Hoon(1991)]{CAD}
Collins, G. and H.~Hoon, \enquote{Partial cylindrical algebraic decomposition
  for quantifier elimination}, Journal of Symbolic Computation \textbf{12}, 3,
  299--328 (1991).

\bibitem[de~Loera and Santos(1996)]{polya_Rn}
de~Loera, J. and F.~Santos, \enquote{An effective version of polya's theorem on
  positive definite forms}, Journal of Pure and Applied Algebra \textbf{108},
  3, 231--240 (1996).

\bibitem[Deitz(2005)]{deitz2005high}
Deitz, S., \emph{High-level programming language abstractions for advanced and
  dynamic parallel computations}, Ph.D. thesis, Computer Science and
  Engineering Department, University of Washington (2005).

\bibitem[Delzell(2008)]{polya_rational}
Delzell, C., \enquote{Impossibility of extending polya's theorem to forms with
  arbitrary real exponents}, Journal of Pure and Applied Algebra \textbf{212},
  12, 2612--2622 (2008).

\bibitem[Dennis~Jr and Schnabel(1996)]{dennis1996numerical}
Dennis~Jr, J.~E. and R.~B. Schnabel, \emph{Numerical methods for unconstrained
  optimization and nonlinear equations}, vol.~16 (SIAM, 1996).

\bibitem[Dickinson and Pohv(2014)]{polya_positivstellensatz_2014}
Dickinson, P. and J.~Pohv, \enquote{On an extension of polya's
  positivstellensatz}, Journal of Global Optimization pp. 1--11 (2014).

\bibitem[Dolzmann and Sturm(1997)]{Redlog}
Dolzmann, A. and T.~Sturm, \enquote{Redlog: Computer algebra meets computer
  logic}, ACM SIGSAM Bulletin \textbf{31}, 2, 2--9 (1997).

\bibitem[Dullerud and Paganini(2000)]{dullerud2000course}
Dullerud, G. and F.~Paganini, \emph{A Course in Robust Control Theory, A Convex
  Approach} (Springer-Verlag New York, 2000).

\bibitem[EPRI-DOE(2003)]{battery_usage}
EPRI-DOE, \enquote{Handbook of energy storage for transmission and distribution
  applications}, 1001834, EPRI, Palo Alto, CA, and the U.S. Department of
  Energy, Washington, DC  (2003).

\bibitem[F.~Alizadeh(1994)]{Alizadeh_method}
F.~Alizadeh, M. L.~O., J. P. A.~Haeberly, \enquote{Primal-dual interior-point
  methods for semidefinite programming}, Math Programming Symposium  (Ann Arbor
  1994).

\bibitem[Farin(2002)]{farin2002curves}
Farin, G.~E., \emph{Curves and surfaces for CAGD: a practical guide} (Morgan
  Kaufmann, 2002).

\bibitem[G.~Blekherman and Thomas(2013)]{SDP_SIAMbook_parrilo}
G.~Blekherman, P.~P. and R.~Thomas, \emph{Semidefinite optimization and convex
  algebraic geometry} (MOS-SIAM Series on Optimization, Philadelphia, 2013).

\bibitem[G.~Chesi and Vicino(2005)]{chesi_2005}
G.~Chesi, A.~T., A.~Garulli and A.~Vicino, \enquote{Lmi-based computation of
  optimal quadratic lyapunov functions for odd polynomial systems},
  International Journal of Robust and Nonlinear Control \textbf{15}, 1, 35--49
  (2005).

\bibitem[G.~Hardy and Polya(1934)]{inequalities}
G.~Hardy, J.~L. and G.~Polya, \emph{Inequalities} (Cambridge University Press,
  1934).

\bibitem[Gahinet \emph{et~al.}(1996)]{affine_dependent0}
Gahinet, P., P.~Apkarian and M.~Chilali, \enquote{Affine parameter-dependent
  lyapunov functions and real parametric uncertainty}, IEEE Transactions on
  Automatic Control \textbf{41}, 3, 436--442 (1996).

\bibitem[Gatermann and Parrilo(2004)]{parrilo_sym}
Gatermann, K. and P.~Parrilo, \enquote{Symmetry groups, semidefinite programs,
  and sums of squares}, Journal of Pure and Applied Algebra \textbf{192}, 1,
  95--128 (2004).

\bibitem[Geromel and de~Oliveira(2001)]{H2_Hinf_filtering}
Geromel, J. and M.~de~Oliveira, \enquote{${H}_2$ and ${H}_{\infty}$; robust
  filtering for convex bounded uncertain systems}, IEEE Transactions on
  Automatic Control \textbf{46}, 1, 100--107 (2001).

\bibitem[Green and Limebeer(1995)]{green_1994}
Green, M. and D.~J.~N. Limebeer, \emph{Linear robust control} (Prentice-Hall,
  Inc., Upper Saddle River, NJ, USA, 1995).

\bibitem[Greenlaw \emph{et~al.}(1995)]{limits}
Greenlaw, R., H.~Hoover and W.~Ruzzo, \emph{Limits to parallel computation:
  P-completeness theory} (Oxford University Press, USA, 1995).

\bibitem[Gripenberg(1996)]{gripenberg1996computing}
Gripenberg, G., \enquote{Computing the joint spectral radius}, Linear Algebra
  and its Applications \textbf{234}, 43--60 (1996).

\bibitem[Gugercin and Antoulas(2004)]{gugercin2004survey}
Gugercin, S. and A.~Antoulas, \enquote{A survey of model reduction by balanced
  truncation and some new results}, International Journal of Control
  \textbf{77}, 8, 748--766 (2004).

\bibitem[Habicht(1939)]{habicht}
Habicht, W., \enquote{Uber die zerlegung strikte definiter formen in quadrate},
  Commentarii Mathematici Helvetici \textbf{12}, 1, 317--322 (1939).

\bibitem[Handelman(1988{\natexlab{a}})]{handelman_1988}
Handelman, D., \enquote{Representing polynomials by positive linear functions
  on compact convex polyhedra}, Pacific Journal of Mathematics \textbf{132}, 1,
  35--62 (1988{\natexlab{a}}).

\bibitem[Handelman(1988{\natexlab{b}})]{handelman1988}
Handelman, D., \enquote{Representing polynomials by positive linear functions
  on compact convex polyhedra}, Pac. J. Math \textbf{132}, 1, 35--62
  (1988{\natexlab{b}}).

\bibitem[Hardy \emph{et~al.}(1934)]{polya_book}
Hardy, G., J.~E. Littlewood and G.~P{\'o}lya, \emph{Inequalities} (Cambridge
  University Press, 1934).

\bibitem[Helmberg and Rendl(2000)]{helmberg2000spectral}
Helmberg, C. and F.~Rendl, \enquote{A spectral bundle method for semidefinite
  programming}, SIAM Journal on Optimization \textbf{10}, 3, 673--696 (2000).

\bibitem[Helmberg \emph{et~al.}(1996)]{helmberg}
Helmberg, C., F.~Rendl, R.~J. Vanderbei and H.~Wolkowicz, \enquote{An
  interior-point method for semidefinite programming}, SIAM Journal of
  Optimization \textbf{6}, 2, 342--361 (1996).

\bibitem[Helmberg \emph{et~al.}(2005)]{helmberg2005interior}
Helmberg, C., F.~Rendl, R.~J. Vanderbei and H.~Wolkowicz, \enquote{An
  interior-point method for semidefinite programming}, Princeton University.
  Princeton, NJ, USA  (2005).

\bibitem[Henze \emph{et~al.}(2004)]{henze2004evaluation}
Henze, G., C.~Felsmann and G.~Knabe, \enquote{Evaluation of optimal control for
  active and passive building thermal storage}, International Journal of
  Thermal Sciences \textbf{43}, 2, 173--183 (2004).

\bibitem[Hilbert(1888)]{hilbert1}
Hilbert, D., \enquote{Uber die darstellung definiter formen als summe von
  formen quadratens}, Math. Ann. \textbf{32}, 342--350 (1888).

\bibitem[Hilbert(1893)]{hilbert2}
Hilbert, D., \enquote{Uber ternare definite formen}, Acta Math. \textbf{17},
  169--197 (1893).

\bibitem[Jeyakumar \emph{et~al.}(2014)]{lasserre_noncompact2014}
Jeyakumar, V., J.-B. Lasserre and G.~Li, \enquote{On polynomial optimization
  over non-compact semi-algebraic sets}, Journal of Optimization Theory and
  Applications \textbf{163}, 3, 707--718 (2014).

\bibitem[Kal{\'e} \emph{et~al.}(1994)]{kale1994charm}
Kal{\'e}, L., B.~Ramkumar, A.~Sinha and A.~Gursoy, \enquote{The {CHARM}
  parallel programming language and system: Part {I}--description of language
  features}, Parallel Programming Laboratory Technical Report No. 95-02
  (1994).

\bibitem[Kamyar and Peet(2012{\natexlab{a}})]{kamyar_ACC2012}
Kamyar, R. and M.~Peet, \enquote{Decentralized computation for robust stability
  analysis of large state-space systems using polya's theorem}, American
  Control Conference  (2012{\natexlab{a}}).

\bibitem[Kamyar and Peet(2012{\natexlab{b}})]{kamyar_CDC2012}
Kamyar, R. and M.~Peet, \enquote{Decentralized computation for robust stability
  of large-scale systems with parameters on the hypercube}, IEEE 51st
  Conference on Decision and Control pp. 6259--6264 (2012{\natexlab{b}}).

\bibitem[Kamyar and Peet(2013)]{kamyar_CDC2013}
Kamyar, R. and M.~Peet, \enquote{Decentralized polya's algorithm for stability
  analysis of large-scale nonlinear systems}, IEEE Conference on Decision and
  Control pp. 5858--5863 (2013).

\bibitem[Kamyar \emph{et~al.}(2013)]{kamyar_TAC2013}
Kamyar, R., M.~Peet and Y.~Peet, \enquote{Solving large-scale robust stability
  problems by exploiting the parallel structure of polya's theorem}, IEEE
  Transactions on Automatic Control \textbf{58}, 8, 1931--1947 (2013).

\bibitem[Keeney and Braun(1997)]{simulation2}
Keeney, K. and J.~E. Braun, \enquote{Application of building precooling to
  reduce peak cooling requirements}, ASHRAE transactions \textbf{103}, 1,
  463--469 (1997).

\bibitem[Kim \emph{et~al.}(2005)]{kim2005generalized}
Kim, S., M.~Kojima and H.~Waki, \enquote{Generalized lagrangian duals and sums
  of squares relaxations of sparse polynomial optimization problems}, SIAM
  Journal on Optimization \textbf{15}, 3, 697--719 (2005).

\bibitem[Kintner-Meyer and Emery(1995)]{kintner1995optimal}
Kintner-Meyer, M. and A.~F. Emery, \enquote{Optimal control of an {HVAC} system
  using cold storage and building thermal capacitance}, Energy and Buildings
  \textbf{23}, 1, 19--31 (1995).

\bibitem[Kuhn \emph{et~al.}(1951)]{KKT}
Kuhn, H.~W., A.~W. Tucker and J.~Neyman, \enquote{Nonlinear programming},
  Proceedings of the Second Berkeley Symposium on Mathematical Statistics and
  Probability pp. 481--492 (1951).

\bibitem[L.~Blum and Smale(1998)]{blum}
L.~Blum, M.~S., F.~Cucker and S.~Smale, \emph{Complexity and real computation}
  (Springer-Verlag, New York, 1998).

\bibitem[Lasserre(2001)]{lasserre2001global}
Lasserre, J.~B., \enquote{Global optimization with polynomials and the problem
  of moments}, SIAM Journal on Optimization \textbf{11}, 3, 796--817 (2001).

\bibitem[Laurent(2009)]{laurent_survey}
Laurent, M., \enquote{Sums of squares, moment matrices and optimization over
  polynomials}, Springer: The IMA Volumes in Mathematics and its Applications
  \textbf{149}, 157--270 (2009).

\bibitem[Lavaei and Aghdam(2008)]{lavaei2008robust}
Lavaei, J. and A.~G. Aghdam, \enquote{Robust stability of {LTI} systems over
  semialgebraic sets using sum-of-squares matrix polynomials}, IEEE
  Transactions on Automatic Control \textbf{53}, 1, 417--423 (2008).

\bibitem[Leroy(2012)]{leroy}
Leroy, R., \enquote{Convergence under subdivision and complexity of polynomial
  minimization in the simplicial bernstein basis}, Reliable Computing
  \textbf{17}, 11--21 (2012).

\bibitem[Li and Fu(1997)]{li1997linear}
Li, H. and M.~Fu, \enquote{A linear matrix inequality approach to robust
  ${H}_{\infty}$ filtering}, IEEE Transactions on Signal Processing,
  \textbf{45}, 9, 2338--2350 (1997).

\bibitem[Li \emph{et~al.}(2011)]{storage_benefit}
Li, N., L.~Chen and S.~Low, \enquote{Optimal demand response based on utility
  maximization in power networks}, Proceedings of IEEE Power and Energy Society
  General Meeting pp. 1--8 (2011).

\bibitem[Lombardi(1991)]{lombardi1991effective}
Lombardi, H., \enquote{Effective real nullstellensatz and variants}, in
  \enquote{Effective Methods in Algebraic Geometry}, pp. 263--288 (Springer,
  1991).

\bibitem[Lu \emph{et~al.}(2005)]{lu2005global2}
Lu, L., W.~Cai, Y.~S. Chai and L.~Xie, \enquote{Global optimization for overall
  {HVAC} systems - {P}art {II} problem solution and simulations}, Energy
  Conversion and Management \textbf{46}, 7, 1015--1028 (2005).

\bibitem[Ma \emph{et~al.}(2014)]{Topcu}
Ma, W.-J., V.~Gupta and U.~Topcu, \enquote{On distributed charging control of
  electric vehicle with power network capacity constraints}, in
  \enquote{American Control Conference, Portland},  (2014).

\bibitem[Monteiro(1997)]{monteiro}
Monteiro, R., \enquote{Primal-dual path-following algorithms for semidefinite
  programming}, SIAM Journal of Optimization \textbf{7}, 3, 663--678 (1997).

\bibitem[Motzkin(1967)]{motzkin}
Motzkin, T., \enquote{The arithmetic-geometric inequality}, Symposium on
  Inequalities, Academic Press pp. 205--224 (1967).

\bibitem[Nayakkankuppam(2007)]{mahdu}
Nayakkankuppam, M., \enquote{Solving large-scale semidefinite programs in
  parallel}, Mathematical programming \textbf{109}, 2, 477--504 (2007).

\bibitem[Nemirovskii(1993)]{nemirovskii1993several}
Nemirovskii, A., \enquote{Several {NP}-hard problems arising in robust
  stability analysis}, Mathematics of Control, Signals and Systems \textbf{6},
  2, 99--105 (1993).

\bibitem[Oldewurtel and Morari(2010)]{old_2010}
Oldewurtel, F. and M.~Morari, \enquote{Reducing peak electricity demand in
  building climate control using real-time pricing and model predictive
  control}, IEEE Conf. on Decision and Control pp. 1927--1932 (2010).

\bibitem[Oliveira \emph{et~al.}(2008)]{peres_multisimplex}
Oliveira, R., P.~Bliman and P.~Peres, \enquote{Robust {LMI}s with parameters in
  multi-simplex: Existence of solutions and applications}, IEEE 47th Conference
  on Decision and Control pp. 2226--2231 (2008).

\bibitem[Oliveira and Peres(2007)]{peres2007}
Oliveira, R.~C. and P.~L. Peres, \enquote{Parameter-dependent {LMI}s in robust
  analysis: characterization of homogeneous polynomially parameter-dependent
  solutions via {LMI} relaxations}, IEEE Transactions on Automatic Control
  \textbf{52}, 7, 1334--1340 (2007).

\bibitem[Oliveira and Peres(2001)]{affine_dependent2}
Oliveira, R. C. L.~F. and P.~L.~D. Peres, \enquote{A less conservative {LMI}
  condition for the robust stability of discrete-time uncertain systems}, Syst.
  Control Lett. \textbf{43}, 4, 371--378 (2001).

\bibitem[Oliveira and Peres(2005)]{affine_dependent1}
Oliveira, R. C. L.~F. and P.~L.~D. Peres, \enquote{Stability of polytopes of
  matrices via affine parameter-dependent {Lyapunov} functions: Asymptotically
  exact {LMI} conditions}, Linear Algebra Appl. \textbf{405}, 3, 209--228
  (2005).

\bibitem[Olsson and Nelson(1975)]{olsson_1975}
Olsson, D.~M. and L.~S. Nelson, \enquote{The {N}elder-{M}ead simplex procedure
  for function minimization}, Technometrics \textbf{17}, 45--51 (1975).

\bibitem[P.~Gahinet(1994)]{Gahinet_1994_LMI_hinf}
P.~Gahinet, P.~A., \enquote{A linear matrix inequality approach to h infinity
  control}, International Journal of Robust and Nonlinear Control \textbf{4},
  4, 421--448 (1994).

\bibitem[Packard and Doyle(1990)]{quadratic1}
Packard, A. and J.~Doyle, \enquote{Quadratic stability with real and complex
  perturbations}, IEEE Transactions on Automatic Control \textbf{35}, 2,
  198--201 (1990).

\bibitem[Papachristodoulou \emph{et~al.}(2013)]{sostools2013}
Papachristodoulou, A., J.~Anderson, G.~Valmorbida, S.~Prajna, P.~Seiler and
  P.~A. Parrilo, \enquote{{SOSTOOLS}: Sum of squares optimization toolbox for
  {MATLAB}}, preprint, arXiv:1310.4716  (2013).

\bibitem[Papachristodoulou \emph{et~al.}(2009)]{papa2009analysis}
Papachristodoulou, A., M.~M. Peet and S.~Lall, \enquote{Analysis of polynomial
  systems with time delays via the sum of squares decomposition}, IEEE
  Transactions on Automatic Control \textbf{54}, 5, 1058--1064 (2009).

\bibitem[Papchristodoulou and Prajna(2009)]{papchristodoulou2009robust}
Papchristodoulou, A. and S.~Prajna, \enquote{Robust stability analysis of
  nonlinear hybrid systems}, IEEE Transactions on Automatic Control
  \textbf{54}, 5, 1034--1041 (2009).

\bibitem[Parrilo(2000)]{parillo_thesis}
Parrilo, P., \enquote{Structured semidefinite programs and semialgebraic
  geometry methods in robustness and optimization}, Ph.D thesis, California
  Institute of Technology  (2000).

\bibitem[Parrilo(2005)]{parrilo_sructure}
Parrilo, P., \enquote{Exploiting algebraic structure in sum of squares
  programs}, Positive polynomials in control pp. 580--580 (2005).

\bibitem[Parrilo(2006)]{parrilo2006polynomial}
Parrilo, P.~A., \enquote{Polynomial games and sum of squares optimization}, in
  \enquote{Decision and Control, 2006 45th IEEE Conference on}, pp. 2855--2860
  (IEEE, 2006).

\bibitem[Patterson and Rao(2013)]{gpops}
Patterson, M.~A. and A.~V. Rao, \enquote{{GPOPS}- {II}: A {MATLAB} software for
  solving multiple-phase optimal control problems}, ACM Transactions on
  Mathematical Software \textbf{39}, 3, 1--41 (2013).

\bibitem[Peaucelle and Arzelier(2001)]{robust_performanceTAC2001}
Peaucelle, D. and D.~Arzelier, \enquote{Robust performance analysis with
  lmi-based methods for real parametric uncertainty via parameter-dependent
  lyapunov functions}, IEEE Transactions on Automatic Control \textbf{46}, 4,
  624--630 (2001).

\bibitem[Peet(2009)]{peet2009exponentially}
Peet, M., \enquote{Exponentially stable nonlinear systems have polynomial
  lyapunov functions on bounded regions}, Automatic Control, IEEE Transactions
  on \textbf{54}, 5, 979--987 (2009).

\bibitem[Peet and Peet(2010)]{peet_acc}
Peet, M.~M. and Y.~V. Peet, \enquote{A parallel-computing solution for
  optimization of polynomials},  (Proceedings of the American Control
  Conference, 2010).

\bibitem[Permenter and Parrilo(2012)]{permenter2012selecting}
Permenter, F. and P.~A. Parrilo, \enquote{Selecting a monomial basis for sums
  of squares programming over a quotient ring.}, in \enquote{CDC}, pp.
  1871--1876 (2012).

\bibitem[Powers and Reznick(2000)]{Reznick_powers_interval2000}
Powers, V. and B.~Reznick, \enquote{Polynomials that are positive on an
  interval}, Transactions of the American Mathematical Society \textbf{352},
  10, 4677--4692 (2000).

\bibitem[Powers and Reznick(2001)]{reznick_powers_polyhedra}
Powers, V. and B.~Reznick, \enquote{A new bound for polya's theorem with
  applications to polynomials positive on polyhedra}, Journal of Pure and
  Applied Algebra \textbf{164}, 221--229 (2001).

\bibitem[Powers and Reznick(2006)]{polya_corner}
Powers, V. and B.~Reznick, \enquote{A quantitative polya's theorem with corner
  zeros}, ACM International Symposium on Symbolic and Algebraic Computation
  (2006).

\bibitem[Prajna and Papachristodoulou(2003)]{prajna_hybrid}
Prajna, S. and A.~Papachristodoulou, \enquote{Analysis of switched and hybrid
  systems-beyond piecewise quadratic methods}, in \enquote{Proceedings of the
  2003 American Control Conference}, vol.~4, pp. 2779--2784 (2003).

\bibitem[Prestel and Delzell(2004)]{delzell_book}
Prestel, A. and C.~Delzell, \emph{Positive polynomials: from Hilbert’s 17th
  problem to real algebra} (Springer New York, 2004).

\bibitem[Putinar(1993)]{putinar}
Putinar, M., \enquote{Positive polynomials on compact semi-algebraic sets},
  Indiana University Mathematics Journal \textbf{42}, 969--984 (1993).

\bibitem[Ramos and Peres(2001)]{affine_dependent3}
Ramos, D. and P.~Peres, \enquote{An {LMI} approach to compute robust stability
  domains for uncertain linear systems},  (Proceedings of the American Control
  Conference, 2001).

\bibitem[Ramshaw(1987)]{blossoming}
Ramshaw, L., \emph{A connect-the-dots approach to splines} (Digital Systems
  Research Center, 1987).

\bibitem[Repin(1965)]{repin1965quadratic}
Repin, I.~M., \enquote{Quadratic liapunov functionals for systems with delay},
  Journal of Applied Mathematics and Mechanics \textbf{29}, 3, 669--672 (1965).

\bibitem[Reznick(2000)]{Hil17_reznick}
Reznick, B., \enquote{Some concrete aspects of hilbert's 17th problem},
  Contemporary Mathematics \textbf{253}, 251--272 (2000).

\bibitem[Reznick(2005)]{reznick_no_denominator}
Reznick, B., \enquote{On the absence of uniform denominators in hilbert’s
  17th problem}, Proceedings of the American Mathematical Society \textbf{133},
  10, 2829--2834 (2005).

\bibitem[Rumolo(2013)]{APS_plan}
Rumolo, D., \enquote{{APS} rate schedule {ECT}-2 residential service},
  available at: \url{http://votesolar.org/wp-content/uploads/2013/07/ECT-2.pdf}
   (2013).

\bibitem[S.~Sankaranarayanan and Abrahám(2013)]{Handelman_Sankaranarayanan}
S.~Sankaranarayanan, X.~C. and E.~Abrahám, \enquote{Lyapunov function
  synthesis using handelman representations}, The 9th IFAC Symposium on
  Nonlinear Control Systems  (2013).

\bibitem[Sankaranarayanan \emph{et~al.}(2013)]{sankaranarayanan2013lyapunov}
Sankaranarayanan, S., X.~Chen and E.~Abrah{\'a}m, \enquote{Lyapunov function
  synthesis using handelman representations}, in \enquote{The 9th IFAC
  Symposium on Nonlinear Control Systems}, pp. 576--581 (2013).

\bibitem[Sassi \emph{et~al.}(2012)]{sassi2012reachability}
Sassi, M. A.~B., R.~Testylier, T.~Dang and A.~Girard, \enquote{Reachability
  analysis of polynomial systems using linear programming relaxations},  pp.
  137--151 (2012).

\bibitem[Sassi and Girard(2012)]{girard_blossom}
Sassi, M.~B. and A.~Girard, \enquote{Computation of polytopic invariants for
  polynomial dynamical systems using linear programming}, Automatica
  \textbf{48}, 12, 3114--3121 (2012).

\bibitem[Scheiderer(2006)]{scheiderer2006sums}
Scheiderer, C., \enquote{Sums of squares on real algebraic surfaces}, Springer:
  manuscripta mathematica \textbf{119}, 4, 395--410 (2006).

\bibitem[Scheiderer(2009)]{scheiderer_survey}
Scheiderer, C., \enquote{Positivity and sums of squares: a guide to recent
  results}, Emerging applications of algebraic geometry, Springer New York pp.
  271--324 (2009).

\bibitem[Scherer(2006)]{scherer2006lmi}
Scherer, C.~W., \enquote{{LMI} relaxations in robust control}, European Journal
  of Control \textbf{12}, 1, 3--29 (2006).

\bibitem[Scherer and Hol(2006)]{sos3}
Scherer, C.~W. and C.~W.~J. Hol, \enquote{Matrix sum-of squares relaxations for
  robust semi-definite programs}, Math. programming Ser. B \textbf{107}, 1-2,
  189--211 (2006).

\bibitem[Schmudgen(1991)]{schmudgen}
Schmudgen, M., \enquote{The k-moment problem for compact semi-algebraic sets},
  Mathematische Annalen \textbf{289}, 203--206 (1991).

\bibitem[Schweighofer(2005)]{schweighofer_nonnegativity}
Schweighofer, M., \enquote{Certificates for nonnegativity of polynomials with
  zeros on compact semialgebraic sets}, manuscripta mathematica \textbf{117},
  4, 407--428 (2005).

\bibitem[Shear(2014)]{peak_to_average}
Shear, T., \enquote{Today in energy: Peak-to-average electricity demand ratio
  rising in {N}ew {E}ngland and many other {U.S.} regions}, US Energy
  Information Administration (EIA), Independent statistics and Analysis
  (2014).

\bibitem[Sherali and Adams(1990)]{sherali_1990}
Sherali, H. and W.~Adams, \enquote{A hierarchy of relaxations between the
  continuous and convex hull representations for zero-one programming
  problems}, SIAM Journal on Discrete Mathematics \textbf{3}, 3, 411--430
  (1990).

\bibitem[Sherali and Liberti(2009)]{sherali_global_2007}
Sherali, H. and L.~Liberti, \emph{Reformulation-linearization technique for
  global optimization} (Encyclopedia of Optimization, Springer, USA, 2009).

\bibitem[Sherali and Tuncbilek(1992)]{sherali_1992}
Sherali, H. and C.~Tuncbilek, \enquote{A global optimization algorithm for
  polynomial programming problems using a reformulation-linearization
  technique}, Journal of Global Optimization \textbf{2}, 101--112 (1992).

\bibitem[Sherali and Tuncbilek(1997)]{sherali_1997}
Sherali, H. and C.~Tuncbilek, \enquote{New reformulation- linearization
  technique based relaxations for univariate and multivariate polynomial
  programming problems}, Operations Research Letters \textbf{21}, 1, 1--10
  (1997).

\bibitem[Sivaramakrishnan(2010)]{krishnan}
Sivaramakrishnan, K.~K., \enquote{A parallel interior point decomposition
  algorithm for block angular semidefinite programs}, Comput. Optim. Appl.
  \textbf{46}, 1, 1--29 (2010).

\bibitem[Slater(2014)]{slater2014lagrange}
Slater, M., \enquote{Lagrange multipliers revisited}, Springer: Traces and
  Emergence of Nonlinear Programming pp. 293--306 (2014).

\bibitem[Stengle(1974)]{stengle}
Stengle, G., \enquote{A {N}ullstellensatz and a {P}ositivstellensatz in
  semialgebraic geometry}, Mathematische Annalen \textbf{207}, 2, 87--97
  (1974).

\bibitem[Sturm(1999)]{sedumi}
Sturm, J., \enquote{Using {SeDuMi} 1.02, a {MATLAB} toolbox for optimization
  over symmetric cones}, Optimization methods and software \textbf{11}, 1-4,
  625--653 (1999).

\bibitem[Sun \emph{et~al.}(2013)]{sun2013peak}
Sun, Y., S.~Wang, F.~Xiao and D.~Gao, \enquote{Peak load shifting control using
  different cold thermal energy storage facilities in commercial buildings: a
  review}, Energy Conversion and Management \textbf{71}, 101--114 (2013).

\bibitem[Tan and Packard(2008)]{packard_ROA}
Tan, W. and A.~Packard, \enquote{Stability region analysis using polynomial and
  composite polynomial lyapunov functions and sum-of-squares programming}, IEEE
  Transactions on Automatic Control \textbf{53}, 2, 565--571 (2008).

\bibitem[Tarski(1951)]{tarski}
Tarski, A., \emph{A Decision Method for Elementary Algebra and Geometry}
  (Random Corporation monograph, Berekley and Los Angeles, 1951).

\bibitem[Topcu \emph{et~al.}(2010)]{topcu2010robust}
Topcu, U., A.~K. Packard, P.~Seiler and G.~J. Balas, \enquote{Robust
  region-of-attraction estimation}, IEEE Transactions on Automatic Control
  \textbf{55}, 1, 137--142 (2010).

\bibitem[Tutuncu \emph{et~al.}(2003)]{sdpt3}
Tutuncu, R., K.~Toh and M.~Todd, \enquote{Solving semidefinite-quadratic-linear
  programs using {SDPT3}, mathematical programming}, Mathematical Programming
  Series B \textbf{95}, 189--217 (2003).

\bibitem[Waki \emph{et~al.}(2008)]{waki}
Waki, H., S.~Kim, M.~Kojima, M.~Muramatsu and H.~Sugimoto, \enquote{Algorithm
  883: Sparse{POP}---a sparse semidefinite programming relaxation of polynomial
  optimization problems}, ACM Trans. Math. Softw. \textbf{35}, 2 (2008).

\bibitem[Walker and Dongarra(1996)]{walker1996mpi}
Walker, D. and J.~Dongarra, \enquote{{MPI}: a standard message passing
  interface}, Supercomputer \textbf{12}, 56--68 (1996).

\bibitem[Wang \emph{et~al.}(2005)]{lall_west2005polynomial}
Wang, T.-C., S.~Lall and M.~West, \enquote{Polynomial level-set methods for
  nonlinear dynamical systems analysis}, in \enquote{Allerton conference on
  communication, control and computing}, pp. 640--649 (2005).

\bibitem[Wang \emph{et~al.}(2013)]{wang2013polynomial}
Wang, T.-C., S.~Lall and M.~West, \enquote{Polynomial level-set method for
  polynomial system reachable set estimation}, Automatic Control, IEEE
  Transactions on \textbf{58}, 10, 2508--2521 (2013).

\bibitem[Witrant \emph{et~al.}(2007)]{Tokamak}
Witrant, E., E.~Joffrin, S.~Br{\'e}mont, G.~Giruzzi, D.~Mazon, O.~Barana and
  P.~Moreau, \enquote{A control-oriented model of the current profile in
  tokamak plasma}, Plasma Physics and Controlled Fusion \textbf{49}, 1075--1105
  (2007).

\bibitem[Yamashita \emph{et~al.}(2003)]{sdpara}
Yamashita, M., K.~Fujisawa and M.~Kojima, \enquote{{SDPARA}: Semidefinite
  programming algorithm parallel version}, Parallel Computing \textbf{29},
  1053--1067 (2003).

\bibitem[Yamashita \emph{et~al.}(2010)]{sdpa}
Yamashita, M., K.~Fujisawa and K.~Nakata, \enquote{A high-performance software
  package for semidefinite programs: {SDPA} 7}, Technical report B-460, Dep. of
  Mathematical and Computing Sciences, Tokyo Institute of Technology  (2010).

\bibitem[Zhang and Tsiotras(2003)]{bound1}
Zhang, X. and P.~Tsiotras, \enquote{Parameter-dependent lyapunov functions for
  stability analysis of {LTI} parameter dependent systems}, pp. 5168--5173 (in
  Proceedings of the IEEE 42nd Conference on Decision and Control, 2003).

\bibitem[Zhang \emph{et~al.}(2005)]{bound2}
Zhang, X., P.~Tsiotras and P.~A. Bliman, \enquote{Multi-parameter dependent
  lyapunov functions for the stability analysis of parameter-dependent {LTI}
  systems}, pp. 1263--1268 (in Proceedings of IEEE International Symposium on,
  Mediterrean Conference on Control and Automation, 2005).

\bibitem[Zhou and Doyle(1998)]{zhou_1998}
Zhou, K. and J.~Doyle, \emph{Essentials of Robust Control} (Prentice Hall,
  1998).

\bibitem[Zhou \emph{et~al.}(1996)]{zhou1996robust}
Zhou, K., J.~C. Doyle, K.~Glover \emph{et~al.}, \emph{Robust and optimal
  control}, vol.~40 (Prentice hall New Jersey, 1996).

\end{thebibliography}
\clearpage
\renewcommand{\chaptername}{APPENDIX}
\addtocontents{toc}{APPENDIX \par}
\addcontentsline{toc}{part}{BIOGRAPHICAL}
\include{vita}
\end{document}